\documentclass{amsart}

\usepackage{amssymb}
\usepackage{xcolor}
\usepackage{hyperref}
\usepackage{mathtools}
\usepackage{comment}
\usepackage{enumitem}
\usepackage{tikz}
\usepackage{tikz-cd}
\usepackage{mathpartir}
\usepackage{setspace}
\usepackage{bbm}
\usepackage{dsfont}
 \usepackage{fullpage}
\usepackage{epstopdf}
\usepackage{graphicx}
\usepackage[backend=biber]{biblatex}

\newsavebox{\pullback}
\sbox\pullback{%
\begin{tikzpicture}%
\draw (0,0) -- (1ex,0ex);%
\draw (1ex,0ex) -- (1ex,1ex);%
\end{tikzpicture}}

\newcommand{\id} {\mathrm{id}}
\newcommand{\Id} {\mathrm{Id}}
\newcommand{\gennattrans}{\rightsquigarrow}
\newcommand{\ob} {\mathrm{ob \,}}
\newcommand{\op}{\mathrm{op}}

\newcommand{\dom}{\mathrm{dom}}
\newcommand{\cod}{\mathrm{cod}}
\newcommand{\Dom}{\mathrm{Dom}}
\newcommand{\Cod}{\mathrm{Cod}}
\newcommand{\mor}{\mathrm{mor \,}}
\newcommand{\Hom}{\mathrm{Hom}}
\newcommand{\Loc}{\mathrm{Loc}}
\newcommand{\To}{\Rightarrow}

\newcommand{\nat}{\mathbb{N}}
\newcommand{\pr}[1]{\mathrm{pr}_{#1 } }

\newcommand{\Eq}[1]{\mathrm{Eq}[#1]}
\newcommand{\CoEq}[1]{\mathrm{CoEq}[#1]}

\newcommand{\pHeyt}{\mathsf{Heytpre}}

\newcommand{\im}{\mathrm{im}}
\newcommand{\trans}{\mathrm{transfer}}

\newcommand{\mono}{\hookrightarrow}

\newcommand{\Sq}{\mathrm{Sq}}
\newcommand{\imply}{\Rightarrow}
\newcommand{\const}{\mathrm{const}}
\newcommand{\app}[1]{\mathrm{app}(#1)}
\newcommand{\RT}[1]{\mathrm{RT}[#1]}
\newcommand{\EqRA}[1]{\mathrm{EqRel}(\Asm(#1))}

\newcommand{\nm}[1]{\lceil #1 \rceil}

\newcommand{\Preord}{\mathsf{Preord}}
\newcommand{\Alg}{\mathrm{Alg}}

\newcommand{\pupw}{\pitchfork}
\newcommand{\lift}{\mathrm{lift}}
\newcommand{\pullp}{\mathrm{pullp}}
\newcommand{\htpy}[1]{\partial_{#1}}
\newcommand{\coe}[2]{#1 \downarrow #2}
\newcommand{\I}{\mathbb{I}}
\newcommand{\Grp}{\mathbbm{Grpd}}
\newcommand{\Asm}{\mathrm{Asm}}
\newcommand{\pAsm}{\mathrm{pAsm}}
\newcommand{\pGrA}[1]{\mathrm{p}\mathbbm{Grpd}(\mathrm{Asm}(#1))}

\newcommand{\GrpA}[1]{\Grp(\Asm(#1))}
\newcommand{\Fib}{\mathrm{Fib}}
\newcommand{\Top}{\mathrm{Top}}
\newcommand{\Bot}{\mathrm{Bot}}

\newcommand{\arr}{\mathbbm{2}}
\newcommand{\twarr}{\mathrm{3}}

\newcommand{\epi}{\twoheadrightarrow}
\newcommand{\lmbd}[2]{\langle #1 \rangle #2 }
\newcommand{\ifthen}[3]{\text{if }#1\text{ then }#2\text{ else }#3}
\newcommand{\bimp}{\Leftrightarrow}
\newcommand{\pow}[1]{\mathcal{P}(#1)}
\newcommand{\powf}[1]{\mathcal{P^{*}}(#1)}
\newcommand{\poly}{\mathcal{P}}

\newcommand{\SubTree}{\mathrm{SubTree}}
\newcommand{\DirST}{\mathrm{DirST}}
\newcommand{\Paths}{\mathrm{Paths}}
\newcommand{\adj}{\dashv}

\newcommand{\loc}[2]{#1[#2^{-1}]}
\newcommand{\GE}{\mathrm{GrpdEq}}
\newcommand{\RE}{\mathrm{RegEq}}
\newcommand{\Ho}{\mathrm{Ho}}
\newcommand{\HGA}[1]{\Ho(\GrpA{#1})}
\newcommand{\EGA}[1]{\loc{\GrpA{#1}}{\RE}}
\newcommand{\Part}{\mathrm{Part}}
\newcommand{\prt}[2]{(#1.#2,#2^{*})}

\newcommand{\atom}{\mathrm{atom}}

\newcommand{\Clus}{\mathrm{Clus}}
\newcommand{\incl}{\mathrm{incl}}
\newcommand{\PEq}[1]{\Set\langle #1 \rangle}

\newcommand{\Equ}{\mathrm{Equiv}}
\newcommand{\nno}{\mathbbm{N}}
\newcommand{\fixf}{\text{fixf}}
\newcommand{\fix}{\text{fix}}
\newcommand{\SOA}{\mathrm{SOA}}

\DeclareSymbolFont{bbold}{U}{bbold}{m}{n}
\DeclareSymbolFontAlphabet{\mathbbb}{bbold}

\newcommand{\cat}[1]{\textup{\textsf{#1}}}
\newcommand{\Cat}{\mathbb{C}\cat{at}}

\newcommand{\Set}{\mathrm{Set}}
\newcommand{\Circ}{\mathbb{S}^{1}}

\newcommand{\cC}{{\mathcal{C}}}
\newcommand{\cD}{{\mathcal{D}}}
\newcommand{\cE}{{\mathcal{E}}}

\newtheorem{thm}{Theorem}[subsection]
\newtheorem{cor}[thm]{Corollary}
\newtheorem{prop}[thm]{Proposition}
\newtheorem{lem}[thm]{Lemma}
\newtheorem{conj}[thm]{Conjecture}

\theoremstyle{definition}
\newtheorem{defn}[thm]{Definition}

\newtheorem{rmk}[thm]{Remark}

\newtheorem{ex}[thm]{Example}

\theoremstyle{remark}

\makeatletter
\let\c@equation\c@thm
\makeatother
\numberwithin{equation}{subsection}

\title[ Type Theory in Groupoid Assemblies]{A Model of Type Theory in Groupoid Assemblies}

\author{Anthony Agwu}
\email{aagwu1@jh.edu}
\address{Department of Mathematics, Johns Hopkins University, Baltimore, MD 21218}

\thanks{This work was partially supported by the NSF through grant DMS-2204304 and also the Air Force Office of Scientific Research through grant FA9550-21-1-0009.}

\usepackage{biblatex}
 
\begin{document}

\newpage

\begin{abstract} 
We consider the category $\GrpA{A}$ of groupoids defined internally to the category of assemblies on a partial combinatory algebra $A$. In this thesis we exhibit the structure of a $\pi$-tribe on $\GrpA{A}$ showing the category to be a model of type theory. We also show that $\GrpA{A}$ has $W$-types with reductions and univalent object classifier for assemblies and modest assemblies, where the latter is an impredicative object classifier. Using the $W$-types with reductions, we show that $\GrpA{A}$ has a model structure. Finally, we construct $\pGrA{A}$, the full subcategory of partitioned groupoid assemblies, and show that $\pGrA{A}$ has finite bilimits and bicolimits as well as showing that the homotopy category of the full subcategory of the $0$-types of $\pGrA{A}$ is $\RT{A}$, the realizability topos of $A$.

\end{abstract}

\maketitle

\setcounter{tocdepth}{2}
\tableofcontents

\section{Introduction}

Homotopy Type Theory is an extension of Martin-L\"{o}f type theory jointly discovered by Steve Awodey, Michael Warren, and Vladimir Voevodsky as a system of type theory better suited for reasoning about ``higher-dimensional structures'' and computer verification. Steve Awodey and Michael Warren discovered the notion of paths as an interpretation of paths as proofs of equality in \cite{AWODEY_2009}. Voevodsky was motivated to propose homotopy type theory as a new foundation system after finding mistakes in papers concerning his Fields medal work on motivic cohomology \cite{Voevodsky2014}.\\

One feature of Homotopy Type Theory that is usually distinct from regular Martin-L\"{o}f type theory is the rejection of uniqueness of identity proofs (UIP). In Homotopy Type Theory, for any type $X$ and elements $x,y:X$, the type of identity proofs $(x =_{X} y)$ can have more then one element. Instead the identity type is interpreted as the type of paths from $x$ to $y$ in $X$. There exists a canonical element $\id_{x}: (x =_{X} x)$ which is a proof of the reflexive property of identity. We also have an induction rule for path types:\\

Given a universe $U$ and predicates $P: \Pi_{x,y:X} (x =_{X} y) \to U$ and $i: \Pi_{x:X} P(x,x,\id_{x})$, there is a predicate $f: \Pi_{x,y:X} \Pi_{p:(x = _{X} y)} P(x,y,p)$ such that $f(x,x,\id_{x}) \equiv i(x) $.\\

Consequently, for $p,q: (x =_{X} y)$, there is an iterated identity type $(p =_{(x =_{X} y)} q)$, and this process continues. As a result types in Homotopy Type Theory can be interpreted as ``infinity groupoids". This further induces a shift from the notion of equality as being ``sameness" to ``interchangeability". A basic example is found in algebra. Given two three-element groups: $\{e_{X},x,x^{-1}\}$ and $\{e_{Y},y,y^{-1}\}$, these are different sets but are obviously the ``same" group. This is exhibited by an isomorphism between the groups. It is well understood that the proper notion of ``equality" of groups is that of isomorphism of groups as the  only thing that matters when discussing groups is that of the structure of the groups. In addition, there are two isomorphisms between the three-element group. Hence, equality between groups fails UIP. This understanding of equality is pervasive in category theory, as notions in category that make use of strict equality or ``sameness" are considered ``evil".\\

In light of this, Homotopy Type Theory can be considered a more complete extension of the Curry-Howard correspondence as it considers the type of proofs of equality and accounts for situations discussed in the previous paragraph. Of course there is the question of what are the equality types of universes (i.e. for $A,B:U$, what is $(A =_{U} B)$?). This leads us to the {\bf Univalence Axiom} which states $(A =_{U} B) \simeq (A \simeq B)$ where $(A \simeq B)$ is the type of equivalences between $A$ and $B$. Universes satisfying this axiom are called {\bf univalent}. \\

Realizability entails a study of objects such as partial combinatory algebras which are models of untyped lambda calculus. Realizability is intended to give a semantics for intuitionistic logic and the Curry-Howard correspondence. In particular, realizability comes out of the work of S.C. Kleene \cite{Kleene1973}. In the realizability paradigm, a proposition $\phi$ are true if there exists a term $a  \in A$ which serves as ``evidence" for $\phi$. We call $a$ a {\bf realizer} of $\phi$. The term $a$ is typically represents an encoding of some computation. In the case of Kleene, $A$ is the natural numbers object $\nno$ where each natural number encodes a Turing machines or a partial recursive function on $\nno$. This is known as Kleene's First Algebra.\\

	Thomas Streicher and Martin Hoffman gave an interpretation of type theory in groupoids in their 1996 paper showing the independence of Martin-L\"{o}f type theory with Uniqueness of Identity Proofs \cite{HofmannStreicher1998}. They also exhibit the groupoid universe of sets giving the first example of a univalent universe making this paper a precursor to homotopy type theory. Much later it was shown by Mike Shulman that homotopy type theory can be modeled in any $(\infty,1)$-topos \cite{Shulman2019}. With realizability, there comes a desire to construct an realizability $(\infty,1)$-topos, a place where one can do computable homotopy type theory. It has been show by Taichi Uemura that there are non-trivial models of homotopy type theory and realizability; in particular models that have a univalent and impredicative universe \cite{Uemura2018}. The goal of obtaining an realizability $(\infty,1)$-topos is at least feasible. What this paper hopes to accomplish is to be a stepping stone towards the realizability $(\infty,1)$-topos. To do that, this paper seeks to present a model of Martin-L\"{o}f type theory in groupoid assemblies on a partial combinatory algebra $A$. \\
	
We chose internal groupoids in assemblies since that will not only serve as a update of sorts to the 1996 Streicher-Hoffman paper but it will also provide a glimpse into what to expect for constructing the eventual realizability $(\infty,1)$-topos. Additionally, this may even give insight into questions regarding a proper definition and understanding of what an elementary ``2-topos" might be as well as questions about groupoidal realizability. Finally, one would not be wrong to see this as an extension of the usual 1-categorical setting of realizability into a 2-categorical setting.\\

	Given a partial combinatory algebra $A$, one can construct its category of assemblies $\Asm(A)$ which is a locally cartesian closed regular category with object classifiers. The category of groupoid assemblies on $A$, $\Grp(\Asm(A))$, has finite limits and colimits, exponents, dependent sums and a natural numbers object; however it does not have dependent products. For this, we impose an algebraic weak factorization system on $\Grp(\Asm(A))$ whose right maps are the normal isofibrations. It turns out that dependent products along the normal isofibrations exist and satisfy the Frobenius property. This will make $\Grp(\Asm(A))$ a $\pi$-tribe and thus a model of dependent type theory. $\Grp(\Asm(A))$ will also have $W$-types with reductions which gives us $W$-types, localizations, truncations and factorizations similar to the ones given by the small object argument. From this we will have an algebraic model structure on $\Grp(\Asm(A))$. We will also show how to construct homotopy colimits like the circle and suspensions. $\Grp(\Asm(A))$ will most certainly have univalent universes of assemblies and modest assemblies. The univalent universe of modest assemblies will also be an impredicative universe.\\
	
	We would also like to talk about the $0$-types of $\Grp(\Asm(A))$ which are the internal equivalence relations of $\Asm(A)$. In $\Grp(\Asm(A))$ they aren't of particular note, but if we restrict to $\pGrA{A}$, a full subcategory of groupoid assemblies whose underlying assembly of objects is a partitioned assembly, the $0$-types present a category equivalent to the realizability topos on $A$ by inverting the categorical equivalences. \\

We give an overview of what we will discuss in each section:
\begin{itemize}
\item In Section 2, we will review the notion of partial combinatory algebra as well as that of the realizability topos on $A$ and the category of assemblies. Finally we will introduce our main setting: the category of groupoid assemblies $\Grp(\Asm(A))$.\\

\item In Section 3, we will construct a fibrant algebraic weak factorization system for $\Grp(\Asm(A))$. The right maps of this algebraic weak factorization system are the normal isofibrations. We will also construct dependent products along the normal isofibrations.\\

\item In Section 4, we will construct $W$-types with reductions of groupoid assemblies. As an application, we will construct a version of the small object argument in $\Grp(\Asm(A))$ which will produce algebraic weak factorization systems.\\

\item In Section 5, we will exhibit $\Grp(\Asm(A))$ as a model of type theory as well as construct a univalent universe of assemblies and an impredicative univalent universe of modest assemblies.\\

\item In Section 6, we characterize $\Ho{\Grp(\Asm(A))}$ as well as showing that $\Grp(\Asm(A))$ has finite bilimits and bicolimits. By constructing two biadjunctions one between $\Grp(\Asm(A))$ and $\pGrA{A}$ and the other on $\Grp(\Asm(A))$ we show that $\pGrA{A}$ also has finite bilimits and bicolimits. We also show that the homotopy category of the $0$-types of $\pGrA{A}$ is $\RT{A}$, the realizability topos on $A$.\\

\item Finally, Section 7 will review the key results of this thesis as well as present future problems and questions.\\
\end{itemize}

There is similar work being done in \cite{HC2025} by Calum Hughes where he constructs an algebraic model structure on $\Cat(\cE)$,the category of internal categories of a lextensive cartesian closed category $\cE$. He then restricts the algebraic model structure to $\Grp{\cE}$, the category of internal groupoids of $\cE$, to obtain a model of type theory. There is also the work of Steve Awodey and Jacopo Emmenegger where they construct a potential candidate of the effective $2$-topos \cite{ASEJ2025}.\\

I would like to thank Steve Awodey, Jonas Frey, Andrew Swan, and Emily Riehl for the discussion we had that helped me to formulate the ideas in this thesis.\\




\section{Preliminaries}

In this section we will give an overview an partial combinatory algebras, important combinators, and exhibiting partial combinatory algebras as models of first order logic by constructing a $\Set$-tripos from a partial combinatory algebra. From there will will introduce realizability topos, $\RT{A}$, on a partial combinatory algebra $A$ and the category of assemblies, $\Asm(A)$. The category of assemblies will be our main focus as we will show that $\Asm(A)$ has finite limits and colimits, exponents, dependent products and sums, and $W$-types with reductions as well as showing other important properties of $\Asm(A)$. We will also introduce partitioned assemblies and their important properties which will especially come into play in section 7. Finally, we will introduce the category of groupoid assemblies, $\GrpA{A}$, the main setting of this thesis.\\

\subsection{Partial Combinatory Algebras}
In this section, we will introduce the notion of partial combinatory algebra as well as important combinators that will be used throughout this thesis.

\begin{defn}
A set $A$ with a partial binary operation $\cdot : A \times A \rightharpoonup A$ is a {\bf partial combinatory algebra} if there exists elements $k,s : A$ such that:  

\begin{enumerate}

  \item $(k \cdot x) \cdot y = x$
  \item $((s\cdot a) \cdot b)$ is defined
  \item $((s \cdot a) \cdot b) \cdot c = (a \cdot c) \cdot (b \cdot c)$ when both sides are defined
\end{enumerate}

Elements of $A$ are called {\bf combinators}. One can think of an element $a \in A$ as representing a partial function on $A$. In particular $k$ maps a combinator $a$ to the combinator $k \cdot a$  which represents the constant function on $a$ and $s$ takes combinators $a$ and $b$ to $(s \cdot a) \cdot b$ which can be though of as function application dependent on a combinator $c$.


\end{defn}

Given $a \in A$, we have that $((s \cdot k) \cdot k) \cdot a = (k \cdot a) \cdot (k \cdot a) = a$. We denote $\iota = ((s \cdot k) \cdot k) $ and call $\iota$ the identity combinator.\\

Given a partial combinatory algebra $A$, we can define the set of terms on $A$, $T(A)$, by induction as follows:
\begin{enumerate}

  \item $A \subseteq T(A)$
  \item $V \subseteq T(A)$ where $V$ is a set of variables 
  \item For $a,b \in T(A)$, $a \cdot b \in T(A)$ 

\end{enumerate} 

\begin{defn}

We define a new term $\lmbd{x}{t} \in T(A)$ for every term $t \in T(A)$ and variable $x$:

\begin{enumerate}

  \item $\lmbd{x}{t} = k \cdot t$ when $t \in A$ or $t$ is variable different from $x$
  \item $\lmbd{x}{x} = \iota$
  \item $\lmbd{x}{t_{1} \cdot t_{2}} = (s \cdot (\lmbd{x}{t_{1}})) \cdot (\lmbd{x}{t_{2}})$

\end{enumerate} 

Whenever we have a term $\lmbd{x}{\lmbd{y}{t}}$, we denote it as $\lmbd{x,y}{t}$. This gives us a way to represent lambda terms in a partial combinatory algebra.

\end{defn}

This is useful when defining important terms such as:

\begin{itemize}
\item $\overline{k} = \lmbd{x,y}{y} = k \cdot \iota.$ \\
\item $[-,-] = \lmbd{x,y,z}{((z\cdot x)\cdot y)} \text{: This term is used to create a pair of elements. } [a,b] = \lmbd{z}{((z \cdot a)\cdot b)}$.\\
\item $p_{0} = \lmbd{x}{(x \cdot k)}\text{: This combinator projects the first element of a pair.}$\\
\item $p_{1} = \lmbd{x}{(x \cdot \overline{k})}\text{: This combinator projects the second element of a pair.}$\\
\item $\ifthen{t}{x}{y} = \lmbd{t,x,y}{(((t \cdot \lmbd{z}{x}) \cdot \lmbd{z}{y}) \cdot k )}\text{: This combinator works by setting }k\text{ as true and }\overline{k}\text{ as false.}$\\
\item $\app{f,x} = \lmbd{f,x}{(f \cdot x)}\text{: This is the function application combinator.}$

\end{itemize}

We also have fixed points operators which enables recursive constructions:\\

We first take the realizers $$\kappa_{0} = \lmbd{w,u,v}{u \cdot ((w \cdot w) \cdot u) \cdot v}$$ and $$\kappa_{1} = \lmbd{u,v}{v \cdot (u \cdot u \cdot v)}$$ and define realizers $\fix = \kappa_{1} \cdot \kappa_{1}$ and $\fixf = \kappa_{0} \cdot \kappa_{0}$ so that $$\text{fixf} \cdot f = (\kappa_{0} \cdot \kappa_{0}) \cdot f = \lmbd{v}{f \cdot ((\kappa_{0} \cdot \kappa_{0}) \cdot f) \cdot v} = \lmbd{v}{f \cdot (\text{fixf} \cdot f) \cdot v}$$ and $$\text{fix} \cdot f = (\kappa_{1} \cdot \kappa_{1}) \cdot f = \lmbd{v}{v \cdot (\kappa_{1} \cdot \kappa_{1} \cdot v)} \cdot f = f \cdot (\kappa_{1} \cdot \kappa_{1} \cdot f) = f \cdot (\fix \cdot f)$$ when both sides of both equations are defined.



\subsection{Tripos-to-Topos Construction}
In this section, we will introduce the notion of $\Set$-tripos which is a model of first order logic. From there we will construct a topos from any $\Set$-tripos. For any partial combinatory algebra $A$, we will construct the $\Set$-tripos $\pow{A}^{-}$ and the subsequent realizability topos $\RT{A}$.

\begin{defn}
A set $H$ is called a {\bf Heyting prealgebra} iff it is a cartesian closed preorder with coproducts.

\end{defn} 

$\pHeyt$ is the category of Heyting prealgebras and preorder morphisms that preserve the bicartesian closed structure.
$\pHeyt$ is enriched in $\Preord$ by imposing a pointwise order on morphisms in the set $\pHeyt(G,H)$.

\begin{defn}
A pseudofunctor $P: \Set^{op} \to \pHeyt$ is a $\bf \Set$-{\bf Tripos} when for each function $f:X \to Y$, the morphism $P(f): P(Y) \to P(X)$ has both a left and right adjoint $\exists_{f}, \forall_{f}:P(X) \to P(Y)$ in $\Preord$ satisfying the Beck-Chevalley condition. \\

Furthermore, there exists a object $\Sigma$ and an {\bf generic element} $\sigma \in P(\Sigma)$ s.t. for all sets $X$ and  $\phi \in P(X)$ there exists a function $[\phi]:X \to \Sigma$ s.t. $P([\phi])(\sigma) \cong \phi$.
\end{defn}

\begin{rmk}
If $\cC$ is a cartesian closed category, you get a $\cC${-Tripos} by replacing $\Set$ with $\cC$.
\end{rmk}


\begin{defn}
Given a functor $P: \Set^{op} \to \pHeyt$, the {\bf Beck-Chevalley condition} states that given a pullback square in $\Set$:
\[\begin{tikzcd}
	X &&&& Y \\
	\\
	\\
	Z &&&& W
	\arrow["f", from=1-1, to=1-5]
	\arrow["h", from=1-5, to=4-5]
	\arrow["g"', from=1-1, to=4-1]
	\arrow["k"', from=4-1, to=4-5]
	\arrow["\lrcorner"{anchor=center, pos=0.125}, draw=none, from=1-1, to=4-5]
\end{tikzcd}\]
we have that $\forall_{f} \circ P(g) \cong P(h) \circ \forall_{k}$. As a corollary, we have that $\exists_{g} \circ P(f) \cong P(k) \circ \exists_{h}$.\\

\end{defn}


\begin{lem}

Given a $\Set${-tripos} $P$, a function $f: X \to Y$, and predicates $\phi \in P(X)$ and $\psi \in P(Y)$, we have $\exists_{f}(\phi \land P(f)(\psi)) \bimp (\exists_{f}(\phi)) \land \psi$. We call this the {\bf Frobenius condition}.

\end{lem}
\begin{proof}
We give the argument in  \cite[pg. 60]{vanOosten2008} since it is straightforward. First we observe that $P(f)$ commutes with products and $\psi \leq P(f)(\exists_{f}(\psi))$ is the unit of $\exists_{f} \adj P(f)$ so that we have $\psi \land P(f)(\phi) \leq P(f)(\exists_{f}(\psi)) \land P(f)(\phi) \leq P(f)(\exists_{f}(\psi) \land \phi)$, hence $\exists_{f}(\phi \land P(f)(\psi)) \leq (\exists_{f}(\phi)) \land \psi$.\\

Since $P(f)$ also commutes with implication, we have the following chain of inequalities:
\begin{align*} 
\phi \land P(f)(\psi) \leq P(f)(\exists_{f}(\phi \land P(f)(\psi))) \\
\phi \leq P(f)(\psi) \imply P(f)(\exists_{f}(\phi \land P(f)(\psi))) \\
\phi \leq P(f)(\psi \imply \exists_{f}(\phi \land P(f)(\psi)))\\
\exists_{f}(\phi) \leq \psi \imply \exists_{f}(\phi \land P(f)(\psi))\\
\exists_{f}(\phi)\land \psi \leq \exists_{f}(\phi \land P(f)(\psi))
\end{align*}
where the first inequality is the unit of $\exists_{f} \adj P(f)$. Therefore, $\exists_{f}(\phi \land P(f)(\psi)) \bimp (\exists_{f}(\phi)) \land \psi$.

\end{proof}

Given a partial combinatory algebra $A$, a set $X$, and $\phi_{0}, \phi_{1} \in \pow{A}^{X}$, we say that $\phi_{0} \leq \phi_{1}$ when there exists an element $l \in A$ s.t for all $x \in X$ and $ a \in \phi_{0}(x)$, $l \cdot a$ exists and $l \cdot a \in \phi_{1}(x)$. In this case we say that $l$ realizes $\phi_{0} \leq \phi_{1}$. This induces a preorder structure on $\pow{A}^{X}$. 

\begin{lem} The preorder $\pow{A}^X$ has a bicartesian closed structure, and moreover any function $f \colon X \to Y$ induces an adjoint triple: 
\[ 
\begin{tikzcd}[column sep=large]
\pow{A}^X  \arrow[r, bend left, "{\exists_f}", "\perp"'] \arrow[r, bend right, "{\forall_f}"', "\perp"]  & \pow{A}^Y \arrow[l, "{\pow{A}^f}" description]
\end{tikzcd}
\] 
Furthermore, the adjoint triple satisfies the Beck-Chevalley condition.
\end{lem}

\begin{proof}
We define the finite products, finite coproducts, and exponentials as follows:
 \begin{enumerate}

\item $(\phi \land \psi)(x) = \{[a,b] | a \in \phi(x), b \in \psi(x) \}$
\item $(\phi \lor \psi)(x) = \{[k,a]  | a \in \phi(x)  \} \cup \{[\overline{k}, b]  | b \in \psi(x)  \}$
\item $(\phi \imply \psi)(x) = \{c \in A | \forall a \in \phi(x) \space ~  ~ c \cdot a \in \psi(x)\}$ 
\item $\top = \const_{A}$ 
\item $\bot = \const_{\emptyset}$ 
\end{enumerate}

For products, we have $(\phi \land \psi) \leq \phi$ realized by $p_{0}$ and $(\phi \land \psi) \leq \psi$ realized by $p_{1}$. Suppose we have an element $\alpha \in \pow{A}^{X}$ and realizers $p$ and $q$ respectively for $\alpha \leq \phi$ and $\alpha \leq \psi$, then we have $\alpha \leq (\phi \land \psi) $ realized by the term $\lmbd{y}{[(p \cdot y),(q \cdot y)]}$. For any $\alpha \in \pow{A}^{X}$, $ \alpha \leq \top$ is realized by $\iota$.\\

For coproducts, $\phi \leq (\phi \lor \psi)$ is realized by $\lmbd{y}{[k,y]}$ and $\psi \leq (\phi \lor \psi)$ is realized by $\lmbd{y}{[\overline{k},y]}$. Suppose we have $\alpha \in \pow{A}^{X}$ and realizers $p$ and $q$ respectively for $\phi \leq \alpha$ and $\psi \leq \alpha$, then $(\phi \lor \psi) \leq \alpha$ is realized by $\lmbd{y}{((\ifthen{p_{0} \cdot y}{p}{q}) \cdot (p_{1} \cdot y))}$. For any $\alpha \in \pow{A}^{X}$, $ \bot \leq \alpha$ is realized trivially by $\iota$.\\

For exponentials, $((\phi \imply \psi) \land \phi) \leq \psi$ is realized by $\lmbd{y}{\app{p_{0} \cdot y, p_{1} \cdot y}}$. Suppose we have $\alpha \in \pow{A}^{X}$ and realizer $p$ for $(\alpha \land \phi) \leq \psi $, then $\alpha \leq (\phi \imply \psi)$ is realized by $\lmbd{y,z}{(p \cdot [y,z])}$.\\

Given a function on sets $f:X \to Y$, we have a function on Heyting prealgebras $\pow{A}^{f}: \pow{A}^{Y} \to \pow{A}^{X}$ taking $\phi : \pow{A}^{Y}$ to $\phi \circ f : \pow{A}^{X}$. 
We define $\forall_{f}, \exists_{f}: \pow{A}^{X} \to \pow{A}^{Y}$ as follows:
\begin{align*} 
\forall_{f}(\phi)(y) &=  \{a \in A | \forall b \in A \forall x \in X (f(x) = y \Rightarrow a \cdot b \space ~ \text{is defined and}\space ~ a \cdot b \in \phi(x))\}   \\
\exists_{f}(\phi)(y) &= \bigcup_{f(x) = y} \phi(x)   
\end{align*}

It should be clear that $\pow{A}^{f}$ is a pre-Heyting algebra morphism as the ordering relation as well as $\land$, $\lor$, $\imply$, $\top$ and $\bot$ are defined pointwise. $\exists_{f}$ is order preserving since if $c$ realizes $\phi(x) \leq \psi(x)$ for all $x \in X$, then $c$ also realizes $$\bigcup_{f(x) = y} \phi(x)  \leq \bigcup_{f(x) = y} \psi(x).$$ Finally if $c$ realizes $\phi(x) \leq \psi(x)$ for all $x \in X$, then take $y \in Y$ and $a \in \forall_{f}(\phi)(y)$ then for $b \in A$ and $x \in X$ if $f(x) = y$ then $a \cdot b$ is defined and $a \cdot b \in \phi(x)$. Then if $f(x) = y$, $c\cdot (a \cdot b) \in \psi(x)$. So $\forall_{f}(\phi)(y) \leq \forall_{f}(\psi)(y)$ is realized by $\lmbd{v_{0},v_{1}}{(c \cdot (v_{0} \cdot v_{1}))}$. Therefore $\forall_{f}$ is order preserving.\\

We will show the following: $$\phi \circ f \leq  \psi \text{ if and only if }\phi \leq \forall_{f}(\psi).$$ Suppose we have a realizer $p$ for $\phi \circ f   \leq  \psi $, then $\phi \leq \forall_{f}(\psi)$ is realized by $\lmbd{z}{(k \cdot (p \cdot z))}$. To see this, take $y \in Y$ and $c \in \phi(y)$. For all $b \in A$ and $x \in X$, if $f(x) = y$ then $\phi(f(x)) = \phi(y)$, then $c \in \phi(f(x))$ and $p \cdot c$ is defined and $p \cdot c \in \psi(x)$.  $k \cdot (p \cdot c)$ gives the constant function on $p\cdot c$ so $(k \cdot (p \cdot c)) \cdot b = p \cdot c$. Thus $ k \cdot (p \cdot c) \in \forall_{f}(\psi)(y)$.\\

Now suppose we have a realizer $p$ for $\phi \leq \forall_{f}(\psi)$. Take $x \in X$ and $c \in \phi(f(x))$, then $p \cdot c \in \forall_{f}(\psi)(f(x))$. Then for all $b \in A$, $(p \cdot c) \cdot b$ is defined and $(p \cdot c) \cdot b \in \psi(x)$. Then $(p \cdot c) \cdot (p \cdot c) \in \psi(x)$. Thus $\phi \circ f \leq  \psi $ is realized by the term $\lmbd{z}{((p \cdot z) \cdot (p \cdot z))}$. With this we have shown that $\forall_{f}$ is a right adjoint to $\pow{A}^{f}$. \\

Now we will show that: $$\psi  \leq \phi \circ f \text{ if and only if }\exists_{f}(\psi)  \leq \phi.$$ Suppose we have a realizer $p$ for $\psi  \leq \phi \circ f$. Take $y \in Y$ and $c \in \exists_{f}(\psi)(y)$. Then there exists an $x \in X$ such that $f(x)  = y$ and $c \in \psi(x)$. Then $p\cdot c \in \phi(f(x))$. So $p$ is also a realizer for $\exists_{f}(\psi)  \leq \phi$.\\

Suppose that $p$ is a realizer for $\exists_{f}(\psi)  \leq \phi$, then $\psi  \leq \phi \circ f$ is also realized by $p$. This because if we have a realizer $c \in \psi(x)$, then $c \in \exists_{f}(\psi)(f(x))$ and $p\cdot c \in \phi(f(x))$. Thus $\exists_{f}$ is the left adjoint to $\pow{A}^{f}$. 

Given a pullback square in $\Set$:
\[\begin{tikzcd}
	X &&&& Y \\
	\\
	\\
	Z &&&& W
	\arrow["f", from=1-1, to=1-5]
	\arrow["h", from=1-5, to=4-5]
	\arrow["g"', from=1-1, to=4-1]
	\arrow["k"', from=4-1, to=4-5]
	\arrow["\lrcorner"{anchor=center, pos=0.125}, draw=none, from=1-1, to=4-5]
\end{tikzcd}\]

We seek to find a realizer for $\forall\phi \in \pow{A}^{Z} (\forall_{f}(\phi \circ g) \cong \forall_{k}(\phi) \circ h )$ where $\phi \cong \psi$ denotes $\phi \leq \psi$ and $\psi \leq \phi$. Given $\phi \in \pow{A}^{Z}$, $y \in Y$ and a realizer $c \in \forall_{f}(\phi \circ g)(y)$, then for all $b \in A$, $x \in X$ if $f(x) = y$, then $c \cdot b$ is defined and $c \cdot b \in \phi(g(x))$. Then given some $b \in A$ and $z \in Z$, if $h(y) = k(z)$ then by the universal property of pullbacks there is a unique $x$ such that $f(x) = y$ and $g(x) = z$. Then $c \cdot b$ is defined and $c \cdot b \in \phi(g(x))$. Hence $c \in \forall_{k}(\phi)(h(y))$ and $(\forall_{f}(\phi \circ g) \leq \forall_{k}(\phi) \circ h )$ is realized by $\iota$.\\

Given a realizer $c \in \forall_{k}(\phi)(h(y))$ then then for all $b \in A$, $z \in Z$ if $k(z) = h(y)$, then $ c \cdot b$ is defined and $c \cdot b \in \phi(z)$. Then given some $b \in A$ and $x \in X$, if $f(x) = y$ then $k(g(x)) = h(f(x)) = h(y)$ and thus $c \cdot b$ is defined and $c \cdot b \in \phi(g(x))$. Hence $c \in \forall_{f}(\phi \circ g)(y)$ and $(\forall_{f}(\phi \circ g) \geq \forall_{k}(\phi) \circ h )$ is also realized by $\iota$.\\  

Thus $\forall\phi \in \pow{A}^{Z} (\forall_{f}(\phi \circ g) \cong \forall_{k}(\phi) \circ h )$ is realized by $[\iota,\iota]$. This means that $\forall_{f} \circ \pow{A}^{g} \cong \pow{A}^{h} \circ \forall_{k}$ which is the Beck-Chevalley condition. This completes the proof. \qedhere \\

\end{proof}

The generic element is the identity morphism on $\pow{A}$. This should be plain as any element $\phi: \pow{A}^{X}$ is literally a function $\phi:X \to \pow{A}$. So $\pow{A}^{\phi}(\id_{\pow{A}}) = \id_{\pow{A}} \circ \phi = \phi$. This proves the following theorem:\\

\begin{thm}
$\pow{A}^{-}$ is s $\Set$-tripos.

\end{thm}

$\Set$ already interprets classical first order logic. However, using $P$ we can equip $\Set$ with a model of intuitionistic first order logic. Consequently, we can construct categories that can interpret intuitionstic first order logic.\\

\begin{defn}

Given a $\Set${-Tripos} $P$, a set $X$, an element $ \sim \space ~ \in P(X\times X)$ is called a {\bf partial equivalence relation} on $X$ iff it satisfies the following in the internal language of $P$:

\begin{enumerate}

\item $\forall x,y \in X (x \sim y  \imply  y \sim x)$\text{ (Symmetry)}
\item $\forall x,y,z \in X(x \sim y \land  y \sim z   \imply   x \sim z)$\text{ (Transitivity)}

\end{enumerate}

These conditions are interpreted respectively as: 

\begin{enumerate}

\item $\top_{P(1)} \leq  \forall_{!_{X\times X}}(\sim \space ~  \imply P(\text{swap}_{X})(\sim))$ 
\item $\top_{P(1)} \leq  \forall_{  !_{X\times X \times X}  }( P(\pr{0,1}^{X,X,X})(\sim) \land P(\pr{1,2}^{X,X,X})(\sim)  \imply P(\pr{0,2}^{X,X,X})(\sim) )$

\end{enumerate}

Where $!_{Z}$ is the morphism from $Z$ into the terminal object $1$ and $\text{swap}_{X}: X\times X \to X \times X$ takes a pair $(x,y)$ to $(y,x)$.

\end{defn}


\begin{lem}
Given a $\Set${-Tripos} $P$, one can construct a category $\Set[P]$ whose objects are pairs $(X, \sim_{X})$ s.t $X$ is a set and $\sim_{X} \in  P(X)$ is a partial equivalence relation on $X$, and, for objects, $(X, \sim_{X})$ and $(Y, \sim_{Y})$ the morphisms are isomorphism classes of elements $F \in P(X \times Y)$ which are functional relations.
\end{lem}
\begin{proof}

 For objects $(X, \sim_{X})$ and $(Y, \sim_{Y})$, we call $F \in P(X \times Y) $ a { \bf functional relation} iff it satisfies the following propositions in $P$:

\begin{enumerate}

\item $\forall x,y(F(x,y) \imply x \sim_{X} x \land  y \sim_{Y} y)$
\item $\forall x,x',y,y'(F(x,y) \land x \sim_{X} x' \land y \sim_{Y} y' \imply F(x',y')) $
\item $\forall x,y,y' (F(x,y) \land F(x,y') \imply y \sim_{Y} y')$
\item $\forall x (x \sim_{X} x \imply \exists y F(x,y))$

\end{enumerate}
interpreted respectively as:

\begin{enumerate}

\item $\top_{P(1)} \leq \forall_{!_{X \times Y}}(F \imply (P( \Delta_{X} \circ \pr{1}^{X,Y})( \sim_{X})  \land  P(\Delta_{Y} \circ \pr{0}^{X,Y})(\sim_{Y}))  )$
\item $\top_{P(1)} \leq \forall_{!_{X \times X \times Y \times Y}}(P(\pr{0,2}^{X,X,Y,Y})(F) \land P(\pr{0,1}^{X,X,Y,Y})(\sim_{X}) \land P(\pr{2,3}^{X,X,Y,Y})(\sim_{Y}) \imply P(\pr{1,3}^{X,X,Y,Y}(F))   )$

\item $\top_{P(1)} \leq \forall_{!_{X\times Y \times Y}}(P(\pr{0,1}^{X,Y,Y})(F) \land P(\pr{0,2}^{X,Y,Y})(F) \imply P(\pr{1,2}^{X,Y,Y})(\sim_{Y})  )$
\item $\top_{P(1)} \leq \forall_{!_{X}}(P(\Delta_{X})(\sim_{X}) \imply \exists_{\pr{1}^{X,Y}}(F)  )$

\end{enumerate}

We refer to the isomorphism class of $F$ as $[F]$. We define identity morphisms by $\id_{(X, \sim_{X})}  = [\sim_{X}]$.
Now we show that $\sim_{X}$ is a functional relation on $(X,\sim_{X})$. Suppose that $x \sim_{X} x'$, by symmetry we have $x' \sim_{X} x$ and by transitivity we have both $x \sim_{X} x$ and $x' \sim_{X} x'$ satisfying $(i)$.Suppose we have $x \sim_{X} y \land x \sim_{X} x' \land y \sim_{X} y'$. Applying transitivity to the first and third formula and symmetry to the second gives us $x \sim_{X} y' \land x' \sim_{X} x$. Apply transitivity gives us $x' \sim_{X} y'$ satisfying $(ii)$. If we have $x \sim_{X} y \land x \sim_{X} y'$, $y \sim_{X} y'$ follows from symmetry and transitivity satisfying $(iii)$. If $x \sim_{X} x$, then obviously $\exists y(x \sim_{X} y)$ satisfying $(iv)$. Therefore $\sim_{X}$ is a function relation.\\

 Given morphisms $[F]: (X, \sim_{X}) \to (Y, \sim_{Y})$ and $[G]: (Y, \sim_{Y}) \to (Z, \sim_{Z})$, we set $$[G] \circ [F] = [\exists y (F(x,y) \land G(x,y))] = [\exists_{\pr{0,2}^{X,Y,Z}}(P(\pr{0,1}^{X,Y,Z})(F') \land  P(\pr{1,2}^{X,Y,Z})(G'))]$$ for any given choice $F' \in [F]$ and $G' \in [G]$. This is well defined since the operations and functors preserve isomorphisms. Furthermore this composition of functional relations can be shown to be a functional relation by applying the properties of functional relations to the original functional relations $F$ and $G$.\\


Given a functional relation $F$ on $(X, \sim_{X})$ and $(Y,\sim_{Y})$, we seek to show that $[F]  = [\sim_{Y}] \circ [F] = [F] \circ [\sim_{X}]$. For all $x \in X$ and $y \in Y$, we have $F(x,y) \imply y \sim_{Y} y$ and $F(x,y) \imply x \sim_{X} x$ by condition $(i)$. Thus we have $F(x,y) \imply \exists y'(F(x,y') \land y' \sim_{Y} y)$ and $F(x,y) \imply \exists x'(x \sim_{X} x' \land F(x',y))$. If for $x \in X$ and $y \in Y$, we have $\exists y'(F(x,y') \land y' \sim_{Y} y)$, then by conditions $(i)$ and $(ii)$ we have $\exists y'(F(x,y') \land y' \sim_{Y} y) \imply F(x,y)$. We have a similar case for $\sim_{X}$. So we have that $F(x,y) \Leftrightarrow \exists y'(F(x,y') \land y' \sim_{Y} y)  \bimp  \exists x'(x \sim_{X} x' \land F(x',y))$ for all $x$ and $y$. Therefore, $[F]  = [\sim_{Y}] \circ [F] = [F] \circ [\sim_{X}]$.\\

Finally we will show associativity of composition. Given a sequence of functional relations:

\[\begin{tikzcd}
	{(X,\sim_{X})} && {(Y,\sim_{Y})} && {(Z,\sim_{Z})} && {(B,\sim_{B})}
	\arrow["F", from=1-1, to=1-3]
	\arrow["G", from=1-3, to=1-5]
	\arrow["H", from=1-5, to=1-7]
\end{tikzcd}\]

we want to show that for all $x \in X$ and $b \in B$:

$$\exists z (\exists y (F(x,y) \land G(y,z)) \land H(z,b)) \bimp \exists y (F(x,y) \land \exists z (G(y,z) \land H(z,b))).$$
 
This follows from the fact that if either statement held, then we have a $y$ and a $z$ such that $F(x,y)$, $G(y,z)$ and $H(z,b)$ holds. Thus the opposing statement can easily be formulated. Hence, the above proposition holds and $[H] \circ ([G] \circ [F]) = ([H] \circ [G]) \circ [F]$.\\

Therefore, $\Set[P]$ is a category.
\end{proof}

\begin{rmk}

Given two objects $(X, \sim_{X})$ and $(Y,\sim_{Y})$ and two functional relations $F$ and $G$ between those objects, if either $F \leq G$ or $G \leq F$ then $F \bimp G$. To see why, suppose $F \leq G$, and $G(x,y)$ holds. Then $x \sim_{X} x$ which means that for some $y' \in Y$, $F(x,y')$ holds. Thus $G(x,y')$ holds. Which means that $y \sim_{Y} y'$ holds. Finally we have that $F(x,y)$ holds. Hence $F \leq G$ implies $G \leq F$.\\

\end{rmk}

\begin{defn}

We call a category $\cC$ an {\bf elementary topos} iff it is has finite limits, is cartesian closed and contains an object $\Omega$ and a morphism $\top_{\Omega}: 1 \to \Omega$ such that for all monomorphisms $\phi : U \to X$ there exists a unique $[\phi]: X \to \Omega$ such that we have a pullback square:

\[\begin{tikzcd}
	U && 1 \\
	\\
	X && \Omega
	\arrow[from=1-1, to=1-3]
	\arrow["\phi"', from=1-1, to=3-1]
	\arrow["\lrcorner"{anchor=center, pos=0.125}, draw=none, from=1-1, to=3-3]
	\arrow["{\top_{\Omega}}", from=1-3, to=3-3]
	\arrow["{[\phi]}"', from=3-1, to=3-3]
\end{tikzcd}\]

We call $\Omega$ the {\bf subobject classifier}.

\end{defn}


\begin{prop}
$\Set[P]$ is an elementary topos.

\end{prop}
\begin{proof}

Given objects $(X, \sim_{X})$ and $(Y, \sim_{Y})$, we set $(X,\sim_{X}) \times (Y, \sim_{Y}) := (X \times Y, \sim_{X,Y})$ where $(x,y) \sim_{X,Y} (x',y') \bimp x \sim_{X} x' \land y \sim_{Y} y'$. We define the projections $[\pr{0}]: (X,\sim_{X}) \times (Y, \sim_{Y}) \to (X, \sim_{X})$ and $[\pr{1}]: (X,\sim_{X}) \times (Y, \sim_{Y}) \to (Y, \sim_{Y})$ as: $\pr{0}((x,y),x') \bimp (x \sim_{X} x'  \land  y \sim_{Y} y)$ and $\pr{1}((x,y),y') \bimp (x \sim_{X} x  \land  y \sim_{Y} y')$. Suppose we have morphisms $[Z_{0}]: (Z,\sim_{Z}) \to (X,\sim_{X})$ and $[Z_{1}]: (Z,\sim_{Z}) \to (Y,\sim_{Y})$, we define a new morphism $[Z^{*}]:(Z,\sim_{Z}) \to (X,\sim_{X})\times (Y, \sim_{Y})$ as $Z^{*}(z,(x,y)) \bimp Z_{0}(z,x) \land Z_{1}(z,y)$.\\

First we have $[\pr{0}] \circ [Z^{*}] = [\exists (x,y) Z^{*}(z,(x,y)) \land \pr{0}((x,y),x')]$. If $\exists (x,y) Z^{*}(z,(x,y)) \land \pr{0}((x,y),x')$ then  $\exists x Z_{0}(z,x) \land x \sim_{X} x'$, then $Z_{0}(z,x')$ holds. Since $\exists (x,y) Z^{*}(z,(x,y)) \land \pr{0}((x,y),x') \leq Z_{0}(z,x')$ for all $z$ and $x'$ then $ \exists (x,y) Z^{*}(z,(x,y)) \land \pr{0}((x,y),x') \bimp Z_{0}(z,x')$ and $[\pr{0}] \circ [Z^{*}] = [Z_{0}]$.  A similar argument holds for $[Z_{1}]$.\\

Suppose there exists some morphism $[Q]:(Z,\sim_{Z}) \to (X,\sim_{X})\times (Y, \sim_{Y})$ such that $[\pr{0}] \circ [Q] = [Z_{0}]$ and $[\pr{1}] \circ [Q] = [Z_{1}]$. If $Q(z,(x,y))$ holds, then $\pr{0}((x,y),x)$ and $\pr{0}((x,y),y)$ holds then $Z_{0}(z,x) \land Z_{1}(z,y) = Z^{*}(z,(x,y))$ holds. Hence $Q \leq Z^{*}$, therefore $Q \bimp Z^{*}$.\\ 

The terminal object is the object $(1,\top_{P(1\times 1)})$. The unique morphism $!_{(X, \sim_{X})}:(X, \sim_{X}) \to (1, \top_{1 \times 1})$ is the morphism $[\top_{1 \times 1}\space~ \land \space~ \sim_{X} ]$. Any functional relation $F:(X, \sim_{X}) \to (1, \top_{1 \times 1})$ implies $\top_{1 \times 1}\space~ \land \space~ \sim_{X} $.\\

Suppose we have the following diagram:
\[\begin{tikzcd}
	{(X,\sim_{X})} &&& {(Y,\sim_{Y})}
	\arrow["{[G]}"', shift right=2, from=1-1, to=1-4]
	\arrow["{[F]}", shift left=2, from=1-1, to=1-4]
\end{tikzcd}\]

We define $\Eq{[F],[G]} := (X, \sim_{[F],[G]})$ where $x \sim_{[F],[G]} x'  \bimp (x \sim_{X} x' \land \exists y (F(x,y) \land G(x,y)) ) $. We define the equalizer map $[E]:(X,\sim_{[F],[G]}) \to (X,\sim_{X}) $, $E\bimp\space ~ \sim_{[F],[G]}$.\\

First we show that $[F] \circ [E] = [G] \circ [E]$. Suppose that $\exists x' (E(x,x') \land F(x',y))$, then there is an $x'$ such that $x \sim_{X} x'$ and $F(x',y)$ holds and there is a $y'$ such that $F(x,y')$ and $G(x,y')$ holds. Then $F(x,y)$ holds. Then $y \sim_{Y} y'$ holds, which means that $G(x,y)$ holds. Then $F(x,y) \land G(x,y)$ holds, therefore $\exists x' E(x,x') \land G(x',y)$. Since $\exists x' (E(x,x') \land F(x',y)) \leq  \exists x' E(x,x') \land G(x',y)$, then $\exists x' (E(x,x') \land F(x',y)) \bimp  \exists x' E(x,x') \land G(x',y)$ and $[F] \circ [E] = [G] \circ [E]$.\\

Suppose there is a morphism $[K]: (L, \sim_{L}) \to (X, \sim_{X})$ such that $[F] \circ [K] = [G] \circ [K]$, then we argue that $K$ is a functional relation from $(L, \sim_{L}) $ to $(X, \sim_{[F],[G]})$. Suppose $K(l,x)$ holds, then there are $y$ and $y'$ such that $F(x,y)$ and $G(x,y')$. This means that $y \sim_{Y} y'$ because $\exists x' K(l,x') \land G(x',y) \bimp \exists x' K(l,x') \land F(x',y)$ and both $\exists x K(l,x) \land F(x,y)$ and $\exists x K(l,x) \land G(x,y')$ hold. Then $F(x,y')$ and $\exists y' F(x,y') \land G(x,y')$ hold. Therefore $x \sim_{[F],[G]} x$. As an obvious consequence, any functional relation $K': (L, \sim_{L}) \to (X, \sim_{[F],[G]})$ with $[K] = [E] \circ [K'] $ is equivalent to $K$ since $[E]$ is also the identity morphism on $(X, \sim_{[F],[G]})$.  \\

Since $\Set[P]$ has finite products and equalizers, it has finite limits.\\

Given $(X,\sim_{X})$ and $(Y,\sim_{Y})$, we define $(X,\sim_{X}) \to (Y,\sim_{Y}) := (\Sigma^{X\times Y}, \sim_{X \to Y})$ where $f  \sim_{X \to Y} g \bimp \forall x,y (f(x,y) \bimp g(x,y)) \land FR_{1}(f) \land FR_{2}(f)\land FR_{3}(f) \land FR_{4}(f)$. We define $\bimp\space ~ \in P(\Sigma \times \Sigma)$ as $P(\imply)(\sigma) \land \space ~ P(\imply \circ \text{ swap}_{\Sigma})(\sigma)$ where $\imply: \Sigma \times \Sigma \to \Sigma$. The $FR_{i}$ are the conditions of a functional relation which we will reiterate:
\begin{align*}
FR_{1}(f) &\bimp \forall x,y(f(x,y) \imply x \sim_{X} x \land  y \sim_{Y} y)\\
FR_{2}(f) &\bimp  \forall x,x',y,y'(f(x,y) \land x \sim_{X} x' \land y \sim_{Y} y' \imply f(x',y')) \\
FR_{3}(f) &\bimp \forall x,y,y' (f(x,y) \land f(x,y') \imply y \sim_{Y} y')\\
FR_{4}(f) &\bimp \forall x (x \sim_{X} x \imply \exists y f(x,y))
\end{align*}

We define $[\app{X,Y}]: (\Sigma^{X\times Y}, \sim_{X \to Y}) \times (X, \sim_{X}) \to (Y,\sim_{Y}) $ where $\app{X,Y}((f,x),y) \bimp (f(x,y) \land FR_{1}(f) \land FR_{2}(f)\land FR_{3}(f) \land FR_{4}(f))$. Suppose we have $[G]: (Z, \sim_{Z}) \times (X, \sim_{X}) \to (Y,\sim_{Y})$, we will define $[G^{*}]: (Z,\sim_{Z}) \to (\Sigma^{X\times Y}, \sim_{X \to Y})$ as $G^{*}(z,f) \bimp ((z \sim_{Z} z) \land \forall x,y(G((z,x),y) \bimp f(x,y)))$. \\

Now we will show that $[G] = [\app{X,Y}] \circ ([G^{*}]\times \id_{(X, \sim_{X})})$. Suppose $\exists (f,x) (G^{*}(z,f) \land x' \sim_{X} x \land \app{X,Y}((f,x), y))$ holds, then $f(x,y)$ holds for some $f$ and $x$ and $f$ is a functional relation. Since $f$ is a functional relation and $x \sim_{X} x'$ holds then $f(x',y)$ holds. Since $G^{*}(z,f)$ holds then $G((z,x'),y)$ holds. Therefore, $[G] = [\app{X,Y}] \circ ([G^{*}]\times \id_{(X, \sim_{X})})$.\\ 

Suppose we have $[H]: (Z,\sim_{Z}) \to (\Sigma^{X\times Y}, \sim_{X \to Y})$ such that $[G] = [\app{X,Y}] \circ ([H]\times \id_{(X, \sim_{X})})$. Suppose both $H(z,f)$ and $f(x,y)$ hold, then $z \sim_{Z} z$ holds and $f$ is a functional relation. Furthermore $\exists (f,x') (H(z,f) \land x \sim_{X} x' \land \app{X,Y}((f,x'), y)))$ since $f(x,y)$ implies $x \sim_{X} x$ and $\app{X,Y}((f,x), y) \bimp f(x,y) \land FR_{1}(f) \land FR_{2}(f)\land FR_{3}(f) \land FR_{4}(f))$. $\exists (f,x') (H(z,f) \land x \sim_{X} x' \land \app{X,Y}((f,x'), y)))$ is equivalent to $G((z,x),y)$. This means that $f(x,y) \leq G((z,x),y)$ so $f(x,y) \bimp G((z,x),y)$. Thus $H(z,f) \leq G^{*}(z,f)$ therefore $H(z,f) \bimp G^{*}(z,f)$.\\

With this, we have shown that $\Set[P]$ is cartesian closed.\\

 The subobject classifier is $(\Sigma, \bimp)$. The map $[\top_{(\Sigma, \bimp)}]: (1, \top_{1\times 1}) \to (\Sigma, \bimp)$ is defined as follows: $\top_{(\Sigma, \bimp)}(*, p) \bimp \sigma(p)$. Suppose we have a monomorphism $[i]: (U, \sim_{U}) \to (X, \sim_{X})$, we define $[\phi_{i}]:(X, \sim_{X}) \to (\Sigma, \bimp)$ where $\phi_{i}(x,p) \bimp (\exists u(i(u,x)) \bimp p)$. Now we want to show that we have the following commutative diagram:

\[\begin{tikzcd}
	(U, \sim_{U}) && (1, \top_{1\times 1}) \\
	\\
	(X, \sim_{X}) &&(\Sigma, \bimp)
	\arrow[from=1-1, to=1-3]
	\arrow["{[i]}"', from=1-1, to=3-1]
	\arrow["{\top_{(\Sigma, \bimp)}}", from=1-3, to=3-3]
	\arrow["{[\phi_{i}]}"', from=3-1, to=3-3]
\end{tikzcd}\]

Suppose $\exists x (i(u,x) \land \phi_{i}(x,p))$, then for some $x$, $\exists u(i(u,x))$ holds. Thus $p$ holds by $\phi_{i}(x,p)$. Therefore, $(!_{(U, \sim_{U})}(u,*) \land \top_{(\Sigma, \bimp)}(*, p))$. Suppose we have the following commutative diagram: 

\[\begin{tikzcd}
	(V, \sim_{V}) && (1, \top_{1\times 1}) \\
	\\
	(X, \sim_{X}) &&(\Sigma, \bimp)
	\arrow[from=1-1, to=1-3]
	\arrow["{[j]}"', from=1-1, to=3-1]
	\arrow["{\top_{(\Sigma, \bimp)}}", from=1-3, to=3-3]
	\arrow["{[\phi_{i}]}"', from=3-1, to=3-3]
\end{tikzcd}\]

We set $J_{i}(v,u) \bimp \exists x(j(v,x) \land i(u,x))$. Now we will show that $[j] = [i] \circ [J_{i}]$. Suppose $j(v,x)$ holds, then $\exists p (\phi_{i}(x,p))$ and for such $p$ we have $(!_{(V, \sim_{V})}(v,*) \land \top_{(\Sigma, \bimp)}(*, p))$. Since $\top_{(\Sigma, \bimp)}(*, p) \bimp \sigma(p)$, $p$ holds. Then $\exists u(i(u,x))$ since $\exists u(i(u,x)) \bimp p$. For such a $u$ we have $J_{i}(v,u)$, therefore $\exists u(J_{i}(v,u) \land i(u,x))$. Since $j(v,x) \leq \exists u(J_{i}(u,v) \land i(u,x))$, $[j] = [i] \circ [J_{i}]$. Since $[i]$ is a monomorphism, any morphism $[K_{i}]$ with $[j] = [i] \circ [K_{i}]$ is equal to $[J_{i}]$. Thus the commutative square above with $(U, \sim_{U})$ is a pullback square.\\

Suppose we have another pullback square:

\[\begin{tikzcd}
	(U, \sim_{U}) && (1, \top_{1\times 1}) \\
	\\
	(X, \sim_{X}) &&(\Sigma, \bimp)
	\arrow[from=1-1, to=1-3]
	\arrow["{[i]}"', from=1-1, to=3-1]
	\arrow["\lrcorner"{anchor=center, pos=0.125}, draw=none, from=1-1, to=3-3]
	\arrow["{\top_{(\Sigma, \bimp)}}", from=1-3, to=3-3]
	\arrow["{[F_{i}]}"', from=3-1, to=3-3]
\end{tikzcd}\]

and suppose $F_{i}(x,p)$ and $p$ hold. Then $\top_{\Sigma, \bimp}(*, p)$ holds, thus We have a commutative square:

\[\begin{tikzcd}
	(1, \top_{1\times 1}) && (1, \top_{1\times 1}) \\
	\\
	(X, \sim_{X}) &&(\Sigma, \bimp)
	\arrow[from=1-1, to=1-3]
	\arrow["{[x]}"', from=1-1, to=3-1]
	\arrow["{\top_{(\Sigma, \bimp)}}", from=1-3, to=3-3]
	\arrow["{[F_{i}]}"', from=3-1, to=3-3]
\end{tikzcd}\]

where $x(*,x') \bimp x \sim_{X} x'$. By the universal property of pullbacks, we have a morphism $[u]:(1, \top_{1\times 1}) \to (U, \sim_{U})$ such that $[x] = [i] \circ [u]$. Thus $\exists u' (u(*,u')) $, for such a $u'$, $\exists x' i(u',x')$ and $x' \sim_{X} x$. Therefore $\exists u (i(u,x))$. If $\exists u (i(u,x))$ holds, for such a $u$, $\exists x'( i(u,x') \land F_{i}(x',p))$. Then $!_{(U, \sim_{U})}(u,*) \land \top_{(\Sigma, \bimp)}(*, p)$. Therefore $p$. Thus $p \bimp \exists u (i(u,x))$ and $F_{i}\leq \phi_{i}$. Therefore $[F_{i}] = [\phi_{i}]$.\\

Since $\Set[P]$ has finite limits, is cartesian closed, and has a subobject classifier, $\Set[P]$ is an elementary topos.\\

\end{proof}

Given a partial combinatory algebra $A$, we call $\RT{A} := \Set[\pow{A}^{-}]$ the {\bf realizability topos on A}.\\

\subsection{Category of Assemblies}
In this section, after proving a few preliminary results, we define the category of assemblies on a partial combinatory algebra $A$, $\Asm(A)$, as well as proving that it is regular and locally cartesian closed.



\begin{lem}

For a $\Set$-tripos $P$ and $\phi \in P(Y)$, define $\sim_{\phi} \space ~ :=  \exists_{\Delta_{Y}}(\phi)$. We have $y \sim_{\phi} y \bimp \phi(y)$ and $y \sim_{\phi} y' \bimp \bot$ when $y \neq y'$.

\end{lem}
\begin{proof}

We will employ the Beck-Chevalley condition as well as the following two pullback squares:

\[\begin{tikzcd}
	Y &&& Y && \emptyset &&& Y \\
	\\
	\\
	Y &&& {Y\times Y} && {\text{not}(Y)} &&& {Y\times Y}
	\arrow["\id_{Y}",from=1-1, to=1-4]
	\arrow["\id_{Y}"',from=1-1, to=4-1]
	\arrow["\lrcorner"{anchor=center, pos=0.125}, draw=none, from=1-1, to=4-4]
	\arrow["\Delta", from=1-4, to=4-4]
	\arrow["\emptyset_{Y}", from=1-6, to=1-9]
	\arrow["\emptyset_{\text{not}(Y)}"',from=1-6, to=4-6]
	\arrow["\lrcorner"{anchor=center, pos=0.125}, draw=none, from=1-6, to=4-9]
	\arrow["\Delta",from=1-9, to=4-9]
	\arrow["\Delta"', from=4-1, to=4-4]
	\arrow["i"',from=4-6, to=4-9]
\end{tikzcd}\]\\

where $\text{not}(Y) = \{(y,y')|  y \neq y'\}$. Given $\sim_{\phi}$, we have $P(\Delta) \circ \exists_{\Delta_{Y}}(\phi) = P(\Delta)(\sim_{\phi}) \bimp \exists_{\id_{Y}} \circ P(\id_{Y})(\phi) = \phi$ and $P(i) \circ \exists_{\Delta_{Y}}(\phi) \bimp \exists_{\emptyset_{\text{not}(Y)}} \circ P(\emptyset_{Y})(\phi)$. $ P(\emptyset_{Y})(\phi)$ is the unique element of $P(\emptyset)$ up to isomorphism and since $\exists_{\emptyset_{\text{not}(Y)}}$ is a left adjoint and thus preserves initial objects we have $\exists_{\emptyset_{\text{not}(Y)}} \circ P(\emptyset_{Y})(\phi) \bimp \bot_{P(Y)}$. Therefore $P(i) \circ \exists_{\Delta_{Y}}(\phi) \bimp \bot_{P(Y)}$.

\end{proof}

\begin{defn}

We define a functor $ \nabla:\Set \to \Set[P]$ where $\nabla(X) = (X, =_{X})$ and for $f:X \to Y$, $\nabla(f)(x,y) \bimp (f(x) =_{Y} y)$. We call this functor the {\bf constant objects functor}. We define $(x =_{X} x') \bimp  x \sim_{\top_{P(X)}} x'$.

\end{defn}

\begin{lem}
For any set $X$, $[i]:(U, \sim_{U}) \to \nabla(X)$ is a monomorphism iff $(U,\sim_{U})$ is isomorphic to an object $(X,\sim_{\phi})$ for $\phi \in P(X)$.

\end{lem}

\begin{proof}

The monomorphism $[i]:(U, \sim_{U}) \to \nabla(X)$ is characterized by the morphism $[\phi_{i}]: \nabla(X) \to (\Sigma,\bimp)$ where $\phi(x,p) \bimp (\exists_{u}(i(u,x)) \bimp p)$. It suffices to show that morphisms of the form $\psi_{j}(x,p) \bimp (\phi(x) \bimp p)$ characterize the morphisms $[\sim_{\phi}]:(X,\sim_{\phi}) \to \nabla(X)$. The morphism $[\sim_{\phi}]$ is a monomorphism since $\sim_{\phi}$ exhibits the identity on $(X,\sim_{\phi})$, thus if $[\sim_{\phi}] \circ [F] = [\sim_{\phi}] \circ [G]$, then $[F] = [G]$.  So $[\sim_{\phi}]$ is characterized by $\psi_{\sim_{\phi}}(x,p) \bimp  (\exists_{x'}(x \sim_{\phi} x') \bimp p)$. We have that $\exists_{x'}(x \sim_{\phi} x') =  \exists_{x'}(\exists_{\Delta_{X}}(\phi)(x,x'))  = \exists_{\pr{0}^{X,X}} \circ \exists_{\Delta_{X}}(\phi)$. The functor $\exists_{\pr{0}^{X,X}} \circ \exists_{\Delta_{X}}: P(X)  \to P(X)$ is left adjoint to $P(\Delta_{X}) \circ P(\pr{0}^{X,X}) = P(\pr{0}^{X,X} \circ \Delta_{X}) = P(\id_{X})$. Thus $\exists_{\pr{0}^{X,X}} \circ \exists_{\Delta_{X}}$ is equivalent to $\id_{P(X)}$ and $\phi(x) \bimp \exists_{x'}(x \sim_{\phi} x') $. So $\psi_{\sim_{\phi}}(x,p) \bimp (\phi(x) \bimp p)$. This completes the proof.

\end{proof}

\begin{defn}

In the case of a realizability tripos on $A$, we define the category $\Asm(A)$ to be the full subcategory of $\RT{A}$ whose objects are $(X,\sim_{\phi})$ for $\phi \in \pow{A}^{X}$. We call $\Asm(A)$ the category of {\bf assemblies on A}.

\end{defn}

\begin{lem}

For a pca A, $ \nabla:\Set \to \RT{A}$ has a left adjoint $\Gamma: \RT{A} \to \Set$. Furthermore, we have that $\Gamma \circ \nabla \cong \Id_{\Set}$ making $\Set$ a reflective subcategory of $\RT{A}$.

\end{lem}
\begin{proof}

First we define $\Gamma: \RT{A} \to \Set$. For $(X, \sim_{X})$, we set $\Gamma(X, \sim_{X}) := \{x \in X|  (x \sim_{X} x) \neq \emptyset \}/\approx_{X}$ where $\approx_{X}$ is the equivalence relation defined as follows: $x \approx_{X} x'$ iff $x \sim_{X} x' \neq \emptyset$. $\approx_{X}$ is indeed an equivalence relation since it is only defined on $x \in X$ such that $x \sim_{X} x \neq \emptyset$ and symmetry and transitivity of $\sim_{X}$ carry over to $\approx_{X}$. For $[F]: (X,\sim_{X}) \to (Y,\sim_{Y})$, we define $\Gamma([F])([x]) = [y]$ if and only if $F(x,y) \neq \emptyset$. $\Gamma([F])$ is well-defined since $F$ is a functional relation. Suppose $x \sim_{X} x $ is non-empty, since $x \sim_{X} x \leq \exists_{y}F(x,y)$, then there exists a $y$ such that $F(x,y) $ is non-empty.  If $x \sim_{X} x'$ and $F(x,y)$ is non-empty since $x \sim_{X} x' \land F(x,y) \leq F(x',y)$ then $F(x',y)$ is non-empty. If $F(x,y)$ and $F(x,y')$ are non-empty, then since $F(x,y) \land F(x,y') \leq y \sim_{Y} y'$, so is $y \sim_{Y} y'$.\\

Given $(X, \sim_{X})$, $\Gamma(\id_{(X,\sim_{X})})$ is clearly $\id_{\Gamma(X,\sim_{X})}$ since $\id_{(X,\sim_{X})} = [\sim_{X}]$. Given $[F]: (X, \sim_{X}) \to (Y, \sim_{Y})$ and $[G]: (Y,\sim_{Y}) \to (Z,\sim_{Z})$, suppose we have $[x] \in \Gamma(X, \sim_{X})$, then for $x' \in [x]$, $x' \sim_{X} x' \neq \emptyset $, and there exists a $y$ such that $F(x',y) $ is non-empty. Since $F(x',y) \leq (x' \sim_{X} x') \land (y \leq_{Y} y)$ then $y \sim_{Y} y $ is non-empty so $\Gamma([F])([x]) = [y]$. This means that there is a $z$ such that $G(y,z) \neq \emptyset $. Then $\exists_{y'} (F(x',y') \land G(y',z)) $ is inhabited. Therefore $\Gamma([G]) \circ \Gamma([F])([x]) = [z] = \Gamma([G] \circ [F])([x])$.\\

Now we will exhibit the following bijection: $\RT{A}((X,\sim_{X}),(Y,=_{Y})) \cong \Set(\Gamma(X,\sim_{X}),Y)$. Firstly, we define $\alpha: \RT{A}(X,\sim_{X}),(Y,=_{Y})) \to \Set(\Gamma(X,\sim_{X}),Y)$. $\alpha([F])([x]) = y$ iff $F(x,y)$ non-empty.  $\alpha([F])$ is well defined since if $y \neq y'$, then either $F(x,y) = \emptyset$ or $F(x,y')= \emptyset$ since $y =_{Y} y'$ is empty. However there must exist a $y$ such that $F(x,y)$ is inhabited, thus $y$ is unique. For $f \in \Set(\Gamma(X,\sim_{X}),Y)$, we define $[f^{*}]:(X,\sim_{X}) \to (Y,=_{Y})$ where $f^{*}(x,y) \bimp (x \sim_{X} x) \land x \approx_{f}y$ where $x \approx_{f} y \bimp \exists(x' \in X_{h})(x' = x \land f([x']) =_{Y} y)  := \exists_{X_{h} \times \id_{Y}} \circ P(X_{q}\times \id_{Y})(\nabla(f))$ where $X_{h}:\{x \in X|  (x \sim_{X} x) \neq \emptyset \} \mono X$ is the inclusion into $X$ and $X_{q}:\{x \in X|  (x \sim_{X} x) \neq \emptyset \} \to \Gamma(X, \sim_{X})$ is the quotient map.\\


Now suppose we have $[x] \in \Gamma(X,\sim_{X})$, then $\alpha(f^{*})([x]) = y$ iff $\exists(x' \in X_{h})(x' = x \land f([x']) = y)$. Thus $\alpha(f^{*})([x]) = f([x])$. Suppose there is $[F]\in \RT{A}((X,\sim_{X}),(Y,=_{Y}))$ such that $\alpha([F]) = f$. Suppose $F(x,y)$ holds, then $x \sim_{X} x$ holds. Then $x \sim_{X} x $ is inhabited and $x \in X_{h}$. Since $\alpha([F])([x]) = y$, then $f([x]) = y$, therefore $\exists(x' \in X_{h})(x' = x \land f([x']) =_{Y} y)$. Thus $F(x,y) \leq (x \sim_{X} x) \land (x \approx_{f}y)$ meaning that they are equivalent. Therefore, $[F] = [f^{*}]$ and $\alpha$ is a bijection exhibiting $\nabla \vdash \Gamma$.\\

For a set $Y$, it is easy to see that $\Gamma \circ \nabla(Y) = Y$ since $y =_{Y} y  \bimp \top$ and $y = _{Y} y' \bimp \bot$ when $y \neq y'$. Furthermore for a function $f:X \to Y$, it is clear that $\Gamma \circ \nabla(f) = f$ since $f(x) =_{Y} y \bimp \top$ only when $f(x) = y$. Therefore $\Gamma \circ \nabla \cong \Id_{\Set}$.

\end{proof}


\begin{lem}
\label{asma}

$\Asm(A)$ is equivalently the category whose objects are pairs $(X,\phi:X \to \pow{A}/\{\emptyset\})$ and morphisms $f: (X,\phi) \to (Y,\psi)$ are functions $f:X \to Y$ such that there exists a realizer $r \in A$ where for all $x \in X$ and $a \in \phi(x)$ $r \cdot a$ is defined and $r \cdot a \in \psi(f(x))$.

\end{lem}
\begin{proof}

Given an object $(X,\sim_{\phi}) \in \Asm(A)$, we can define an object $(X',\phi^{*})$ where $X' = \Gamma(X, \sim_{\phi})$. We have that $x \sim_{\phi} x'$ iff $x = x'$, thus $\Gamma(X, \sim_{\phi}) = \{x \in X|  (x \sim_{\phi} x) \neq \emptyset \}$. So we can define $\phi^{*}(x) = x \sim_{\phi} x$. Given $[F]: (X,\sim_{\phi}) \to (Y,\sim_{\psi})$, we have a function $\Gamma([F]):(X',\phi^{*}) \to (Y',\psi^{*})$. Now to show that $\Gamma([F])$ is a morphism first observe that we have realizers $r_{0}$ and $r_{1}$ respectively for $\forall x (x \sim_{\phi} x  \leq  \exists_{y} F(x,y))$ and $\forall xy(F(x,y) \leq (x \sim_{\phi} x) \land (y \sim_{\psi} y))$. So given $x \in X'$ and $a \in \phi^{*}(x) = x \sim_{\phi} x $, then $r_{0} \cdot a \in F(x,y)$ for some $y \in Y$. Furthermore $y = \Gamma([F])(x)$. So, $r_{1} \cdot (r_{0} \cdot a) \in (x \sim_{\phi} x) \land (y \sim_{\psi} y)$. Therefore, $p_{1} \cdot (r_{1} \cdot (r_{0} \cdot a)) \in (y \sim_{\psi} y) = \psi^{*}(y) = \psi^{*}(\Gamma([F])(x))$. So $\lmbd{x}{p_{1} \cdot (r_{1} \cdot (r_{0} \cdot x)) }$ realizes $\Gamma([F]):(X',\phi^{*}) \to (Y',\psi^{*})$ as a morphism. This construction is functorial since $\Gamma$ is a functor and we will refer to this construction as $T$.\\

Given an object $(X,\phi:X \to \pow{A}/\{\emptyset\})$, we can construct an object $S(X,\phi) = (X, \sim_{\phi})$ and given a morphism $f: (X,\phi) \to (Y,\psi)$ to a morphism $[S(f)]: (X,\sim_{\phi}) \to (Y,\sim_{\psi})$ where $S(f)(x,y) = \{[a,b]| a \in \phi(x),\, b \in (f(x) \sim_{\psi} y)\}$ so that $S(f)(x,y) \bimp \phi(x) \land (f(x) \sim_{\psi} y)$. if we have $\id_{X}: (X,\phi) \to (X,\phi)$, then for all $x, x' \in X$, $S(\id_{X})(x,x') \bimp \phi(x) \land (x \sim_{\phi} x')$. We obviously have that $\phi(x) \land (x \sim_{\phi} x') \bimp (x \sim_{\phi} x')$ and $(x \sim_{\phi} x') \imply (x' \sim_{\phi} x)$ since $\sim_{\phi}$ is symmetric thus $(x \sim_{\phi} x') \imply (x \sim_{\phi} x) = \phi(x)$. \\

Therefore, $(x \sim_{\phi} x') \bimp (\phi(x) \land (x \sim_{\phi} x')) \bimp S(\id_{X})(x,x')$. So $[S(\id_{X})] = \id_{(X,\sim_{\phi})}$. If we have $f: (X,\phi) \to (Y,\psi)$ and $g:(Y,\psi) \to (Z,\zeta)$, then observe that $\exists_{y} S(f)(x,y) \land S(g)(y,z)  \bimp \exists_{y} \phi(x) \land (f(x) \sim_{\psi} y) \land \psi(y) \land (g(y) \sim_{\zeta} z)$.  If we have a realizer $a \in \exists_{y} \phi(x) \land (f(x) \sim_{\psi} y) \land \psi(y) \land (g(y) \sim_{\zeta} z)$, then $f(x) \sim_{\psi} y$ and $g(y) \sim_{\zeta} z$ are inhabited thus $f(x) = y$ and $g(f(x)) = z$. So $g(f(x)) \sim_{\zeta} z$, thus $\exists_{y} S(f)(x,y) \land S(g)(y,z) \imply \phi(x) \land (g(f(x)) \sim_{\zeta} z) = S(g\circ f)(x) $. The converse clearly holds thus $\exists_{y} S(f)(x,y) \land S(g)(y,z) \bimp S(g \circ f)(x,z)$ and so $[S(g)] \circ [S(f)] = [S(g\circ f)]$. Thus $S$ is functorial.\\

For the functor $T \circ S $ given $(X,\phi)$, we have $\Gamma(X,\sim_{\phi}) = X$ so $T \circ S(X,\phi) = (X,\phi)$. For $f:(X,\phi) \to (Y,\psi)$, $T \circ S(f)(x) = y$ iff $\phi(x) \land (f(x) \sim_{\psi} y) \neq \emptyset$. We have that  $\phi(x) \neq \emptyset$ and $(f(x) \sim_{\psi} f(x)) = \psi(f(x)) \neq \emptyset$, so $T \circ S(f)(x) = f(x)$.\\

Given an assembly, $(X, \sim_{\phi})$, we define $[R(X, \sim_{\phi})]: (X,\sim_{\phi}) \to (\Gamma(X,\phi), \sim_{\phi^{*}})$ where $R(X,\phi)(x,x') = x \sim_{\phi} x'$. Given a morphism $[F]: (X,\sim_{\phi}) \to (Y, \sim_{\psi})$, we seek to show that $[R(Y,\psi)] \circ [F] = S(\Gamma([F])) \circ [R(X,\phi)]$. Suppose that $\exists_{y}(F(x,y) \ \land \ y \sim_{\psi} y')$ is inhabited. Then $\phi(x)$ is inhabited, $x \in \Gamma(X,\phi)$ and for some $y$, $F(x,y) \neq \emptyset$. Thus $\Gamma([F])(x) = y$ and $\Gamma([F])(x) \sim_{\psi} y'$ holds. Therefore, $\exists_{x'}(x \sim_{\phi} x' \ \land \ \phi(x') \ \land \ \Gamma([F])(x') \sim_{\psi} y')$ holds. Therefore $[R(Y,\psi)] \circ [F] = S(\Gamma([F])) \circ [R(X,\phi)]$. This gives us a natural transformation $[R]: \Id_{\Asm(P)} \to S \circ T$. $[R]$ becomes a natural isomorphism with inverse $R^{-1}(X,\phi)(x,x') \bimp x \sim_{\phi} x'$.  Therefore, $\Id_{\Asm(P)} \cong S \circ T$. With this, $S$ is an equivalence.\\

\end{proof}

\begin{rmk}
\label{asmnabla}

Because $ \nabla:\Set \to \RT{A}$ takes a set $X$ to an assembly and $\Asm(A)$ is a full sub-category of $\RT{A}$, the $\Gamma \adj \nabla$ adjunction restricts to $\Asm(A)$ making $\Set$ a reflective subcategory of $\Asm(A)$. In particular, $\nabla(X) = (X, \const_{A})$ and $\Gamma(X,\phi) = X$.\\

\end{rmk}

\begin{lem}

$\Asm(A)$ has finite limits and colimits.

\end{lem}
\begin{proof}
Given two assemblies $(X,\phi)$ and $(Y,\psi)$, we define $(X,\phi)\times (Y,\psi)$ as $(X\times Y, \phi \otimes \psi)$ where $\phi \otimes \psi \in \pow{A}^{X\times Y}$ where $\phi \otimes \psi(x,y) = \{ [a,b]| a \in \phi(x)\text{ and } b \in \psi(y) \}$. We have projections $\pr{0}:(X,\phi)\times (Y,\psi) \to (X,\phi)$ and $\pr{1}:(X,\phi)\times (Y,\psi) \to (Y,\psi)$ realized by $p_{0}$ and $p_{1}$ respectively. Suppose we have $s_{0}:(Z,\zeta) \to (X,\phi)$ realized by $\alpha_{0}$ and $s_{1}:(Z,\zeta) \to (X,\phi)$ realized by $\alpha_{1}$, we have a unique morphism $s: (Z, \zeta) \to (X,\phi)\times (Y,\psi) $      
mapping $z \in Z$ to $(s_{0}(z),s_{1}(z)) \in X \times Y$ since $X \times Y$ is the product of $X$ and $Y$ in $\Set$ and this morphism $s$ is realized by $\lmbd{t}{[s_{0} \cdot t, s_{1} \cdot t]}$. Thus $\Asm(A)$ has products.\\

The terminal object is $(1, A)$. This is easy to verify.\\

Given a diagram

\[\begin{tikzcd}
	{(X,\phi)} &&& {(Y,\psi)}
	\arrow["{g}"', shift right=2, from=1-1, to=1-4]
	\arrow["{f}", shift left=2, from=1-1, to=1-4]
\end{tikzcd}\]

where $f$ is realized by $r_{f}$ and $g$ is realized by $r_{g}$, we define $\Eq{f,g}:= (E_{f,g}, \phi')$ where $E_{f,g} := \{x | f(x) = g(x)\}$ is the equalizer of $f$ and $g$ in $\Set$ and $\phi' \in \pow{A}^{E_{f,g}}$ where $\phi'(x) = \phi(x)$. The canonical inclusion $i: E_{f,g} \mono X$ is plainly an assembly morphism. Suppose we have $h: (Z,\zeta) \to (X,\phi)$ such that $f \circ h = g  \circ h$, then we have the unique morphism in $\Set$, $h': Z \to E_{f,g}$ such that $h = i \circ h'$ and it is an assembly morphism since $h(z) = h'(z)$ for all $z$ and thus any realizer of $h$ realizes $h'$. Therefore $\Asm(A)$ has all finite limits\\

We define $(X,\phi) + (Y,\psi) := (X + Y,\phi + \psi)$ where $\phi + \psi \in \pow{A}^{X + Y}$ defined as $\phi + \psi(x) = \{ [k,a] |a \in \phi(x)\}$ and $\phi + \psi(y) = \{ [\overline{k},a] | a \in \psi(y)\}$.  We have inclusions $i_{0}: (X,\phi) \to (X,\phi) + (Y,\psi)$ realized by $\lmbd{t}{[k,t]}$ and $i_{1}: (Y,\psi) \to (X,\phi) + (Y,\psi)$ realized by $\lmbd{t}{[\overline{k},t]}$\\

Suppose we have $h_{0}: (X,\phi) \to (Z,\zeta)$ realized by $r_{0}$ and $h_{1}: (X,\phi) \to (Z,\zeta)$ realized by $r_{1}$, we have the unique morphism $h:X + Y \to Z$ in $\Set$ such that $h \circ i_{0} = h_{0}$ and $h \circ i_{1} = h_{1}$. $h$ is realized by $\lmbd{t}{\ifthen{p_{0} \cdot t}{r_{0} \cdot (p_{1} \cdot t)}{r_{1} \cdot (p_{1} \cdot t)}} $.\\

The initial object is $(\emptyset, \phi)$ for $\phi \in \pow{A}^{\emptyset}$.\\

Using the same diagram used for equalizers, we define $\CoEq{f,g} := (CE_{f,g}, \psi')$ where $CE_{f,g}$ is the coequalizer of $f$ and $g$ and $\psi' = \exists_{\kappa}(\psi)$ where $\kappa: Y \to CE_{f,g}$ is the coequalizer surjection. Alternatively, $\psi'(t) = \bigcup_{y \in \kappa^{-1}(t)} \psi(x)$. One can see that $\kappa $ is realized by $\iota$. Suppose we have a morphism $j: (Y,\psi) \to (Z, \zeta)$ such that $j \circ f = j \circ g$, then the unique morphism $j':CE_{f,g} \to Z$ such that $j' \circ \kappa = j$ exists. Furthermore, $j'$ is realized by identity morphism $\iota$. Thus $\Asm(A)$ has all finite colimits.\\

Therefore, $\Asm(A)$ has finite limits and colimits.\\

\end{proof}



\begin{defn}

Given a category $\cC$, a morphism $f:X \to Y$ is called a {\bf regular epimorphism} if it is the coequalizer of some diagram:

\[\begin{tikzcd}
	{Z} &&& {X}
	\arrow["{h_{0}}"', shift right=2, from=1-1, to=1-4]
	\arrow["{h_{1}}", shift left=2, from=1-1, to=1-4]
\end{tikzcd}\]

The fact that $f:X \to Y$ is an epimorphism follows from the universal property of coequalizers.

\end{defn}

\begin{lem}

Let $\cC$ be a category with pullbacks and coequalizers and let $f:X \to Y$ be a regular epimorphism then $f$ is the coequalizer of its kernel pair.

\end{lem}
\begin{proof}
For the diagram:
\[\begin{tikzcd}
	{X \times_{Y}X} &&& X && Y
	\arrow["{\pr{1}}"', shift right=2, from=1-1, to=1-4]
	\arrow["{\pr{0}}", shift left=2, from=1-1, to=1-4]
	\arrow["{f}",from=1-4, to=1-6]
\end{tikzcd}\]

We have $f \circ \pr{0} = f \circ \pr{1}$. Given a $g:X \to Y'$ such that $g \circ \pr{0} = g \circ \pr{1}$, using the original coequalizer diagram for $f$:
\[\begin{tikzcd}
	{Z} &&& X && Y
	\arrow["{h_{1}}"', shift right=2, from=1-1, to=1-4]
	\arrow["{h_{0}}", shift left=2, from=1-1, to=1-4]
	\arrow["{f}",from=1-4, to=1-6]
\end{tikzcd}\]

since $f \circ h_{0} = f \circ h_{1}$, we have a unique $h:Z \to X \times_{Y}X$ such that $h_{0} = \pr{0} \circ h$ and $h_{1} = \pr{1} \circ h$. Thus $g \circ h_{0} = g \circ \pr{0} \circ h = g \circ \pr{1} \circ h = g \circ h_{1}$. So there must exist a unique $i:Y \to Y'$ such that $i \circ f = g$. This $i$ this exhibits the first diagram as a coequalizer diagram.

\end{proof}


\begin{lem}

Given a pca $A$, $f:(X,\phi) \to (Y,\psi)$ is a regular epimorphism in $\Asm(A)$ iff $f$ is surjective and $\psi \bimp \exists_{f}(\phi)$.

\end{lem}
\begin{proof}

It suffices to show the ``only if" direction since the ``if" direction is precisely the characterization of coequalizers in $\Asm(A)$. First we construct the kernel pair of $f$ and present the following diagram in $Asm(A)$:

\[\begin{tikzcd}
	{(X,\phi)\times_{(Y,\psi)} (X,\phi)} &&& (X,\phi) && (Y,\psi)
	\arrow["{\pr{1}}"', shift right=2, from=1-1, to=1-4]
	\arrow["{\pr{0}}", shift left=2, from=1-1, to=1-4]
	\arrow["{f}",from=1-4, to=1-6]
\end{tikzcd}\]

The underlying diagram in $\Set$ is coequalizer diagram since every surjection in $\Set$ is a regular epimorphism and the kernel pair of $f$ in $\Set$ is precisely the underlying kernel pair of $f:(X,\phi) \to (Y,\psi)$ in $\Set$. The fact that $\psi \bimp \exists_{f}(\phi)$ then makes $f$ the coequalizer of the above diagram since this is precisely the construction of coequalizers in $\Asm(A)$.

\end{proof}

\begin{rmk}

We can also characterize a regular epimorphism $f:(X,\phi) \to (Y,\psi)$ as being the epimorphism such that $\forall_{y \in Y} \psi(y) \imply \exists_{x \in X} (f(x) = y \land \phi(x))$ is realized. This is also the same as stating that $f$ is a surjection in the internal language of $\Asm(A)$.

\end{rmk}

\begin{defn}

We call a category $\cC$ a {\bf regular category} iff 
\begin{itemize}
\item $\cC$ has finite limits.

\item Given a morphism $f: X \to Y$ in $\cC$, the kernel pair of $f$ has a coequalizer

\item Regular epimorphisms are stable under pullbacks.

\end{itemize}

\end{defn}

\begin{lem}

Given a pca $A$, $\Asm(A)$ is regular.

\end{lem}
\begin{proof}

$\Asm(A)$ has all finite limits and colimits so it suffices to show that regular epimorphisms are preserved by pullback. Given a pullback diagram:

\[\begin{tikzcd}
	{(X \times_{Y} Z, \phi \times \zeta)} &&& {(X, \phi)} \\
	\\
	\\
	{(Z,\zeta)} &&& {(Y, \psi)}
	\arrow[ "c", from=1-1, to=1-4]
	\arrow["h"', from=1-1, to=4-1]
	\arrow["\lrcorner"{anchor=center, pos=0.125}, draw=none, from=1-1, to=4-4]
	\arrow["f", from=1-4, to=4-4]
	\arrow["g"', from=4-1, to=4-4]
\end{tikzcd}\]

where $f$ is a surjection and $\psi  \bimp \exists_{f}(\phi)$, we need to show that $h$ is a surjection and $\zeta \bimp \exists_{h}(\phi \times \zeta)$. For any $z \in Z$, there exists $x \in X$ such that $g(z) = f(x)$, thus $h(z,x) = z$ which shows that $h$ is surjective. We have $\exists_{h}(\phi \times \zeta) \leq \zeta$ since $\phi \times \zeta \leq \pow{A}^{h}(\zeta)$. \\

For the other direction, first observe that $\phi \times \zeta \bimp \pow{A}^{c}(\phi) \land \pow{A}^{h}(\zeta)$ since $\phi \times \zeta  \leq \pow{A}^{c}(\phi)$ and $\phi \times \zeta  \leq \pow{A}^{h}(\zeta)$ meaning $\phi \times \zeta  \leq \pow{A}^{c}(\phi) \land \pow{A}^{h}(\zeta)$. The morphisms $c: (X \times_{Y} Z, \pow{A}^{c}(\phi) \land \pow{A}^{h}(\zeta)) \to (X, \phi)$ and $h: (X \times_{Y} Z, \pow{A}^{c}(\phi) \land \pow{A}^{h}(\zeta)) \to (Z, \zeta)$ correspond to the morphism $\id: (X \times_{Y} Z, \pow{A}^{c}(\phi) \land \pow{A}^{h}(\zeta)) \to (X \times_{Y} Z, \phi \times \zeta)$ giving us $\pow{A}^{c}(\phi) \land \pow{A}^{h}(\zeta) \leq \phi \times \zeta$. \\

So $\exists_{h}(\phi \times \zeta) \bimp \exists_{h}(\pow{A}^{c}(\phi) \land \pow{A}^{h}(\zeta)) \bimp \exists_{h}(\pow{A}^{c}(\phi)) \land \zeta$ by the Frobenius condition. So is suffices to show $\zeta \leq \exists_{h}(\pow{A}^{c}(\phi))$. We have $\exists_{h}(\pow{A}^{c}(\phi)) \bimp \pow{A}^{g}(\exists_{f}(\phi)) \bimp \pow{A}^{g}(\psi)$ by Beck-Chevalley and our assumption, thus $\zeta \leq \exists_{h}(\pow{A}^{c}(\phi))$ holds because $\zeta \leq \pow{A}^{g}(\psi)$ holds. Thus $\zeta \bimp \exists_{h}(\phi \times \zeta)$ making $h$ is regular epimorphism.

\end{proof}


\begin{lem}
Given a regular category $\cC$, every morphism $f: X \to Y$ in $\cC$ has a factorization:
\[\begin{tikzcd}
	X && \im(f) && Y
	\arrow["e", from=1-1, to=1-3]
	\arrow["m", from=1-3, to=1-5]
\end{tikzcd}\]

where $e$ is a regular epimorphism and $m$ is a monomorphism. We call this the {\bf image factorization}. Furthermore, for any factorization $f = m' \circ e'$ where $m'$ is a monomorphism, there exists a morphism $\mu: m \to m'$ (i.e. m is the smallest subobject of $Y$ that factors through $f$).

\end{lem}
\begin{proof}

The fact that $f$ can be factored as $m \circ e$ where $m$ is a monomorphism and $e$ is a regular epimorphism is proven as Proposition 4.2 of \cite{vanOosten1995}. $e$ is given as the coequalizer of the kernel pair of $f$ and $m$ is the unique morphism such that $f = m \circ e$ that follows from the universal property of coequalizers. Suppose we have a factorization $f = m' \circ e'$ such that $m'$ is a monomorphism, then let $\pr{0}, \pr{1}: X \times_{Y} X \to X$ and be the kernel pair of $f$. Since $m' \circ e' \circ \pr{0} = m' \circ e' \circ \pr{1}$, then $e' \circ \pr{0} = e' \circ \pr{1}$. By the universal property of coequalizers, there exists a $\mu: \im(f) \to Z'$ such that $\mu \circ e = e'$. Since $m \circ e = f = m' \circ e' = m' \circ \mu \circ e$, then $m  = m' \circ \mu$ since $e$ is an epimorphism. This exhibits the morphism $\mu: m \to m'$.

\end{proof}


\begin{defn}
In a category $\cC$, we call a morphism $f:X \to Y$ a {\bf cover} iff the only subobject of $Y$ that $f$ factors through is $\id_{Y}$.

\end{defn}

\begin{lem}
In a regular category $\cC$, $f:X \to Y$ is a regular epimorphism iff $f$ is a cover.

\end{lem}
\begin{proof}
Suppose $f$ is a cover, then we have the following factorization:
\[\begin{tikzcd}
	X && \im(f) && Y
	\arrow["e", from=1-1, to=1-3]
	\arrow["m", from=1-3, to=1-5]
\end{tikzcd}\]

so that $m \cong \id_{Y}$. Thus $m$ is an isomorphism and since $e$ is a regular epimorphism, $f$ is also a regular epimorphism. \\

Suppose $f$ is a regular epimorphism, let $g,h: Z \to X$ be the coequalizer pair of $f$. Since we have the image factorization of $f = m \circ e$ where $e$ is a regular epimorphism and $m$ is a monomorphism and $m \circ e \circ g = m \circ e \circ h$,  then $e \circ g =  e \circ h$. There exists a unique morphism $m': Y \to \im(f)$ such that $m' \circ f = e$. Furthermore, $m \circ m' = \id_{Y}$ since by the universal property of coequalizers and $m \circ m' \circ f = m \circ e = f$. Then $m$ is a split epimorphism and a monomorphism, therefore $m$ is an isomorphism. Obviously, $m \cong \id_{Y}$, thus $m$ is both the largest and smallest subobject that factors through $f$ meaning that $f$ is a cover.\\

Therefore, $f:X \to Y$ is a regular epimorphism iff $f$ is a cover.

\end{proof}


\begin{lem}
\label{asmdepprod}

Given a pca $A$, for any morphism $f:(X,\phi) \to (Y,\psi)$, we have the following adjoint triple:
\[ 
\begin{tikzcd}[column sep=large]
 \Asm(A)/(X,\phi)  \arrow[r, bend left, "{\sum_f}", "\perp"'] \arrow[r, bend right, "{\Pi_f}"', "\perp"]  &  \Asm(A)/(Y,\psi) \arrow[l, "{f^{*}}" description]
\end{tikzcd}
\] 

where $f^{*}$ is the pullback functor. Thus $\Asm(A)$ is locally cartesian closed.

\end{lem}
\begin{proof}

$\sum_{f}$ is $f \circ -$ and is left adjoint to $f^{*}$ by the universal property of pullbacks.\\

For $g:(Z,\zeta) \to (X,\phi)$, we define $\Pi_{f}(g) = (\Pi_{x \in X(y)} Z(x),\phi \To \zeta) \to (Y, \psi)$ where $\Pi_{x \in X(y)} Z(x) := \{(y, F: f^{-1}(y) \to Z)| g \circ F = \id_{X}|_{f^{-1}(y)}\text{ and }\{ [a,b] \in A| a \in \phi(y)\text{ and }\forall_{x \in f^{-1}(y)}\forall_{c \in \phi(x)} b \cdot c \in \zeta(F(x))   \} \neq \emptyset \}$ and $\phi \To \zeta \in \pow{A}^{\Pi_{x \in X(y)} Z(x)}$ where $\phi \To \zeta(y,F) = \{ [a,b] \in A| a \in \phi(y)\text{ and }\forall_{x \in f^{-1}(y)}\forall_{c \in \phi(x)} b \cdot c \in \zeta(F(x))   \}$. \\

Given a morphism $\beta: g \to h$, we define $\Pi_{f}(\beta):\Pi_{f}(g) \to \Pi_{f}(h)$, where $\Pi_{f}(\beta)(y,F) = (y,\beta \circ F)$. Suppose $(y,F) \in \dom(\Pi_{x \in X(y)} Z(x))$ and given a realizer $[a,b] \in \phi \To \zeta(y,F)$, if $r_{\beta}$ realizes $\beta$, then $\lmbd{t}{r_{\beta} \cdot(b \cdot t)}$ realizes $\beta \circ F$ so $[a,\lmbd{t}{r_{\beta} \cdot(b \cdot t)}]$ realizes $(y,\beta \circ F)$. So $\Pi_{f}(\beta)$ is realized by $\lmbd{s}{[p_{0} \cdot s, \lmbd{t}{r_{\beta} \cdot((p_{1} \cdot s) \cdot t)}]}$. Verification of composition, identity and that $\beta \circ F \in \Pi_{f}(h)$ is straightforward. \\ 

For $p:(N,\nu) \to (X,\phi)$ and $q:(M,\mu) \to (Y,\psi)$, suppose we have $\alpha: f^{*}(q) \to p$, define $\hat{\alpha}: q \to \Pi_{f}(p)$ where $\hat{\alpha}(m) = (q(m), \alpha(-,m):f^{-1}(q(m))\to N )$. For $m \in M$, $r_{m} \in \mu(m)$, if $r_{q}$ is the realizer for $q$ and $r_{\alpha}$ realizes $\alpha$, then $\lmbd{t}{r_{\alpha} \cdot [t,r_{m}]}$ realizes $\alpha(-,m):f^{-1}(q(m))\to N$ so $[r_{q} \cdot r_{m}, \lmbd{t}{r_{\alpha} \cdot [t,r_{m}]}]$ realizes $\hat{\alpha}(m)$. Thus, $\hat{\alpha}$ is realized by $\lmbd{s}{[r_{q} \cdot s, \lmbd{t}{r_{\alpha} \cdot [t,s]}]}$. \\

Given $\beta:q \to \Pi_{f}(p)$, we define $\overline{\beta}:f^{*}(q) \to p$, where $\overline{\beta}(x,m) = \pr{1}(\beta(m))(x)$. Suppose $(x,m) \in \dom(f^{*}(q))$, and $[r_{x},r_{m}]$ realizes $(x,m)$ , if $r_{\beta}$ realizes $\beta$ then $(p_{1}\cdot (r_{\beta} \cdot r_{m})) \cdot r_{x}$ realizes $\overline{\beta}(x,m)$. So $\overline{\beta}$ is realized by $\lmbd{t}{(p_{1}\cdot (r_{\beta} \cdot (p_{1} \cdot t))) \cdot (p_{0} \cdot t)}$.\\

For $\alpha: f^{*}(q) \to p$, we have $\overline{\hat{\alpha}}(x,m) = \pr{1}(\hat{\alpha}(m))(x) = \alpha(x,m)$. Thus $\alpha = \overline{\hat{\alpha}}$. For $\beta:q \to \Pi_{f}(p)$, we have $\hat{\overline{\beta}}(m) = (q(m),\overline{\beta}(-,m)) = (q(m),\pr{1}(\beta(m))) = \beta(m)$. Thus $\beta = \hat{\overline{\beta}}$. We have a bijection between morphisms $\alpha: f^{*}(q) \to p$ and $\beta:q \to \Pi_{f}(p)$. Therefore, $\Asm(A)$ is locally cartesian closed.

\end{proof}

\subsection{W-Types}
Here we give a construction of dependent $W$-types in $\Set$ and extend dependent $W$-types to $\Asm(A)$.

\begin{defn}

Given a locally cartesian closed category $\cC$ and a diagram of the following form:

\[\begin{tikzcd}
	& X &&& Y \\
	\\
	Z &&&&& Z
	\arrow["f", from=1-2, to=1-5]
	\arrow["g"', from=1-2, to=3-1]
	\arrow["h", from=1-5, to=3-6]
\end{tikzcd}\]

is called a {\bf polynomial} in $\cC$. This induces the following functor $\poly_{f,g,h} = \sum_{h} \circ \Pi_{f} \circ g^{*}: \cC/Z \to \cC/Z$. 

\end{defn}

\begin{defn}
We call $\poly_{f,g,h}$ a {\bf polynomial functor}. An endofunctor $F: D \to D$ induces a category $\Alg(F)$ called {\bf the category of algebras over $F$} whose objects are pairs $(c, \alpha_{c}: F(c) \to c)$ and morphism are morphisms $\beta: c \to d$ such that $\beta \circ \alpha_{c} = \alpha_{d} \circ F(\beta)$. An initial object $(W,\omega)$ of $\Alg(F)$ is called the {\bf initial algebra for $F$}. In the case of $\Alg(\poly_{f,g,h})$, we call $W$ a {\bf dependent W-type}.

\end{defn}


Given a function $f:X \to Y$ in $\Set$, the explicit description for polynomial functor is $\poly_{f}(X) = \sum_{y \in Y} X^{f^{-1}(y)}$. We will give a construction of the $W_{f}$-type for $\poly_{f}$ which is found in \cite{Hoffman2006}. First we define $(X + Y)^{*}$ to be the set of non-zero finite sequences of $X + Y$. Next we take a subset $W_{f} \subseteq \pow{(X + Y)^{*}}$ which we will call the {\bf well-founded f-trees}. 

\begin{defn}

Given $T \in \pow{(X + Y)^{*}}$, $T$ is an {\bf (X + Y)-labeled tree} iff it contains exactly one sequence of length $1$ and for $(z_{1},\ldots, z_{n + 1}) \in T$, then $(z_{1},\ldots, z_{n }) \in T$ for all $n \geq 1$.\\

We call $T$ an {\bf f-tree} iff it is an $(X + Y)$-labeled tree with the following properties:
\begin{itemize}
\item If $(z_{1},\ldots,z_{n}) \in T$, then $z_{i} \in Y$ if $i$ is odd and $z_{i} \in X$ if $i$ is even.
\item If $(z_{1},\ldots,z_{2n}) \in T$, then there is a unique $z_{2n + 1} \in X$ such that $(z_{1},\ldots,z_{2n + 1}) \in T$.
\item If $(z_{1},\ldots,z_{2n + 1}) \in T$, then $\{z | (z_{1},\ldots,z_{2n + 1},z) \in T\} = \{x \in X| f(x) = z_{2n + 1}\}$

\end{itemize}

For any $f$-tree $T$, there is an ordering $<$ on $T$ where for any two elements $(z_{1},\ldots,z_{n}),(z'_{1},\ldots,z'_{m}) \in T$, $(z_{1},\ldots,z_{n})< (z'_{1},\ldots,z'_{m})$ iff there is a non-zero finite sequence $(w_{1},\ldots,w_{k})$ such that $(z'_{1},\ldots,z'_{m},w_{1},\ldots,w_{k}) = (z_{1},\ldots,z_{n})$. We call $T$ {\bf well-founded} if it satisfies the following property with respect to $<$: \\

For any subset $P \subseteq T$, if $P$ satisfies the condition: $\forall x\in T\forall y\in T(((y < x) \imply y \in P) \imply x \in P)$, then $P = T$.\\

We call the above condition {\bf the principle of $\epsilon$-induction} and we call a subset $P$ satisfying the premise of the principle of $\epsilon$-induction {\bf $<$-inductive}.

\end{defn}

\begin{lem}

$W_{f} \cong \poly_{f}(W_{f}) = \sum_{y \in Y} W_{f}^{f^{-1}(y)}$.

\end{lem}
\begin{proof}

First we construct $\alpha_{f}: \sum_{y \in y} W_{f}^{f^{-1}(y)} \to W_{f}$. Given $(y,t) \in \sum_{y \in y} W_{f}^{f^{-1}(y)} $ and $x \in f^{-1}(y)$, define $T'_{x,y,t} = \{(y,x,z_{1},\ldots,z_{n})| (z_{1},\ldots,z_{n}) \in t(x) \}$. Then we have $\alpha_{f}(y,t) = \{y\} \cup \{(y,x)| x \in f^{-1}(y)\} \cup \bigcup_{x \in f^{-1}(y)} T'_{x,y,t}$. $\alpha_{f}(y,t)$ is an $(X + Y)$-labeled tree since $y \in \alpha_{f}(y,t)$ is the unique element of length $1$ and for $n \geq 1$ and $z = (z_{1},\ldots,z_{n + 1}) \in \alpha_{f}(y,t)$ for some $x \in f^{-1}(y)$, either  $z = (y,x)$ or $z = (y,x,z'_{1},\ldots,z'_{n})$ for $(z'_{1},\ldots,z'_{n}) \in t(x)$. Thus $(z_{1},\ldots,z_{n}) \in \alpha_{f}(y,t)$ since $y, (y,x) \in \alpha_{f}(y,t)$ and $t(x)$ is an $(X + Y)$-labeled tree. $\alpha_{f}(y,t)$ is an f-tree as witness by the $y$ and $(y,x)$'s in $\alpha_{f}(y,t)$ and by the fact that $t(x)$ is an f-tree. Suppose there is a subset $P \subseteq \alpha_{f}(y,t)$ that is $<$-inductive, then set $p(x) = \{(z_{1},\ldots,z_{n})  \in t(x)| (y,x,z_{1},\ldots,z_{n}) \in P \cap T'_{x,y,t}  \}$. p(x) is $<$-inductive since if $(z_{1},\ldots,z_{n}) < (z'_{1},\ldots,z'_{m}) $ in $t(x)$, then $(y,x,z_{1},\ldots,z_{n}) < (y,x,z'_{1},\ldots,z'_{m}) $ in $T'_{x,y,t}$. Thus $p(x) = t(x)$. So $T'_{x,y,t} \subseteq P$ for all $y$ and $x \in f^{-1}(y)$. Since $P$ is $<$-inductive, then $ \{y\} \cup \{(y,x)| x \in f^{-1}(y)\} \subseteq P$ for all $y$. Therefore, $P = \alpha_{f}(y,t)$.\\

Given $w \in W_{f}$, we construct $(y_{w},t_{w}) \in \sum_{y \in y} W_{f}^{f^{-1}(y)}$, where $y_{w}$ is the unique element in $w$ of length $1$ and $t_{w}(x) = \{(z_{1},\ldots, z_{n})|  (y_{w},x,z_{1},\ldots,z_{n}) \in w\}$. Each $t_{w}(x)$ is an $(X + Y)$-labeled tree since for $(y_{w},x) \in w$, there is an unique $z$ such that $(y_{w},x,z) \in w$ and $w$ is also an $(X + Y)$-labeled tree. $t_{w}(x)$ is also an f-tree since $w$ is. Suppose for each $x$, there is a $<$-inductive subset $p(x) \subseteq t_{w}(x)$, then set $p'(x) = \{(y,x,z_{1},\ldots,z_{n})| (z_{1},\ldots,z_{n}) \in p(x)\}$ and $P =\bigcup_{x \in f^{-1}(y)} p'(x) \cup \{y_{w}\} \cup \{(y_{w},x)| x \in f^{-1}(y_{w})\} $. P is $<$-inductive since for $(y_{w},x, z_{1},\ldots, z_{n}) \in T'_{x,y_{w},t_{w}}$, if $(y_{w},x',z'_{1},\ldots, z'_{m}) < (y_{w},x, z_{1},\ldots, z_{n})$, then $x' = x$, $(z'_{1},\ldots, z'_{m}) \in t_{w}(x)$, and $(z'_{1},\ldots, z'_{m}) < (z_{1},\ldots,z_{n})$ in $t_{w}(x)$ and $ \{y_{w}\} \cup \{(y_{w},x)| x \in f^{-1}(y_{w})\} $ are already in $P$. Thus $P = w$. So $p'(x) = T'_{x,y_{w},t_{w}}$, therefore $t_{w}(x) = p(x)$.\\

It is clear that $\alpha_{f}(y_{w},t_{w}) = w$. Suppose we have $(y',t')$ such that $\alpha_{f}(y',t') = w$. Then $y' = y_{w}$ and $t'(x) = \{(z_{1},\ldots, z_{n})| (y_{w},x,z_{1},\ldots,z_{n}) \in w\} = t_{w}(x)$. Thus $(y',t') = (y_{w},t_{w})$, and $\alpha_{f}$ is a bijection.

\end{proof}

\begin{cor}

$(W_{f},\alpha_{f})$ is a $\poly_{f}$-algebra.

\end{cor}

\begin{defn}

We define a relation $\DirST$ on $W_{f}$ where $\DirST(w,w')$ holds when there exists $(y,t) \in \poly_{f}(W_{f})$ and $x \in f^{-1}(y)$ such that $w = \alpha_{f}(y,t)$ and $t(x) = w$. We define the relation $\SubTree$ on $W_{f}$ as the reflexive and transitive closure of $\DirST$. \\

\end{defn}

\begin{lem}

$W_{f}$ is well-founded with respect to $\DirST$

\end{lem}
\begin{proof}

Here we adapt an argument from Proposition $2$ of Section $5$ of \cite{Blass1983}. Suppose we have a $\DirST$-inductive subset $P \subseteq W_{f}$, then for $w \in W_{f}$ define $P_{w} = \{w' \in P | \SubTree(w',w)\}$. First we observe that for any subtree $w'$ of $w$, what $\SubTree$ states is that either there is a finite sequence $((y_{0},t_{0},x_{0}),\ldots, (y_{n},t_{n},x_{n}))$ such that $\alpha_{f}(y_{0},t_{0}) = w$, for $i < n$, $t_{i}(x_{i}) = \alpha_{f}(y_{i + 1},t_{i + 1})$ and $t_{n}(x_{n}) = w'$ or $w' = w$. This means that $P_{w}$ induces a subset $[P_{w}] \subseteq w$ where $[P_{w}]$ is the set of sequences $(y_{0},x_{0},\ldots,y_{n},x_{n},y_{n + 1})$ and $(y_{0},x_{0},\ldots,y_{n},x_{n},y_{n + 1},x_{n + 1})$ such that there exists a $w' \in P_{w}$ and a sequence $((y_{0},t_{0},x_{0}),\ldots, (y_{n},t_{n},x_{n}))$ which makes $w'$ a subtree of $w$, such that $(y_{n + 1}) \in w'$ and $x_{n + 1} \in f^{-1}(y_{n + 1})$. Furthermore, if $\alpha_{f}(y,t) = w \in P_{w}$, then $(y), (y,x) \in [P_{w}]$ for all $x \in f^{-1}(y)$.\\

We seek to show that $[P_{w}]$ is $<$-inductive. Given a sequence $(z_{1},\ldots,z_{n}) \in w$, suppose for $(z_{1},\ldots, z_{n},z'_{1},\ldots, z'_{n'}) < (z_{1},\ldots,z_{n}) $, $(z_{1},\ldots, z_{n},z'_{1},\ldots, z'_{n'}) \in [P_{w}]$; if $n$ is odd then that sequence witnesses a subtree $w'$ of $w$ in $P_{w}$. Then any sequence $(z_{1},\ldots,z_{n},x,y) < (z_{1},\ldots,z_{n})$ corresponds to a direct subtree $t(x)$ where $t: f^{-1}(z_{n}) \to W_{f}$ such that $\alpha_{f}(z_{n},t) = w'$ and $(y) \in t(x)$. Then each $(z_{1},\ldots,z_{n},x,y)$ witnesses $t(x) \in P_{w} \subseteq P$. Then $w' \in P_{w} \subseteq P$ since $P$ is $\DirST$-inductive. So $(z_{1},\ldots,z_{n}) \in [P_{w}]$ since the sequence witnesses $w' \in P$ as a subtree of $w$. \\

If $n$ is even then $(z_{1},\ldots,z_{n - 1})$ exhibits some $w' = \alpha_{f}(z_{n - 1},t)$ as a subtree of $w$. Then $(z_{1},\ldots,z_{n - 1},x) \in [P_{w}]$ for $x \in f^{-1}(z_{n -1})$ by definition of $[P_{w}]$. Then $(z_{1},\ldots,z_{n - 1},z_{n}) \in [P_{w}]$. Therefore, $[P_{w}]$ is $<$-inductive and $[P_{w}] = w$. By definition of $[P_{w}]$, $w \in P_{w} \subseteq P$. Therefore $P = W_{f}$.\\

\end{proof}

\begin{lem}

If $f$ is surjective, then $W_{f} = \emptyset$.

\end{lem}

\begin{proof}

Suppose for $\alpha_{f}(y,t) = w \in W_{f}$. Since $f^{-1}(y)$ is inhabited for all $y$, there will always exists a $t(x)$ for some $x \in f^{-1}(y)$. So the condition $\forall (w' \in W_{f})(\DirST(w',w) \imply w' \in \emptyset)$ is always false for $w \in W_{f}$. Hence $\emptyset$ is vacuously $\DirST$-inductive. Therefore $W_{f} = \emptyset$.

\end{proof}

\begin{lem}
\label{fwtype}

$(W_{f},\alpha_{f})$ is the $W$-type of $f:X \to Y$.

\end{lem}
\begin{proof}

We will use the argument in Theorem 2.1.5 of \cite{Berg2006}. Suppose we have a $\poly_{f}$-algebra $(X,\beta)$. For $w \in W_{f}$, we define $\Paths_{w} := \{(w_{0},\ldots,w_{n})| w_{0} = w \land \DirST(w_{i + 1},w_{i})\text{ for }i < n \}$. We call a function $\gamma:\Paths_{w} \to X$ an attempt on $w$ when for $(w_{0},\ldots,w_{n}) \in \Paths_{w}$, $\gamma(w_{0},\ldots,w_{n}) = \beta(y,[x \in f^{-1}(y) \mapsto \gamma(w_{0},\ldots,w_{n},t(x))])$ where $w_{n} = \alpha_{f}(y,t)$. Define $U_{f} \subseteq W_{f}$ where $U_{f} = \{w \in W_{f}| w\text{ has a unique attempt }\gamma_{w}:\Paths_{w}\to X\}$. We will show that $U_{f}$ is $\DirST$-inductive. Suppose for $w = \alpha_{f}(y,t)$ such that for all $x \in f^{-1}(y)$, $t(x) \in U_{f}$; Define $\gamma_{w}:\Paths_{w} \to X$ where $\gamma_{w}((w)) = \beta(y, [x \in f^{-1}(y) \mapsto \gamma_{t(x)}( (t(x)) )])$ and $\gamma_{w}((w,t_{0},\ldots,t_{n})) = \gamma_{t_{0}}(t_{0},\ldots,t_{n})$. One should be able to see that this is an attempt on $w$. Suppose that there is another attempt $\gamma'_{w}:\Paths_{w} \to X$ on $w$. We define functions $\gamma'_{t(x)}:\Paths_{t(x)}\to X$ where $\gamma'_{t(x)}(t(x),\ldots,t_{n}) = \gamma'_{w}(w,t(x),\ldots,t_{n})$. $\gamma'_{t(x)}$ is an attempt on $t(x)$ since $\gamma'_{w}$ is an attempt on $w$. Then $\gamma'_{t(x)} = \gamma_{t(x)}$. As for $(w) \in \Paths_{w}$, we have $\gamma'_{w}((w)) = \beta(y, [x \in f^{-1}(y) \mapsto \gamma'_{w}( (w,t(x)) )]) = \beta(y, [x \in f^{-1}(y) \mapsto \gamma'_{t(x)}( (t(x)) )]) = \beta(y, [x \in f^{-1}(y) \mapsto \gamma_{t(x)}( (t(x)) )]) = \gamma_{w}(w)$. Then $\gamma'_{w} = \gamma_{w}$ and $w \in U_{f}$. Therefore $U_{f}$ is $\DirST$-inductive and $U_{f} = W_{f}$.\\

This allows us to define a function $\gamma:W_{f} \to X$ where $\gamma(w) = \gamma_{w}((w))$. $\gamma$ is a $\poly_{f}$-morphism since if $w = \alpha_{f}(y,t)$ then $\gamma(w) =  \gamma_{w}((w)) = \beta(y,[x \in f^{-1}(y) \mapsto \gamma_{w}(w,t(x))]) = \beta(y,[x \in f^{-1}(y) \mapsto \gamma_{t(x)}((t(x)))]) = \beta(y,[x \in f^{-1}(y) \mapsto \gamma(t(x))]) = \beta\circ \poly_{f}(\gamma)(y,t)$.\\

Suppose we have another $\poly_{f}$-algebra morphism $h: W_{f} \to X$, then we take the equalizer of $h$ and $\gamma$, $\Eq{\gamma,h} = \{w \in W_{f}| h(w) = \gamma(w)\}$. We show that $\Eq{\gamma,h}$ is $\DirST$-inductive. Suppose for $w = \alpha_{f}(y,t)$, and for $x \in f^{-1}(y)$, $t(x) \in \Eq{\gamma,h}$, then $$\beta(y,[x \in f^{-1}(y) \mapsto h(t(x))]) = \beta(y,[x \in f^{-1}(y) \mapsto \gamma(t(x))]) $$ $$\beta \circ \poly_{h}(y,t) = \beta \circ \poly_{\gamma}(y,t)$$ $$h(\alpha(y,t)) = \gamma(\alpha(y,t))$$ Hence $w \in \Eq{\gamma,h}$ making $\Eq{\gamma,h}$ $\DirST$-inductive. Therefore $\Eq{\gamma,h} = W_{f}$ and $\gamma = h$.\\

Therefore, $(W_{f},\alpha_{f})$ is the initial algebra of $\poly_{f}$.


An important example of $W$-types is the natural numbers object:\\

\begin{defn}

An object $\nat$ in a finite product category $\cC$ is called a {\bf natural numbers object} when equipped with morphisms $\nm{0}: 1 \to \nat$ and $s: \nat \to \nat $ such that for every object $C$ in $\cC$, if there exists morphisms $\nm{c}: 1 \to C$ and $f: C \to C$, there exists a unique morphism $\overline{f}: \nat \to C$ making the following diagram commute:

\[\begin{tikzcd}
	1 && \nat && \nat \\
	\\
	&& C && C
	\arrow["{\nm{0}}", from=1-1, to=1-3]
	\arrow["{\nm{c}}"', from=1-1, to=3-3]
	\arrow["s", from=1-3, to=1-5]
	\arrow["{\overline{f}}"', from=1-3, to=3-3]
	\arrow["{\overline{f}}", from=1-5, to=3-5]
	\arrow["f"', from=3-3, to=3-5]
\end{tikzcd}\]

\end{defn}

The natural numbers object $\nat$ is given by the morphism $1 \to 1 + 1$.\\

\end{proof}


Given a diagram:
\[\begin{tikzcd}
	& X &&& Y \\
	\\
	Z &&&&& Z
	\arrow["f", from=1-2, to=1-5]
	\arrow["g"', from=1-2, to=3-1]
	\arrow["h", from=1-5, to=3-6]
\end{tikzcd}\]\\

in $\Set$, we construct an initial algebra for $\poly_{f,g,h}$. 

\begin{defn}
\label{fghtree}
Given the above diagram and the $W$-type $W_{f}$, we call a element $w \in W_{f}$ an {\bf (f,g,h)-tree} if when $w = \alpha_{f}(y,t)$, for $x \in f^{-1}(y)$, $t(x)$ is a $(f,g,h)$-tree and $g(x) = h(y_{x})$ where $t(x) = \alpha_{f}(y_{x},t_{x})$.

\end{defn}

We define $W_{f,g,h} \subseteq W_{f}$ as the set of {\bf well-founded (f,g,h)-trees}, and define a function $c_{f,g,h}:W_{f,g,h} \to Z$ where $c_{f,g,h}(w) = h(y)$ where $w = \alpha_{f}(y,t)$. When $Z$ is the terminal object, then $W_{f,g,h} = W_{f}$.\\


\begin{lem}
\label{fixpoint}

$c_{f,g,h} \cong \poly_{f,g,h}(c_{f,g,h})$, thus $c_{f,g,h}$ is a $\poly_{f,g,h}$-algebra.

\end{lem}

\begin{proof}

The domain of $\poly_{f,g,h}(c_{f,g,h})$ can be expressed as the set $\{z \in Z, y \in h^{-1}(z), w: f^{-1}(y) \to W_{f,g,h}\space ~| \space ~ c_{f,g,h} \circ w = g|_{f^{-1}(y)}\}$. We have a morphism $\alpha_{f,g,h}: \dom(\poly_{f,g,h}(c_{f,g,h})) \to W_{f,g,h}$ where $\alpha_{f,g,h}(z,y,w) = \alpha_{f}(y,w)$. $\alpha_{f}(y,w)$ is an $(f,g,h)$-tree since for $x \in f^{-1}(y)$ where $w(x) = \alpha_{f}(y_{x},t_{x})$, $g(x) = c_{f,g,h} \circ w(x) = h(y_{x})$. $\alpha_{f,g,h}$ also extends to a morphism $\alpha_{f,g,h}: \poly_{f,g,h}(c_{f,g,h}) \to c_{f,g,h} $ which makes $c_{f,g,h}$ a $\poly_{f,g,h}$-algebra.\\

We also have a morphism $r: W_{f,g,h} \to \dom(\poly_{f,g,h}(c_{f,g,h}))$ where $r(w) = (h(y),y,t)$ where $w = \alpha_{f}(y,t)$. To see that $(h(y),y,t) \in \dom(\poly_{f,g,h}(c_{f,g,h}))$, observe that for $x \in f^{-1}(y)$, $t(x)$ is an $(f,g,h)$-tree and $c_{f,g,h} \circ t(x) = h(y_{x}) = g(x)$ for $t(x) = \alpha_{f}(y_{x},t_{x})$ by the fact that $w$ is a $(f,g,h)$-tree. $r$ also extends to a morphism $r: c_{f,g,h} \to \poly_{f,g,h}(c_{f,g,h})$.\\

Clearly $r = \alpha_{f,g,h}^{-1}$ follows from the fact that $\alpha_{f}$ is an isomorphism.\\

\end{proof}

\begin{lem}
\label{wellf}
$c_{f,g,h} $ is its own smallest  $\poly_{f,g,h}$ sub-algebra. 

\end{lem}

\begin{proof}

Suppose we have an sub-algebra $d_{f,g,h}:D \to Z$ of $c_{f,g,h}$; then $D$ is a subobject of $W_{f,g,h}$. Define a subobject of $W_{f}$, $D' = \{w \in W_{f}| w \in W_{f,g,h} \imply w \in D\}$. We seek to show that $D'$ is $\DirST$-inductive. Suppose for $w  = \alpha(y,t) \in W_{f}$ and $x \in f^{-1}(y)$ we have that $t(x) \in D'$ and suppose that $w \in W_{f,g,h}$; then $t(x) \in W_{f,g,h}$ meaning that $t(x) \in D$. Furthermore we have that $g(x) = c_{f,g,h} \circ t(x) = d \circ t(x)$ since $d_{f,g,h}$ is a sub-algebra of $c_{f,g,h}$. Therefore $w = \alpha_{f}(y,t) \in D$. Therefore, $w \in D'$ making $D'$ $\DirST$-inductive. Then $D' = W_{f}$ and by implication $D = W_{f,g,h}$ making $c_{f,g,h} = d_{f,g,h}$.

\end{proof}

\begin{defn}
A $\poly_{f,g,h}$-algebra $c$ that satisfies both Lemma $\ref{wellf}$ and Lemma $\ref{fixpoint}$ is called the {\bf well-founded fixpoint} of $\poly_{f,g,h}$

\end{defn}

\begin{lem}

$(c_{f,g,h}, \alpha_{f,g,h})$ is the initial algebra of $\poly_{f,g,h}$.

\end{lem}
\begin{proof}
We will use an argument similar to the one in Lemma $\ref{fwtype}$. Given a $\poly_{f,g,h}$-algebra, $(b:X \to Z,\beta:\poly_{f,g,h}(b) \to b )$ and an element $w \in W_{f,g,h}$, take $\Paths_{w}$ as defined in Lemma $\ref{fwtype}$ and define $c_{w}: \Paths_{w} \to Z$ where $c_{w}(w,\ldots, w_{n}) = c_{f,g,h}(w_{n})$. We call a morphism $\gamma: c_{w} \to b$ an attempt on $w$ iff for $(w,\ldots, w_{n}) \in \Paths_{w}$, $\gamma(w,\ldots,w_{n}) = \beta(h(y),y,[x \in f^{-1}(y) \mapsto \gamma(w,\ldots,w_{n},t(x))])$ where $w_{n} = \alpha_{f,g,h}(h(y),y, t)$. This definition still works in this setting because $b(\gamma(w,\ldots,w_{n},t(x))) = c_{w}(w,\ldots,w_{n},t(x)) = c_{f,g,h}(t(x)) = g(x)$ by the fact that $w$ is a $(f,g,h)$-tree.\\

Define a subset $U_{f,g,h} \subseteq W_{f,g,h}$ where $U_{f,g,h} = \{w \in W_{f,g,h}\space ~ |\space ~ w\text{ has a unique attempt} \}$. We define a function $u_{f,g,h}: U_{f,g,h} \to Z$ where $u_{f,g,h}(w) = c_{f,g,h}$. Similar to Lemma $\ref{fwtype}$, we can show that $u_{f,g,h}$ is a subalgebra of $W_{f,g,h}$ since for $w = \alpha_{f,g,h}(h(y),y,t)$ and $t(x) \in U_{f,g,h}$ for $x \in f^{-1}(y)$, we can easily construct a unique attempt for $w$, $\gamma_{w}:c_{w} \to \beta$ where $\gamma_{w}(w, z_{1},\ldots, z_{n}) = \gamma_{z_{1}}(z_{1},\ldots, z_{n})$ and $\gamma_{w}(w) = \beta(h(y),y,[x \in f^{-1}(y) \mapsto \gamma_{t(x)}(t(x))])$. This defines a morphism $\gamma : c_{f,g,h} \to b$ where $\gamma(w) = \gamma_{w}(w)$.\\

Also similar to Lemma $\ref{fwtype}$, one can show that $\gamma$ is the unique $\poly_{f,g,h}$-algebra morphism into $b$ since for any other $\poly_{f,g,h}$-algebra morphism $\delta: c_{f,g,h} \to b$ one can form the equalizer $\Eq{\gamma, \delta} \mono W_{f,g,h}$ and show that $\Eq{\gamma, \delta}$ has a $\poly_{f,g,h}$-algebra structure.\\

Therefore, $(c_{f,g,h}, \alpha_{f,g,h})$ is the initial algebra of $\poly_{f,g,h}$.

\end{proof}

All of the above results from Definition $\ref{fghtree}$ can be generalized in the internal language of locally cartesian closed categories with $W$-types. We explicitly state the generalized results below:

\begin{lem}
\label{depwtype}

Suppose $\cC$ is a locally cartesian closed category with $W$-types, then $\cC$ has all dependent $W$-types. Thus any elementary topos with a natural numbers object has dependent $W$-types.

\end{lem}

This lemma are also meant to be a citation of Theorem 14 of \cite{Gambino2003} and a generalization of Proposition 3.6 of \cite{Moerd2000}.\\

Suppose we are working in the category of assemblies on a partial combinatory algebra $\Asm(A)$, by Lemma $\ref{depwtype}$ and Section 2.2.6 of \cite{Berg2006}, $\Asm(A)$ has dependent $W$-types. We give the construction of dependent $W$-types in $\Asm(A)$ given the following diagram:

Given a diagram:
\[\begin{tikzcd}
	& (X,\phi) &&& (Y,\psi) \\
	\\
	(Z,\zeta) &&&&& (Z,\zeta)
	\arrow["f", from=1-2, to=1-5]
	\arrow["g"', from=1-2, to=3-1]
	\arrow["h", from=1-5, to=3-6]
\end{tikzcd}\]\\

First we construct the dependent $W$-type of the underlying diagram in $\Set$, $c_{f,g,h}: W_{f,g,h} \to Z$. First we define a function $\omega: W_{f} \to \pow{A}$ where for $w = \alpha_{f}(y_{w},t_{w})$, $$\omega(w) = \{[r,s] | r \in \psi(y_{w})\text{ and }\forall x \in f^{-1}(x)\forall u \in \phi(x)( s \cdot u  \in \omega(t_{w}(x))) \}$$. This technique is known in \cite{Berg2006} as decoration. $r$ is a realizer of $y_{w}$ and $s$ is essentially a realizer of $t_{w}$. Set $W'_{f}$ as the subset of $W_{f}$ with $\omega(w)$ inhabited. By \cite{Berg2006}, $(W'_{f},\omega)$ is the $W$-type for $f$.
We set $W'_{f,g,h}$ as the subset of $W'_{f}$ of well-founded (f,g,h)-trees, so $c_{f,g,h}:(W'_{f,g,h}, \omega) \to (Z, \zeta)$, where $c_{f,g,h}(\alpha_{f}(y_{w},t_{w})) = h(y_{w})$, is the dependent $W$-type corresponding to the diagram above. \\

\subsection{Partitioned Assemblies}
In this section, we introduce the category of partitioned assemblies and show that under the axiom of choice, partition assemblies are the projective objects of $\Asm(A)$. Using this, we prove that $\RT{A}$ is the ex/lex completion of the category of partition assemblies on $A$.
\begin{defn}

Given a regular category $\cC$, we define the {\bf reg/ex completion} of $\cC$, $\cC_{reg/ex}$ to be the exact category such that there is a regular functor $\eta : \cC \to \cC_{reg/ex}$ and for any regular functor $F: \cC \to \cD$ where $\cD$ is an exact category there exist uniquely up to natural isomorphism a functor $F':\cC_{reg/ex} \to \cD$ where $F' \circ \eta \cong F$.\\

Similarly, given a lex category $\cC$, we define the {\bf reg/lex completion} of $\cC$, $\cC_{reg/lex}$ to be the regular category such that there is a regular functor $\eta : \cC \to \cC_{reg/lex}$ and for any regular functor $F: \cC \to \cD$ where $\cD$ is a regular category there exist uniquely up to natural isomorphism a functor $F':\cC_{reg/lex} \to \cD$ where $F' \circ \eta \cong F$.\\

\end{defn}

\begin{lem}

$\RT{A}$ is the ex/reg completion of $\Asm(A)$.

\end{lem}
\begin{proof}
This follows from Corollary $2.4.5$ of \cite{vanOosten2008} and Lemma \ref{asma}.

\end{proof}

\begin{defn}
Given a regular category $\cC$, we call an object $P$ of $\cC$ a {\bf projective object} if given a regular epimorphism $f: X \to Y$ and a morphism $P \to Y$, we have the following lift:

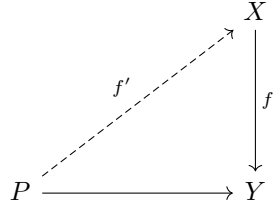
\[\begin{tikzcd}
	&&& X \\
	\\
	\\
	P &&& Y
	\arrow["f", from=1-4, to=4-4]
	\arrow["{f'}", dashed, from=4-1, to=1-4]
	\arrow[from=4-1, to=4-4]
\end{tikzcd}\]

\end{defn}

\begin{defn}

Given an assembly $(X, \phi)$, we called said assembly a {\bf partitioned assembly} if $(X,\phi)$ is isomorphic to an assembly $(X,\phi')$ where each $\phi'(x)$ is a singleton set. 

\end{defn}

We refer to the full subcategory of partitioned assemblies as $\pAsm(A)$.

\begin{lem}
$\pAsm(A)$ has finite limits and coproducts.

\end{lem}
\begin{proof}
Given that finite limits and coproducts exists in $\Asm(A)$, it suffice to show that the finite limits and coproducts of partitioned assemblies are still partitioned assemblies. Given partitioned assemblies $(X,[x \mapsto \{r_{x}\}])$ and $(Y,[y \mapsto \{r_{y}\}])$, note that the product of said assemblies is $(X\times Y, [(x,y) \mapsto \{[r_{x},r_{y}] \} = ])$ and the coproduct is $(X + Y, [x \mapsto \{[k,r_{x}] | a \in \phi(x)\}, y \mapsto \{[\overline{k},r_{y}] | b \in \psi(y)\})]$. One can see that both the product and coproduct are partitioned assemblies. Furthermore, given the following diagram:

\[\begin{tikzcd}
	{(X,[x \mapsto \{r_{x}\}])} &&& {(Y,[y \mapsto \{r_{y}\}])}
	\arrow["{f}"', shift right=2, from=1-1, to=1-4]
	\arrow["{g}", shift left=2, from=1-1, to=1-4]
\end{tikzcd}\]

the equalizer of the diagram is $(\{x | f(x) = g(x)\}, [x \mapsto \{r_{x}\}])$ which is a partitioned assembly. Therefore, $\pAsm(A)$ has finite limits and coproducts.

\end{proof}

\begin{lem}

For any assembly $(X,\phi)$, there exists an partitioned assembly $(X.\phi,\phi^{*})$ and a regular epimorphism $(X.\phi,\phi^{*}) \to (X,\phi)$. Furthermore, every assembly $(X,\phi)$ embeds into a partitioned assembly.

\end{lem}
\begin{proof}
Given an assembly $(X,\phi)$, we define the assembly $(X.\phi,\phi^{*})$ where $X.\phi = \Sigma_{x:X} \phi(x)$ and $\phi^{*}(x,a) = \{a\}$.  $(X.\phi,\phi^{*})$ is obviously a partitioned assembly, furthermore we have a morphism induced by first projection $\pr{X}:(X.\phi,\phi^{*}) \to (X,\phi)$ realized by $\iota$. Recall that $\pr{X}$ is a regular epimorphism iff and only if $\pr{X}$ is a surjection and $\exists_{\pr{X}}(\phi^{*}) \bimp \phi$. However $\exists_{\pr{X}}(\phi^{*})(x)  = \bigcup_{a \in \phi(x)} \{a\} = \phi(x)$ by definition of $(X.\phi,\phi^{*})$. Since $\pr{X}$ is surjection, therefore $\pr{X}$ is a regular epimorphism.\\

There is an obvious monomorphism $\id_{X}:(X,\phi) \to (X,\const_{\{\iota\}})$ realized by $\lmbd{z}{\iota}$.

\end{proof}

\begin{rmk}
We will refer to $(X.\phi,\phi^{*})$ as the {\bf partitioning } of $(X,\phi)$.

\end{rmk}

\begin{lem}

Any projective object $(X,\phi)$ is a partitioned assembly.

\end{lem}
\begin{proof}

Suppose $(X,\phi)$ is a projective object, since we have a regular epimorphism $\pr{X}:(X.\phi,\phi^{*}) \to (X,\phi)$ so we have a section of $\pr{X}$:
\[\begin{tikzcd}
	&&& (X.\phi, \phi^{*}) \\
	\\
	\\
	(X,\phi) &&& (X,\phi)
	\arrow["\pr{X}", from=1-4, to=4-4]
	\arrow["{s}", dashed, from=4-1, to=1-4]
	\arrow["\id_{(X,\phi)}"',from=4-1, to=4-4]
\end{tikzcd}\]

$s$ is realized by $r_{s} \in A$. We construct a partitioned assembly $(X, \phi')$ where $\phi'(x) = \{r_{s} \cdot a |  a \in \phi(x) \}$. Each $\phi'(x)$ is a singleton since $r_{s}$ realizes a morphism from $(X,\phi)$ into $(X.\phi,\phi^{*})$. Furthermore, $r_{s} \cdot a \in \phi(x)$ for $a \in \phi(x)$. So we have an isomorphism pair $\id_{X}: (X,\phi) \to (X,\phi')$ and $\id_{X}: (X,\phi') \to (X,\phi)$ realized by $r_{s}$ and $\iota$ respectively. Therefore, $(X,\phi)$ is a partitioned assembly.

\end{proof}


\begin{lem}
Suppose Axiom of Choice holds in $\Set$, then $\pAsm(A)$ is precisely the full subcategory of projective objects in $\Asm(A)$.

\end{lem}

\begin{proof}
It suffices to show that every partitioned assembly is a projective object assuming Axiom of Choice. Given the following:
\[\begin{tikzcd}
	&&& (X,\phi) \\
	\\
	\\
	(P,\pi) &&& (Y,\psi)
	\arrow["f", from=1-4, to=4-4]
	\arrow["g"',from=4-1, to=4-4]
\end{tikzcd}\]
 where $(P,\pi)$ is a partitioned assembly. Without loss of generality, assume that each $\pi(p)$ is a singleton set. Let $r_{f}$ and $r_{g}$ be the realizers of $f$ and $g$ respectively.  Suppose $f$ is a regular epimorphism, then we have a realizer $r_{re}$ for $\forall_{y \in Y} \psi(y) \imply \exists_{x \in X} (f(x) = y \land \phi(x))$. This means that for $y \in Y$ and $a \in \psi(y)$, $r_{re} \cdot a = [r_{x,y}, b]$ where $r_{x,y}$ denotes that existence of some $x \in X$ such that $f(x) = y$ and $b \in \phi(x')$ where $f(x') = y$. We construct a surjection $j:Q \to P$ where for $p \in P$, $j^{-1}(p) = \{x \in X | g(p) = f(x)\text{ and }p_{1} \cdot (r_{re} \cdot (r_{g} \cdot c)) \in \phi(x)\text{ where }\pi(p) = \{c\} \}$. By Axiom of Choice, j has a section $s:P \to Q$. This induces a morphism $s':(P,\pi) \to (X,\phi)$ where $s'(p) = s(p)$ realized by $\lmbd{z}{p_{1} \cdot (r_{re} \cdot (r_{g} \cdot z))}$. For $p \in P$, $s(p) \in X$ and $g(p) = f(s(p))$, thus $g  = f \circ s'$ making $(P,\pi)$ a projective object. Therefore, under Axiom of Choice the class of projective objects and partitioned assemblies coincide.

\end{proof}


\begin{lem}
Suppose Axiom of Choice holds in $\Set$, then $\Asm(A)$ is the reg/lex completion of $\pAsm(A)$.

\end{lem}
\begin{proof}

We know that every assembly is covered by a partitioned assembly and embeds into a partitioned assembly. By axiom of choice, every partitioned assembly is a projective object, so every assembly embeds into and is covered by a projective object. We also have that $\pAsm(A)$ is the subcategory of the projective objects of $\Asm(A)$. By proposition 9 of \cite{CARBONI199879} which states that any regular category $\cC$ is the reg/lex completion of the full subcategory of projective objects when any object of $\cC$, $X$, can be covered by and embeds into a projective object, this makes $\Asm(A)$ the reg/lex completion of $\pAsm(A)$.

\end{proof}

\begin{cor}

Assuming Axiom of Choice hold in $\Set$, $\RT{A}$ is the ex/lex completion of $\pAsm(A)$.

\end{cor}
\begin{proof}
 $\RT{A}$ is the ex/reg completion of $\Asm(A)$. With axiom of choice, $\Asm(A)$ is the reg/lex completion of $\pAsm(A)$. Therefore, $\RT{A}$ is the ex/lex completion of $\pAsm(A)$.

\end{proof}

\subsection{W-Types with Reductions}
Here we prove that category of assemblies has $W$-types with reductions.

\label{wtwr}

\begin{defn}
Suppose we have the following diagram in a locally cartesian closed category $\cC$ with pushouts:

\[\begin{tikzcd}
	R && X &&& Y \\
	\\
	& Z &&&&& Z
	\arrow["k", from=1-1, to=1-3]
	\arrow["f", from=1-3, to=1-6]
	\arrow["g"', from=1-3, to=3-2]
	\arrow["h", from=1-6, to=3-7]
\end{tikzcd}\]

We say that the diagram above satisfies the {\bf coherence condition} iff $h \circ f \circ k = g \circ k$. In which case, we call the above diagram a {\bf polynomial with reductions}. This induces an endofunctor $\poly_{f,g,h,k}: \cC/Z \to \cC/Z$ which is constructed as the following pushout:
\[\begin{tikzcd}
	{\sum_{g}\circ \sum_{k}\circ k^{*}\circ f^{*}\circ\Pi_{f}\circ g^{*}} &&&& {\Id_{\cC/Z}} \\
	\\
	\\
	\\
	{\poly_{f,g,h}} &&&& {\poly_{f,g,h,k}}
	\arrow[from=1-1, to=1-5]
	\arrow[from=1-1, to=5-1]
	\arrow[from=1-5, to=5-5]
	\arrow[from=5-1, to=5-5]
	\arrow["\lrcorner"{anchor=center, pos=0.125, rotate=180}, draw=none, from=5-5, to=1-1]
\end{tikzcd}\]

where the morphisms of the pushout diagram are as follows:

\[\begin{tikzcd}
	{\sum_{g}\circ \sum_{k}\circ k^{*}\circ f^{*}\circ\Pi_{f}\circ g^{*}} && {\sum_{g}\circ \sum_{k}\circ k^{*}\circ  g^{*}} && {\sum_{g}\circ g^{*}} && {\Id_{\cC/Z}}
	\arrow[from=1-1, to=1-3]
	\arrow[from=1-3, to=1-5]
	\arrow[from=1-5, to=1-7]
\end{tikzcd}\]

where each morphism is constructed using the counits $f^{*} \circ \Pi_{f} \To \Id_{\cC/X}$, $\sum_{k} \circ k^{*} \To \Id_{\cC/X}$, and $\sum_{g} \circ g^{*} \To \Id_{\cC/Z}$ respectively and 

\[\begin{tikzcd}
	 {\sum_{h} \circ \sum_{f}\circ \sum_{k}\circ k^{*}\circ f^{*}\circ\Pi_{f}\circ g^{*}} && {\sum_{h} \circ \sum_{f}\circ f^{*}\circ\Pi_{f}\circ g^{*}} && {\poly_{f,g,h}}
	\arrow[from=1-1, to=1-3]
	\arrow[from=1-3, to=1-5]
\end{tikzcd}\]

where $\sum_{h} \circ \sum_{f}\circ \sum_{k} = \sum_{g} \circ \sum_{k}$ by the coherence condition and the morphism is constructed by counits $\sum_{k} \circ k^{*} \To \Id_{\cC/X}$ and $\sum_{f} \circ f^{*} \To \Id_{\cC/Y}$ respectively.\\

The initial algebra of $\poly_{f,g,h,k}$ is called the {\bf W-type with reductions} for $\poly_{f,g,h,k}$.

\end{defn}


\begin{defn}

Given the commutative square in a category $\cC$:
\[\begin{tikzcd}
	A &&& B \\
	\\
	\\
	C &&& D
	\arrow["q", from=1-1, to=1-4]
	\arrow["g"', from=1-1, to=4-1]
	\arrow["f", from=1-4, to=4-4]
	\arrow["p"', from=4-1, to=4-4]
\end{tikzcd}\]

we called this square {\bf covering} iff $p$ is a regular epimorphism and $[g,q]: A \epi B \times_{D} C$ is a regular epimorphism. \\

We call the above square a {\bf collection} square iff in the internal language of $\cC$: for all $d \in D$, and for every cover $e:E \epi f^{-1}(d)$ , there exists a $c \in p^{-1}(d)$ and a morphism $t: g^{-1}(c) \to E$ such that $q_{d,c} = e \circ t$ where $q_{d,c} = q|_{g^{-1}(c)}$ with codomain $f^{-1}(d)$.

\end{defn}

\begin{defn}
\label{wiscasm}

Given a regular category $\cC$, we say that $\cC$ satisfies the {\bf weakly initial set of covers} axiom or {\bf WISC} iff for any morphism $f: X \to Y$ in $\cC$ there exists a covering and collection square: 

\[\begin{tikzcd}
	A &&& X \\
	\\
	\\
	B &&& Y
	\arrow[ from=1-1, to=1-4]
	\arrow[ from=1-1, to=4-1]
	\arrow["f", from=1-4, to=4-4]
	\arrow[ from=4-1, to=4-4]
\end{tikzcd}\]

or in other words, $f$ admits a {\bf cover base}.

\end{defn}

\begin{lem}
\label{swan}
Given a covering and collection square:

\[\begin{tikzcd}
	A &&& X \\
	\\
	\\
	B &&& Y
	\arrow["h", from=1-1, to=1-4]
	\arrow[ "k"',from=1-1, to=4-1]
	\arrow["f", from=1-4, to=4-4]
	\arrow["g"', from=4-1, to=4-4]
\end{tikzcd}\]

Suppose that for $y \in Y$, we have a surjection $t: Z \epi f^{-1}(y)$, then there exists $b \in g^{-1}(y)$ and an element
$\alpha \in \Pi_{a \in k^{-1}(b)} t^{-1}(h(a))$.

\end{lem}
\begin{proof}
This is Lemma 3.6 of \cite{swan2018wtypes}.
\end{proof}

For notation, we will refer to $f:(X,\phi) \to (Y,\psi)$ as $f^{-1}(y) =  (X,\phi)_{y}$ for $y \in (Y,\psi)$ in the internal language of $\Asm(A)$ when convenient.

\begin{lem}
\label{wisc}

If WISC holds in $\Set$, then it holds in $\Asm(A)$.

\end{lem}
\begin{proof}

Suppose we have a morphism $f: (X,\phi) \to (Y,\psi)$, we construct a pullback square:

\[\begin{tikzcd}
	{(X.\phi,\phi^{*}) \times_{(Y,\psi)}(Y.\psi,\psi^{*} )} &&&& {(X.\phi,\phi^{*})} \\
	\\
	\\
	\\
	{(Y.\psi,\psi^{*} )} &&&& {(Y,\psi)}
	\arrow["{\pi_{X}}", from=1-1, to=1-5]
	\arrow["{\pi_{Y}}"', from=1-1, to=5-1]
	\arrow["\lrcorner"{anchor=center, pos=0.125}, draw=none, from=1-1, to=5-5]
	\arrow["{f \circ \pr{X}}", from=1-5, to=5-5]
	\arrow["{{\pr{Y}}}"', from=5-1, to=5-5]
\end{tikzcd}\]

where the bottom morphism is the partitioning of $(Y,\psi)$ and $(X.\phi,\phi^{*}) \times_{(Y,\psi)}(Y.\psi,\psi^{*} ) = ((X.\phi) \times_{Y} (Y.\psi), \phi^{*} \times \psi^{*} )$ such that $\phi^{*} \times \psi^{*} ((x,a),(f(x),a')) = \{[a,a']\}$. This induces a square:

\[\begin{tikzcd}
	{((X.\phi) \times_{Y} (Y.\psi), \phi^{*} \times \psi^{*} )} &&&& {(X,\phi)} \\
	\\
	\\
	\\
	{(Y.\psi,\psi^{*} )} &&&& {(Y,\psi)}
	\arrow["{{\pr{X} \circ \pi_{X}}}", from=1-1, to=1-5]
	\arrow["{{\pi_{Y}}}"', from=1-1, to=5-1]
	\arrow["{{f }}", from=1-5, to=5-5]
	\arrow["{{{\pr{Y}}}}"', from=5-1, to=5-5]
\end{tikzcd}\]

$\pr{Y}: (Y.\psi,\psi^{*} ) \to (Y,\psi)$ is a regular epimorphism and the canonical morphism $$(\pi_{Y},\pr{X} \circ \pi_{X}): ((X.\phi) \times_{Y} (Y.\psi), \phi^{*} \times \psi^{*} ) \to (X,\phi) \times_{(Y,\psi)} (Y.\psi,\psi^{*})$$ is a regular epimorphism since for a realizer $[a,a']$ for the pair $(x,(f(x),a'))$ is a realizer for $((x,a),(f(x),a'))$ in $((X.\phi) \times_{Y} (Y.\psi), \phi^{*} \times \psi^{*} )$. Now using WISC in $\Set$, we can construct a covering and collection square:

\[\begin{tikzcd}
	J &&&& {(X.\phi) \times_{Y} (Y.\psi)} \\
	\\
	\\
	\\
	I &&&& {Y.\psi}
	\arrow["{q_{0}}", from=1-1, to=1-5]
	\arrow["p"', from=1-1, to=5-1]
	\arrow["{{{\pi_{Y}}}}", from=1-5, to=5-5]
	\arrow["{q_{1}}"', from=5-1, to=5-5]
\end{tikzcd}\]

this gives a new diagram in $\Asm(A)$:

\[\begin{tikzcd}
	{(J,(\phi^{*} \times \psi^{*}) \circ q_{0})} &&&& {((X.\phi) \times_{Y} (Y.\psi), \phi^{*} \times \psi^{*} )} &&&& {(X,\phi)} \\
	\\
	\\
	\\
	{(I,\psi^{*} \circ q_{1})} &&&& {(Y.\psi,\psi^{*} )} &&&& {(Y,\psi)}
	\arrow["{q_{0}}", from=1-1, to=1-5]
	\arrow["p"', from=1-1, to=5-1]
	\arrow["{{{\pr{X} \circ \pi_{X}}}}", from=1-5, to=1-9]
	\arrow["{{{\pi_{Y}}}}"', from=1-5, to=5-5]
	\arrow["{{{f }}}", from=1-9, to=5-9]
	\arrow["{q_{1}}"', from=5-1, to=5-5]
	\arrow["{{{{\pr{Y}}}}}"', from=5-5, to=5-9]
\end{tikzcd}\]

The square on the right is a covering square since $q_{1}$ is a regular epimorphism and the canonical morphism 

$$(p, q_{0}):(J,(\phi^{*} \times \psi^{*}) \circ q_{0}) \to ((X.\phi) \times_{Y} (Y.\psi), \phi^{*} \times \psi^{*} ) \times_{(Y.\psi,\psi^{*} )} (I,\psi^{*} \circ q_{1}) $$

is a regular epimorphism since for a realizer $[[b,b'],b']$ for $(((x,b),(y,b')),i)$, there exists a $j \in J$ such that $p(j) = i$ and $q_{0}(j) = ((x,b),(y,b'))$ so $[b,b']$ is a realizer for $j$. Hence the above diagram induces a covering square for $f$. For $y \in Y$, and realizer $a \in \psi(y)$ and a regular epimorphism $g:(E,\epsilon) \epi (X,\phi)_{y}$ since $(X,\phi)_{y} = (X_{y},\phi)$ we have a pullback square:

\[\begin{tikzcd}
	{(T, \tau)} &&& {(E,\epsilon)} \\
	\\
	\\
	{(X_{y}.\phi,\phi^{*}) \times(\{(y,a)\},\psi^{*})} &&& {(X,\phi)_{y} }
	\arrow["v", from=1-1, to=1-4]
	\arrow["{g'}"', from=1-1, to=4-1]
	\arrow["\lrcorner"{anchor=center, pos=0.125}, draw=none, from=1-1, to=4-4]
	\arrow["g", from=1-4, to=4-4]
	\arrow["{\pr{X} \circ \pi_{X}}"', from=4-1, to=4-4]
\end{tikzcd}\]

where $((X.\phi) \times_{Y} (Y.\psi), \phi^{*} \times \psi^{*} )_{(y,a)} = (X_{y}.\phi,\phi^{*}) \times(\{(y,a)\},\psi^{*})$ and $g'$ is also a regular epimorphism. Let $re \in A$ realizes $g'$ as a regular epimorphism. We construct an assembly $(T',\tau) \mono (T,\tau)$ where $$T' = \{t \in T| \exists a \in \phi^{*} \times \psi^{*}(g'(t))( re \cdot a \in \tau(t))  \}$$ This induces a regular epimorphism $g'': (T',\tau)  \epi (X_{y}.\phi,\phi^{*}) \times(\{(y,a)\},\psi^{*})$ where $g''(t) = g'(t)$. Using the covering square we constructed in $\Set$:

\[\begin{tikzcd}
	J &&&& {(X.\phi) \times_{Y} (Y.\psi)} \\
	\\
	\\
	\\
	I &&&& {Y.\psi}
	\arrow["{q_{0}}", from=1-1, to=1-5]
	\arrow["p"', from=1-1, to=5-1]
	\arrow["{{{\pi_{Y}}}}", from=1-5, to=5-5]
	\arrow["{q_{1}}"', from=5-1, to=5-5]
\end{tikzcd}\]

there exists an $i \in I$ such that $q_{1}(i) = (y,a)$ and a function $w:J_{i} \to T'$ such that $g'' \circ w = q_{0}|_{J_{i}}$. Now we need to extend this to a morphism $w: (J_{i}, (\phi^{*} \times \psi^{*}) \circ q_{0}) \to (T',\tau)$. Given $j \in J_{i}$ and $[b,a'] \in \phi^{*} \times \psi^{*}(q_{0}(j))$, we have $a' = a$ since $q_{1} \circ p(j) = \pi_{Y} \circ q_{0}(j) = (y,a)$. We have $g''(w(j)) = ((x,b), (y,a))$ and $[b,a]$ is the unique realizer of $((x,b), (y,a))$ in $(X_{y}.\phi,\phi^{*}) \times(\{(y,a)\},\psi^{*})$ and of $j$ in $(J_{i}, (\phi^{*} \times \psi^{*}) \circ q_{0})$. Furthermore, for $t \in T'$, $g''(t) = ((x,b), (y,a))$ implies $re \cdot [b,a] \in \tau(t)$ hence $re \cdot [b,a] \in \tau(w(j))$. So $w$ is realized by $re$. $w$ then induces a morphism $w': (J_{i}, (\phi^{*} \times \psi^{*}) \circ q_{0}) \to (E,\epsilon)$ where $w'(j) = v \circ w(j)$. So we have $$ g \circ w' = g \circ v \circ w = \pr{X} \circ \pi_{X} \circ g' \circ w = \pr{X} \circ \pi_{X} \circ q_{0}|_{J_{i}} = (\pr{X} \circ \pi_{X} \circ q_{0})|_{J_{i}}$$ So we have covering and collection square:

\[\begin{tikzcd}
	{(J,(\phi^{*} \times \psi^{*}) \circ q_{0})} &&&& {(X,\phi)} \\
	\\
	\\
	\\
	{(I,\psi^{*} \circ q_{1})} &&&& {(Y,\psi)}
	\arrow["{\pr{X} \circ \pi_{X} \circ q_{0}}", from=1-1, to=1-5]
	\arrow["p"', from=1-1, to=5-1]
	\arrow["{{{{f }}}}", from=1-5, to=5-5]
	\arrow["{{\pr{Y} \circ q_{1}}}"', from=5-1, to=5-5]
\end{tikzcd}\]

which completes the proof.

\end{proof}

We seek to give a construction of W-types with reductions as shown in \cite{swan2018wtypes} in $\Asm(A)$ for partial combinatory algebra $A$. This construction requires that WISC holds in $\Asm(A)$ which we have by Lemma \ref{wisc}. This construction also requires $\Asm(A)$ be an exact category:

\begin{defn}
Given a regular category $\cC$ and a monomorphism $R \mono X \times X$ in $\cC$, we call $R$ a {\bf congruence} on $X$ if $R$ is an internal equivalence relation in $\cC$. \\

$\cC$ is an {\bf exact category} if any congruence $R \mono X \times X$ is a kernel pair or in other words we can express $R$ as the following pullback:

\[\begin{tikzcd}
	{R} &&& {X} \\
	\\
	\\
	{X} &&& {Y}
	\arrow[ "\pr{0}", from=1-1, to=1-4]
	\arrow["\pr{1}"', from=1-1, to=4-1]
	\arrow["\lrcorner"{anchor=center, pos=0.125}, draw=none, from=1-1, to=4-4]
	\arrow["f", from=1-4, to=4-4]
	\arrow["f"', from=4-1, to=4-4]
\end{tikzcd}\]

\end{defn}

However, $\Asm(A)$ is not exact as shown in the following example:

\begin{ex}

Suppose we have a set $ \{1,2,3\}$ and the following congruence $\Delta':( \{1,2,3\}, \overline{\{-\}}) \mono ( \{1,2,3\}, \const_{A}) \times ( \{1,2,3\},\const_{A})$ where $\overline{n}$ is church numeral for $n$. This gives us a diagram

\[\begin{tikzcd}
	{( \{1,2,3\}, \overline{\{-\}}) } &&& {( \{1,2,3\},\const_{A})}
	\arrow["{\id}"', shift right=2, from=1-1, to=1-4]
	\arrow["{\id}", shift left=2, from=1-1, to=1-4]
\end{tikzcd}\]

which gives us the regular epimorphism $\id_{( \{1,2,3\},\const_{A})} $ whose kernel pair is a pair of the same identity. This corresponds to the congruence $\Delta:( \{1,2,3\}, \const_{A}) \mono ( \{1,2,3\}, \const_{A}) \times ( \{1,2,3\},\const_{A})$. $( \{1,2,3\}, \const_{A})$ and $( \{1,2,3\}, \overline{\{-\}})$ are not equivalent since you would need to find a realizer mapping all of $A$ into three different constants. Thus $\Delta'$ is not a kernel pair.

\end{ex}

Though $\Asm(A)$ itself is not exact, we can still leverage the fact the the underlying category of $\Set$ is exact. We will use a construction similar to the construction of $W$-types in \cite{Berg2006} as well as give an overview of the construction of $W$-types with reductions in $\Set$. \\


First let us have a $W$-type with reductions diagram:

\[\begin{tikzcd}
	(R, \rho) && (X, \phi) &&& (Y,\psi) \\
	\\
	& (Z,\zeta) &&&&& (Z,\zeta)
	\arrow["k", from=1-1, to=1-3]
	\arrow["f", from=1-3, to=1-6]
	\arrow["h"', from=1-3, to=3-2]
	\arrow["g", from=1-6, to=3-7]
\end{tikzcd}\]

which we will call $D$. We have a covering and collection square using the construction in Lemma \ref{wisc}:

\[\begin{tikzcd}
	{(J,(\phi^{*} \times \psi^{*}) \circ q_{0})} &&&& {(X,\phi)} \\
	\\
	\\
	\\
	{(I,\psi^{*} \circ q_{1})} &&&& {(Y,\psi)}
	\arrow["{\pr{X} \circ \pi_{X} \circ q_{0}}", from=1-1, to=1-5]
	\arrow["p"', from=1-1, to=5-1]
	\arrow["{{{{f }}}}", from=1-5, to=5-5]
	\arrow["{{\pr{Y} \circ q_{1}}}"', from=5-1, to=5-5]
\end{tikzcd}\]

and a diagram:

\[\begin{tikzcd}
	& {(J,(\phi^{*} \times \psi^{*}) \circ q_{0})} &&& {(I,\psi^{*} \circ q_{1})} \\
	\\
	{(Z,\zeta)} &&&&& {(Z,\zeta)}
	\arrow["p", from=1-2, to=1-5]
	\arrow["{h \circ\pr{X} \circ \pi_{X} \circ q_{0}}"{description}, from=1-2, to=3-1]
	\arrow["{g \circ \pr{Y} \circ q_{1}}"{description}, from=1-5, to=3-6]
\end{tikzcd}\]

which we call $D_{w}$. It is not difficult to show that this induces a covering and collection square in $\Set$:

\[\begin{tikzcd}
	J &&&& X \\
	\\
	\\
	\\
	I &&&& Y
	\arrow["{{\pr{X} \circ \pi_{X} \circ q_{0}}}", from=1-1, to=1-5]
	\arrow["p"', from=1-1, to=5-1]
	\arrow["{{{{{f }}}}}", from=1-5, to=5-5]
	\arrow["{{{\pr{Y} \circ q_{1}}}}"', from=5-1, to=5-5]
\end{tikzcd}\]

as any surjection $u:E \epi X_{y}$ induces a regular epimorphism $u: (E, \phi \circ u) \to (X_{y}, \phi)$. Using the construction of $W$-types with reductions in \cite{swan2018wtypes}, we have morphisms $c:W_{base} \to Z$ and $c_{r}:W_{f,g,h,r} \to Z$ so that $c_{r}$ is the initial algebra of the underlying diagram of $D$ in $\Set$, $c$ is the initial algebra of the underlying diagram of $D_{w}$ in $\Set$. Furthermore $\alpha_{p,g,h}$ is the algebra structure for $W_{base}$ and $\alpha_{f,g,h,k}$ is the algebra structure of $W_{f,g,h,r}$. Follow the steps of the construction, we have a covering and collection square:

\[\begin{tikzcd}
	M &&& {J \times_{X} J} \\
	\\
	\\
	N &&& {I \times_{Y} I}
	\arrow["{s_{0}}", from=1-1, to=1-4]
	\arrow[ from=1-1, to=4-1]
	\arrow["{<p,p>}", from=1-4, to=4-4]
	\arrow["{s_{1}}"', from=4-1, to=4-4]
\end{tikzcd}\]

and a dependent $W$-type $T \to W_{base} \times_{Z} W_{base}$ generated by the following steps:

\begin{itemize}

\item If $w' \in W_{base, z}$, $t_{1} \in T_{(w_{0},w')}$ and $t_{2} \in T_{(w',w_{1})}$, then we have that $trans(q_{1},q_{2}) \in T_{(w_{0},w_{1})}$.\\

\item If $y \in Y_{z}$,$i \in I_{y}$, and $\beta \in \Pi_{j \in J_{i}} W_{base, h \circ \pr{X} \circ \pi_{X} \circ q_{0}(j)}$, and $w_{0} = \alpha_{p,g,h}(z,i,\beta)$ and there is $x \in X_{y}$,  $r \in R_{x}$, $j \in J_{x}$, $ n \in N_{(i,i)}$, and $\gamma \in \Pi_{m \in M_{n}} T_{\beta(\pr{0}(s_{0}(m))),\beta(\pr{1}(s_{0}(m)))}$, then we have that $rl(j,n,\beta,\gamma) \in T _{(w_{0},\beta(j))}$ and $rr(j,n,\beta,\gamma) \in T_{(\beta(j),w_{0})}$.\\

\item If $y \in Y_{z}$, $i_{0},i_{1} \in I_{y}$, $n \in N_{(i_{0},i_{1})}$, $\beta_{0} \in \Pi_{j \in J_{i_{0}}} W_{base, h \circ \pr{X} \circ \pi_{X} \circ q_{0}(j)}$ and $\beta_{1} \in \Pi_{j \in J_{i_{1}}} W_{base, h \circ \pr{X} \circ \pi_{X} \circ q_{0}(j)}$ where $w_{0} = \alpha_{p,g,h}(z,i_{0},\beta_{0})$ and $w_{1} = \alpha_{p,g,h}(z,i_{1},\beta_{1})$ and $\gamma \in \Pi_{m \in M_{n}} T_{\beta_{0}(\pr{0}(s_{0}(m))),\beta_{1}(\pr{1}(s_{0}(m)))}$, then we have that $ex(\beta_{0},\beta_{1},\gamma) \in T_{(w_{0},w_{1})}$.\\

\end{itemize}

The $W$-type $T$ induces a partial equivalence relation $\sim$ on each $W_{base,z}$ which can be shown to be generated under the following conditions using the properties of covering and collection squares and surjections:

\begin{itemize}

\item $w \sim w'$  and  $w' \sim w''$ implies $w \sim w''$.\\

\item For $w = \alpha_{p,g,h}(i,t:\Pi_{j \in J_{i}} W_{base,h \circ \pr{X} \circ \pi_{X} \circ q_{0}(j)})$ and $w' = \alpha_{p,g,h}(i',t':\Pi_{j \in J_{i'}} W_{base,h \circ \pr{X} \circ \pi_{X} \circ q_{0}(j)})$ with $i, i' \in I_{y}$, if for all $j \in J_{i}$ and $j' \in J_{i'}$ with $j, j' \in J_{x}$ we have $t(j) \sim t'(j')$ then $w \sim w'$.\\

\item For $w = \alpha_{p,g,h}(i,t:\Pi_{j \in J_{i}} W_{base,h \circ \pr{X} \circ \pi_{X} \circ q_{0}(j)})$ such that $w \sim w$ by the above condition and $j \in J_{i}$ if there exists an $x \in X$ and $r \in R_{x}$ such that $j \in J_{x}$, then $w \sim t(j)$.

\end{itemize}

We have that $w \sim w'$ implies $c(w) = c(w')$. We can derive a subset $W'_{base} \mono W_{base}$ of elements $w \in W_{base}$ such that $w \sim w$ and $W_{f,g,h,r} = W'_{base}/\sim  $. \\

This means we have a surjection $s:W'_{base} \to W_{f,g,h,r}$ such that $c_{r} \circ s(w) = c(w)$. Recall that $D_{w}$ has an initial algebra $c': (W_{p,g,h}, \omega) \to (Z,\zeta)$ where $W_{p,g,h}$ is the subset of $W_{p}$, the initial algebra of the underlying function $p$ in $\Set$,  of the well founded $(f,g,h)$-trees that can be realized by an element in $A$. Of course we have $W_{p,g,h} \mono W_{base}$ since $W_{base}$ is the subset of $W_{p}$ of well founded $(f,g,h)$-trees. Let $\sim'$ be the partial equivalence relation on $W_{p,g,h} \times_{Z} W_{p,g,h}$ generated as follows

\begin{itemize}

\item For $w,w'' \in W_{p,g,h,z}$ if there exists $w' \in W_{p,g,h,z}$ such that $w \sim' w'$  and  $w' \sim' w''$ then $w \sim' w''$.\\

\item For $w = \alpha_{p,g,h}(i,t:\Pi_{j \in J_{i}} W_{p,g,h,h \circ \pr{X} \circ \pi_{X} \circ q_{0}(j)})$ and $w' = \alpha_{p,g,h}(i',t':\Pi_{j \in J_{i'}} W_{p,g,h,h \circ \pr{X} \circ \pi_{X} \circ q_{0}(j)})$ with $i, i' \in I_{y}$, if for all $j \in J_{i}$ and $j' \in J_{i'}$ with $j, j' \in J_{x}$ we have $t(j) \sim' t'(j')$ then $w \sim' w'$.\\

\item For $w = \alpha_{p,g,h}(i,t:\Pi_{j \in J_{i}} W_{p,g,h,h \circ \pr{X} \circ \pi_{X} \circ q_{0}(j)})$ such that $w \sim w$ by the above condition and $j \in J_{i}$ if there exists an $x \in X$ and $r \in R_{x}$ such that $j \in J_{x}$, then $w \sim' t(j)$.

\end{itemize}

 Let $W'_{p,g,h}  = \{w \in W_{p,g,h} | w \sim' w \}$. We set $c^{*}: W'_{p,g,h} \to Z$ be the restriction of $c'$ to $W'_{p,g,h}$. We set $W'_{f,g,h,r} =  W'_{p,g,h}/\sim'$. So we have a surjection $s': W'_{p,g,h} \to W'_{f,g,h,r}$ with $w \sim' w'$ implying $s'(w) = s(w')$ and a unique morphism $c'_{r}: W'_{f,g,h,r} \to Z$ such that $c^{*} = c'_{r} \circ s'$. We construct the assembly $(W'_{f,g,h,r}, \exists_{s'}(\omega))$ where $\exists_{s'}(\omega)(w) = \bigcup_{w' \in s'^{-1}(w)} \omega(w')$ and a morphism $c'_{r}: (W'_{f,g,h,r}, \exists_{s}(\omega)) \to (Z, \zeta)$. $c'_{r}$ is realized by the same realizer as $c'$. We will show that $c'_{r}$ is a $\poly_{f,g,h,k}$-algebra. First we construct a morphism $\alpha: \poly_{f,g,h}(c'_{r}) \to c'_{r}$. Suppose we have $(y, t: (X,\phi)_{y} \to (W'_{f,g,h,r}, \exists_{s}(\omega)))$ realized by $[r_{y},r_{t}]$, we have a regular epimorphism $s': (W'_{p,g,h}, \omega) \to (W'_{f,g,h,r}, \exists_{s}(\omega))$ where $\iota$ realizes both $s'$ and the fact that $s'$ is a regular epimorphism. We can find a $i \in I$ so that $\psi^{*} \circ q_{1}(i) = \{r_{y}\}$ and commutative square:

\[\begin{tikzcd}
	{(J,(\phi^{*} \times \psi^{*}) \circ q_{0})_{i}} &&&& {(W'_{p,g,h},\omega)} \\
	\\
	\\
	\\
	{(X,\phi)_{y} } &&&& {(W'_{f,g,h,r},\exists_{s'}(\omega))}
	\arrow["{t'}", from=1-1, to=1-5]
	\arrow["{\pr{X} \circ \pi_{X} \circ q_{0}}"{description}, from=1-1, to=5-1]
	\arrow["s'", from=1-5, to=5-5]
	\arrow["t"', from=5-1, to=5-5]
\end{tikzcd}\]

so that $t'$ is realized by $\lmbd{b}{r_{t} \cdot (p_{0} \cdot b)}$ by examining the covering and collection square construction in Lemma \ref{wisc} and by the fact that regular epimorphisms are preserved by pullback.

Suppose we have another such square:

\[\begin{tikzcd}
	{(J,(\phi^{*} \times \psi^{*}) \circ q_{0})_{i'}} &&&& {(W'_{p,g,h},\omega)} \\
	\\
	\\
	\\
	{(X,\phi)_{y} } &&&& {(W'_{f,g,h,r},\exists_{s'}(\omega))}
	\arrow["{t''}", from=1-1, to=1-5]
	\arrow["{\pr{X} \circ \pi_{X} \circ q_{0}}"{description}, from=1-1, to=5-1]
	\arrow["s'", from=1-5, to=5-5]
	\arrow["t"', from=5-1, to=5-5]
\end{tikzcd}\]

then for $j \in J_{i}$ and $j' \in J_{i'}$ such that $\pr{X} \circ \pi_{X} \circ q_{0}(j) = \pr{X} \circ \pi_{X} \circ q_{0}(j')$ we have $s'(t'(j)) = s'(t''(j'))$. Since $\Set$ is an exact category then $t'(j) \sim' t''(j')$. Thus $\alpha_{p,g,h}(i,t') \sim' \alpha_{p,g,h}(i',t'')$ and so $s'(\alpha_{p,g,h}(i,t')) = s'(\alpha_{p,g,h}(i',t''))$. We set $\alpha(y,t) = s'(t')$ which is realized by the combinator taking $[r_{y},r_{t}]$ to $[r_{y}, \lmbd{b}{r_{t} \cdot (p_{0} \cdot b)}]$. Furthermore, $\alpha$ induces a $\poly_{f,g,h,k}$-algebra structure on $c'_{r}$ by the third condition of $\sim'$ and since for $x \in X_{y}$ and $i \in I_{y}$ there exists $j \in J$ such that $\pr{X} \circ \pi_{X} \circ q_{0}(j) = x$ and $p(j) = i$.\\

Suppose we have a $\poly_{f,g,h,k}$-algebra $v: (D, \delta) \to (Z,\zeta)$ with algebra structure $\beta: \poly_{f,g,h}(v) \to v$, we give a construction of the relation $N \mono W'_{p,g,h} \times_{Z}  D$ which is similar to the construction found in section 3.2.5 of \cite{swan2018wtypes}:\\

Given $(w,d) \in W'_{p,g,h} \times_{Z}  D$ if we have:
\begin{itemize}
\item $w = \alpha_{p,g,h}(i,k:  \Pi_{j \in (J,(\phi^{*} \times \psi^{*}) \circ q_{0})_{i} } (W'_{p,g,h},\omega)_{h \circ \pr{X} \circ \pi_{X} \circ q_{0}(j)} )$\\

\item $k' : \Pi_{x \in (X,\phi)_{\pr{Y} \circ q_{1}(i)}} (D,\delta)_{h(x)}$ such that $d = \beta(\pr{Y} \circ q_{1}(i), k')$\\

\item For all $j \in J_{i}$,  $(k(j),k'(\pr{X} \circ \pi_{X} \circ q_{0}(j))) \in N$\\

\end{itemize}

then  $(w,d) \in N$. Let $r_{\beta}$ be a realizer for $\beta$. Given some $w \in W'_{p,g,h}$ and $r \in \omega(w)$, we seek to construct a realizer $I \cdot r$ for some element $d \in (D,\delta)$. Recall that $$w = \alpha_{p,g,h}(i,k:  \Pi_{j \in (J,(\phi^{*} \times \psi^{*}) \circ q_{0})_{i} } (W_{p,g,h},\omega)_{h \circ \pr{X} \circ \pi_{X} \circ q_{0}(j)} )$$ where $i$ is realized by a unique realizer $r_{y} \in \psi(\pr{Y} \circ q_{1}(i))$ and $k$ is realized by some $r_{k}$ which takes $[r_{y},r_{x}]$ where $f(x) = y = \pr{Y} \circ q_{1}(i)$ to a realizer of $k(j)$ such that $[r_{y},r_{x}]$ is the unique realizer of $j \in J_{i}$ and $\pr{X} \circ \pi_{X} \circ q_{0}(j) = x$ with $r_{x} \in \phi(x)$. every realizer of $w$ is of the form $[r_{y},r_{k}]$. Also remember that for $x \in X_{y}$ and $r_{x} \in \phi(x)$, there exists $j \in J_{i}$ realized uniquely by $[r_{y}],r_{x}$. With all this mind we construct the realizer:

$$uni :=   \lmbd{f,l}{r_{\beta} \cdot [p_{0} \cdot l, \lmbd{l'}{f \cdot ((p_{1} \cdot l) \cdot [p_{0} \cdot l, l'])} ]}$$
So we have $UM =  \fixf \cdot uni$. When we have a realizer $[r_{y},r_{k}]$ of $w$, we have:
\begin{align*}
UM \cdot [r_{y},r_{k}] &= (\fixf \cdot uni) \cdot [r_{y},r_{k}]\\
&= uni \cdot (\fixf \cdot uni) \cdot [r_{y},r_{k}]\\
&= r_{\beta} \cdot [p_{0} \cdot [r_{y},r_{k}], \lmbd{l'}{(\fixf \cdot uni) \cdot ((p_{1} \cdot [r_{y},r_{k}]) \cdot [ p_{0} \cdot [r_{y},r_{k}],  l'])} ]\\
&= r_{\beta} \cdot [r_{y}, \lmbd{l'}{(\fixf \cdot uni) \cdot (r_{k} \cdot [r_{y},l'])} ]\\
&=  r_{\beta} \cdot [r_{y}, \lmbd{l'}{UM \cdot (r_{k} \cdot [r_{y},l'])} ]\\
\end{align*}

Ideally we want $ [r_{y}, \lmbd{l'}{UM \cdot (r_{k} \cdot [r_{y},l'])} ]$ to be a realizer for a pair $(y, k' : \Pi_{x \in (X,\phi)_{\pr{Y} \circ q_{1}(i)}} (D,\delta)_{h(x)})$ in $\poly_{f,g,h}(v)$. We now seek to prove the following lemma:

\begin{lem}

For $w \in W'_{p,g,h}$, there exists a unique $d \in D$ such that $(w,d) \in N$. In addition, if $r_{w}$ realizes $w$ then $UM \cdot r_{w}$ realizes $d$ and if $w \sim' w'$ for $w' \in W'_{p,g,h}$ then both elements have the same unique $d$.

\end{lem}
\begin{proof}

We will prove this by cases on $\sim'$.  Suppose we have $w \sim' w'$ by transitivity then there exists $w'' \in W'_{p,g,h}$ such that $w \sim' w''$ and $w'' \sim' w'$. Suppose the Lemma holds for both $w \sim' w''$ and $w'' \sim' w'$ then obviously it holds for $w \sim' w'$.\\

Suppose $w \sim' w'$ holds by the second case then we have 

$$w = \alpha_{p,g,h}(i,k:  \Pi_{j \in (J,(\phi^{*} \times \psi^{*}) \circ q_{0})_{i} } (W_{p,g,h},\omega)_{h \circ \pr{X} \circ \pi_{X} \circ q_{0}(j)} )$$ and $$w' = \alpha_{p,g,h}(i',k':  \Pi_{j \in (J,(\phi^{*} \times \psi^{*}) \circ q_{0})_{i'} } (W_{p,g,h},\omega)_{h \circ \pr{X} \circ \pi_{X} \circ q_{0}(j)} )$$ such that for $j \in J_{i}$ and $j' \in J_{i'}$ such that $j,j' \in J_{x}$ for $x \in X_{y}$ then $k(j) \sim k'(j')$. Suppose that the lemma holds for all such cases by induction, then we construction a function $k^{*}: X_{y} \to D$ where $k^{*}(x)$ is the unique $d$ for all $k(j)$ and $k'(j')$ with $j,j' \in J_{x}$. Furthermore $k^{*}(x)  \in D_{h(x)}$ by construction of $N$, so we have $k^{*}: \Pi_{x \in (X,\phi)_{y}} (D,\delta)_{h(x)}$. \\

Let $[r_{y},r_{k}]$ and $[r'_{y},r_{k'}]$ be realizers for $w$ and $w'$ respectively then $r_{y}$ and $r'_{y}$ are both realizers for $y$. Take a realizer $r_{x}$ for $x \in X_{y}$, then so that $[r_{y},r_{x}]$ and $[r'_{y},r_{x}]$ realizes both some $j \in J_{i}$ and $j' \in J_{i}$ so that $j, j' \in J_{x}$. Then $r_{k} \cdot [r_{y},r_{x}]$ realizes $k(j)$ and $r_{k'} \cdot [r'_{y},r_{x}]$ realizes $k'(j')$. Thus both $UM \cdot (r_{k} \cdot [r_{y},r_{x}]) $ and $UM \cdot (r_{k'} \cdot [r'_{y},r_{x}]) $ realizes $k^{*}(x)$. So both $\lmbd{l'}{UM \cdot (r_{k} \cdot [r_{y},l'])} $ and $\lmbd{l'}{UM \cdot (r_{k'} \cdot [r'_{y},l'])} $ realizes $k^{*}: (X, \phi)_{y} \to (D, \delta)$. So $UM \cdot [r_{y},r_{k}]$ and $UM \cdot [r'_{y},r_{k'}]$ realizes $\beta(y,k^{*})$. $\beta(y,k^{*})$ is unique for both $w$ and $w'$ by construction of $N$.\\

Suppose $w \sim' w'$ by the third case, then without loss of generality suppose $w = \alpha_{p,g,h}(i,k)$ and $k(j) = w'$ for some $j \in J_{i}$. Then $w \sim' w$ by the second case. Then there exists a unique $d \in D$ such that $(w,d) \in N$ and for every realizer of $w$, $[r_{y}, r_{k}]$, $UM \cdot [r_{y},r_{k}]$ realizes $d$. Let $d = \beta(y,k')$ with $y = \pr{Y} \circ q_{1}(i)$. Let there be $x \in X_{y}$, $r \in R_{x}$ and $j \in J_{x}$ by third case of $\sim'$, then $d = k'(x)$ since $v:(D,\delta) \to (Z,\zeta)$ is a $\poly_{f,g,h,k}$-algebra. Furthermore by the second case, for any realizer of $k(j)$, $r_{k(j)}$, $UM \cdot r_{k(j)}$ realizes $k'(x) = d$ and $k'(x)$ is the unique element of $D$ corresponding to $k(j)$. Thus $(k(j), d) \in N$ uniquely. This completes the proof.\\

\end{proof}

By implication of the above lemma , we have a morphism $um: (W'_{p,g,h},\omega) \to (D,\delta)$ over $(Z,\zeta)$ realized by $UM$. Furthermore when $w \sim' w'$, $um(w) = um(w')$. In addition, it is the unique morphism such that for a pair $(i ,  k:  \Pi_{j \in (J,(\phi^{*} \times \psi^{*}) \circ q_{0})_{i} } (W_{p,g,h},\omega)_{h \circ \pr{X} \circ \pi_{X} \circ q_{0}(j)})$ there exists a $(y, k') \in \poly_{f,g,h}(v)$ such that for $x \in X_{y}$ and $j \in J_{i}$ such that $j \in J_{x}$, $um(k(j)) = k'(x)$ and $um(\alpha_{p,g,h}(i,k)) = \beta(y,k')$. This gives a unique morphism $[um]: (W'_{f,g,h,r}, \exists_{s'}(\omega)) \to (D, \delta)$ over $(Z,\zeta)$ such that $[um] \circ s' = um$. In addition this is the unique $\poly_{f,g,h,k}$-algebra morphism into $v:(D,\delta)  \to (Z,\zeta)$ since the algebra structure on $(W'_{f,g,h,r}, \exists_{s'}(\omega))$ is derived from the algebra structure on $(W'_{p,g,h},\omega)$ and $um$ is the unique morphism preserving the algebra structure on $(W'_{p,g,h},\omega)$ up to equivalence by $\sim$. Thus we have the following lemma:

\begin{lem}

$c'_{r}:(W'_{f,g,h,r}, \exists_{s'}(\omega)) \to (Z,\zeta)$ is the initial $\poly_{f,g,h,k}$-algebra.

\end{lem}

\subsection{Universes and Impredicativity}
In this section, we will construct object classifiers in particular, we will construct an impredicative object classifiers for modest assemblies.

Given Grothendieck universes in $\Set$, we can construct object classifiers in $\Asm(A)$ as follows:

Take a Grothendieck universe $U$ in $\Set$, We define an universal assembly $(U[A] = \Sigma_{X:U} \powf{A}^{X}, \const_{A})$ and we define $(\Sigma_{(X,\phi):U[A]} X, \upsilon)$ where $\upsilon((X,\phi),x) = \phi(x)$. The full object classifier is the projection morphism $\pr{U[A]}:(\Sigma_{(X,\phi):U[A]} X, \upsilon) \to ( U[A], \const_{A})$ realized by $\iota$. We call assemblies whose underlying set is in $U$, {\bf $U$-assemblies}. We refer to $U/(X,\phi)$ as the set of morphisms over $(X,\phi)$ whose fibers are $U$-assemblies

\begin{lem}

$U/(X,\phi) \cong \Asm(A)((X,\phi),(U[A], \const_{A}))$

\end{lem}
\begin{proof}

Straightforward.

\end{proof}

\begin{defn}

We call  $(X,\phi)$ a {\bf modest assembly} if for $x,x' \in X$, $x \neq x'$ implies $\phi(x) \cap \phi(x') = \emptyset$.

\end{defn}

\begin{defn}

Given a locally cartesian closed category $\cC$ with finite limits and an object classifier $u:U^{*} \to U$, we define $\{U\}$ to be the full subcategory of $\cC^{\to}$ whose objects are pullbacks of $u$. We have a subfibration $\{U\} \to \cC$ which we call the {\bf externalization of u}. We call $U$ {\bf impredicative} if for any morphism $f:X \to Y$ in $\cC$, the pullback functor $f^{*} : \{U\}/Y \to \{U\}/X$
has a right adjoint.

\end{defn}

We now construct an object classifier for the modest assemblies. First we construct a set $M \subseteq \pow{\pow{A}}$ where $$M =  \{ S \subseteq \powf{A}| \forall B,C \in S(B \cap C \neq \emptyset) \} $$ We now have our universe of assemblies $(M, \const_{A})$. We construct the set $M^{*} = \sum_{s \in M} s$ and the assembly $(M^{*}, \mu)$ where $\mu(s, B \in s) = B$. we have the projection morphism $m: (M^{*}, \mu) \to (M,\const_{A})$ which is our proposed object classifier for modest assemblies.

\begin{lem}
\label{modu}
$m: (M^{*}, \mu) \to (M,\const_{A})$ is an object classifier for modest assemblies.

\end{lem}
\begin{proof}
This is easy to see once you realize that any modest assembly $(X,\phi)$ is represent by the set $\{\phi(x)| x \in X \}$ and is equivalent to $(\{\phi(x)| x \in X \}, \id)$, so this applies to fibers of morphisms.

\end{proof}

The externalization of $(M,\const_{A})$ has as its objects the morphisms whose fibers are modest assemblies.

\begin{lem}

$(M,\const_{A})$ is impredicative.

\end{lem}
\begin{proof}
Given a morphism $f: (X,\phi) \to (Y,\psi)$ in $\Asm(A)$ and a morphism $r: (S,\sigma) \to (X,\phi) $ in $\{(M,\const_{A})\}$, we construct a morphism $\Pi_{f}: \{(M,\const_{A})\}/(X,\phi) \to \{(M,\const_{A})\}/(Y,\psi)$ to just be the dependent product functor defined in Lemma \ref{asmdepprod} where $\Pi_{f}(r)_{y} = (\{t: (X,\phi)_{y} \to (S,\sigma) | r \circ t = \id_{(X,\phi)}|_{(X,\phi)_{y}}\}, [t \mapsto \{[r_{y},r_{t}]| r_{y} \text{ realizes } y \text{ and }r_{t} \text{ realizes } t\} ])$. \\

First we show that $\Pi_{f}(r)_{y}$ is modest. Suppose we have $t, t' \in \Pi_{f}(r)_{y}$ such that $t \neq t'$, then there exists some $x \in (X,\phi)$ such that $t(x) \neq t'(x)$. We also have that $(S,\sigma)_{x}$ is modest so $\sigma(t(x)) \cap \sigma(t'(x)) = \emptyset$. Thus for any pair of realizers $[r_{y},r_{t}]$ and $[r'_{y},r_{t'}]$ for $t$ and $t'$ and a realizer $r_{x}$ for $x$, $r_{t} \cdot r_{x} \neq r_{t'} \cdot r_{x}$, so $r_{t} \neq r_{t'}$. So $[r_{y},r_{t}] \neq [r'_{y},r_{t'}]$. Therefore $\Pi_{f}(r)_{y}$ is modest.\\

It should then be clear that $\Pi_{f}$ is the right adjoint to $f^{*}$ since the construction of $\Pi_{f}$ is just the construction of dependent products in $\Asm(A)$.

\end{proof}

\subsection{Groupoid Assemblies} 
In this section, we will introduce the category of groupoid assemblies, $\GrpA{A}$, which is the main setting of this paper as well as important properties of the category.


\begin{defn}
The category of groupoid assemblies $\Grp(\Asm(A))$ is the category of internal groupoids and functors in $\Asm(A)$.

\end{defn}

This category will be the main category of discourse for this entire thesis. When dealing with said category, we will treat $\Asm(A)$ as a generic regular category $\cC$ with dependent products and work inside the internal language of said category; though there will be moments where we use the specific properties of $\Asm(A)$.  For the rest of this paper, we will enrich $\Grp(\Asm(A))$ in the category of $\Asm(A)$ by constructing a bifunctor: \[\Hom: (\GrpA{A})^{\op} \times\GrpA{A} \to \Asm(A)\] where:

 $$\Hom(X,Y) := (\{\text{functors }F: X \to Y\}, [F \mapsto \{[r_{F,o},r_{F,m}]| r_{F,o}\text{ and }r_{F,m}\text{ realizes the object and morphism part of }F \}])$$ and functors $G_{0}:X_{1} \to X_{0}$ and $G_{1}: Y_{0} \to Y_{1}$ get taken to $G_{1} \circ - \circ G_{0}: \Hom(X_{0},Y_{0}) \to \Hom(X_{1},Y_{1})$.


\begin{rmk}
\label{grpnabla}

It is straightforward to show that the adjunction in Remark \ref{asmnabla}, extends to $\GrpA{A}$. In other words we have an adjunction:

$$\nabla: \Grp \leftrightarrows  \GrpA{A}:\Gamma  $$ where $\Gamma \circ \nabla \cong \Id_{\Grp}$.

\end{rmk}

\begin{defn}

We call a functor $F: X \to Y$ in $\Grp(\Asm(A))$ a {\bf groupoid equivalence} if there exists a a functor $F':Y \to X$ and natural isomorphisms $\alpha: F' \circ F \cong \Id_{X}$ and $\beta: F \circ F' \cong \Id_{Y}$.

\end{defn}

\begin{lem}

$\Grp(\Asm(A))$ has finite limits and colimits.

\end{lem}
\begin{proof}

Suppose we had two groupoid assemblies $A$ and $B$, then we define $A \times B$ as the groupoid assembly whose object and morphism assemblies are the product of the object and morphism assemblies of $A$ and $B$. So if we have a groupoid assembly $X$ and functors $F:X \to A$ and $G:X \to B$, then we have a unique functor $H:X \to A \times B$ where $H(x) = (F(x),G(x))$ for $x$ being either an object or morphism of $X$. Thus $A \times B$ is the product of $A$ and $B$.\\

Similarly for coproducts, $A + B$ is the groupoid assembly whose object and morphism assemblies is the coproduct of object and morphism assemblies of $A$ and $B$.\\

Suppose we have the following diagram in $\Grp(\Asm(A))$:
\[\begin{tikzcd}
	{A} &&& {B}
	\arrow["{f}"', shift right=2, from=1-1, to=1-4]
	\arrow["{g}", shift left=2, from=1-1, to=1-4]
\end{tikzcd}\]

For equalizers, the groupoid assembly $\Eq{f,g}$ has as its object and morphism assembly the equalizers of the object and morphism assemblies of $A$ and $B$.\\

For coequalizers, we cannot simply take the coequalizers of the morphism assemblies. As an example of this, take the following diagram:

\[\begin{tikzcd}
	{1} &&& {\{a \to b, c \to d \}}
	\arrow["{b}"', shift right=2, from=1-1, to=1-4]
	\arrow["{c}", shift left=2, from=1-1, to=1-4]
\end{tikzcd}\]

we would expect the coequalizer to be $\{a \to [b,c], [b,c] \to d, a \to d\}$, however the coequalizer of the morphism assemblies wouldn't include the morphism $a \to d$. Going back to the original equalizer diagram, we not only need to take the coequalizer of the morphism assemblies but we also need to consider finite chains of morphisms. First let $f_{\ob}$ and $f_{\mor}$ be the object and morphism parts of the functor $f$. Take the coequalizers $\CoEq{f_{\ob}, g_{\ob}}$ and $\CoEq{f_{\mor},g_{\mor}}$. We have morphisms $\cod, \dom: \CoEq{f_{\mor},g_{\mor}} \to \CoEq{f_{\ob}, g_{\ob}}$ and $\id: \CoEq{f_{\ob}, g_{\ob}} \to \CoEq{f_{\mor},g_{\mor}} $ which follows from applying the universal property of coequalizers to the diagram:

\[\begin{tikzcd}
	{A_{\mor}} & {} & {B_{\mor}} && {\CoEq{f_{\mor},g_{\mor}}} \\
	\\
	{A_{\ob}} && {B_{\ob}} && {\CoEq{f_{\ob},g_{\ob}}}
	\arrow["{g_{\mor}}"', shift right=2, from=1-1, to=1-3]
	\arrow["{f_{\mor}}", shift left=2, from=1-1, to=1-3]
	\arrow["\cod", shift left=4, from=1-1, to=3-1]
	\arrow["\dom"', shift right=4, from=1-1, to=3-1]
	\arrow[from=1-3, to=1-5]
	\arrow["\dom"', shift right=4, from=1-3, to=3-3]
	\arrow["\cod", shift left=4, from=1-3, to=3-3]
	\arrow["\dom"', shift right=5, dashed, from=1-5, to=3-5]
	\arrow["\cod", shift left=5, dashed, from=1-5, to=3-5]
	\arrow["\id"{description}, from=3-1, to=1-1]
	\arrow["{g_{\ob}}"', shift right=2, from=3-1, to=3-3]
	\arrow["{f_{\ob}}", shift left=2, from=3-1, to=3-3]
	\arrow["\id"{description}, from=3-3, to=1-3]
	\arrow[from=3-3, to=3-5]
	\arrow["\id"{description}, dashed, from=3-5, to=1-5]
\end{tikzcd}\]

As for inverses, Note that the morphism $[b \mapsto [b^{-1}]]: B_{\mor} \to \CoEq{f_{\mor},g_{\mor}}$ satisfies the universal property of coequalizers since if for $a \in A_{\mor}$, $[f(a)] = [g(a)]$, then $[f(a)^{-1}] = [f(a^{-1})] = [g(a^{-1})] = [g(a)^{-1}]$. Thus we get $(-)^{-1}: \CoEq{f_{\mor},g_{\mor}} \to \CoEq{f_{\mor},g_{\mor}}$ where $[b]^{-1} = [b^{-1}]$ from the morphism $[b \mapsto [b^{-1}]]$.\\

$\Asm(A)$ has a natural numbers object $\nno$ since $\Asm(A)$ has W-types. Given $n \in \nno$, define $[n] = \{i | i \leq n\}$. We define $$M'(f_{\mor},g_{\mor}) = \{  (m,t) : \sum_{n : \nno} [n] \to \CoEq{f_{\mor},g_{\mor}}| m > 0 \implies \forall j < m (\cod(t(j)) = \dom(t(succ(j)))) \} $$

We have functions $\cod, \dom: M'(f_{\mor},g_{\mor}) \to \CoEq{f_{\ob},g_{\ob}}$ and $\id: \CoEq{f_{\ob},g_{\ob}} \to M'(f_{\mor},g_{\mor})$ where $\cod(m,t) = \cod(t(m))$, $\dom(m,t) = \dom(t(0))$ and $\id(x) = (0,(\id_{x}))$. We also have $(-)^{-1}: M'(f_{\mor},g_{\mor})  \to M'(f_{\mor},g_{\mor}) $ where $(m,t)^{-1} = (m,([t_{m}]^{-1},\cdots,[t_{0}]^{-1}))$. Now we define a concatenation operator $\cdot : \{(m,t),(m',t'):M'(f_{\mor},g_{\mor})\times M'(f_{\mor},g_{\mor})| \cod(m,t) = \dom(m',t')   \} \to M'(f_{\mor},g_{\mor})$ where $(m,t) \cdot (m',t') = (succ(m + m'), (t(0),\cdots, t(m), t'(0),\cdots,t'(m')))$. Associativity of concatenation is straightforward. We define a equivalence relation $\sim$ on $M'(f_{\mor},g_{\mor})$ generated by 

\begin{itemize}

\item $(m,t) \sim (m,t) \cdot (0,(\id_{\cod(m,t)}))$

\item$(m,t) \sim (0,(\id_{\dom(m,t)})) \cdot (m,t) $

\item $(m,t) \sim (m',t')\text{ implies }(m,t)^{-1} \sim (m',t')^{-1}$

\item $(0,([t_{0} \cdot t_{1}])) \sim (succ(0),([t_{1}],[t_{0}]))\text{ for composable }t_{0},t_{1} \in B_{\mor}$

\item $\text{For finite sequence of }(s_{l,i} \sim s_{r,i})\text{, }(s_{l,0} \cdot s_{l,1} \cdots s_{l,i} \cdots) \sim  (s_{r,0} \cdot s_{r,1} \cdots s_{r,i} \cdot \cdots)\text{ when composable.} $

\end{itemize}

Note that by the first three relations, when $s_{l} \sim s_{r}$, $\dom(s_{l}) = \dom(s_{r})$ and $\cod(s_{l}) = \cod(s_{r})$. 
Also take note that $(m,t) \cdot (m,t)^{-1} \sim \id_{\dom(m,t)}$ by applying all of the relations to $(m,t) \cdot (m,t)^{-1} = (succ(m + m),([t_{0}],\cdots, [t_{m}],[t_{m}]^{-1},\cdots, [t_{0}]^{-1}))$. We have a similar situation with $(m,t)^{-1} \cdot (m,t) \sim \id_{\cod(m,t)}$.\\

We define the morphism assembly $\mor \CoEq{f,g}$ as the quotient of $\sim$. By the universal property of coequalizers, we have morphisms $\dom, \cod: \mor \CoEq{f,g} \to \CoEq{f_{\ob},g_{\ob}}$ where for $t \in M'(f_{\mor},g_{\mor})$, $\dom(t) = \dom(r(t))$ and $\cod(t) = \cod(r(t))$ with $r: M'(f_{\mor},g_{\mor})  \to \mor \CoEq{f,g}$ as the canonical regular epimorphism. We have $\id:  \CoEq{f_{\ob},g_{\ob}} \to  \mor \CoEq{f,g}$ by composing $\id:  \CoEq{f_{\ob},g_{\ob}} \to  M'(f_{\mor},g_{\mor}) $ with $r: M'(f_{\mor},g_{\mor})  \to \mor \CoEq{f,g}$. \\

Now we define composition on $\mor \CoEq{f,g}$. Suppose we have $t, t' \in \mor \CoEq{f,g}$ with $\dom(t') = \cod(t)$, since $r$ is regular there exists $p,p' \in M'(f_{\mor},g_{\mor}) $ with $r(p) = t$, $r(p') = t'$, and $\dom(p') = \cod(p)$. For any $q$ and $q'$ with $r(q) = t$ and $r(q') = t'$, we have that $q \cdot q'$ exists.  To show that $q \sim p$ and $q' \sim p'$ we make note of the fact that this construction is being done in $\Asm(A)$; hence $q$,$q'$, $p'$ and $p$ are elements of the underlying set of $M'(f_{\mor},g_{\mor})$. Since $\Set$ is exact and $r(q) = r(p) = t$ and $r(q') = r(p') = t$, then $(q,p)$ and $(q',p')$ exist in the underlying set of $\sim$, thus $q \sim p$ and $q' \sim p'$, therefore $q \cdot q' \sim p \cdot p'$. We now have $t' \cdot t = r(p \cdot p')$ for any $p \in M'(f_{\mor},g_{\mor})(t)$ and $p' \in M'(f_{\mor},g_{\mor})(t')$. The satisfaction of associativity and identity laws of composition follows from the satisfaction of such laws with concatenation of elements in $M'(f_{\mor},g_{\mor})$. \\

As for inverses, if we have $t \in \mor \CoEq{f,g}$, then as above  there exists $p \in M'(f_{\mor},g_{\mor}) $ with $r(p) = t$. Furthermore, for any $q$ with $r(q) = t$, we have that $q^{-1}$ exists and $q^{-1} \sim p^{-1}$. Thus $t^{-1} = r(p^{-1})$ for all $p \in M'(f_{\mor},g_{\mor})(t)$. The fact that $t^{-1}$ is the proper inverse of $t$ follows from the definition of $\sim$.\\

Thus we have our potential coequalizer groupoid assembly $\CoEq{f,g}$ with the object assembly $\CoEq{f_{\ob}, g_{\ob}}$ and morphism assembly $\mor \CoEq{f,g}$. We define the canonical morphism $[-]: B \to \CoEq{f,g}$ where the object part is just the canonical morphism $[-]:B_{\ob} \to  \CoEq{f_{\ob},g_{\ob}}$ and the morphism part takes a morphism $p$ to $r(0,([p]))$ where $[p] \in  \CoEq{f_{\mor},g_{\mor}}$.  Preservation of identity follows from construction of $\id: \CoEq{f_{\ob}, g_{\ob}} \to \CoEq{f_{\mor},g_{\mor}}$ using the universal property of coequalizers. Preservation of composition and inverses follows from construction of $\sim$.\\

Now suppose we have a functor $h: B \to X$ such that $h \circ f = h \circ g$, we construct a functor $y:\CoEq{f,g} \to X$. The object part of $y$ is simply the unique morphism $y'_{\ob}:\CoEq{f_{\ob}, g_{ob}} \to X_{\ob}$ such that $h_{\ob} = y'_{\ob} \circ [-]$. We construct the morphism part via induction on $M'(f_{\mor},g_{\mor}) $: when we have $(0,t) \in M'(f_{\mor},g_{\mor}) $, then $t = ([p])$ for $p \in B_{\mor}$. Since we have the unique morphism $y':\CoEq{f_{\mor},g_{\mor}} \to X_{\mor}$ such that $h_{\mor} = y' \circ [-]$, we set $y(r(0,t)) = y'([p])$. Given we have $y(r(n,t))$, for $(succ(n),t)$, we have $(succ(n),t) = (n,(t_{0},\cdots,t_{n})) \cdot (0,(t_{succ(n)}))$ so we have $y(r(succ(n),t)) = y(r(0,(t_{succ(n)}))) \cdot y(r(n,(t_{0},\cdots,t_{n})))$. $y$ is well defined since composition on $\CoEq{f,g}$ is defined by concatenation on $M'(f_{\mor},g_{\mor}) $. Needless to say $y$ preserves composition. $y$ preserves identities since for an identity morphism $\id_{b}$ in B, $y'([\id_{b}]) = h(\id_{b})$ and the identities of $\CoEq{f,g}$ are precisely the morphisms $r(0,[\id_{b}])$. $y$ preserves inverses since for morphism $p \in B_{\mor} $, $y'([p]^{-1}) = y'([p^{-1}]) = h(p^{-1}) = h(p)^{-1} = y'([p])^{-1}$.\\

We clearly have $y \circ [-] = h$. This necessary equality makes it so that $y_{\ob} = y'_{\ob}$ and $y(r(0,([p]))) = h(p) = y'([p])$. Since $y$ preserves composition we must have $y(r(n,t)) = y'(t_{n}) \cdots y'(t_{succ(0)}) \cdot y(t_{0})$. Thus $y$ is completely determined by $y'_{\ob}$ and $y'$ which are unique. Thus $y$ must be unique. Therefore, $\Grp(\Asm(A))$ has coequalizers.\\

Therefore $\Grp(\Asm(A))$ has finite limits and colimits.\\

\end{proof}

\begin{lem}
\label{ccgrpd}

$\Grp(\Asm(A))$ is cartesian closed.

\end{lem}
\begin{proof}

Suppose we have groupoid assemblies $A$ and $B$, we define the groupoid assembly $A \to B$ whose object assembly is $\Hom(A,B)$ and whose morphism object is the set of triples $\{(F,G,\alpha:\ob A \to \mor B )| \alpha:F \To G\text{ is a natural isomorphism}\}$ where each triple is realized by $[r_{F},r_{G},r_{\alpha}]$ where the individual  realizers realizes $F$, $G$, and $\alpha:\ob A \to \mor B $ respectively. The morphisms $\dom, \cod: \mor (A \to B) \to \ob (A \to B)$ are just projections, $\id: \ob(A \to B) \to \mor(A \to B) $ takes $F$ to $(F,F,[a \in \ob A \mapsto \id_{F(a)}])$, and $(F,G,\alpha)^{-1} = (G,F,[a \in A \mapsto \alpha(a)^{-1}])$. Composition is composition of natural transformations. \\

We construct a functor $\app{-}: (A \to B)\times A  \to B$ where $\app{F,a} = F(a)$ for objects and $\app{(F,G,\alpha),(a,a',m)} = G(m) \circ \alpha(a) = \alpha(a') \circ F(m)$ for morphisms. Preservation of identity and inverses is straightforward to show. If we have $((G,H, \beta),(a',a'',m'))$ and $((F,G, \alpha),(a,a',m))$, then $\app{(F,H,\beta\circ \alpha), (a,a'',m' \circ m)} = H(m' \circ m) \circ \beta(a) \circ \alpha(a) = H(m') \circ H(m) \circ \beta(a) \circ \alpha(a) = H(m') \circ \beta(a') \circ G(m) \circ \alpha(a) = \app{(G,H,\beta),(a',a'',m')} \circ \app{(F,G,\alpha),(a,a',m)}$.\\

Suppose we have a groupoid assembly $X$ and a functor $h:X \times A \to B$, we seek to define unique functor $y:X \to (A \to B)$ giving the following commutative diagram:

\[\begin{tikzcd}
	{X \times A} \\
	\\
	{(A \to B) \times A} &&& B
	\arrow["{y \times \id_{A}}"', from=1-1, to=3-1]
	\arrow["h", from=1-1, to=3-4]
	\arrow["{\app{-}}"', from=3-1, to=3-4]
\end{tikzcd}\]

Given $x  \in \ob X$, the functor $y(x):A \to B$ gives us $y(x)(a) = h(x,a)$ and $y(x)(m) = h(id_{x},m)$. For morphism $n:x \to x' \in \mor X$, $y(n): y(x) \to y(x')$ gives us $y(n)(a) = h(n,\id_{a})$. The fact that $y(n)$ and a natural isomorphism and $y(x)$ is a functor follows from functoriality of $h$. Furthermore, the functoriality of $y$ also follows from the functoriality of $h$. Now we have that $\app{-} \circ (y \times \id_{A})(x,a) = \app{y(x),a} = y(x)(a) = h(x,a)$ and $\app{-} \circ (y \times \id_{A})(n:x \to x',m:a \to a') = \app{y(n),m} = y(n)(a') \circ y(x)(m) = h(n,\id_{a'}) \circ h(\id_{x},m) = h(n,m)$.   Furthermore if if we have $y': X \to A$ such that $h = \app{-} \circ (y' \times \id_{A})$ then $y'(x)(a) = h(x, a)$ and $\app{y'(n),m} = y'(n)(a') \circ y'(x)(m) = y'(x')(m) \circ y'(n)(a) = h(n,m)$. So $y'(x)(m) =  \id_{y'(x)(a')} \circ y'(x)(m) = y'(\id_{x})(a) \circ y'(x)(m) = h(\id_{x},m)$ and $y'(n)(a) = y'(x')(\id_{a}) \circ y'(n)(a) = y'(n,\id_{a})$. Thus $y' = y$. Therefore $\Grp(\Asm(A))$ is cartesian closed.

\end{proof}

\section{An Algebraic Weak Factorization System on Groupoid Assemblies}
In this section we introduce lifting properties and construct an algebraic weak factorization system for $\GrpA{A}$ whose right maps are the normal isofibrations and left maps are the strong deformations sections. This algebraic weak factorization system will serve as the fibrant replacement of our eventual model structure on $\GrpA{A}$. We also introduce groupoid equivalences and prove various properties about groupoid equivalences, normal isofibrations, and strong deformations sections. Finally, we prove that the normal isofibrations are exponentiable and that normal isofibrations satisfies the Frobenius property.\\

We construct the interval groupoid $\I$ in $\GrpA{A}$ as the groupoid generated by two objects $0,1 \in \ob \I$ and an isomorphism $\text{line}:0 \to 1$ in the internal language of $\Asm(A)$.\\

\subsection{ Lifting properties of morphisms}
Recall in the previous section, we defined the category $\GrpA{A}$, the category of internal groupoids in $\Asm(A)$, enriched over $\Asm(A)$. In this section we will define the notion of lifting properties using the enrichment of $\GrpA{A}$ in $\Asm(A)$. We will also prove some basic properties of lifting properties.\\

\begin{defn}

Given functors $i: A \to B$ and $f: X \to Y$ in $\GrpA{A}$, the object $\Sq(i,f)$ in $\Asm(A)$  is defined as the pullback of the following diagram:

\[
\begin{tikzcd}
\Sq(i,f)  \arrow[r] \arrow[d]  \arrow[dr,phantom,"\scalebox{1.5}{$\lrcorner$}" , very near start, color=black]   &  \Hom(B,Y) \arrow[d,"- \circ i"] \\
 \Hom(A,X) \arrow[r, "f \circ -"'] &  \Hom(A,Y)
\end{tikzcd}
\]

\end{defn}

\begin{rmk}[on the bifunctoriality of the square construction]\label{square-bifunctor}
 We have that the above construction induces a bifunctor \[\Sq: (\GrpA{A}^{\arr})^{\op} \times\GrpA{A}^{\arr} \to \Asm(A).\]
\end{rmk}

\begin{defn}

The morphism $\pullp(i,f): \Hom(B,X) \to \Sq(i,f)$ is the unique morphism that is derived by apply the universal property of pullbacks:

\[
\begin{tikzcd}
\Hom(B,X) \arrow[dr, dashed, "{\pullp(i,f)}"] \arrow[drr, bend left, "{f \circ -}"] \arrow[ddr, bend right, "{-\circ i}"']  
\\
 & \Sq(i,f)  \arrow[r] \arrow[d]  \arrow[dr,phantom,"\scalebox{1.5}{$\lrcorner$}" , very near start, color=black]   &  \Hom(B,Y) \arrow[d,"- \circ i"] \\ &
 \Hom(A,X) \arrow[r, "f \circ -"'] &  \Hom(A,Y)
\end{tikzcd}
\]
Explicitly it sends a functor $j:B \to X$ to a pair $(f \circ j, j \circ i)$.

\end{defn}

\begin{defn}

We say $f$ has a \emph{right lifting property} with $i$ or that $i$ has a \emph{ left lifting property} iff $\pullp(i,f)$ has a section $\lift(i,f)$.

\end{defn}

\begin{rmk}
An element of $\Sq(i,f)$ is a square:
\[
\begin{tikzcd}
A \arrow[r,"u"] \arrow[d,"i"']    &  X \arrow[d,"f"] \\
 B \arrow[r, "v"'] &  Y
\end{tikzcd}
\]

A lifting problem essentially asks whether there exists a morphism $z: B \to X$ giving us the following commutative diagram:

\[
\begin{tikzcd}
A \arrow[r,"u"] \arrow[d,"i"']    &  X \arrow[d,"f"] \\
 B \arrow[ur,"z",dashed] \arrow[r, "v"'] &  Y
\end{tikzcd}
\]

A solution to a lifting problem is called a lift.\\

The morphism $\lift(i,f)$ takes the square $(u,v)$ and provides a solution to lifting problem posed by $(u,v)$ . In fact, any section of $\pullp(i,f)$ will provide a solution to the lifting problem for any square between $i$ and $f$.\\

\end{rmk}

\begin{ex}\label{section-lift}
Given the unique morphism from the initial object to the final object $!:0 \to 1$ and an arbitrary functor $f$, the functor $\lift(!,f)$ gives a section of the object part of $f$.
\end{ex}


\begin{lem}

If  $f:X \to Y$ has a right lifting property with $i:A \to B$, and $i':B \to C$ then $f$ has a  right lifting property with $i' \circ i$.

\end{lem}


\begin{proof}
We seek to construct the dotted line for the following commutative diagram:

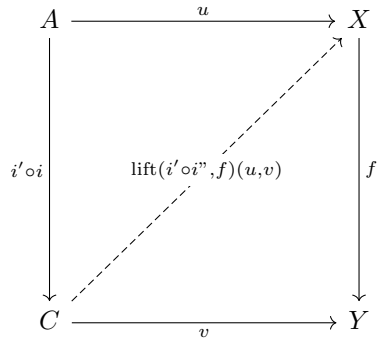
\[
\begin{tikzcd}[column sep=10em, row sep=10em]
A \arrow[r,"u"] \arrow[d,"i' \circ i"']    &  X \arrow[d,"f"] \\
 C \arrow[ur,"{\lift(i' \circ i",f)(u,v)}" description,dashed] \arrow[r, "v"'] &  Y
\end{tikzcd}
\]

We construct said arrow using the following two diagrams where the diagonal arrow of the second diagram is the arrow in question:

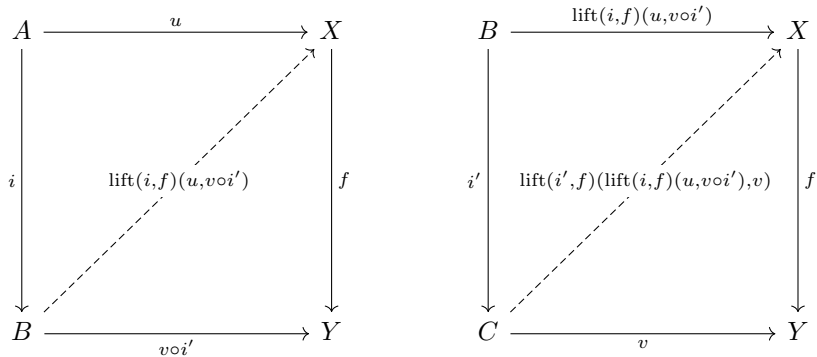
\[
\begin{tikzcd}[column sep=10em, row sep=10em]
A \arrow[r,"u"] \arrow[d,"i"']    &  X \arrow[d,"f"] \\
 B \arrow[ur,"{\lift( i,f)(u, v \circ i')}" description,dashed] \arrow[r, "v \circ i'"'] &  Y
\end{tikzcd}
\qquad\qquad
\begin{tikzcd}[column sep=10em, row sep=10em]
B \arrow[r,"{\lift(i,f)(u, v \circ i')}" ] \arrow[d,"i'"']    &  X \arrow[d,"f"] \\
 C \arrow[ur,"{\lift(i',f)(\lift(i,f)(u, v \circ i'), v)}" description,dashed] \arrow[r, "v"'] &  Y
\end{tikzcd}
\] \\

We have that $\lift(i' \circ i",f)(u,v) = \lift(i',f)(\lift(i,f)(u, v \circ i'), v)$. To see that this is clearly a lift observe that: $$\lift(i',f)(\lift(i,f)(u, v \circ i'), v) \circ i' \circ i  =  \lift(i,f)(u, v \circ i') \circ i = u$$ and that $f \circ \lift(i',f)(\lift(i,f)(u, v \circ i'), v)  = v$. Thus we have a lift for the first square which completes the proof.
\end{proof}

\begin{lem}

If  $f:X \to Y$ has a right lifting property with $i:A \to B$, then $f$ has a right lifting property with retracts of $i$. Dually, if $i:A \to B$ has a left lifting property with $f:X \to Y$, then $i$ has a right lifting property with retracts of $f$

\end{lem}

\begin{proof}
Suppose $i':A' \to B'$ is a retract of $i$, then we have the following square: 
\[
\begin{tikzcd}
A' \arrow[r,"a_{0}"] \arrow[d,"i'"]   &  A\arrow[r,"a_{1}"]  \arrow[d,"i"] & A \arrow[d,"i'"]  \\
B' \arrow[r, "b_{0}"] &  B \arrow[r,"b_{1}"]  & B'
\end{tikzcd}
\]

where $b_{1} \circ b_{0} = \id_{B}$ and $a_{1} \circ a_{0} = \id_{A}$. We define the section $\lift(i',f): \Sq(i', f) \to \Hom(B',X)$ to be $(- \circ b_{0}) \circ \lift(i,f)\circ (- \circ (a_{1},b_{1}))$. Given a square $(s,t)$, we have $\lift(i',f)(s,t) \circ i' =\lift(i,f)(s \circ a_{1}, t \circ b_{1}) \circ b_{0} \circ i' = \lift(i,f)(s \circ a_{1}, t \circ b_{1})\circ i \circ a_{0} = s \circ a_{1} \circ a_{0} = s$ and $f \circ \lift(i',f)(s,t) = f \circ \lift(i,f)(s \circ a_{1}, t \circ b_{1}) \circ b_{0} = t \circ b_{1} \circ b_{0} = t$. Therefore $f$ has a right lifting property with retracts of $i$. \\

The second part of the lemma follows from duality.\\

\end{proof}

\subsection{The fibrations of groupoid assemblies}
In this section, we introduce normal isofibrations and prove some properties about them. We also introduce the factorization for the eventually algebraic weak factorization system.\\




\begin{defn}

Given $\htpy{0}: 1 \to \I$, We call $f$ a normal  isofibration iff $f$ has a  right lifting property with $\htpy{0}$ and the section $\lift(\htpy{0},f)$ 
sends any pair of the form $(\lceil x \rceil: 1 \to X, \lceil \id_{f(x)} \rceil: \I \to Y)$ to $\lceil \id_{x} \rceil: \I \to X$.

\end{defn}

Recall from Lemma \ref{ccgrpd} that $\GrpA{A}$ is cartesian closed. Given functors $f:X \to Y$ and $g:B \to Z$ we can construct the {\bf pullback power}, $g \pupw f$:

\[
\begin{tikzcd}
Z \to X \arrow[dr, dashed, "{g \pupw f}"] \arrow[drr, bend left, "{f \circ -}"] \arrow[ddr, bend right, "{-\circ g}"']  
\\
 & (B \to X) \times_{B \to Y} (Z \to Y)  \arrow[r] \arrow[d]  \arrow[dr,phantom,"\scalebox{1.5}{$\lrcorner$}" , very near start, color=black]   &  Z \to Y \arrow[d,"- \circ g"] \\ &
 B \to X \arrow[r, "f \circ -"'] &  B \to Y
\end{tikzcd}
\]
 
and the {\bf pushout product}, $g \otimes f$:

\[\begin{tikzcd}
	{B \times X} &&& {B \times Y} \\
	\\
	\\
	{Z \times X} &&& {(Z \times X) +_{B \times X} (B \times Y)} \\
	\\
	&&&& {Z \times Y}
	\arrow["{B \times f}", from=1-1, to=1-4]
	\arrow["{g \times X}"', from=1-1, to=4-1]
	\arrow[from=1-4, to=4-4]
	\arrow["{g \times Y}",  bend left,from=1-4, to=6-5]
	\arrow[from=4-1, to=4-4]
	\arrow["{Z \times f}"', bend right,from=4-1, to=6-5]
	\arrow["\lrcorner"{anchor=center, pos=0.125, rotate=180}, draw=none, from=4-4, to=1-1]
	\arrow["{g \otimes f}", dotted, from=4-4, to=6-5]
\end{tikzcd}\]

We can extend both constructions to functors $$\otimes, \pupw : (\GrpA{A}^{\arr})^{\op} \times \GrpA{A}^{\arr} \to \GrpA{A}^{\arr}$$ where the functor $f \otimes -$ is left adjoint to $f \pupw -$.


\begin{lem}
\label{pathspace}

For any groupoid assembly $X$, The morphism $(\I \to X) \to X \times X$ which projects the domain and codomain of every path is a normal isofibration.

\end{lem}

\begin{proof}

To show that $(\I \to X) \to X\times X$ is a normal isomorphism observe that a square of the form 

\[\begin{tikzcd}
	{1} &&& {\I \to X} \\
	\\
	\\
	{\I} &&& {X \times X}
	\arrow["{\nm{p}}", from=1-1, to=1-4]
	\arrow["{\htpy{0}}"', from=1-1, to=4-1]
	\arrow["{- \circ [\htpy{0},\htpy{1}]}",from=1-4, to=4-4]
	\arrow["{\nm{t,t'}}"', from=4-1, to=4-4]
\end{tikzcd}\]

corresponds to a diagram in $X$:

\[\begin{tikzcd}
	x &&& y \\
	\\
	\\
	{x'} &&& {y'}
	\arrow["t", from=1-1, to=1-4]
	\arrow["p"', from=1-1, to=4-1]
	\arrow["{t'}"', from=4-1, to=4-4]
\end{tikzcd}\]

which can be easily filled in:

\[\begin{tikzcd}
	x &&& y \\
	\\
	\\
	{x'} &&& {y'}
	\arrow["t", from=1-1, to=1-4]
	\arrow["p"', from=1-1, to=4-1]
	\arrow["{t' \circ p \circ t^{-1}}", from=1-4, to=4-4]
	\arrow["{{t'}}"', from=4-1, to=4-4]
\end{tikzcd}\]

This becomes the identity square for $p$ when $t$ and $t'$ are identities. Therefore, $(\I \to X) \to X\times X$ is a normal isofibration.

\end{proof}

We define functors $\Dom,\Cod:\Hom(\I,A) \to A$ to be functors that take an isomorphism in $A$ to its domain and codomain respectively.

\begin{prop} 

Suppose we have a normal isofibration $f:X \to Y$ and decidable monomorphism on objects $i:A \mono B$, then $i \pupw f$ is a normal isofibration

\end{prop}

\begin{proof}

We seek to construct a section $\lift(\htpy{0},i \pupw f): \Sq(\htpy{0}, i \pupw f) \to \Hom(\I, (B\to X))$. Suppose we have a square $(s: 1 \to (B\to X), \sigma : \I \to (A \to X)\times_{A \to Y} (B \to Y))$. $s$ is a functor $s: B \to X$ and $\sigma$ is 
a pair of natural isomorphisms $(\sigma_{0}:f \circ s \cong s_{0} , \sigma_{1}: s \circ i \cong s_{1})$ such that $\sigma_{0} \cdot i = f \cdot \sigma_{1}$. Define a functor $s':B \to X$ such that:

\begin{enumerate}

\item For $b:\ob B$ in the image of $i$, $s'(b) = s_{1}(a)$ where $i(a) = b$, for every other $b: \ob B$,  $s'(b) = \Cod(\lift(\htpy{0},f)(s(b), \sigma_{0}(b)))$
\item For $m:\mor B$, $s'(m) = q^{-1}_{\cod(m)} \circ \Cod(\lift(\htpy{0},f)(s(m), \sigma_{0}(m) )) \circ q_{\dom(m)}$ where:
\item $q_{b} = \lift(\htpy{0},f)(s(b), \sigma_{0}(b)) \circ \sigma_{1}^{-1}(b')$ when $b = i(b')$ and $q_{b} = \id_{b}$ otherwise

\end{enumerate}

The functoriality of $s'$ follows from the functoriality of $\lift(\htpy{0},f)$ and $\Cod$. Next we seek to construct a natural isomorphism $t: s \cong s': B \to X$. When $b: \ob B$ is in the image of $i$, $t(b) = \sigma_{1}(a)$ where $i(a) = b$. Otherwise, $t(b) = \lift(\htpy{0},f)(s(b), \sigma_{0}(b))$. For $m : \mor B$, $$s'(m) \circ t(\dom(m)) =   q^{-1}_{\cod(m)} \circ \Cod(\lift(\htpy{0},f)(s(m), \sigma_{0}(m) )) \circ q_{\dom(m)} \circ t(\dom(m)).$$
Since $ q_{\dom(m)} \circ t(\dom(m)) =  \lift(\htpy{0},f)(s(b), \sigma_{0}(b))$. This gives us 
\begin{align*}
s'(m) \circ t(\dom(m))  &=  q^{-1}_{\cod(m)} \circ \Cod(\lift(\htpy{0},f)(s(m), \sigma_{0}(m) )) \circ  \lift(\htpy{0},f)(s(\dom(m)), \sigma_{0}(\dom(m)))\\
&=  q^{-1}_{\cod(m)} \circ \lift(\htpy{0},f)(s(\cod(m)), \sigma_{0}(\cod(m))) \circ s(m)\\
\end{align*}
When $\cod(m)$ is in the image of $i$, then $q^{-1}_{\cod(m)} = \sigma_{1}(b') \circ \lift(\htpy{0},f)(s(\cod(m)), \sigma_{0}(\cod(m)))^{-1}$. Otherwise, $q^{-1}_{\cod(m)} = \id_{\cod{m}}$. Thus we finally get that $s'(m) \circ t(\dom(m)) = t(\cod(m)) \circ s(m)$ showing $t$ to be a natural isomorphism.\\

Finally we seek to show that $s' \circ i = s_{1}$, $f \circ s' = s_{0}$, $f \cdot t = \sigma_{0}$ and $t \cdot i = \sigma_{1}$. $f( \lift(\htpy{0},f)(s(b), \sigma_{0}(b)))) = \sigma_{0}(b)$ and $f(\sigma_{1}(a)) = \sigma_{0}(i(a))$, thus $f \cdot t = \sigma_{0}$. $t(i(a)) = \sigma_{1}(a)$, thus $t \cdot i = \sigma_{1}$. The facts that $s' \circ i = s_{1}$ and $f \circ s' = s_{0}$ follow immediately. Along with the fact that $t \circ \htpy{0} = s$, this demonstrates the sectionality of $ \lift(\htpy{0},i \pupw f)$ on objects. We set $ \lift(\htpy{0},i \pupw f)(s,\sigma)  = t $ \\

Suppose we have $s:B \to X$ and $\sigma = (\id_{f \circ s}, \id_{s \circ i})$, then referring back to the above construction of $s':B \to X$ and $t:s \cong s'$, $s' = s$ and $t = \id_{s}$ since $f$ is a normal isofibration. Therefore $i \pupw f$ is a normal isofibration.\\

\end{proof}

\begin{cor}\label{pullbackpowerlift}

Suppose we have a normal isofibration $f:X \to Y$ and decidable monomorphism on objects $i:A \mono B$, then $i \otimes \htpy{0}$ has the left lifting property against $f$

\end{cor}
\begin{proof}

Suppose we have a square:
\[\begin{tikzcd}
	{(\I \times B) +_{1 \times A} (1 \times B)} &&& {X} \\
	\\
	\\
	{\I \times B} &&& {Y}
	\arrow["{s_{0}}", from=1-1, to=1-4]
	\arrow["{i \otimes \htpy{0}}"', from=1-1, to=4-1]
	\arrow["{f}",from=1-4, to=4-4]
	\arrow["{s_{1}}"', from=4-1, to=4-4]
\end{tikzcd}\]

then we have a square from the $(i \otimes -) \adj (i \pupw -)$ adjunction and a lifting by the previous proposition:

\[\begin{tikzcd}
	1 &&& {B \to X} \\
	\\
	\\
	{\I } &&& {(A \to X) \times_{A \to Y} (B \to Y)}
	\arrow["{{{s'_{0}}}}", from=1-1, to=1-4]
	\arrow["{{{\htpy{0}}}}"', from=1-1, to=4-1]
	\arrow["{{{i \pupw f}}}", from=1-4, to=4-4]
	\arrow["\lambda"{description}, dashed, from=4-1, to=1-4]
	\arrow["{{{s'_{1}}}}"', from=4-1, to=4-4]
\end{tikzcd}\]

then when have a lifting on the original square:

\[\begin{tikzcd}
	{(\I \times B) +_{1 \times A} (1 \times B)} &&& X \\
	\\
	\\
	{\I \times B} &&& Y
	\arrow["{{s_{0}}}", from=1-1, to=1-4]
	\arrow["{{i \otimes \htpy{0}}}"', from=1-1, to=4-1]
	\arrow["{{f}}", from=1-4, to=4-4]
	\arrow["{\lambda^{*}}"{description}, dashed, from=4-1, to=1-4]
	\arrow["{{s_{1}}}"', from=4-1, to=4-4]
\end{tikzcd}\]

where $\lambda^{*}: (\I \times B) \to X$ is the adjunct of $\lambda:\I \to (B \to X)$. 

\end{proof}

\begin{defn}  

Given a functor $F: A \to B$ in $\GrpA{A}$, we construct the fibrant replacement of $F$ using the iso comma category construction:

\[
\begin{tikzcd}
A \arrow[r,"\tilde{F}"]   & \coe{F}{B} \arrow[r,"\hat{F}"] & B 
\end{tikzcd}
\]
 In addition:

\begin{itemize}

\item$\ob( \coe{F}{B}) := \sum_{a:A}\sum_{b:B} B(F(a),b)$.\\

\item $\mor (\coe{F}{B}) := \sum_{(a,b,f):\ob (\coe{F}{B})}\sum_{(a',b',f'):\ob (\coe{F}{B})}\sum_{x:A(a,a')}\sum_{y:B(b,b')} f' \circ F(x) = y \circ f$.\\

\item Given a morphism $((a,b,f),(a',b',f'),x,y)$, we use the shorthand: $(x,y):(a,b,f) \to (a',b',f')$. We define $\dom(x,y) = (a,b,f)$ and $\cod(x,y) = (a',b',f')$.\\

\item Composition $\circ$ sends the following two morphisms $(x,y):(a,b,f) \to (a',b',f')$ and $(x',y'):(a',b',f') \to (a'',b'',f'')$ to $(x' \circ x,y' \circ y):(a,b,f) \to (a'',b'',f'')$ so that $f'' \circ F(x') \circ F(x) = y' \circ y \circ f$ follows from $f' \circ F(x) = y \circ f$ and $f'' \circ F(x') = y' \circ f'$.\\

\item $\id_{(a,b,f)} = (id_{a},\id_{b})$.\\

\item $\hat{F}$ projects the objects and morphisms of $B$. Whereas $\tilde{F}$ sends an object $a$ to $(a,F(a),\id_{F(a)})$ and a morphism $m:a \to a'$ into $(m, F(m))$.
\end{itemize}

\end{defn}

\begin{lem}

Given a functor $F: A \to B$, $\hat{F}$ is a normal isofibration.

\end{lem}

\begin{proof}

We seek to construct a section $\lift(\htpy{0},\hat{F}): \Sq(\htpy{0},\hat{F}) \to \Hom(\I, \coe{F}{B})$.  Suppose we have a square $((a,b,f): 1 \to \coe{F}{B}, g : \I \to B)$. Then  $\lift(\htpy{0},\hat{F})((a,b,f),g) := ((a,b,f), (a,\cod(g),g \circ f), \id_{a}, g , *)$. This obviously gives us a section on objects.

Now we show that $\hat{F}$ is normal. Suppose we have a square  $((a,b,f): 1 \to \coe{F}{B}, \id_{b} : \I \to B)$, then $\lift(\htpy{0},\hat{F})((a,b,f),\id_{b}) = (\id_{a},\id_{b},*) = \id_{(a,b,f)}$. Therefore, $\hat{F}$ is a normal isofibration.

\end{proof}

\subsection{The acyclic cofibrations and weak equivalences}
In the section we introduce groupoid equivalences, strong deformation retracts, and strong deformation sections and prove some properties about them. We also prove that normal isofibrations are characterized by having the right lifting property against strong deformation sections and similarly strong deformation sections are characterized by having the left lifting property against normal isofibrations.\\


\begin{defn}

We call a functor $W:A \to B$ an equivalence of groupoids when there exists a functor $W^{i}:B \to A$ and natural isomorphisms $\alpha: W^{i} \circ W \to \id_{A}$ and $\beta: W \circ W^{i} \to \id_{B}$.

\end{defn}

\begin{lem}

Equivalences of groupoids satisfy the 2 out of 3 principle.

\end{lem}
\begin{proof}

We have the following commutative diagram

\[\begin{tikzcd}
	&& B \\
	\\
	A &&&& C
	\arrow["j", from=1-3, to=3-5]
	\arrow["i", from=3-1, to=1-3]
	\arrow["k", from=3-1, to=3-5]
\end{tikzcd}\]

We denote $h*$ as the homotopy inverse of $h$. \\

Assume $i$ and $j$ are groupoid equivalences:\\

We define $k^{*}(c) = i^{*} \circ j^{*}(c)$. The following chains of isomorphisms give the necessary natural isomorphisms:

$$k^{*} \circ k  = i^{*} \circ j^{*} \circ j \circ i \cong i^{*} \circ i \cong \id_{A}$$
$$k \circ k^{*} = j \circ i  \circ i^{*} \circ j^{*} \cong j \circ j^{*} \cong \id_{C}$$

Assume $k$ and $j$ are groupoid equivalences:\\

We define $i^{*}(b) = k^{*} \circ j(b)$. The following chains of isomorphisms give the necessary natural isomorphisms:

$$i^{*} \circ i  = k^{*} \circ j \circ i = k^{*} \circ k \cong \id_{A}$$
$$i \circ i^{*} = i \circ k^{*} \circ j \cong j^{*} \circ j \circ i \circ k^{*} \circ j  \cong  j^{*} \circ k \circ k^{*} \circ j  \cong j^{*} \circ j \cong  \id_{B}$$

Assume $k$ and $i$ are groupoid equivalences:\\

We define $j^{*}(c) = i  \circ k^{*}(c)$. The following chains of isomorphisms give the necessary natural isomorphisms:

$$j^{*} \circ j  = i  \circ k^{*} \circ j \cong i  \circ k^{*} \circ j \circ i \circ i^{*} = i  \circ k^{*} \circ k \circ i^{*} \cong i \circ i^{*} \cong \id_{B}$$
$$j \circ j^{*} = j \circ i  \circ k^{*} = k \circ k^{*} \cong \id_{C}$$\\

Therefore, equivalences of groupoids satisfy the 2 out of 3 principle.

\end{proof}

\begin{defn}

Given a functor $F: A \to B$ and an object $b$ in $B$. We define a category $F^{-1}(b)$ to be category whose objects are objects $a$ in $A$ such that $F(a) = b$ and morphisms are the morphisms $m$ in $A$ such that $F(m) = \id_{b}$.

\end{defn}

\begin{lem}
\label{fiberpath}

Given a normal isofibration $F: A \to B$, and an isomorphism in $j: b \to b'$ in $B$, we have a equivalence of groupoids $\lift(j): F^{-1}(b) \to F^{-1}(b')$ and if $j = \id_{b}$, $\lift(j) = id_{F^{-1}(b)}$. Furthermore, we have a natural isomorphism $\lift(j,j'): \lift(j') \circ \lift(j) \cong \lift(j' \circ j)$ for $j: b' \to b''$.

\end{lem}

\begin{proof}

For object $a$ in $F^{-1}(b)$, $\lift(j)(a) = \Cod(\lift(\htpy{0},F)(a,j))$ and for a morphism $m:a \to a'$ in $F^{-1}(b)$, $\lift(j)(m) =  \lift(\htpy{0},F)(a', j) \circ m \circ \lift(\htpy{0},F)(a, j)^{-1} $. Functoriality of $\lift(j)$ is straightforward. \\

We will set $\lift(j)^{i} := \lift(j^{-1})$. We seek to construct natural isomorphisms $\lift(j,b):\lift(j^{-1}) \circ \lift(j) \to \id_{F^{-1}(b)}$, $\lift(j^{-1},b'):\lift(j) \circ \lift(j^{-1}) \to \id_{F^{-1}(b)}$. \\

For object $a$ in $F^{-1}(b)$, we set $\lift(j,b)(a) = \lift(\htpy{0},F)(a,j)^{-1} \circ \lift(\htpy{0},F)(\lift(j)(a),j^{-1})^{-1}$.\\

 For morphism $m:a \to a'$ in $F^{-1}(b)$, 
 \begin{align*}
 m \circ \lift(j,b)(a) &= m \circ  \lift(\htpy{0},F)(a,j)^{-1} \circ \lift(\htpy{0},F)(\lift(j)(a),j^{-1})^{-1}\\
&= \lift(\htpy{0},F)(a',j)^{-1} \circ \lift(j)(m) \circ \lift(\htpy{0},F)(\lift(j)(a),j^{-1})^{-1}\\
&= \lift(\htpy{0},F)(a',j)^{-1} \circ \lift(\htpy{0},F)(\lift(j)(a'),j^{-1})^{-1} \circ \lift(j^{-1})(\lift(j)(m))\\
&= \lift(j,b)(a') \circ \lift(j^{-1})(\lift(j)(m)).\\
\end{align*}

Without loss of generality, this gives us natural isomorphisms  $\lift(j,b):\lift(j^{-1}) \circ \lift(j) \to \id_{F^{-1}(b)}$, $\lift(j^{-1},b'):\lift(j) \circ \lift(j^{-1}) \to \id_{F^{-1}(b)}$. Thus $\lift(j)$ is an equivalence of groupoids. Finally if $j = \id_{b}$, $\lift(j) = \id_{F^{-1}(b)}$ since $F$ is normal.\\

Finally, for $a \in \ob F^{-1}(b)$, we define $\lift(j,j')(a)$ as the morphism allowing the following diagram to commute:

\[\begin{tikzcd}
	a &&&& {\lift(j)(a)} &&&&& {\lift(j')(\lift(j)(a))} \\
	\\
	a &&&&&&&&& {\lift(j'\circ j)(a)}
	\arrow["{\lift(\htpy{0},F)(\nm{a}, \nm{j})}", from=1-1, to=1-5]
	\arrow["{\id_{a}}"', from=1-1, to=3-1]
	\arrow["{\lift(\htpy{0},F)(\nm{\lift(j)(a)},\nm{j'})}", from=1-5, to=1-10]
	\arrow["{\lift(j,j')(a)}", from=1-10, to=3-10]
	\arrow["{\lift(\htpy{0},F)(\nm{a},\nm{j'\circ j})}"', from=3-1, to=3-10]
\end{tikzcd}\]

For a morphism $m:a \to a'$ in $F^{-1}(b)$, we have the following
\begin{align*}
\lift(j' \circ j)(m) \circ \lift(j,j')(a) &= \lift(\htpy{0},F)(\nm{a'}, \nm{j' \circ j}) \circ m \circ \lift(\htpy{0},F)(\nm{a}, \nm{j' \circ j})^{-1} \circ \lift(j,j')(a)\\
&=  \lift(\htpy{0},F)(\nm{a'}, \nm{j' \circ j}) \circ m \circ \lift(\htpy{0},F)(\nm{a},\nm{j})^{-1} \circ \lift(\htpy{0},F)(\nm{\lift(j)(a)},\nm{j'})^{-1}\\
&= \lift(j,j')(a') \circ \lift(\htpy{0},F)(\nm{\lift(j)(a')},\nm{j'}) \circ \lift(j)(m) \circ \lift(\htpy{0},F)(\nm{\lift(j)(a)},\nm{j'})^{-1}\\
&= \lift(j,j')(a') \circ \lift(j')(\lift(j)(m))\\
\end{align*}

Thus $\lift(j,j')$ is a natural isomorphism.

\end{proof}


\begin{defn}

We call $W: A \to B$ is strong equivalence of groupoids when $W$ is an equivalence of groupoids with $W^{i}:B \to A$, $\alpha: W^{i} \circ W \to \id_{A}$ and $\beta: W \circ W^{i} \to \id_{B}$ and we have that $\beta \cdot W = W \cdot \alpha$.

\end{defn}

\begin{defn}

We call $W: A \to B$ is strong deformation retract when $W$ is a strong equivalence of groupoids with $W^{i}:B \to A$, $\alpha: W^{i} \circ W \to \id_{A}$ and $\beta: W \circ W^{i} \to \id_{B}$ such that $W^{i}$ is a section of $W$ and $\beta = \id_{\id_{B}}$

\end{defn}

\begin{lem}
\label{deformr}

 If $F:A \to B$ is a normal isofibration and an equivalence of groupoids, then  $F$ is a strong deformation retract.

\end{lem}

\begin{proof}

Suppose $F$ is a normal isofibration and an equivalence of groupoids, we seek to construct a section of $F :A \to B$, $\lift(!, F)$. For object $b : \ob B$, set $\lift(!, F)(b) = \lift(\beta(b))(F^{i}(b))$, where $\beta:F \circ F^{i} \to \id_{B}$. For morphism $m \in \mor B$, $\lift(!, F)(m:b \to b') =  \lift(\htpy{0},F)(F^{i}(b'),\beta(b') ) \circ F^{i}(m) \circ \lift(\htpy{0},F)(F^{i}(b),\beta(b) )^{-1}$. functoriality of $\lift(!, F)$ follows from the functoriality of $F^{i}$. Furthermore, using the lifting property of $F$, one can see that $\lift(!, F)$ is a section of $F$.\\

Now we seek to construct a natural transformation $\alpha(!,F): \lift(!, F) \circ F \to id_{A}$. We have $\lift(\htpy{0},F)(F^{i}(-),\beta(-) ): F^{i} \to \lift(!, F)$. Thus we have $\lift(\htpy{0},F)(F^{i}(-),\beta(-) ) \cdot F: F^{i}\circ F \to \lift(!, F) \circ F$. Set $\alpha' := \alpha \circ (\lift(\htpy{0},F)(F^{i}(-),\beta(-) ) \cdot F)^{-1}: \lift(!, F) \circ F \to id_{A}$. Then set $\alpha(!,F) = \alpha' \circ ((\lift(!, F) \circ F )\cdot \alpha')^{-1}$. Then we have $$F \cdot \alpha(!,F)  = F \cdot \alpha'  \circ F \cdot ((\lift(!, F) \circ F )\cdot \alpha')^{-1} = F\cdot \alpha' \circ ((F \circ \lift(!, F) \circ F )\cdot \alpha')^{-1} = F\cdot \alpha'  \circ (F \cdot \alpha')^{-1} = F \cdot \id_{\id_{B}} = \id_{F}$$.\\

Thus the section $\lift(!, F)$ and the natural transformation $\alpha(!,F)$  make $F$ a strong deformation retract. 

\end{proof}


\begin{defn}

We call $W: A \to B$ is strong deformation section when $W$ is a strong equivalence of groupoids with $W^{i}:B \to A$, $\alpha: W^{i} \circ W \to \id_{A}$ and $\beta: W \circ W^{i} \to \id_{B}$ such that $W^{i}$ is a retraction of $W$ and $\alpha = \id_{\id_{A}}$

\end{defn}


\begin{lem}
\label{deciequi}

Suppose $j: A \to B$ is a decidable monomorphism on objects and an equivalence of groupoids, then $j$ is a strong deformation section

\end{lem}

\begin{proof}

First we will define a retraction $j^{r}:B \to A$. For $b : \ob B$ when $b = j(a)$, $j^{r}(b) = a$; $j^{r}(b) = j^{i}(b)$ otherwise. For $m : \mor B$, $j^{r}(m) = q(\cod(m)) \circ j^{i}(m) \circ q^{-1}(\dom(m))$ where $q(b) = \alpha(a): j^{i}(j(a)) \to a$ when $j(a) = b$ and $q(b) = \id_{j^{i}(b)}$ otherwise. composition and identity are straightforward. It is clear that $j^{r}(j(a)) = a$. Furthermore, $j^{r}(j(m)) = q(\cod(m)) \circ j^{i}(j(m)) \circ q^{-1}(\dom(m)) = m \circ q(\dom(m)) \circ q^{-1}(\dom(m)) = m$. Thus $j^{r}$ is a retract of $j$. \\

Now we want to construct $\beta^{r}:j\circ j^{r} \to \id_{B}$. We have $\alpha':j^{r} \to j^{i}$ where $\alpha'(b) = \alpha(a)^{-1}$ when $b = j(a)$ and $\alpha'(b) = \id_{j^{i}(b)}$ otherwise. Thus we have $j \cdot \alpha':j \circ j^{r} \to j \circ j^{i}$. Define $\beta' := \beta \circ j \cdot \alpha': j\circ j^{r} \to \id_{B}$. Note that $\beta'(j(a)) = \beta(j(a)) \circ j(\alpha'(j(a))) = \beta(j(a)) \circ j(\alpha(a))^{-1}$. So we set $\beta^{r} = \beta' \circ (\beta' \cdot (j \circ j^{r}) )^{-1}$. For $a :\ob A$, 
\begin{align*}
\beta^{r}(j(a)) &= \beta'(j(a)) \circ (\beta'(j(j^{r}(j(a)))) )^{-1}\\
&= \beta'(j(a)) \circ \beta'(j(a))^{-1} = \id_{j(a)} = j(\id_{a})\\
\end{align*}

Thus with the retract $j^{r}$ and natural transformation $\beta^{r}$, we have that $j$ is a strong deformation section.

\end{proof}


\begin{lem}

Given a functor $F: A \to B$, $\tilde{F}$ is a strong deformation section

\end{lem}

\begin{proof}

First we construct a retraction $\tilde{F}^{r}: \coe{F}{B} \to A$ of $\tilde{F}$. We set $\tilde{F}^{r}(a,b,f) = a$ and $$\tilde{F}^{r}((a,b,f),(a',b',f'),x,y,*) = x$$ functoriality of $\tilde{F}^{r}$ is straightforward. Additionally $$\tilde{F}^{r}(a,F(a),\id_{F(a)}) = a$$ and $$\tilde{F}^{r}((a,F(a),\id_{F(a)}),(a',F(a'),\id_{F(a')}),x,F(x),*) = x$$ so $\tilde{F}^{r}$ is a retract of $\tilde{F}$.\\

Now we seek to construct $\beta: \tilde{F} \circ \tilde{F}^{r} \to \id_{\coe{F}{B}}$. For $(a,b,f): \ob \coe{F}{B}$, $\beta(a,b,f) = ((a,F(a),\id_{F(a)}),(a,b,f),\id_{a},f, *)$. We will omit the domain and codomain for the sake of brevity. Given $(x,y,*)\in \mor \coe{F}{B}$, where $\dom(x,y,*) = (a,b,f)$ and $\cod(x,y,*) = (a',b',f')$, we have that $$(x,y,*) \circ (\id_{a},f,*) = (x, y \circ f,*) = (x, f' \circ F(x),*) = (\id_{a'},f',*) \circ (x, F(x),*)$$ So $$(x,y,*) \circ \beta(a,b,f) = \beta(a',b',f') \circ \tilde{F} (\tilde{F}^{r}(x,y,*) )$$ showing $\beta$ to be a natural transformation. We also have that $$\beta(a,F(a),\id_{F(a)}) = (\id_{a},\id_{F(a)},*) = \tilde{F}(\id_{a})$$ Therefore, $\tilde{F}$ is a strong deformation retract.

\end{proof}


\begin{lem}
\label{modstr1}

Given morphisms $i: U \to V$ and $f:A \to B$, $i$ has the left lifting property against all normal isofibrations if and only if $i$ is strong deformation section. $f$ has the right lifting property against all strong deformation sections if and only if $f$ is a normal isofibration.

\end{lem}

\begin{proof}

Suppose $j: A \to B$ is is a strong deformation section  and $F: X \to Y$ is a normal isofibration. We want a section $\lift(j,F): \Sq(j,F) \to \Hom(B,X)$ of $\pullp(j,F)$. First observe the following diagram:

\[\begin{tikzcd}
	A && X \\
	\\
	B && Y
	\arrow["s", from=1-1, to=1-3]
	\arrow["j"', from=1-1, to=3-1]
	\arrow["F", from=1-3, to=3-3]
	\arrow["t"', from=3-1, to=3-3]
\end{tikzcd}\]

Given $b \in \ob B$, we can construct a square:
\[\begin{tikzcd}
	1 && A && X \\
	\\
	\I && B && Y
	\arrow["{\nm{j^{i}(b)}}", from=1-1, to=1-3]
	\arrow["{\htpy{0}}", from=1-1, to=3-1]
	\arrow["s", from=1-3, to=1-5]
	\arrow["j", from=1-3, to=3-3]
	\arrow["F", from=1-5, to=3-5]
	\arrow["{\nm{\beta(b)}}"', from=3-1, to=3-3]
	\arrow["t"', from=3-3, to=3-5]
\end{tikzcd}\]

So that $\lift(j,F)(s,t)(b) = \Cod(\lift(\htpy{0},F)(s(j^{i}(b)),t(\beta(b))))$, for $m:b \to b'$, $$\lift(j,F)(s,t)(m) = \lift(\htpy{0},F)(s(j^{i}(b')),t(\beta(b'))) \circ s(j^{i}(m)) \circ  \lift(\htpy{0},F)(s(j^{i}(b)),t(\beta(b)))^{-1}$$ functoriality of $\lift(j,F)(s,t)$ follows from functoriality of $s \circ j^{i}$. \\

We have that $$F \circ \lift(j,F)(s,t)(b) = F(\Cod(\lift(\htpy{0},F)(s(j^{i}(b)),t(\beta(b))))) = t(b)$$ and 
\begin{align*}
F \circ \lift(j,F)(s,t)(m) &= F(\lift(\htpy{0},F)(s(j^{i}(b')),t(\beta(b'))) \circ s(j^{i}(m)) \circ  \lift(\htpy{0},F)(s(j^{i}(b)),t(\beta(b)))^{-1})\\
&= t(\beta(b')) \circ t(j(j^{i}(m))) \circ t(\beta(b))^{-1}\\
&= t( \beta(b') \circ j(j^{i}(m)) \circ \beta(b)^{-1}) \\
&= t(m)\\
\end{align*}

Thus we have $t = F \circ \lift(j,F)(s,t)$. Note for $a \in \ob A$, $\beta(j(a)) = j(\alpha(a)) = j(\id_{a}) = \id_{j(a)}$ since $j$ is a strong equivalence. Thus $$\lift(j,F)(s,t)(j(a)) = \Cod(\lift(\htpy{0},F)(s(j^{i}(j(a))),t(\beta(j(a))))) = \Cod(\lift(\htpy{0},F)(s(a),\id_{t(j(a))})) = s(a)$$ Similarly for $m:a \to a'$

\begin{align*}
\lift(j,F)(s,t)(j(m)) &= \lift(\htpy{0},F)(s(j^{i}(j(a'))),t(\beta(j(a')))) \circ s(j^{i}(j(m))) \circ  \lift(\htpy{0},F)(s(j^{i}(j(a))),t(\beta(j(a))))^{-1} \\
&= \id_{s(a')} \circ s(m) \circ \id_{s(a)} \\
&= s(m)\\
\end{align*}

Thus we have $s = \lift(j,F)(s,t) \circ j$. \\

Therefore, we have a lifting $\lift(j,F): \Sq(j,F) \to \Hom(B,X)$ which confirms that strong deformation section has a left lifting property with every normal isofibration.\\

Since we have the 'if' direction of both statements, we only need to show the 'only if' direction.\\

Suppose $f$ has a uniform right lifting property against all strong deformation sections, then since $\tilde{f}$ is a strong deformation section, we have the morphism $\lift(\tilde{f},f)$ which is the section of the morphism $\pullp(\tilde{f},f)$. Given the square $(\id_{A}, \hat{f})$, $\lift(\tilde{f},f)$ gives us a functor $f_{*}: \coe{f}{B} \to A $ such that $f_{*}\circ \tilde{f} = \id_{A}$ and $\hat{f} = f \circ f_{*}$. \\

Now we seek to construct a morphism $\lift(\partial_{0}, f):\Sq(\htpy{0},f) \to \Hom(\I, A) $. Given a square $(x:1 \to A, y: \I \to B)$, We set $\lift(\partial_{0}, f)(x,y) = f_{*}(\id_{x},y,*)$ where $(\id_{x},y,*)$ is an isomorphism between $\tilde{f}(x) = (x,f(x),\id_{f(x)})$ to $(x,y_{1}, y)$. We have $$\lift(\partial_{0}, f)(x,y) \circ x = f_{*}(\id_{x},y,*) \circ x = f_{*}(x,f(x),\id_{f(x)}) = f_{*} \circ \tilde{f}(x) = x$$ and $$y \circ \lift(\partial_{0}, f)(x,y) = f \circ f_{*}(\id_{x},y,*) = y$$ Now suppose we have a square $(x:1 \to A, \id_{f(x)}:\I \to B)$, then $$\lift(\partial_{0}, f)(x,\id_{f(x)}) = f_{*}(\id_{x},\id_{f(x)},*) = f_{*} \circ \tilde{f}(\id_{x}) = \id_{x}$$\\
 
Thus $f$ is a normal isofibration.\\

Suppose $i$ has a uniform left lifting property against all normal isofibrations, then since $\hat{i}$ is a normal isofibration, we have the morphism $\lift(i,\hat{i})$ which is the section of the morphism$\pullp(i,\hat{i})$. Given the square $(\hat{i}, \id_{V})$, $\lift(i,\hat{i})$ gives us a functor $i^{*}: V \to \coe{i}{V} $ such that $i^{*} \circ i = \tilde{i}$ and $\hat{i} \circ i^{*} = \id_{V}$.\\

Now we define the functor $i^{j}:V \to U$ as $\pr{0} \circ i^{*}$ where $\pr{0}: \coe{i}{V} \to U$. We have $$i^{j} \circ i(u) = \pr{0} \circ i^{*} \circ i(u) = \pr{0}(u,i(u),\id_{i(u)}) = u$$ Using a similar argument for morphisms, one can see that $i^{j} \circ i = \id_{U}$. For every $v:V$, we have $\pr{2} \circ i^{*}(v): i \circ i^{j}(v) \cong v$. Given a morphism $m:v \to v'$ in $V$, the morphism $i^{*}(m)$ confirms that $\pr{2} \circ i^{*}(v') \circ i^{j}(m) = m \circ \pr{2} \circ i^{*}(v)$. Thus $\beta = \pr{2} \circ i^{*}: i \circ i^{j} \cong \id_{V}$ is a natural isomorphism. Furthermore, $$\beta_{i(u)} = \pr{2} \circ i^{*}(i(u)) = \pr{2}(u,i(u), \id_{i(u)}) = \id_{i(u)}$$ So, $i \cdot \id_{\id_{U}} = \beta \cdot i$.\\

Thus $i$ is a strong deformation section.\\

This completes the proof.\\

\end{proof}

\begin{lem} 

Given two functors $F:A \to B$ and $G:X \to Y$ where $G$ is a normal isofibration, The morphism $ - \circ (\tilde{F},\id_{B}): \Sq(\hat{F},G) \to \Sq(F,G)$ has a section.

\end{lem}

\begin{proof}

We seek to construct a section $\lift(F,G): \Sq(F,G) \to \Sq(\hat{F},G)$. Suppose we have a square:
\[\begin{tikzcd}
	A && X \\
	\\
	B && Y
	\arrow["s_{0}", from=1-1, to=1-3]
	\arrow["F"', from=1-1, to=3-1]
	\arrow["G", from=1-3, to=3-3]
	\arrow["s_{1}"', from=3-1, to=3-3]
\end{tikzcd}\]

then by applying our fibrant factorization to $F$, we get the following diagram and lift:

\[\begin{tikzcd}
	A &&&& X \\
	\\
	{F \downarrow B} && B && Y
	\arrow["{s_{0}}", from=1-1, to=1-5]
	\arrow["{\tilde{F}}"', from=1-1, to=3-1]
	\arrow["G", from=1-5, to=3-5]
	\arrow["{\lambda_{F,G}}"{description}, dashed, from=3-1, to=1-5]
	\arrow["{\hat{F}}"', from=3-1, to=3-3]
	\arrow["{s_{1}}"', from=3-3, to=3-5]
\end{tikzcd}\]

so we set $\lift(F,G)(s_{0},s_{1}) = (\lambda_{F,G},s_{1})$. We also have that $$\lift(F,G)(s_{0},s_{1}) \circ (\tilde{F},\id_{B}) = (\lambda_{F,G},s_{1}) \circ (\tilde{F},\id_{B}) = (\lambda_{F,G} \circ \tilde{F},s_{1}) = (s_{0},s_{1})$$

Therefore $\lift(F,G)$ is a section of $ - \circ (\tilde{F},\id_{B})$.

\end{proof}

\begin{lem}
\label{froben}

The pullback of a strong deformation section along a normal isofibration is a strong deformation section

\end{lem}

\begin{proof}
Suppose we have a diagram:

\[
\begin{tikzcd}
 Z\arrow[r, " {\pr{0}}" ] \arrow[d,"\pr{1}" swap]   \arrow[dr,phantom,"\scalebox{1.5}{$\lrcorner$}" , very near start, color=black]  &  W\arrow[d,"j"] \\
X \arrow[r, " {F}" swap] &  Y
\end{tikzcd}
\]

where $F$ is a normal isofibration and $j$ is a strong deformation section. Let us construct a retract of $\pr{1}$, $w: X \to Z$. For $x : \ob X$, we set $$w(x) = (\Cod( \lift(\htpy{0},F)(x, \beta^{-1}(F(x))) ),  j^{i}(F(x)) )$$ when $\beta: j\circ j^{i}  \to \id_{Y}$. For $m:x \to x' \in \mor X$, $$w(m) = (  \lift(\htpy{0},F)(x', \beta^{-1}(F(x'))) \circ m \circ \lift(\htpy{0},F)(x, \beta^{-1}(F(x)))^{-1},  j^{i}(F(m)) )$$ Functoriality of $w$ follows from functoriality of $j^{i} \circ F$ and from the fact that $X$ is a groupoid assembly. \\

Suppose $(a,b): \ob Z$, then $F(a) = j(b)$. So $\beta(F(a)) = \id_{j(b)}$ and thus $$w(a) = \Cod(\lift(\htpy{0},F)(a, \id_{j(b)} ), j^{i}(j(b))) = (a,b)$$ For $(m:a\to a',n:b \to b'): \mor Z$, we have $F(m) = j(n)$ so
\begin{align*}
w(m) &= (  \lift(\htpy{0},F)(a', \beta^{-1}(j(b'))) \circ m \circ \lift(\htpy{0},F)(a, \beta^{-1}(j(b)))^{-1},  j^{i}(j(n)) )\\
&= (\id_{a'} \circ m \circ \id_{a}, n) \\
&= (m,n)
\end{align*}

Thus $w$ is a retraction of $\pr{1}$. \\

Now we construct a natural transformation $d: \pr{1} \circ w \to \id_{X}$. we set $d(x) = \lift(\htpy{0},F)(x, \beta^{-1}(F(x)))^{-1}$. Naturality should be obvious from reviewing the construction of $w$ on morphisms. Furthermore, for $(a,b) : \ob Z$, $$d(a) = \lift(\htpy{0},F)(a, \id_{j(b)} ) = \id_{a} = \pr{1}(\id_{(a,b)})$$ The functor $w$ and natural transformation $d$ makes $\pr{1}$ is strong deformation section completing the proof.

\end{proof}


\subsection{First algebraic weak factorizations system}
In this section, we take the factorization from previous section and construct an algebraic weak factorization system from it.\\

We define a functorial factorization $\Fib: \GrpA{A}^{\arr} \to \GrpA{A}^{\twarr}$. Given a functor $F: X \to Y$, $\Fib(F) = (\tilde{F},\hat{F})$. Given a square between functors $(\alpha,\beta): F \to F'$, $\Fib(\alpha,\beta) = (\alpha,E(\alpha,\beta),\beta)$. $E(\alpha,\beta) : = \pr{0}(\lift(F,\hat{F'})(\tilde{F'} \circ \alpha, \beta))$. Recall that $\lift(F,\hat{F'})$ is a section of $ - \circ (\tilde{F},\id_{B}): \Sq(\hat{F},\hat{F'}) \to \Sq(F,\hat{F'})$. So we have the following commutative diagram:

\[\begin{tikzcd}
	X &&& {X'} \\
	\\
	\\
	{F \downarrow Y} &&& {F' \downarrow Y'} \\
	\\
	\\
	Y &&& {Y'}
	\arrow["\alpha", from=1-1, to=1-4]
	\arrow["{\tilde{F}}"', from=1-1, to=4-1]
	\arrow["{\tilde{F'}}", from=1-4, to=4-4]
	\arrow["{E(\alpha,\beta)}"', from=4-1, to=4-4]
	\arrow["{\hat{F}}"', from=4-1, to=7-1]
	\arrow["{\hat{F'}}", from=4-4, to=7-4]
	\arrow["\beta"', from=7-1, to=7-4]
\end{tikzcd}\]

We can explicitly present $E(\alpha,\beta)$. For $(a,b,f) \in \ob \coe{F}{Y}$:
\begin{itemize}
\item For $(a,b,f) \in \ob \coe{F}{Y}$, $E(\alpha,\beta)(a,b,f) = (\alpha(a),\beta(b),\beta(f))$.\\
\item For $((a,b,f),(a',b',f'), x,y,*) \in \mor \coe{F}{Y}$,  $E(\alpha,\beta)(x,y,*) =  (\alpha(x),\beta(y),*)$

\end{itemize}

From this it is clear that $\Fib$ preserves identities and composition. $\Fib$ gives us a functorial fibrant factorization.

\begin{defn}

The functor $R_{\Fib}:\GrpA{A}^{\arr} \to \GrpA{A}^{\arr}$ is the functor $\Bot \circ \Fib$. 
The functor $L_{\Fib}:\GrpA{A}^{\arr} \to \GrpA{A}^{\arr}$ is the functor $\Top \circ \Fib$.

\end{defn}


\begin{lem}

$R_{\Fib}$ is a monad

\end{lem}

\begin{proof}

To prove this we need to construct natural transformations $\eta_{R_{\Fib}}:\Id_{\GrpA{A}^{\arr}} \To  R_{\Fib}$ and $\mu_{R_{\Fib}}: R_{\Fib} \circ R_{\Fib} \To R_{\Fib}$ and prove they satisfy the monad equations. First we define $\eta_{R_{\Fib}}$. Given a functor $F:X \to Y$, $\eta_{R_{\Fib}}(F) = (\tilde{F}, \id_{Y})$. Given a square $(\alpha,\beta):F \to F'$, we have that $$ (E(\alpha,\beta), \beta) \circ (\tilde{F}, \id_{Y}) =  (E(\alpha,\beta) \circ \tilde{F}, \beta) = (\tilde{F'} \circ \alpha, \beta) = (\tilde{F'}, \id_{Y'}) \circ (\alpha, \beta)$$ This demonstrates that $R_{\Fib}(\alpha,\beta) \circ  \eta_{R_{\Fib}}(F) = \eta_{R_{\Fib}}(F') \circ (\alpha,\beta)$ proving $\eta_{R_{\Fib}}$ to be a natural transformation. \\

Now we define $\mu_{R_{\Fib}}$. Given a functor $F:X \to Y$, $\mu_{R_{\Fib}}(F) = (\hat{F},\hat{F})(\id_{\coe{F}{Y}},\id_{Y}) = (\coe{F}{\id_{Y}},\id_{Y})$. Given $((a,b,f),c,g) \in \ob \coe{\hat{F}}{Y}$, $$\coe{F}{\id_{Y}}((a,b,f),c,g) = \Cod(\lift(\htpy{0},\hat{F})((a,b,f),g)) = (a,c,g\circ f)$$ Given $(((a,b,f),c,g),((a',b',f'),c',g')(x,y,*),h,*) \in \mor \coe{\hat{F}}{Y}$, $$\coe{F}{\id_{Y}}((x,y,*),h,*) =  \Cod(\lift(\htpy{0},\hat{F})((x,y,*),(g,g'))) = ((a,c,g\circ f),(a',c',g'\circ f'),x,h,*)$$

 Given a square $(\alpha,\beta):F \to F'$, we have that $(E(\alpha,\beta), \beta) \circ (\coe{F}{\id_{Y}},\id_{Y})  = (E(\alpha,\beta) \circ \coe{F}{\id_{Y}},\beta)$, $E(\alpha,\beta)(a,c,g\circ f) = (\alpha(a),\beta(c),\beta(g\circ f))$ and $E(\alpha,\beta)(x,h,*) = (\alpha(x),\beta(h),*)$\\

We also have that $(\coe{F'}{\id_{Y'}},\id_{Y'}) \circ(E(E(\alpha,\beta),\beta), \beta) = (\coe{F'}{\id_{Y'}} \circ E(E(\alpha,\beta),\beta),\beta)$. for $((a,b,f),c,g)\in \ob \coe{\hat{F}}{Y}$, we have $$(\coe{F'}{\id_{Y'}} \circ E(E(\alpha,\beta),\beta),\beta)((a,b,f),c,g) = \coe{F'}{\id_{Y'}}((\alpha(a),\beta(b),\beta(f)),\beta(c),\beta(g)) = (\alpha(a),\beta(b),\beta(g\circ f))$$ Given $(((a,b,f),c,g),((a',b',f'),c',g'),(x,y,*),h,*) \in \mor \coe{\hat{F}}{Y}$, we have $$(\coe{F'}{\id_{Y'}} \circ E(E(\alpha,\beta),\beta),\beta)((x,y,*),h,*) = \coe{F'}{\id_{Y'}}((\alpha(x),\beta(y),*),\beta(h),* ) = (\alpha(x),\beta(h),*)$$ \\

Thus $E(\alpha,\beta) \circ \coe{F}{\id_{Y}} = \coe{F'}{\id_{Y'}} \circ E(E(\alpha,\beta),\beta)$. Therefore, $(\coe{F'}{\id_{Y'}},\id_{Y'}) \circ(E(E(\alpha,\beta),\beta), \beta) = (E(\alpha,\beta), \beta) \circ (\coe{F}{\id_{Y}},\id_{Y}) $, proving $\mu_{R_{\Fib}}$ to be a natural transformation.\\

Now we need to show that $\mu_{R_{\Fib}} \circ (R_{\Fib}\cdot \eta_{R_{\Fib}}) = \mu_{R_{\Fib}} \circ (\eta_{R_{\Fib}} \cdot  R_{\Fib}) = \Id_{R_{\Fib}}$ and $\mu_{R_{\Fib}} \circ (\mu_{R_{\Fib}} \cdot  R_{\Fib}) = \mu_{R_{\Fib}} \circ  (R_{\Fib} \cdot \mu_{R_{\Fib}}) $. Recall that for a functor $F:X \to Y$, $\eta_{R_{\Fib}}(F) = (\tilde{F}, \id_{Y})$ and $\mu_{R_{\Fib}}(F) = (\coe{F}{\id_{Y}},\id_{Y})$

$$(\coe{F}{\id_{Y}},\id_{Y}) \circ R_{\Fib}(\tilde{F}, \id_{Y}) = (\coe{F}{\id_{Y}},\id_{Y}) \circ (E(\tilde{F}, \id_{Y}),\id_{Y})$$
$$\coe{F}{\id_{Y}} \circ E(\tilde{F}, \id_{Y})(a,b,f) = \coe{F}{\id_{Y}}((a,F(a),\id_{F(a)}),b,f) = (a,b,f)$$
$$\coe{F}{\id_{Y}} \circ E(\tilde{F}, \id_{Y})(x,y,*) = \coe{F}{\id_{Y}}((x,F(x),*),y,*) = (x,y,*)$$
$$\coe{F}{\id_{Y}} \circ E(\tilde{F}, \id_{Y}) = \id_{\coe{F}{Y}}$$
$$(\coe{F}{\id_{Y}},\id_{Y}) \circ R_{\Fib}(\tilde{F}, \id_{Y}) =  (\id_{\coe{F}{Y}},\id_{Y}) = \Id_{R_{\Fib}}(F)$$
Thus $\mu_{R_{\Fib}} \circ (R_{\Fib}\cdot \eta_{R_{\Fib}}) = \Id_{R_{\Fib}}$

$$(\coe{F}{\id_{Y}},\id_{Y}) \circ (\tilde{\hat{F}}, \id_{Y}) = (\coe{F}{\id_{Y}} \circ \tilde{\hat{F}}, \id_{Y})$$
$$\coe{F}{\id_{Y}} \circ \tilde{\hat{F}}(a,b,f) = \coe{F}{\id_{Y}}((a,b,f),b,\id_{b}) = (a,b,f)$$
$$\coe{F}{\id_{Y}} \circ \tilde{\hat{F}}(x,y,*) = \coe{F}{\id_{Y}}((x,y,*),y,*) = (x,y,*)$$
$$\coe{F}{\id_{Y}} \circ \tilde{\hat{F}}= \id_{\coe{F}{Y}}$$
$$(\coe{F}{\id_{Y}},\id_{Y}) \circ (\tilde{\hat{F}}, \id_{Y})  =  (\id_{\coe{F}{Y}},\id_{Y}) = \Id_{R_{\Fib}}(F)$$
Thus $\mu_{R_{\Fib}} \circ (\eta_{R_{\Fib}} \cdot  R_{\Fib}) = \Id_{R_{\Fib}}$

$$(\coe{F}{\id_{Y}},\id_{Y}) \circ (E(\coe{F}{\id_{Y}},\id_{Y}),\id_{Y}) = (\coe{F}{\id_{Y}} \circ E(\coe{F}{\id_{Y}},\id_{Y}), \id_{Y})$$
$$\coe{F}{\id_{Y}} \circ E(\coe{F}{\id_{Y}},\id_{Y})(((a,b,f),c,g),d,h) = \coe{F}{\id_{Y}}((a,c,g \circ f),d,h) = (a,d,h \circ g \circ f)$$
$$\coe{F}{\id_{Y}} \circ E(\coe{F}{\id_{Y}},\id_{Y})(((x,y,*),z,*),w,*) = \coe{F}{\id_{Y}}((x,z,*),w,*) = (x,w,*)$$

$$(\coe{F}{\id_{Y}},\id_{Y}) \circ (\coe{\hat{F}}{\id_{Y}},\id_{Y}) = (\coe{F}{\id_{Y}} \circ \coe{\hat{F}}{\id_{Y}}, \id_{Y})$$
$$\coe{F}{\id_{Y}} \circ \coe{\hat{F}}{\id_{Y}}(((a,b,f),c,g),d,h) = \coe{F}{\id_{Y}}((a,b, f),d,h \circ g ) = (a,d,h \circ g \circ f)$$
$$\coe{F}{\id_{Y}} \circ \coe{\hat{F}}{\id_{Y}}(((x,y,*),z,*),w,*) = \coe{F}{\id_{Y}}((x,y,*),w,*) = (x,w,*)$$
$$\coe{F}{\id_{Y}} \circ E(\coe{F}{\id_{Y}},\id_{Y}) =  \coe{F}{\id_{Y}} \circ \coe{\hat{F}}{\id_{Y}}$$
Thus $\mu_{R_{\Fib}} \circ (\mu_{R_{\Fib}} \cdot  R_{\Fib}) = \mu_{R_{\Fib}} \circ  (R_{\Fib} \cdot \mu_{R_{\Fib}})$.\\

Therefore, $R_{\Fib}$ is a monad.

\end{proof}

\begin{defn}

A functor $F: X \to Y$ is an $R_{\Fib}$ algebra when there exists a square $(s,t):\hat{F} \to F$ such that $(s,t) \circ \eta_{R_{\Fib}}(F) = \id_{F}$ and $(s,t) \circ \mu_{R_{\Fib}}(F) = (s,t) \circ R_{\Fib}(s,t)$. Since $\eta_{R_{\Fib}}(F) = (\tilde{F},\id_{Y})$, $t = \id_{Y}$. Thus we will present an $R_{\Fib}$ algebra as $(F,s)$

\end{defn}




\begin{defn}

A morphism between two $R_{\Fib}$ algebras $(F:X \to Y, s)$ and $(F':X' \to Y',s')$ is a square $(x,y):F \to F'$ such that $s' \circ E(x,y) = x \circ s$

\end{defn}


\begin{lem}

There exists a composition structure on $R_{\Fib}$ algebras: Given two $R_{\Fib}$ algebras  $(F:X \to Y, s)$ and $(G:Y \to Z, \sigma)$, one can construct an $R_{\Fib}$ algebra $(G \circ F, s \cdot \sigma)$ such that for any two $R_{\Fib}$ algebra morphisms $(x,y):(F,s) \to (F',s')$ and $(y,z): (G,\sigma) \to (G', \sigma')$, $(x,z):(G \circ F, s \cdot \sigma) \to (G' \circ F', s' \cdot \sigma')$ is an $R_{\Fib}$ algebra morphism.

\end{lem}

\begin{proof}
Given $(F:X \to Y, s)$ and $(G:Y \to Z, \sigma)$, we will construct $s \cdot \sigma: \coe{(G \circ F)}{Z} \to X$. Given $(x,z,h): \ob \coe{(G \circ F)}{Z}$, we have a morphism $(\id_{F(x)},h,*):(F(x), G\circ F(x),\id_{G\circ F(x)}) \to (F(x), z, h)$, so $\sigma(\id_{F(x)},h,*): F(x) \to \sigma(F(x), z, h)$ where $\sigma\circ \tilde{G}(F(x)) = F(x)$. So we have an object $(x,\sigma(F(x), z, h),\sigma(\id_{F(x)},h,*)) : \ob \coe{F}{Y}$. Thus we set $s \cdot \sigma(x,z,h) = s(x,\sigma(F(x), z, h),\sigma(\id_{F(x)},h,*))$.\\

 Given  $((x,z,h),(x',z',h'),k,l,*): \mor \coe{(G \circ F)}{Z}$, we have a square $((\id_{F(x)}, h,*),(\id_{F(x')},h',*)): (F(k),G\circ F(k),*) \to (F(k),l,*)$. This becomes the square $(\sigma(\id_{F(x)}, h,*),\sigma(\id_{F(x')},h',*)): F(k) \to \sigma(F(k),l,*)$ where where $\sigma\circ \tilde{G}(F(k)) = F(k)$. So we have a morphism $(k,\sigma(F(k),l,*),*) \in \mor \coe{F}{Y}$. Thus we set $s \cdot \sigma(k,l,*) = s(k,\sigma(F(k),l,*),*)$. functoriality of $s \cdot \sigma$ follows from functoriality of $s$ and $\sigma$.\\

Now we need to show that $G \circ F \circ (s \cdot \sigma) = \hat{G \circ F}$. For $(x,z,h): \ob \coe{(G \circ F)}{Z}$, $G \circ F \circ (s \cdot \sigma) (x,z,h) = G \circ F(s(x,\sigma(F(x), z, h),\sigma(\id_{F(x)},h,*))) = G(\sigma(F(x), z, h)) = z$. For $((x,z,h),(x',z',h'),k,l,*): \mor \coe{(G \circ F)}{Z}$, $G \circ F \circ (s \cdot \sigma)(k,l,*) = G \circ F(s(k,\sigma(F(k),l,*),*)) = G(\sigma(F(k),l,*)) = l$. Therefore, $G \circ F \circ (s \cdot \sigma) = \hat{G \circ F}$.\\

Now we seek to show that $s \cdot \sigma$ gives a $R_{\Fib}$-algebra structure to $G \circ F$. First we need to show that $(s \cdot \sigma) \circ \tilde{G \circ F} = \id_{X}$. For $x \in \ob X$, $(s \cdot \sigma) \circ \tilde{G \circ F}(x) = s \cdot \sigma(x, G\circ F(x),\id_{G\circ F(x)}) = s(x,\sigma(F(x),G\circ F(x),\id_{G\circ F(x)}), \sigma(\id_{F(x)},\id_{G\circ F(x)},*)) = s(x, F(x),\id_{F(x)}) = x$. For $m \in \ob X$, $(s \cdot \sigma) \circ \tilde{G \circ F}(m) = s \cdot \sigma(m, G\circ F(m),*) = s(m, \sigma(F(m),G\circ F(m),*),*) = s(m, F(m),*) = m$. Thus $(s \cdot \sigma) \circ \tilde{G \circ F} = \id_{X}$.\\

Now we will show that $(s \cdot \sigma) \circ E(s \cdot \sigma, \id_{Z}) = (s \cdot \sigma) \circ \mu_{\Fib}(G \circ F)$. 

\begin{align*}
(s \cdot \sigma) \circ E(s \cdot \sigma, \id_{Z})((a,b,f),c,g) &= s \cdot \sigma(s \cdot \sigma(a,b,f),c,g)\\
&= s (s \cdot \sigma(a,b,f),\sigma(F(s \cdot \sigma(a,b,f)),c,g), \sigma(\id_{F(s \cdot \sigma(a,b,f))},g,*))\\
\end{align*}
since $s \cdot \sigma(a,b,f) = s(a, \sigma(F(a),b,f),\sigma(\id_{F(a)},f,*))$ we have 
\begin{align*}
(s \cdot \sigma) \circ E(s \cdot \sigma, \id_{Z})&((a,b,f),c,g)  \\
&=  s \circ E(s,\id_{Y})((a, \sigma(F(a),b,f),\sigma(\id_{F(a)},f,*)), \sigma(F(s \cdot \sigma(a,b,f)),c,g), \sigma(\id_{F(s \cdot \sigma(a,b,f))},g,*))\\
&= s \circ \mu_{\Fib}(F)((a, \sigma(F(a),b,f),\sigma(\id_{F(a)},f,*)), \sigma(F(s \cdot \sigma(a,b,f)),c,g), \sigma(\id_{F(s \cdot \sigma(a,b,f))},g,*))\\
\end{align*}
since $F(s(a, \sigma(F(a),b,f),\sigma(\id_{F(a)},f,*))) = \sigma(F(a),b,f)$
\begin{align*}
(s \cdot \sigma) \circ E(s \cdot \sigma, \id_{Z})((a,b,f),c,g)  &= s \circ \mu_{\Fib}(F)((a, \sigma(F(a),b,f),\sigma(\id_{F(a)},f,*)), \sigma(\sigma(F(a),b,f),c,g), \sigma(\id_{\sigma(F(a),b,f)},g,*))\\
&= s(a,\sigma(\sigma(F(a),b,f)),c,g), \sigma(\id_{\sigma(F(a),b,f)},g,*) \circ \sigma(\id_{F(a)},f,*))\\
\end{align*}

\begin{align*}
(s \cdot \sigma) \circ \mu_{\Fib}(G \circ F)((a,b,f),c,g) &= s \cdot \sigma(a,c,g \circ f)\\
&= s (a, \sigma(F(a),c,g \circ f), \sigma(\id_{F(a)}, g \circ f,*))\\
&= s (a, \sigma \circ \mu_{\Fib}(G)((F(a),b,f)c,g), \sigma\circ \mu_{\Fib}(\id_{F(a)},f,*)), g,*))\\
&= s (a, \sigma \circ E(\sigma, \id_{Z})((F(a),b,f)c,g), \sigma\circ E(\sigma, \id_{Z})((\id_{F(a)},f,*) g,*))\\
&= s (a, \sigma(\sigma(F(a),b,f)c,g), \sigma(\sigma(\id_{F(a)},f,*) g,*)\\
\end{align*}

First we have that $\sigma(\id_{\sigma(F(a),b,f)},g,*)  = \sigma \circ E(\sigma,\id_{Z})((\id_{F(a)},\id_{b},*),g,*) = \sigma \circ \mu_{\Fib}(G)((\id_{F(a)},\id_{b},*),g,*) = \sigma(\id_{F(a)},g,*)$. So $\sigma(\id_{\sigma(F(a),b,f)},g,*) \circ \sigma(\id_{F(a)},f,*) = \sigma(\id_{F(a)},g,*) \circ \sigma(\id_{F(a)},f,*) = \sigma(\id_{F(a)}, g \circ f,*)$. Thus $(s \cdot \sigma) \circ E(s \cdot \sigma, \id_{Z})((a,b,f),c,g) = (s \cdot \sigma) \circ \mu_{\Fib}(G \circ F)((a,b,f),c,g) $. A similar argument can be made for morphisms. Therefore, $(s \cdot \sigma) \circ E(s \cdot \sigma, \id_{Z}) = (s \cdot \sigma) \circ \mu_{\Fib}(G \circ F)$.\\

Suppose we have  $R_{\Fib}$ algebra morphisms $(\alpha,\beta):(F,s) \to (F',s')$ and $(\beta,\gamma): (G,\sigma) \to (G', \sigma')$, we need to show that $(s' \cdot \sigma') \circ E(\alpha,\gamma) = \alpha \circ (s \cdot \sigma)$. Suppose we have $(x,z,n): \coe{(G\circ F)}{Z}$, we have that 
\begin{align*}
\alpha \circ (s \cdot \sigma)(x,z,n) &= \alpha(s(x,\sigma(F(x),z,n),\sigma(\id_{F(x)},n,*) ))\\
&= s'(E(\alpha,\beta)(x,\sigma(F(x),z,n),\sigma(\id_{F(x)},n,*) ))\\
&= s'(\alpha(x),\beta\circ \sigma(F(x),z,n),\beta \circ \sigma(\id_{F(x)},n,*) )\\
&= s'(\alpha(x),\sigma' \circ E(\beta,\gamma)(F(x),z,n),\sigma' \circ E(\beta,\gamma)(\id_{F(x)},n,*) )\\
&= s'(\alpha(x),\sigma'(\beta \circ F(x),\gamma(z),\gamma(n)),\sigma'(\id_{\beta \circ F(x)},\gamma(n),*) )\\
&= s'(\alpha(x),\sigma'(G \circ \alpha(x),\gamma(z),\gamma(n)),\sigma'(\id_{G \circ \alpha(x)},\gamma(n),*) )\\
&= s' \cdot \sigma'(\alpha(x),\gamma(z),\gamma(n))\\
&= (s' \cdot \sigma')\circ E(\alpha,\gamma)(x,z,n)\\
\end{align*}
A similar argument can be made for morphisms. Therefore, $(s \cdot \sigma) \circ E(s \cdot \sigma, \id_{Z}) = (s \cdot \sigma) \circ \mu_{\Fib}(G \circ F)$ showing $R_{\Fib}$ algebras have a composition structure.

\end{proof}

\begin{defn}

We call the triple $(\Fib,R_{\Fib},L_{\Fib})$ an {\bf algebraic weak factorization system} if:
\begin{itemize}

\item $\Fib$ is a weak factorization system
\item $R_{\Fib}$ is a monad
\item $L_{\Fib}$ is a comonad
\item The canonical natural transformation $L_{\Fib} \circ R_{\Fib} \To R_{\Fib} \circ L_{\Fib}$ as described in  \cite[pgs. 8-9]{Riehl2011} is a distributive law of the comonad over the monad.

\end{itemize}

\end{defn}

\begin{cor}
\label{awfsgrpd}
$(\Fib,R_{\Fib},L_{\Fib})$ presents an algebraic weak factorization system on $\GrpA{A}$ where the $R$-morphisms are normal isofibrations and the $L$-morphism are strong deformation sections

\end{cor}
\begin{proof}
Examining $\eta_{R_{\Fib}}$ shows that $R_{\Fib}$ is a monad over $\Bot$. Furthermore $R_{\Fib}$ has a composition structure. The fact that $(\Fib,R_{\Fib},L_{\Fib})$ presents an algebraic weak factorization system on $\GrpA{A}$ follows from Theorem 2.24 of \cite{Riehl2011}. Furthermore, $R$-morphisms are retracts of normal isofibrations, and $L$-morphisms are retracts of strong deformation sections.

\end{proof}

\subsection{Construction of dependent products}
Here we give a construction of pushforwards over normal isofibrations.\\

Given a groupoid assembly $X$, we have the comma category $\GrpA{A}/X$. Here we adapt a definition from \cite{wreo22517} which will be useful in constructing dependent products:

\begin{defn}

Given a normal isofibration $F:N \to M$, a functor $G:P \to N$, a morphism $f:m \to m'$ in $M$, and $T: N_{m} \to P$ and $T': N_{m'} \to P$ such that $G \circ T:N_{m} \mono N$ and $G \circ T':N_{m'} \mono N$ are inclusions, a {\bf generalized natural transformation} $\eta: T \rightsquigarrow T'$ over $f$ is a dependent morphism $\eta: \Pi_{u \in (\mor N)_{f}} P(T(\dom(u)),T'(\cod(u)))$ such that for $u,v \in (\mor N)_{f}$ when we have the following square 

\[\begin{tikzcd}
	X && Y \\
	\\
	{X'} && {Y'}
	\arrow["u", from=1-1, to=1-3]
	\arrow["i"', from=1-1, to=3-1]
	\arrow["j", from=1-3, to=3-3]
	\arrow["v"', from=3-1, to=3-3]
\end{tikzcd}\]

where $F(i) = \id_{A}$ and $F(j) = \id_{B}$ then we have the following commutative square:

\[\begin{tikzcd}
	T(X) && T'(Y) \\
	\\
	{T(X')} && {T'(Y')}
	\arrow["\eta_{u}", from=1-1, to=1-3]
	\arrow["T(i)"', from=1-1, to=3-1]
	\arrow["T'(j)", from=1-3, to=3-3]
	\arrow["\eta_{v}"', from=3-1, to=3-3]
\end{tikzcd}\]\\

\end{defn}





\begin{lem}
Given a normal isofibration $F:N \to M$, $\Delta_{F}: \GrpA{A}/M \to \GrpA{A}/N$ which is the pullback functor has a right adjoint $\Pi_{F}: \GrpA{A}/N \to \GrpA{A}/M$.

\end{lem}
\begin{proof}

We define the functor $\Pi_{F}: \GrpA{A}/N \to \GrpA{A}/M$ where for the functor $G: P \to N$, we define the groupoid $\Pi_{F}(G)$ whose objects are the pairs $(m,T:N_{m} \to P)$ where $G \circ T: N_{m} \mono N$ is an inclusion and morphisms are the pairs $(f:m \to m', \eta: T \rightsquigarrow T')$ where for $u \in (\mor N)_{f}$, $G(\eta(u)) = u$. Given morphisms $(f:m \to m', \eta: T \rightsquigarrow T')$ and $(f':m' \to m'', \omega: T' \rightsquigarrow T'')$, we define their composition to be the pair $(f' \circ f, \omega \circ \eta)$ where for $w \in (\mor N)_{f' \circ f}$, $\omega \circ \eta(w) = \omega(w \circ r^{-1}) \circ \eta(r )$ where $r = \lift(\htpy{0},F)(\dom(w),f)$. Note that the choice of factorization for $w$ is irrelavant since if we have a factorization $w = w_{1} \circ w_{0}$ where $F(w_{0}) = f$ and $F(w_{1}) = f'$ then we have the following commutative square:

\[\begin{tikzcd}
	{\dom(w)} && X && {\cod(w)} \\
	\\
	{\dom(w)} && Y && {\cod(w)}
	\arrow["r", from=1-1, to=1-3]
	\arrow["{\id_{\dom(w)}}"', from=1-1, to=3-1]
	\arrow["{w \circ r^{-1}}", from=1-3, to=1-5]
	\arrow["{w_{0} \circ r^{-1}}"{description}, from=1-3, to=3-3]
	\arrow["{\id_{\cod(w)}}", from=1-5, to=3-5]
	\arrow["{w_{0}}"', from=3-1, to=3-3]
	\arrow["{w_{1}}"', from=3-3, to=3-5]
\end{tikzcd}\]

thus by definition of a generalized natural transformation we then have the following square:

\[\begin{tikzcd}
	{T(\dom(w))} && T'(X) && {T''(\cod(w))} \\
	\\
	{T(\dom(w))} && T'(Y) && {T''(\cod(w))}
	\arrow["\eta(r)", from=1-1, to=1-3]
	\arrow["{\id_{T(\dom(w))}}"', from=1-1, to=3-1]
	\arrow["{\omega(w \circ r^{-1})}", from=1-3, to=1-5]
	\arrow["{T'(w_{0} \circ r^{-1})}"{description}, from=1-3, to=3-3]
	\arrow["{\id_{T''(\cod(w))}}", from=1-5, to=3-5]
	\arrow["{\eta(w_{0})}"', from=3-1, to=3-3]
	\arrow["{\omega(w_{1})}"', from=3-3, to=3-5]
\end{tikzcd}\]

making $\omega(w \circ r^{-1}) \circ \eta(r ) = \omega(w_{1}) \circ \eta(w_{0})$. Given an object $(m, T)$, $\id_{(m,T)} = (\id_{m}, [u \in (\mor N)_{\id_{m}} \mapsto T(u)])$. For a morphism $(f, \eta)$, the inverse is $(f^{-1}, [u \mapsto \eta(u^{-1})^{-1}])$. Given a morphism $Q:G \to G'$ in $ \GrpA{A}/N $, we construct a functor $\Pi_{F}(Q): \Pi_{F}(G) \to \Pi_{F}(G')$ where the object $(m,T)$ gets sent to $(m, Q \circ T)$ and the morphism $(f,\eta)$ gets sent to $(f,[w \mapsto Q(\eta(w))])$.\\

Given $X,Y \in \GrpA{A}/N $, we define $\GrpA{A}/N(X,Y) \in \ob \Asm(A)$ in the internal language as $\GrpA{A}/N(X,Y) :=\{Z \in \Hom(\dom(X),\dom(Y))|  Y \circ Z = X \}$. So we seek to show the following equivalence: $$\GrpA{A}/N(\Delta_{F}(X),Y) \cong \GrpA{A}/M(X,\Pi_{F}(Y))$$ 

Suppose we have $Z \in \GrpA{A}/N(\Delta_{F}(X),Y)$, then we will define $\phi(Z) \in \GrpA{A}/M(X,\Pi_{F}(Y))$. Take $x \in \ob( \dom(X))$, setting $m_{x} = X(x)$ we define $\phi(Z)(x): N_{m_{x}} \to \dom(Y)$ where $\phi(Z)(x)(n) = Z(x,n)$ for $n \in \ob N_{m_{x}}$ and $\phi(Z)(x)(w) = Z(\id_{x},w)$ for $w \in \mor N_{m_{x}}$. functoriality of $\phi(Z)(x)$ follows from functoriality of $Z$. Taking $u:x \to x' \in \mor( \dom(X))$, we define $\phi(Z)(u): \phi(Z)(x) \to \phi(Z)(x')$ to be the pair $(X(u), [h \in (\mor N)_{X(u)}  \mapsto Z(u, h)])$. It should be straightforward that $[h \in (\mor N)_{X(u)}  \mapsto Z(u, h)]$ gives a generalized natrual transformation as well as the fact that $Y(Z(u, h)) = \Delta_{F}(X)(u,h) = h$.\\

We have that $\phi(Z)(\id_{x'}) = (X(\id_{x'}) = \id_{m_{x'}}, [h \in (\mor N)_{X(\id_{x'})}  \mapsto Z(\id_{x'}, h) = \phi(Z)(x')(h)]) = \id_{(m_{x'},  \phi(Z)(x'))}$. Also since $F$ is a normal isofibration, we can always express a morphism over some $X(u) \circ X(v) \in \mor M$ as a composition of morphisms over $X(u)$ and $X(v)$ thus $\phi(Z)$ preserve composition of morphisms. Furthermore every morphism over $X(u)^{-1}$ is just the inverse of a morphism over $X(u)$ thus $\phi(Z)$ preserves inverses. So $\phi(Z)$ is a functor.\\

Suppose $W \in  \GrpA{A}/M(X,\Pi_{F}(Y))$, we construct $\psi(W) \in \GrpA{A}/N(\Delta_{F}(X),Y)$, take $(x,n) \in \ob( \dom(\Delta_{F}(X)))$, then $\psi(W)(x,n) = W(x)(n)$ and for $(u,j) \in  \mor( \dom(\Delta_{F}(X))) $, $\psi(W)(u,j) = W(u)(j)$. functoriality of $\psi(W) $ follows from functoriality of $W$ as well as the construction of composition, identity and inverses of morphisms in $\Pi_{F}(Y)$.\\

We have that $\phi(\psi(W))(x)(n) = \psi(W)(x,n) = W(x)(n)$ where $x$ and $n$ are either objects or morphisms and $\psi(\phi(Z))(x,n) = \phi(Z)(x)(n) = Z(x,n)$ where $x$ and $n$ are both either objects or morphisms. Thus $\GrpA{A}/N(\Delta_{F}(X),Y) \cong \GrpA{A}/M(X,\Pi_{F}(Y))$. Therefore. $\Pi_{F}$ is the right adjoint of $\Delta_{F}$.

\end{proof}

\begin{rmk}
\label{altdepprod}

We have a alternative characterization of $\Pi_{F}(X)$ where the objects are the same but the morphisms $(f,\eta): (m,T) \to (m',T')$ are natural transformations $\eta: T \To T' \circ \lift(f)$ where $X(\eta_{n}) = \lift(\htpy{0},F)(\nm{n},\nm{f})$. $\id_{(m,T)} = (\id_{m}, \id_{T})$ and for $(f,\eta): (m,T) \to (m',T')$  and $(f',\eta'): (m',T') \to (m'',T'')$, $(f', \eta') \circ (f,\eta) = (f' \circ f , [n \in N_{m} \mapsto T''(\lift(f,f')(n)) \circ \eta'_{\lift(f)(n)} \circ \eta_{n}])$. For inverses we have $(f:m \to m',\eta: T \To T' \circ \lift(f))^{-1} = (f^{-1},[n \in N_{m'} \mapsto \eta_{\lift(f^{-1})(n)}^{-1} \circ T'(\lift(f^{-1},f)(n)^{-1}) ])$.\\

The reason for this has to do with existence of generalized natural transformations $ \lift(f)_{-}:  \lift(f) \rightsquigarrow \Id_{N_{m'}}$ for $f \in \mor M$ where for $t:n \to n'$ over $f:m \to m'$, $\lift(f)_{t}$ is the morphism allowing the following diagram to commute:

\[\begin{tikzcd}
	&&& {\lift(f)(n)} \\
	\\
	n &&& {n'}
	\arrow["{{\lift(f)_{t}}}", from=1-4, to=3-4]
	\arrow["{{\lift(\htpy{0},F)(\nm{n},\nm{f})}}"{description}, from=3-1, to=1-4]
	\arrow["t"', from=3-1, to=3-4]
\end{tikzcd}\]

 Whiskering $\lift(f)_{-}$ with some $T': N_{m'} \to \dom(X)$ also produces a generalized natural transformation and each $T' \cdot \lift(f)_{-}$ for $f:m \to m' \in \mor M$ and $T': N_{m'} \to \dom(X)$ can be used to establish a bijection between generalized natural transformations $(f:m \to m, \eta:T \rightsquigarrow T')$ and natural transformations $(f:m \to m', \eta: T \To T' \circ \lift(f))$.\\

\end{rmk}

\subsection{Dependent products preserve fibrations}
Here we prove that normal isofibrations satisfies the Frobenius property.\\


\begin{cor}
The right adjoint of the pullback of along normal isofibrations preserve normal isofibrations.

\end{cor}
\begin{proof}

By Corollary \ref{awfsgrpd}, the class of strong deformation sections and normal isofibrations forms a weak factorization system. Thus this corollary is satisfied by applying Proposition 2.4 of \cite{GAMBINO20173027} to Lemma \ref{froben}.\\

\end{proof}

We refer to $\Fib(X)$ as the full subcategory of $\GrpA{A}/X$ containing the normal isofibrations. In which case we have the following corollary:
\begin{cor}

Given a normal isofibration $F:X \to Y$ we have the following adjoint pair:

$$\Delta_{F}: \Fib(Y) \leftrightarrows  \Fib(X):\Pi_{F}  $$

\end{cor}
\begin{proof}

By the previous corollary we have that dependent products on normal isofibrations preserve isofibrations and we already have that pullbacks preserve normal isofibrations since normal isofibrations are right maps of a weak factorization system by Lemma \ref{modstr1}. Therefore, we can restrict the $\Pi_{F} \vdash \Delta_{F}$ adjunction from $\GrpA{A}/-$ to $\Fib(-)$.

\end{proof}

\begin{cor}

$\Fib(X)$ for every $X \in \ob \GrpA{A}$ is cartesian closed.\\ 

\end{cor}
\begin{proof}

Given normal isofibrations $F,G \in \Fib(X)$, it is easy to show that $F \times G = \Sigma_{F} \circ \Delta_{F}(G)$ where $\Sigma_{F} = F \circ -$ using the universal property of pullbacks. Thus we define $F \to G = \Pi_{F} \circ \Delta_{F}(G)$ giving us the following isomorphisms:
\begin{align*}
\Fib(X)(H, F \to G) &= \Fib(X)(H, \Pi_{F} \circ \Delta_{F}(G))\\
&\cong \Fib(X)(\Delta_{F}(H),\Delta_{F}(G))\\
&\cong \Fib(X)(\Sigma_{F} \circ \Delta_{F}(H),G)\\
&= \Fib(X)(F \times H, G)
\end{align*}

\end{proof}

\section{W-Types, Model Structures and Localizations}
This section is dedicated to $W$-types with reductions and their applications. We first start with construct plain $W$-types in $\GrpA{A}$ mainly as a premier for the next section where we will show the existence of $W$-types with reductions. We will then apply $W$-types with reductions to construct a version of the small object argument in $\GrpA{A}$. We will apply this small object argument to construct a cofibrant algebraic weak factorization system for the model structure on $\GrpA{A}$ and to construct localization modalities.\\

\subsection{Existence of W-Types of Groupoid Assemblies}
In this section, we show the existence of $W$-types manly to showcase a methodology similar to the one we will use in proving the existence of $W$-types with reductions in $\GrpA{A}$.\\

Given the existence of the following diagram in $\GrpA{A}$:

\[\begin{tikzcd}
	& X &&& Y \\
	\\
	Z &&&&& Z
	\arrow["f", from=1-2, to=1-5]
	\arrow["g"', from=1-2, to=3-1]
	\arrow["h", from=1-5, to=3-6]
\end{tikzcd}\]

where $f$ is a normal isofibration, we define the functor $\poly_{f,g,h} = \sum_{h} \circ \Pi_{f} \circ g^{*}: \GrpA{A}/Z \to \GrpA{A}/Z$. Given some $G: D \to Z$, the domain of $\poly_{f,g,h}(G)$ is constructed as follows:

\begin{itemize}

\item The objects of $\poly_{f,g,h}(G)$ are $(z, y, T: X_{y} \to D \times_{Z} X)$ where $\Delta_{g}(G) \circ T = \id_{X}|_{X_{y}}$ and $h(y) = z$. 

\item The morphisms are $(c:z \to z',v:y \to y',  \phi: T \To T' \circ \lift(v))$ where $\Delta_{g}(G)(\phi_{x}) = \lift(\htpy{0},f)(\nm{x},\nm{v})$ for $x \in \ob X_{y}$ and $h(v) = c$.

\end{itemize}

Here will we expound on the results of sections 3.1 and 6.1 of \cite{wreo22517} to show that there is a $W$-type for $\poly_{f,g,h}$ in the internal language of $\GrpA{A}$.\\

First we seek to construct a $W$-type for $f:X \to Y$.  Given a morphism $v:y \to y'$ in $Y$, we define the sets: $$L_{v} := \{\lift(\htpy{0},f)(\nm{x},\nm{v})|  f(x) = y\}$$ $$\overline{X}_{y} := \{x \in \ob X|  f(x) = y\} \cup \{u \in \mor X|  f(u) = \id_{x} \}$$ $$\overline{X}_{v} :=  \overline{X}_{y} + \overline{X}_{y'} + L_{v}$$ $$\overline{Y} := \ob Y + \mor Y$$ $$\overline{X} := \Sigma_{y \in \ob Y}\overline{X}_{y}  +  \Sigma_{v \in \mor Y}\overline{X}_{v}$$

Now we define a morphism $\overline{f}: \overline{X} \to \overline{Y}$ in $\Asm(A)$ as a projection morphism. Set $W$ as the $W$-type corresponding to $\poly_{\overline{f}}$ where $\alpha: \poly_{\overline{f}}(W) \to W$ is the canonical algebra on $W$.
By Lemma 3.1.1 of \cite{wreo22517}, we can construct a reflexive graph $W_{g}$ using elements of $W$ as follows:

\begin{itemize}

\item The vertices of $W_{g}$ are the elements $\alpha(y, T: \overline{X}_{y} \to W)$ for $y \in \ob Y$

\item The edges of $W_{g}$ are the elements $\alpha(v:y \to y',\phi: \overline{X}_{v} \to W)$ for $v \in \ob Y$

\item The source of $\alpha(v,\phi: \overline{X}_{v} \to W)$ is $\alpha(y,\phi|_{\overline{X}_{y}})$ and the target is $\alpha(y',\phi|_{\overline{X}_{y'}})$

\item The edge $\id_{\alpha(y,T)}$ defined as $\alpha(\id_{y}, \id_{T}: \overline{X}_{\id_{y}} \to W)$ where $\id_{T}(x) = T(x)$ which works since $\overline{X}_{\id_{y}} = L_{\id_{y}} + \overline{X}_{y} + \overline{X}_{y}$ and $L_{\id_{y}}  \subseteq \overline{X}_{y}$

\end{itemize}

Finally we can consider $\overline{X}_{u}$ where $u \in \overline{Y}$ a reflexive graph where the vertices are the object, the edges are the morphisms and the reflexive edges are the identity morphisms. Now we can present the following definition:

\begin{defn}

We call an $\alpha(u,T) \in W$ a {\bf hereditary graph morphism} if $T: \overline{X}_{u} \to W$ is a graph morphism from $\overline{X}_{u}$ into $W_{g}$ and for all $x \in \overline{X}_{u}$, $T(x)$ is a hereditary graph morphism.

\end{defn}

By Lemma 3.1.3 of  \cite{wreo22517}, we can take $\tilde{W_{g}}$ as a subgraph of $W_{g}$ whose vertices and edges are all hereditary graph morphisms. In fact by Proposition 3.1.4 of \cite{wreo22517}, $\tilde{W_{g}}$ is inductively generated by the following conditions:

\begin{itemize}
\item For $y \in \ob Y$ and graph morphism $T: \overline{X}_{y} \to \tilde{W_{g}}$, (a, T) is a vertex in $ \tilde{W_{g}}$

\item For $v:y \to y' \in \mor Y$, $(y,T)$, $(y',T')$ and a collection of edges in $\tilde{W_{g}}$ of the form: $$\phi = \{\phi_{x}: T(x) \to T'(\lift(v)(x)) | f(x) = y\}$$ then $(v,\phi)$ is an edge from $(y,T)$ to $(y',T')$

\end{itemize}

The arrow $\id_{(y,T)}$ is defined using the pair $(\id_{y}, \{T(\id_{x})|  f(x) = y\})$. Note that $\lift(\htpy{0},f)(\nm{x},\nm{\id_{y}}) = \id_{x}$ since $f$ is a normal isofibration, thus $\lift(\id_{y})(x) = x$.\\

Given edges $(v, \phi): (y,T) \to (y',T')$ and $(v^{*},\phi^{*}): (y',T') \to (y'',T'')$, we define $(v^{*},\phi^{*}) \circ (v, \phi)$ inductively as $(v^{*} \circ v, \phi^{*} \circ \phi)$ where $$(\phi^{*} \circ \phi)_{x} := \{(T''(\lift(v,v^{*})(x)) \circ \phi^{*}_{\lift(v)(x)}) \circ \phi_{x}| f(x) = y \}$$ 

Note the parallels with Remark \ref{altdepprod}.\\

Now we define a new set of vertices and edges:

\begin{defn}

We call a vertex $(y,T: \overline{X}_{y} \to \tilde{W_{g}})$ in $\tilde{W_{g}}$ {\bf functorial} if $T(m \circ n) = T(m) \circ T(n)$ for composable $m,n \in \mor X_{y}$, $T(\id_{x}) = \id_{T(x)}$ for $x \in \ob X_{y}$ and each $T(x)$ for $x \in \ob X_{y}$ is functorial and each $T(m)$ for $m \in \mor X_{y}$ is natural.\\

We call an edge $(v, \phi): (y, T) \to (y', T')$ in $\tilde{W_{g}}$ {\bf natural} if for $n:x \to x' \in \mor X_{y}$ we have that $$\phi_{x'} \circ T(n) = T'(\lift(v)(n)) \circ \phi_{x}$$ and each $\phi_{x}$ for $x \in \ob X_{y}$ is natural.

\end{defn}

Let $W_{f}$ be the subgraph of $\tilde{W_{g}}$ with functorial vertices and natural edges, then we have thee following:

\begin{lem}
$W_{f}$ is a groupoid assembly. 

\end{lem}
\begin{proof}

The argument stated Proposition 6.1.2 of \cite{wreo22517} that tells us that $W_{f}$ is a groupoid can be replicated in the internal language of $\Asm(A)$.

\end{proof}

\begin{lem}

There is an algebra structure $\alpha_{f}: \poly_{f}(W_{f}) \to W_{f}$ on $W_{f}$. Furthermore, $\alpha_{f}$ is an isomorphism.

\end{lem}
\begin{proof}

This is proposition 6.1.3 of \cite{wreo22517}.

\end{proof}

\begin{lem}
\label{subalgpoly}

$W_{f}$ is its own smallest $\poly_{f}$-subalgebra.

\end{lem}
\begin{proof}

We will replicate the argument from Lemma \ref{wellf} using the $W$-type $W$. Suppose we have a subalgebra $D \mono W_{f}$ and suppose we have the assembly $D' = \{w \in W |  w \in W_{f} \imply  w \in D\}$ where $w$ can either be an object or a morphism. We seek to show that $D'$ is $\DirST$-inductive. \\

Suppose for $w = \alpha(u,T) \in W$ and for $x \in \overline{X}_{u}$ being an object or morphism we have that $T(x)  \in D'$ and suppose $w \in \ob W_{f}$, then $T$ is functorial and $T(x) \in W_{f}$ thus $T(x) \in D$. This means that we have a functor $T: X_{u} \to D$. Thus $\alpha(u,T) \in \ob D$. \\

If $w \in \mor W_{f}$, then $T$ is natural and can be characterized as a collection of morphisms $$\{T_{x}: T'(x) \to T''(\lift(u)(x))| f(x) = \dom(u) \}$$ 

Plus, for $x \in \ob X_{\dom(u)}$, $T_{x} \in \mor W_{f}$ thus $T_{x} \in \mor D$. Additionally since the pairs $(\dom(u),T')$ and $(\cod(u), T'')$ exist in $W_{f}$,  by the previous argument for objects, the pairs must exist in $D$. So we have a natural transformation $T$ in $D$, thus $\alpha(u,T) \in \ob D$.\\

This means that $D'$ is $\DirST$-inductive. Thus $D' = W$ and by implication $D = W_{f}$.

\end{proof}

\begin{lem}
\label{wgini}

 $W_{f}$ is the $W$-type associated with $f$.

\end{lem}
\begin{proof}

Here we will imploy a proof technique from Lemma $\ref{fwtype}$. Given an object $w \in \ob W_{f}$, we construct the subgroupoid assembly $\Paths_{w} \subseteq W_{f}$ inductively as follows:

\begin{enumerate}
\item $w \in \ob \Paths_{w}$
\item $\id_{w} \in \Paths_{w}(w,w)$
\item For $w' \in \ob \Paths_{w}$, when $w' = \alpha_{f}(y,T)$, for $x \in \ob X_{y}$,  $T(x) \in \ob\Paths_{w}$ and for $m \in \mor X_{y}$, $T(m) \in \Paths_{w}(T(\dom(m)),T(\cod(m)))$
\item For $v \in \Paths_{w}(w',w'')$, when $v = \alpha_{f}(v,\phi )$, $w' = (y,T)$, and $w'' = (y',T')$ for $x \in X_{y}$, $\phi_{x} \in \Paths_{w}(T(x),T'(\lift(v)(x)))$
\item $\Paths_{w}$ is closed under composition

\end{enumerate}

Identities in $\Paths_{w}$ can be construct via parts $(ii)$ and $(iv)$ and inverses follow from the fact that $X_{y}$ is a groupoid assembly and parts $(iii)$ and $(iv)$. \\

Given a morphism $u \in \mor W_{f}$, we construct $\Paths_{u} \subseteq W_{f}$ similarly:

\begin{enumerate}
\item $\dom(u), \cod(u) \in \ob \Paths_{u}$
\item $u \in \Paths_{u}(\dom(u),\cod(u))$, $u^{-1} \in \Paths_{u}(\cod(u),\dom(u))$
\item For $w' \in \ob \Paths_{u}$, when $w' = \alpha_{f}(y,T)$, for $x \in \ob X_{y}$,  $T(x) \in \ob\Paths_{u}$ and for $m \in \mor X_{y}$, $T(m) \in \Paths_{u}(T(\dom(m)),T(\cod(m)))$
\item For $v \in \Paths_{u}(w',w'')$, when $v = \alpha_{f}(j,\phi )$, $w' = (y,T)$, and $w'' = (y',T')$ for $x \in X_{y}$, $\phi_{x} \in \Paths_{u}(T(x),T'(\lift(j)(x)))$
\item $\Paths_{u}$ is closed under composition.

\end{enumerate}

Finally, given composable $u,u' \in \mor W_{f}$, we construct $\Paths_{u,u'} \subseteq W_{f}$:
\begin{enumerate}
\item $\dom(u), \cod(u) = \dom(u'),\cod(u') \in \ob \Paths_{u,u'}$
\item $u \in \Paths_{u,'u}(\dom(u),\cod(u))$, $u^{-1} \in \Paths_{u,u'}(\cod(u),\dom(u))$
\item $u' \in \Paths_{u,u'}(\dom(u'),\cod(u'))$, $u'^{-1} \in \Paths_{u,u'}(\cod(u'),\dom(u'))$
\item For $w' \in \ob \Paths_{u}$, when $w' = \alpha_{f}(y,T)$, for $x \in \ob X_{y}$,  $T(x) \in \ob\Paths_{u}$ and for $m \in \mor X_{y}$, $T(m) \in \Paths_{u}(T(\dom(m)),T(\cod(m)))$
\item For $v \in \Paths_{u}(w',w'')$, when $v = \alpha_{f}(j,\phi )$, $w' = (y,T)$, and $w'' = (y',T')$ for $x \in X_{y}$, $\phi_{x} \in \Paths_{u}(T(x),T'(\lift(j)(x)))$
\item $\Paths_{u}$ is closed under composition.

\end{enumerate}

Just like with $\Paths_{w}$, we can say that $\Paths_{u}$ and $\Paths_{u,u'}$ has inverses and identities. It should be obvious that $\Paths_{u'},\Paths_{u} \subseteq \Paths_{u,u'}$ and $\Paths_{\dom(u)}, \Paths_{\cod(u)} \subseteq \Paths_{u}$. But also observe that $\Paths_{u' \circ u} \subseteq \Paths_{u,u'}$ since if we have $v = \alpha_{f}(j,\phi) \in \Paths_{u}(w,w')$ and $v' = \alpha_{f}(j',\phi') \in \Paths_{u'}(w',w'')$, for $x \in \ob X_{\dom(u)}$, $\lift(j,j')(x) \in \mor X_{\cod(u')}$ so $(\phi' \circ \phi)_{x} = (T''(\lift(j,j')(x)) \circ \phi'_{\lift(j)(x)} \circ \phi_{x}) \in \mor\Paths_{u,u'}$.\\

Given a $\poly_{f}$-algebra $\beta: \poly_{f}(C) \to C$, we call a functor $\gamma_{u}: \Paths_{u} \to C$ an attempt on $u$ whether $u$ is an object or a morphism if when we have $w = \alpha_{f}(y,T) \in \ob \Paths_{u}$, $\gamma_{u}(w) = \beta(y, \gamma_{u} \circ T)$ and when we have $v = \alpha_{f}(j, \phi) \in \mor \Paths_{u}$, $\gamma_{u}(v) = \beta(j, \gamma_{u} \cdot \phi)$. A similar notion holds for $\Paths_{u,u'}$ as well. \\

Let $U_{f}$ contain all of the objects and morphisms in $W_{f}$ with a unique attempt, we seek to show that $U_{f}$ is a groupoid assembly. For $x \in \ob U_{f}$, note that $\Paths_{x} = \Paths_{\id_{x}}$ thus any attempt of $x$ must be an attempt of $\id_{x}$ and vice versa. So $\id_{x}$ must have a unique attempt. Given $j \in \mor U_{f}$, note that $\Paths_{j} = \Paths_{j^{-1}}$ so $j^{-1}$ must have a unique attempt. Suppose we have composable $j, j' \in \mor U_{f}$, 
note that unique attempts $\gamma_{j}$ and $\gamma_{j'}$ must agree on the unique attempt $\gamma_{\cod(j)} = \gamma_{\dom(j)}$. With that we can construct an attempt $\gamma_{u,u'}: \Paths_{u,u'} \to C$ where $\gamma_{u,u'}|_{\Paths_{u}} = \gamma_{u}$, $\gamma_{u,u'}|_{\Paths_{u'}} = \gamma_{u'}$ and for composable $k \in \mor \Paths_{u}$ and $k' \in \mor \Paths_{u'}$, $\gamma_{u,u'}(k' \circ k) = \gamma_{u'}(k') \circ \gamma_{u}(k)$. $\gamma_{u,u'}$ is unique since it must agree on $\Paths_{u}$ and $\Paths_{u'}$ with $\gamma_{u}$ and $\gamma_{u'}$ respectively and preserve composition. Thus we have an  attempt $\gamma_{u' \circ u} = \gamma_{u,u'}|_{\Paths_{u' \circ u}}$. Suppose we have another attempt $\gamma'_{u' \circ u}$, we construct a new attempt $\gamma'_{u' \circ u, u'^{-1}}$ in the same way we constructed $\gamma_{u,u'}$. Then $ \gamma'_{u' \circ u,u'^{-1}}|_{\Paths_{u}} = \gamma_{u}$ by uniqueness of $\gamma_{u}$. Then 
\begin{align*}
\gamma_{u}(u) &= \gamma'_{u' \circ u,u'^{-1}}(u) \\
&= \gamma_{u'^{-1}}(u'^{-1}) \circ \gamma'_{u' \circ u}(u' \circ u)\\
&= \gamma_{u'}(u'^{-1})\circ \gamma'_{u' \circ u}(u' \circ u)\\
\end{align*}

Thus $$\gamma'_{u' \circ u}(u' \circ u) = \gamma_{u'}(u') \circ \gamma_{u}(u) = \gamma_{u' \circ u}(u' \circ u)$$

In fact we can make a similar argument for subtrees of $u' \circ u$ and $u$. Given $x \in \ob X_{\dom(u)}$ and $u = (j, \phi): (y,T) \to (y',T')$ and $u' = (j', \phi'): (y',T') \to (y'', T'')$ we have:
\begin{align*}
\gamma_{u}(\phi_{x}) &= \gamma'_{u' \circ u,u'^{-1}}(\phi_{x}) \\
&= \gamma'_{u' \circ u,u'^{-1}}(T'(\lift(j' \circ j, j'^{-1})(x)) \circ \phi'^{i}_{ \lift(j' \circ j)(x)} \circ (\phi' \circ \phi)_{x}) \\
&= \gamma'_{u' \circ u,u'^{-1}}(T'(\lift(j' \circ j, j'^{-1})(x)) \circ \phi_{\lift(j'^{-1})(\lift(j' \circ j)(x))}'^{-1} \circ T''(\lift(j^{-1},j)(\lift(j' \circ j)(x))^{-1}) \circ (\phi' \circ \phi)_{x}) \\
&= \gamma_{u' \circ u, u'^{-1}}(\phi_{\lift(j)(x)}'^{-1} \circ T''(\lift(j')(\lift(j' \circ j, j'^{-1})(x))) \circ T''(\lift(j^{-1},j)(\lift(j' \circ j)(x))^{-1}) \circ (\phi' \circ \phi)_{x} )\\
&= \gamma_{u' \circ u, u'^{-1}}(\phi_{\lift(j)(x)}'^{-1} \circ T''( \lift(j')(\lift(j' \circ j, j'^{-1})(x)) \circ \lift(j^{-1},j)(\lift(j' \circ j)(x))^{-1}) \circ (\phi' \circ \phi)_{x} )\\
&= \gamma_{u' \circ u, u'^{-1}}(\phi_{\lift(j)(x)}'^{-1} \circ T''(\lift(j' \circ j)(x) )^{-1}) \circ (\phi' \circ \phi)_{x} )\\
&= \gamma_{u'}(\phi_{\lift(j)(x)}'^{-1} \circ T''(\lift(j' \circ j)(x) )^{-1})) \circ \gamma'_{u' \circ u}((\phi' \circ \phi)_{x} )\\
\end{align*}

The penultimate step is justified by the following diagram:

\[\begin{tikzcd}[column sep=small]
	&&& {\lift(j' \circ j)(x)} &&&&& {\lift(j' \circ j)(x)} \\
	\\
	x &&& {\lift(j' \circ j)(x)} &&& {\lift(j'^{-1})\circ \lift(j' \circ j)(x)} && {\lift(j') \circ \lift(j'^{-1})\circ \lift(j' \circ j)(x)} \\
	\\
	x &&&&&& {\lift(j)(x)} && {\lift(j')(\lift(j)(x))}
	\arrow["{{\lift(\htpy{0},f)(\nm{\lift(j' \circ j)(x)},\nm{j' \circ j'^{-1}}) = \id_{\lift(j' \circ j)(x)}}}", from=1-4, to=1-9]
	\arrow["\id", from=1-4, to=3-4]
	\arrow["{{\lift(j^{-1},j)(\lift(j' \circ j)(x))^{-1}}}"{description}, from=1-9, to=3-9]
	\arrow["{{\lift(\htpy{0},f)(\nm{x},\nm{j' \circ j})}}", from=3-1, to=3-4]
	\arrow["{{\id_{x}}}"', from=3-1, to=5-1]
	\arrow[from=3-4, to=3-7]
	\arrow[from=3-7, to=3-9]
	\arrow["{\lift(j')(\lift(j' \circ j, j'^{-1})(x)}"{description}, from=3-9, to=5-9]
	\arrow["{{\lift(\htpy{0},f)(\nm{x},\nm{j})}}"', from=5-1, to=5-7]
	\arrow[from=5-7, to=5-9]
\end{tikzcd}\]

By the above we have:$$\gamma_{u' \circ u}((\phi' \circ \phi)_{x}) = \gamma_{\cod(u')}(T''(\lift(j' \circ j)(x))) \circ \gamma_{u'}(\phi'_{\lift(j)(x)}) \circ \gamma_{u}(\phi_{x}) = \gamma'_{u' \circ u}((\phi' \circ \phi)_{x})$$

We can inductively apply this argument to $\gamma_{u'}(\phi'_{\lift(j)(x)}) \circ \gamma_{u}(\phi_{x})$ and to the following morphisms and thus cover every morphism in $\Paths_{u,u'}$ derived from $u' \circ u$. Since $\gamma'_{u' \circ u}$ must coincide with $\gamma_{\dom(u)}$ and $\gamma_{\cod(u')}$ and preserve composition, then $\gamma'_{u' \circ u} = \gamma_{u' \circ u}$ and $u' \circ u \in U_{f}$. Thus $U_{f}$ is a subgroupoid assembly of $W_{f}$.\\

Now we want to show that $U_{f}$ is a subalgebra of $W_{f}$. Suppose we have a functor $T:X_{y} \to U_{f}$, we construct an attempt $\gamma_{\alpha(y,T)}: \Paths_{\alpha(y,T)} \to C$ where 

\begin{itemize}
\item $\gamma_{\alpha(y,T)}(\alpha(y,T)) = \beta(y,[x \in X_{y} \mapsto \gamma_{T(x)}(T(x))])$
\item $\gamma_{\alpha(y,T)}(\id_{\alpha(y,T)}) = \id_{\beta(y,[x \in X_{y} \mapsto \gamma_{T(x)}(T(x))])} = \beta(\id_{y}, [x \in \ob X_{y} \mapsto \gamma_{T(x)}(\id_{T(x)})])$
\item For $x \in X_{y}$ and $w \in \Paths_{T(x)}$, $\gamma_{\alpha(y,T)}(w) = \gamma_{T(x)}(w)$
\item For composable $u,v \in \mor\Paths_{\alpha(y,T)}$, $\gamma_{\alpha(y,T)}(v \circ u) = \gamma_{\alpha(y,T)}(v) \circ \gamma_{\alpha(y,T)}(u)$\\
\end{itemize}

$\gamma_{\alpha(y,T)}$ is an attempt on $\alpha(y,T)$ since each $\gamma_{T(x)}$ is an attempt and by definition of $\gamma_{\alpha(y,T)}(\alpha(y,T)) $ and $\gamma_{\alpha(y,T)}(\id_{\alpha(y,T)})$. Furthermore it is a functor by definition. Furthermore, $\gamma_{\alpha(y,T)}$ is unique since $\gamma_{\alpha(y,T)}(\alpha(y,T)) $ and $\gamma_{\alpha(y,T)}(\id_{\alpha(y,T)}) $ are given by the necessary construction from unique attempts on each $T(x)$ in order for $\gamma_{\alpha(y,T)}$ to be an attempt. Thus $\alpha(y,T) \in \ob U_{f}$.\\

Suppose we have a natural transformation $\phi: T \To T' \circ \lift(v)$ where $v:y \to y'$, $T: X_{y} \to U_{f}$, and $T':X_{y'} \to U_{f}$; with and $\phi_{x} \in \mor U_{f}$ for $x \in X_{y}$ and $T,T' \in \ob U_{f}$. We construct an attempt $\gamma_{\alpha(v,\phi)}:\Paths_{(v,\phi)} \to C$ as follows:

\begin{itemize}
\item $\gamma_{\alpha(v,\phi)}(\alpha(v,\phi)) = \beta(v, [x \in \ob X_{y} \mapsto  \gamma_{\phi_{x}}(\phi_{x})])$
\item $\gamma_{\alpha(v,\phi)}(\alpha(v,\phi)^{-1}) = \beta(v^{-1}, [x \in \ob X_{y'} \mapsto \gamma_{\phi_{\lift(v^{-1})(x)}}(\phi_{\lift(v^{-1})(x)}^{-1}) \circ \gamma_{\alpha(y',T')}(T'(\lift(v^{-1},v)(x)^{-1})) ])$
\item For $x \in X_{y}$ and $w \in \Paths_{\phi_{x}}$, $\gamma_{\alpha(v,\phi)}(w) = \gamma_{\phi_{x}}(w)$
\item For $w \in \Paths_{\alpha(y,T)}$, $\gamma_{\alpha(v,\phi)}(w) = \gamma_{\alpha(y,T)}(w)$
\item For $w \in \Paths_{\alpha(y',T')}$, $\gamma_{\alpha(v,\phi)}(w) = \gamma_{\alpha(y',T')}(w)$
\item For composable $u,v \in \mor\Paths_{\alpha(v,\phi)}$, $\gamma_{\alpha(v,\phi)}(v \circ u) = \gamma_{\alpha(v,\phi)}(v) \circ \gamma_{\alpha(v,\phi)}(u)$\\
\end{itemize}

Note that since we have $\gamma_{\alpha(v,\phi)}(\alpha(v,\phi)^{-1}): \gamma_{\alpha(y',T')} \circ T' \to \gamma_{\alpha(y,T)} \circ T$ the construction of $\gamma_{\alpha(v,\phi)}(\alpha(v,\phi)^{-1})$ is precisely what one would expect in $\poly_{f}(C)$ if referencing Remark \ref{altdepprod} for the construction of inverses. Thus similarly to $\gamma_{\alpha(y,T)}$, we have that $\gamma_{\alpha(v,\phi)}$ is an attempt on $\alpha(v,\phi)$. We have that $\gamma_{\alpha(v,\phi)}$ is a unique attempt for similar reasons as with $\gamma_{\alpha(y,T)}$. Thus $\alpha(v,\phi) \in \mor U_{f}$. Therefore, $U_{f}$ is a subalgebra of $W_{f}$ and thus $U_{f} = W_{f}$. This also gives a $\poly_{f}$-algebra morphism $\gamma: W_{f} \to C$ where $\gamma(w) = \gamma_{w}(w)$.\\

If we have another $\poly_{f}$-algebra morphism $\gamma': W_{f} \to C$, then we can take the equalizer $\Eq{\gamma, \gamma'}$ which can easily be shown to a subalgebra of $W_{f}$ by using the fact that both $\gamma$ and $\gamma'$ preserves the algebra structure of $W_{f}$. Thus $W_{f} = \Eq{\gamma, \gamma'}$ and $\gamma = \gamma'$. Therefore, $W_{f}$ is the initial algebra of $\poly_{f}$.

\end{proof}

Now going back to the original functor $\poly_{f,g,h}$, we seek to construct an initial algebra for it. First we take our $W$-type for $f$, $W_{f}$ then we give the following definition:

\begin{defn}

An object $\alpha(y, T:X_{y} \to W_{f})$ is called {\bf (g,h)-coherent} if for $x \in X_{y}$ either an object or a morphism, $g(x) = h(y_{x})$ where $T(x) = \alpha(y_{x}, T_{x})$ and $T(x)$ is $(g,h)$-coherent. A morphism $\alpha(v:y \to y', \phi: T \To T' \circ \lift(v))$ is called {\bf (g,h)-coherent} if for $x \in \ob X_{y}$, $g(\lift(\htpy{0},f)(\nm{x},\nm{v})) = h(v_{x})$ where $\phi(x) = (v_{x}, \phi_{x})$ and $\phi_{x}$ is $(g,h)$-coherent.

\end{defn}

\begin{lem}
If we take $W_{f,g,h} \subseteq W_{f}$ as containing only the $(g,h)$-coherent objects and morphisms, then $W_{f,g,h}$ is a groupoid assembly. 
\end{lem}
\begin{proof}

Given an $(g,h)$-coherent morphism $(v,\phi)$ between $(g,h)$-coherent objects, the inverse is $$(v^{-1},\phi^{i})  = (v^{-1},[x \in X_{y'} \mapsto \phi_{\lift(v^{-1})(x)}^{-1} \circ T'(\lift(v^{-1},v)(x)^{-1}) ])$$ Suppose each $\phi^{i}(x)$ is $(g,h)$-coherent, since $\phi^{-1}_{\lift(v^{-1})(x)} = \alpha(v_{\lift(v^{-1})(x)}^{-1}, \psi)$ with $$g(\lift(\htpy{0},f)(\nm{\lift(v^{-1})(x)},\nm{v})) = h(v_{\lift(v^{-1})(x)})$$ and $$g(\lift(v^{-1},v)(x)) = h(v_{\lift(v^{-1},v)(x)})$$ for $T'(\lift(v^{-1},v)(x)) = (v_{T'(\lift(v^{-1},v)(x)^{-1})}, \phi_{T'(\lift(v^{-1},v)(x)^{-1})})$, we have the following diagram:

\[\begin{tikzcd}
	x &&&& {\lift(v^{-1})(x)} &&&&& {\lift(v)(\lift(v^{-1})(x))} \\
	\\
	x &&&&&&&&& {x}
	\arrow["{\lift(\htpy{0},f)(\nm{x}, \nm{v^{-1}})}", from=1-1, to=1-5]
	\arrow["{\id_{x}}"', from=1-1, to=3-1]
	\arrow["{\lift(\htpy{0},f)(\nm{\lift(v^{-1})(x)},\nm{v})}", from=1-5, to=1-10]
	\arrow["{\lift(v^{-1},v)(x)}", from=1-10, to=3-10]
	\arrow["{\lift(\htpy{0},f)(\nm{x},\nm{\id_{y}}) = \id_{x}}"', from=3-1, to=3-10]
\end{tikzcd}\]

which shows the following: 
\begin{align*}
g(\lift(\htpy{0},f)(\nm{x}, \nm{v^{-1}})) &= g(\lift(v^{-1},v)(x) \circ \lift(\htpy{0},f)(\nm{\lift(v^{-1})(x)},\nm{v}))^{-1}\\
&= g(\lift(\htpy{0},f)(\nm{\lift(v^{-1})(x)},\nm{v}))^{-1} \circ g(\lift(v^{-1},v)(x) )^{-1}\\
&= h(v_{\lift(v^{-1})(x)})^{-1} \circ h(v_{\lift(v^{-1},v)(x)})^{-1}\\
&= h(v_{\lift(v^{-1})(x)}^{-1} \circ v_{\lift(v^{-1},v)(x)}^{-1})\\
&= h(v_{x})\\
\end{align*}

where $\phi^{i}(x) = \alpha(v_{x}, \phi_{x})$. Thus $(v,\phi)^{-1}$ is coherent.\\

Given an $(g,h)$-coherent object $(y,T)$, $\id_{(y,T)} = (\id_{y},\id_{T})$. Suppose $\id_{T(x)}$ is $(g,h)$-coherent for $x \in \ob X_{y}$. For $T(x) = (y_{x}, T_{x})$, since we have that $g(x) = h(y_{x})$ and $\lift(\htpy{0},f)(\nm{x},\nm{\id_{y}}) = \id_{x}$, then $h(\id_{y_{x}}) = g(\id_{x}) = g(\lift(\htpy{0},f)(\nm{x},\nm{\id_{y}})) $ where $\id_{T(x)} = (\id_{y_{x}}, \id_{T_{x}})$. Therefore, $\id_{(y,T)}$ is $(g,h)$-coherent.\\

Given $(g,h)$-coherent morphisms $(v,\phi): (y,T) \to (y',T')$ and $(v',\phi'): (y',T') \to (y'',T'')$ between $(g,h)$-coherent objects, their composition is $$(v', \phi') \circ (v,\phi) = (v' \circ v , [x \in X_{y} \mapsto T''(\lift(v,v')(x)) \circ \phi'_{\lift(f)(x)} \circ \phi_{x}])$$

Suppose each $T''(\lift(v,v')(x)) \circ \phi'_{\lift(v)(x)} \circ \phi_{x}$ is $(g,h)$-coherent, $T''(\lift(v,v')(x)) = (v''_{\lift(v,v')(x)}, \phi_{\lift(v,v')(x)})$, $\phi'_{\lift(v)(x)} = (v'_{\lift(v)(x)}, \phi'_{\lift(v)(x)})$, and $\phi_{x} = (v_{x}, \psi_{x})$. We have that $$g(\lift(v,v')(x)) = h(v''_{\lift(v,v')(x)})$$ $$g(\lift(\htpy{0},f)(\nm{\lift(v)(x)},\nm{v'})) = h(v'_{\lift(v)(x)})$$ and $$g(\lift(\htpy{0},f)(\nm{x},\nm{v})) = h(v_{x})$$ thus we have:
\begin{align*}
g(\lift(\htpy{0},f)(\nm{x},\nm{v'\circ v})) &= g(\lift(v,v')(x) \circ \lift(\htpy{0},f)(\nm{\lift(v)(x)},\nm{v'}) \circ \lift(\htpy{0},f)(\nm{x},\nm{v}))\\
&= g(\lift(v,v')(x)) \circ g(\lift(\htpy{0},f)(\nm{\lift(v)(x)},\nm{v'})) \circ g(\lift(\htpy{0},f)(\nm{x},\nm{v})) \\
&= h(v''_{\lift(v,v')(x)}) \circ h(v'_{\lift(v)(x)}) \circ h(v_{x})\\
&= h(v''_{\lift(v,v')(x)} \circ v'_{\lift(v)(x)} \circ v_{x})\\
\end{align*}

Therefore, $(v',\phi') \circ (v,\phi)$ is $(g,h)$-coherent. Therefore, $W_{f,g,h}$ is a groupoid assembly.

\end{proof}

We define $c_{f,g,h}: W_{f,g,h} \to Z$ where $c_{f,g,h}(w) = h(y)$ where $w = (y,T)$ is a object or morphism.

\begin{lem}

We have an isomorphism $\alpha_{f,g,h}: \poly_{f,g,h}(c_{f,g,h}) \to c_{f,g,h}$ which induces a $\poly_{f,g,h}$-algebra structure on $c_{f,g,h}$.

\end{lem}
\begin{proof}

First we construct a functor $\alpha_{f,g,h}: \dom(\poly_{f,g,h}(c_{f,g,h})) \to \dom(c_{f,g,h})$. Given $(z, y, T: X_{y} \to  \dom(c_{f,g,h}) \times_{Z} X) \in \ob \dom(\poly_{f,g,h}(c_{f,g,h})) $, we have $\alpha_{f,g,h}(z,y,T) = \alpha_{f}(y,\pr{\dom(c_{f,g,h})} \circ T)$. Given a morphism $(c:z \to z',v:y \to y',  \phi: T \To T' \circ \lift(v)) \in \mor \dom(\poly_{f,g,h}(c_{f,g,h}))$, we have $\alpha_{f,g,h}(c,v,\phi) = \alpha_{f}(c,v,\pr{\dom(c_{f,g,h})} \cdot \phi)$. functoriality of $\alpha_{f,g,h}$ follows from functoriality of $\alpha_{f}$.\\

Now we construct $r: \dom(c_{f,g,h}) \to \dom(\poly_{f,g,h}(c_{f,g,h}))$. Suppose we have  $\alpha(y, T) \in \ob \dom(c_{f,g,h})$, then $r(\alpha(y,T)) = (h(y),y,[x \in X_{y} \mapsto (T(x),x)])$. This construction works since $\alpha(y,T)$ is $(g,h)$-coherent and thus $(T(x),x) \in \dom(c_{f,g,h}) \times_{Z} X$. If we have $\alpha(v,\phi) \in \mor \dom(c_{f,g,h})$, $r(\alpha(v,\phi)) = (h(v),v, [x \in \ob X_{y} \mapsto (\phi_{x},\lift(\htpy{0},f)(\nm{x},\nm{v}))])$. First observe that $[x \in \ob X_{y} \mapsto (\phi_{x},\lift(\htpy{0},f)(\nm{x},\nm{v}))]$ is a natural isomorphism since $\phi$ is and by the construction of the functor $\lift(v)$ found in Lemma \ref{fiberpath}. Second, observe that $g(\lift(\htpy{0},f)(\nm{x},\nm{v})) = h(v_{x})$ where $\phi_{x} = (v_{x},\psi_{x})$, so $(\phi_{x},\lift(\htpy{0},f)(\nm{x},\nm{v})) \in \mor \dom(c_{f,g,h}) \times_{Z} X$. Third we have $(\phi_{x},\lift(\htpy{0},f)(\nm{x},\nm{v})): (T(x),x) \to (T'(\lift(v)(x)),\lift(v)(x)) $. \\

For functoriality of $r$, we have
\begin{itemize}
\item  $r(\alpha(\id_{y},\id_{T})) = (\id_{h(y)},\id_{y},[x \in \ob X_{y} \mapsto (\id_{T(x)}, \lift(\htpy{0},f)(\nm{x},\nm{\id_{y}}) = \id_{x})]) = \id_{r(\alpha(y,T))}$
\item $r(\alpha(v,\phi)^{-1}) = (h(v^{-1}),v^{-1},[x \in \ob X_{y'} \mapsto (\phi_{\lift(v^{-1})(x)}^{-1} \circ T'(\lift(v^{-1},v)(x)^{-1}), \lift(\htpy{0},f)(\nm{\lift(v^{-1})(x)},\nm{v})^{-1} \circ \lift(v^{-1},v)(x)^{-1} )]) = r(\alpha(v,\phi))^{-1}$
\item $r(\alpha(v',\phi') \circ \alpha(v,\phi)) = (h(v') \circ h(v), v' \circ v, [x \in \ob X_{y} \mapsto (T''(\lift(v,v')(x)) \circ \phi'_{\lift(v)(x)} \circ \phi_{x}, \lift(v,v')(x) \circ \lift(\htpy{0},f)(\nm{\lift(v)(x)},\nm{v'}) \circ \lift(\htpy{0},f)(\nm{x},\nm{v}) )]) = r(\alpha(v',\phi')) \circ r(\alpha(v,\phi))$
\end{itemize}

$r$ being the inverse of $\alpha_{f,g,h}$ is straightforward.

\end{proof}


\begin{lem}

$c_{f,g,h}$ is its own smallest $\poly_{f,g,h}$ sub-algebra
\end{lem}
\begin{proof}

This is essentially the same proof as in Lemma \ref{subalgpoly}.

\end{proof}

\begin{lem}

$c_{f,g,h}$ is the initial algebra of $\poly_{f,g,h}$
\end{lem}
\begin{proof}

The proof follows the same line of reason as in Lemma \ref{wgini} except we have functors $\Paths_{u} \to Z$. 

\end{proof}

Thus $\GrpA{A}$ has dependent $W$-types. \\

\subsection{ W-Types with Reductions}
In this section, we prove the existence of $W$-types with reductions in $\GrpA{A}$ by making use of WISC in $\Asm(A)$.\\

Now suppose the following diagram for $W$-types with reductions in $\Grp(\Asm(A))$:
\[\begin{tikzcd}
	R && X &&& Y \\
	\\
	& Z &&&&& Z
	\arrow["f", from=1-3, to=1-6]
	\arrow["g"', from=1-3, to=3-2]
	\arrow["h", from=1-6, to=3-7]
	\arrow["k", tail, from=1-1, to=1-3]
\end{tikzcd}\]

where $f$ is a normal isofibration and $g \circ k = h \circ f \circ k$. The functor $\poly_{f,g,h,k}$ can be defined using the following pushout diagram:

\[\begin{tikzcd}
	{\poly^{*}_{k} = \sum_{g}\circ \sum_{k}\circ k^{*}\circ f^{*}\circ\Pi_{f}\circ g^{*}} &&&& {\Id_{\cC/Z}} \\
	\\
	\\
	\\
	{\poly_{f,g,h}} &&&& {\poly_{f,g,h,k}}
	\arrow["\lambda",from=1-1, to=1-5]
	\arrow["\kappa"',from=1-1, to=5-1]
	\arrow[from=1-5, to=5-5]
	\arrow[from=5-1, to=5-5]
	\arrow["\lrcorner"{anchor=center, pos=0.125, rotate=180}, draw=none, from=5-5, to=1-1]
\end{tikzcd}\]

where we explicitly describe the behavior of $\poly^{*}_{k}$ and $\poly_{f,g,h}$ on some $G:D \to Z$ below using generalized natural transformations:\\


\begin{itemize}

\item The objects of $\poly_{f,g,h}(G)$ are $(z, y, T: X_{y} \to D \times_{Z} X)$ where $\Delta_{g}(G) \circ T = \id_{X}|_{X_{y}}$ and $h(y) = z$. 

\item The morphisms are $(c:z \to z',v:y \to y',  \phi: T \gennattrans T')$ where  $\Delta_{g}(G)(\phi_{u}) = u$ for $u \in \mor X_{v}$ and $h(v) = c$.\\

\end{itemize}

\begin{itemize}

\item The objects of $\poly^{*}_{k}(G)$ are $(z, r , T: X_{f(k(r))} \to D \times_{Z} X)$ where $\Delta_{g}(G) \circ T = \id_{X}|_{X_{f(k(r))}}$ and $h(f(k(r))) = z$. 

\item The morphisms are $(c:z \to z',s:r \to r',  \phi: T \gennattrans T')$ where  $\Delta_{g}(G)(\phi_{u}) = u$ for $u \in \mor X_{h(f(k(s)))}$ and $h(f(k(s))) = c$.\\

\end{itemize}

As for $\lambda$ and $\kappa$:

\begin{itemize}

\item $\lambda_{G}$ sends $(z, r , T: X_{f(k(r))} \to D \times_{Z} X)$ to $\pr{D} \circ T(k(r))$ and $(c:z \to z',s:r \to r',  \phi: T \gennattrans T')$ to $\pr{D}(\phi(k(s)))$

\item $\kappa_{G}$ sends $(z, r , T: X_{f(k(r))} \to D \times_{Z} X)$ to $(z, f(k(r)), T)$ and $(c:z \to z',s:r \to r',  \phi: T \gennattrans T')$ to $(c, f(k(s)), \phi)$\\

\end{itemize}

Though we could do the same for $\poly_{f,g,h,k}$, it would be easier to refer to $\poly_{f,g,h,k}$ using the universal property of pushouts, in particular given a $\poly_{f,g,h,k}$-algebra $\beta: \poly_{f,g,h,k}(d) \to d$ we have the following diagram as a result of the universal property of pushouts:

\[\begin{tikzcd}
	{\poly^{*}_{k}(d)} \\
	\\
	{\poly_{f,g,h}(d)} &&& d
	\arrow["\kappa_{d}"', from=1-1, to=3-1]
	\arrow["\lambda_{d}"{description}, from=1-1, to=3-4]
	\arrow["\beta'"', from=3-1, to=3-4]
\end{tikzcd}\]

where $\beta' = \beta \circ \iota_{0,d}$ with $\iota_{0}$ resulting from the following pushout diagram for $\poly_{f,g,h,k}$:
\[\begin{tikzcd}
	{\poly_{k}^{*}} &&&& {\Id_{\GrpA{A}/Z}} \\
	\\
	\\
	\\
	{\poly_{f,g,h}} &&&& {\poly_{f,g,h,k}}
	\arrow[from=1-1, to=1-5]
	\arrow[from=1-1, to=5-1]
	\arrow["{\iota_{1}}", from=1-5, to=5-5]
	\arrow["{\iota_{0}}"', from=5-1, to=5-5]
	\arrow["\lrcorner"{anchor=center, pos=0.125, rotate=180}, draw=none, from=5-5, to=1-1]
\end{tikzcd}\]

This characterizes $d$ as a $\poly_{f,g,h,k}$-algebra:

\begin{itemize}
\item For $y$ such that $h(y) = z$ and $T \in \Pi_{x \in X_{y}} \dom(d)_{g(x)}$, we have $\beta'(z,y,T) \in \dom(d)_{z}$ and for $v:y \to y'$ such that $h(v) = c:z \to z'$ and $\phi: T \gennattrans T'$, we have $\beta'(c,v,\phi) \in \dom(d)_{c}$

\item For objects $r \in R$ and $T \in \Pi_{x \in X_{f(k(r))}} \dom(d)_{g(x)}$ we have $\beta'(h(f(k(r))),f(k(r)),T) = T(k(r))$ and for morphisms $s \in R$ and $\phi: T \gennattrans T'$ we have $\beta'(h(f(k(s))),f(k(s)),\phi) = \phi_{k(s)}$\\
\end{itemize}

Similarly to the previous section, we will show the existence of the initial $\poly_{f,g,h,k}$-algebra. However here we will combine the methods found in \cite{swan2018wtypes} and \cite{wreo22517}. Unlike with $W$-types, will will use the full diagram

Given a morphism $v:y \to y'$ in $Y$, we define the assemblies: $$L_{v} := \{u \in \mor X  |  f(u) = v\}$$ $$\overline{X}_{y} := \{x \in \ob X|  f(x) = y\} \cup \{u \in \mor X|  f(u) = \id_{x} \}$$ $$\overline{X}_{v} :=  \overline{X}_{y} + \overline{X}_{y'} + L_{v} + L_{v^{-1}}$$ $$\overline{Y} := \ob Y + \mor Y $$ $$\overline{X} := \Sigma_{y \in \ob Y}\overline{X}_{y}  +  \Sigma_{v \in \mor Y}\overline{X}_{v} $$

We set $\tilde{Z}$ as the pullback of $\cod:\mor Z \to \ob Z$ along $\dom:\mor Z \to \ob Z$ making $\tilde{Z}$ the assembly of composable pairs of morphisms in $Z$. We now construct the following diagram:

Now we construct the following diagram: 

\[\begin{tikzcd}
	& {\overline{X }+ \Sigma_{(m,m') \in \tilde{Z}}  \{m,m'\} } &&& {\overline{Y }+ \tilde{Z}} \\
	\\
	{\ob Z + \mor Z} &&&&& {\ob Z + \mor Z}
	\arrow["{\overline{f} +\pr{\tilde{Z}} }", from=1-2, to=1-5]
	\arrow["{t_{1}}"', from=1-2, to=3-1]
	\arrow["{t_{2}}", from=1-5, to=3-6]\\
\end{tikzcd}\]

where $\Sigma_{(m,m') \in \tilde{Z}}  \{m,m'\}$ is defined by the morphism $\overline{Z} \to U$ taking $(m,m')$ to $\{n \in \mor Z| n = m \text{ or }n = m'  \}$ for object classifier $U$. $\overline{f}$ is the projection operator on $\Sigma_{y \in \ob Y}\overline{X}_{y}  +  \Sigma_{v \in \mor Y}\overline{X}_{v}$. $t_{2}$ takes an object or morphism of $Y$, $y$ to $h(y)$ and a composable pair $(m,m')$ to $m' \circ m$. $t_{1}$ takes an object or morphism in $\overline{X}_{y}$ $x$ to $g(x)$ and a morphism in $\{m, m'\}$ to itself in $\mor Z$.\\

One may recall the definition of WISC in Definition \ref{wiscasm}. Using the fact that WISC holds in $\Asm(A)$, we have the following covering and collection square:

\[\begin{tikzcd}
	C' &&& \ob X + \mor X \\
	\\
	\\
	{C} &&& {\ob Y + \mor Y}
	\arrow["i ", from=1-1, to=1-4]
	\arrow["c"', from=1-1, to=4-1]
	\arrow["{f}", from=1-4, to=4-4]
	\arrow["{j }"', from=4-1, to=4-4]\\
\end{tikzcd}\]

We will use this to construct covers for $\pr{\ob Y}: \Sigma_{y \in \ob Y}\overline{X}_{y} \to \ob Y$ and $\pr{\mor Y}: \Sigma_{v \in \mor Y}\overline{X}_{v} \to \mor Y$. For $y \in \ob Y$, we define $Q_{y} := \{ (c_{0},c_{1})| j(c_{0}) = y \text{ and }j(c_{1}) = \id_{y} \}$ and for $m \in \mor Y$, $Q_{m} = \{ (o_{0},o_{1},c_{m},c^{i}_{m})| j(c_{m}) = m,\, j(c^{i}_{m}) = m^{-1},\, o_{0} \in Q_{\dom(m)},\, \text{and }o_{1} \in Q_{\cod(m)} \}$. For $(c,c_{i}) \in Q_{y}$, we define $P_{(c,c_{i})} := C'_{c} + C'_{c_{i}}$ and $(o_{0},o_{1},c_{m},c^{i}_{m}) \in Q_{m} $, $P_{(o_{0},o_{1},c_{m},c^{i}_{m})} := P_{o_{0}} + P_{o_{1}} + C'_{c_{m}} + C'_{c^{i}_{m}} $. \\

We set:

\begin{itemize}

\item For $q = (c_{0},c_{1})$, $\id_{q} = (q,q,c_{1},c_{1})$
\item For $q = (o_{0},o_{1},c_{m},c^{i}_{m})$, $\dom(q) = o_{0}$
\item For $q = (o_{0},o_{1},c_{m},c^{i}_{m})$, $\cod(q) = o_{1}$
\item For $q = (o_{0},o_{1},c_{m},c^{i}_{m})$, $q^{-1} = (o_{1},o_{0},c^{i}_{m},c_{m})$\\

\end{itemize}

We seek to construct the above diagram but with the assemblies $Q_{u}$ and $P_{q}$. First we set $Q = \Sigma_{u \in \overline{Y}} Q_{u}$ and $P = \Sigma_{q \in Q} P_{q}$. Now we will define a set of morphisms $i'_{(u,q)}:P_{q} \to \overline{X}_{u}$ for $q \in Q_{u}$ as follows:

\begin{itemize}

\item when $q = (c_{0},c_{1}) \in Q_{y}$ for $y \in \ob Y$, $i'_{(y,q)}(p) = i(p) \in (\ob X)_{y}$ for $p \in C'_{c_{0}}$ and $i'_{(y,q)}(p) = i(p) \in (\ob X)_{\id_{y}}$ for $p \in C'_{c_{1}}$.\\

\item when $q = (o_{0},o_{1},c_{m},c^{i}_{m}) \in Q_{m}$ for $m \in \ob Y$, $i'_{(m,q)}(p) = i(p) \in \overline{X}_{\dom(m)}$ for $p \in P_{o_{0}}$, $i'_{(m,q)}(p) = i(p) \in \overline{X}_{\cod(m)}$ for $p \in P_{o_{1}}$, $i'_{(m,q)}(p) = i(p) \in L_{m}$ for $p \in C'_{c_{m}}$, and $i'_{(m,q)}(p) = i(p) \in L_{m^{-1}}$ for $p \in C'_{c^{i}_{m}}$. \\

\end{itemize}

We set $i' = \Sigma_{(u,q) \in Q} i'_{(u,q)}$. Now we show the following lemma:

\begin{lem}

The square 
\[\begin{tikzcd}
	P &&& \overline{X} \\
	\\
	\\
	{Q} &&& {\overline{Y}}
	\arrow["i' ", from=1-1, to=1-4]
	\arrow["\pr{Q}"', from=1-1, to=4-1]
	\arrow["{\overline{f}}", from=1-4, to=4-4]
	\arrow["{\pr{\overline{Y}} }"', from=4-1, to=4-4]
\end{tikzcd}\]

 is a covering and collection square.

\end{lem}
\begin{proof}

Each $Q_{u}$ is inhabited since $j$ is a cover, so $\pr{\overline{Y}}$ is a cover. Given $u \in \overline{Y}$, if $u$ is an object then  for every $x \in (\ob X)_{u}$ and $b \in C_{u}$ we have a $b' \in C'_{(b,x)}$ and for every $x_{m} \in (\mor X)_{\id_{u}}$ and $b_{m} \in C_{\id_{u}}$, there is $b'_{m} \in C'_{(b_{m},x_{m})}$.  Thus for a fixed $(b,b_{m}) \in Q_{u}$, $P_{(b,b_{m})} = C'_{b} + C'_{b_{m}}$ covers $\overline{X}_{u}$. We can make a similar argument for if $u$ is a morphism. Thus the square in question is covering.\\

As with covering, it suffices to consider the object case with the collection part of this proof since the proofs for both objects and morphisms are similar. Suppose we have an object $y \in \ob Y$, and a cover $h:H \to \overline{X}_{y}$, then we split $h$ into two covers: $h_{o}: H_{o} \to (\ob X)_{y}$ and $h_{m} : H_{m} \to (\mor X)_{\id_{y}}$. This gives us a $c_{0} \in C_{y}$ and $c_{1} \in C_{\id_{y}}$ with morphisms $t_{0}:C'_{c_{0}} \to H_{o}$ and $t_{1}:C'_{c_{1}} \to H_{m}$ with the necessary property. So we have $(c_{0},c_{1}) \in Q_{y}$ and a cover $t_{0} + t_{1}: P_{(c_{0},c_{1})} \to H$ with the necessary property. Therefore, the above square is a covering and collection square.

\end{proof}

This gives us a new diagram:

\[\begin{tikzcd}
	& {P+ \Sigma_{(m,m') \in \tilde{Z}}  \{m,m'\} } &&& {Q + \tilde{Z}} \\
	\\
	{\ob Z + \mor Z} &&&&& {\ob Z + \mor Z}
	\arrow["{\overline{f}'  = \pr{Q} +\pr{\tilde{Z}}  }", from=1-2, to=1-5]
	\arrow["{t'_{1}  = t_{1} \circ (i' + \id)}"', from=1-2, to=3-1]
	\arrow["{t'_{2} = t_{2} \circ (\pr{\overline{Y}} + \id)}", from=1-5, to=3-6]\\
\end{tikzcd}\]

 and the subsequent dependent $W$-type $\omega_{\overline{f}',t'_{1},t'_{2}}: W_{\overline{f}',t'_{1},t'_{2}} \to \ob Z + \mor Z$ with algebra structure $\alpha: \poly_{\overline{f}',t'_{1},t'_{2}}(\omega_{\overline{f}',t'_{1},t'_{2}}) \to \omega_{\overline{f}',t'_{1},t'_{2}}$. We denote $\alpha((m,m'), (w,w') \in W_{\overline{f}',t'_{1},t'_{2},m}\times W_{\overline{f}',t'_{1},t'_{2},m'} )$ as $w' \circ w$. We set $\overline{Z} := \ob Z + \mor Z$.\\

We will define the following:

\begin{itemize}

\item For $z \in \ob Z$ and $w = \alpha(q,E \in \Pi_{p \in P_{q}} W_{\overline{f}',t'_{1},t'_{2},t'_{1}(p)}  ) \in W_{\overline{f}',t'_{1},t'_{2},z}$, we define $\id_{w} = \alpha(\id_{q},\id_{E})$ where $\id_{E}(p) = E(p)$ which works since $P_{\id_{q}} = C'_{c_{1}} + C'_{c_{1}} + P_{q} + P_{q}$ where $P_{q} = C'_{c_{0}} + C'_{c_{1}}$.\\

\item For $m \in \mor Z$ and $w = \alpha(q,E \in \Pi_{p \in P_{q}} W_{\overline{f}',t'_{1},t'_{2},t'_{1}(p)}  ) \in W_{\overline{f}',t'_{1},t'_{2},m}$, we define $\dom(w) = \alpha(\dom(q),E|_{P_{\dom(q)}})$, $\cod(w) = \alpha(\cod(q),E|_{P_{\cod(q)}})$, $w^{-1} = \alpha(q^{-1}, E)$. Note that $P_{q} = P_{q^{-1}}$.\\

\item For $w' \circ w$, $\dom(w' \circ w) = \dom(w)$, $\cod(w' \circ w) = \cod(w')$, and $(w' \circ w)^{-1} = w^{-1} \circ w'^{-1}$.\\

\end{itemize}

It should be clear that $\dom(\id_{w}) = \cod(\id_{w}) = w$. 

\begin{rmk}
\label{commuteZ}

Given some $\alpha(q,E)$ over some $z \in \ob Z$, we have $\id_{\alpha(q,E)} = \alpha(\id_{q},\id_{E}) $ over $\id_{z}$ since $\id_{\pr{\overline{Y}}(q)} = \pr{\overline{Y}}(\id_{q}) $. Of course a similar thing can be shown for $\dom$, $\cod$, and $(-)^{-1}$. Of course if we have $w$ over $m \in \mor Z$ and $w'$ over $m' \in \mor Z$, then $w' \circ w$ is over $m' \circ m$. Thus $\omega_{\overline{f}',t'_{1},t'_{2}}$ in an certain sense can be said to preserve groupoidal operations.

\end{rmk}

Recall that $\omega_{\overline{f}',t'_{1},t'_{2}}$ gives us an underlying function between sets of assemblies: $$\Gamma(\omega_{\overline{f}',t'_{1},t'_{2}}):\Gamma(W_{\overline{f}',t'_{1},t'_{2}} ) \to \Gamma(\ob Z + \mor Z)$$

Furthermore since finite limits and colimits in $\Asm(A)$ are constructed as finite limits and colimits on the underlying sets, we have that $\Gamma$ preserves finite limits and colimits and thus $ \Gamma(\ob Z + \mor Z) = \Gamma(\ob Z) + \Gamma(\mor Z)$. For an element $x \in X$ in an assembly $X$, we will refer to $\gamma(x) \in \Gamma(X)$ as the corresponding element in the underlying set of $X$. Thus when we have a morphism of assemblies $f:X \to Y$ and $x \in X$, $\Gamma(f)(\gamma(x)) = \gamma(f(x))$. We seek to construct a fibred partial equivalence relation $\sim \ \subseteq \Gamma(W_{\overline{f}',t'_{1},t'_{2}} ) \times_{\Gamma(\ob Z) + \Gamma(\mor Z)} \Gamma(W_{\overline{f}',t'_{1},t'_{2}} ) $. The reflexive elements of $\sim$ should give us functorial and natural elements which will will define below:\\

\begin{defn}

We call a $w$ a {\bf hereditary graph morphism} if when $w = \alpha(q,E \in \Pi_{p \in P_{q}} W_{\overline{f}',t'_{1},t'_{2},t'_{1}(p)} ) $:
\begin{itemize}
\item For $p \in P_{q}$ such that $i'(p)$ is an object, for $p' \in P_{q}$ such that $\id_{i'(p)} = i'(p')$, then $\gamma(E(p')) \sim \gamma(\id_{E(p)})$.

\item For $p \in P_{q}$ such that $i'(p)$ is a morphism for $p' \in P_{q}$ if $\dom(i'(p)) = i'(p')$ then $\gamma(E(p')) \sim \gamma(\dom(E(p)))$, if $\cod(i'(p)) = i'(p')$ then $\gamma(E(p')) \sim \gamma(\cod(E(p)))$, and  if $i'(p)^{-1} = i'(p')$ then $\gamma(E(p')) \sim \gamma(E(p)^{-1})$.

\item for each $p \in P_{q}$, $E(p)$ is a hereditary graph morphism.

\end{itemize}

if $w = w'_{1} \circ w'_{0}$, then both $w'_{1}$ and $w'_{0}$ are hereditary graph morphisms and $\gamma(\cod(w'_{0})) \sim \gamma(\dom(w'_{1}))$.\\

For a hereditary graph morphism $w$ over an object in $Z$, we call $w$ {\bf functorial} if when $w = \alpha(q,E)$, and we have $p, p' \in P_{q}$ such that $i'(p)$ and $i'(p')$ are morphisms and $\dom(i'(p')) = \cod(i'(p))$ then for $p'' \in P_{q}$ such that $i'(p'') = i'(p') \circ i'(p)$, $\gamma(E(p'' )) \sim \gamma(E(p') \circ E(p))$. Furthermore, for $p \in P_{q}$, when $i'(p)$ is an object then $E(p)$ is functorial and when $i'(p)$ is a morphism then $E(p)$ is natural.\\

For a hereditary graph morphism $w$ over a morphism in $Z$, we call $w$ {\bf natural} if $\dom(w)$ and $\cod(w)$ are functorial and when $w = \alpha(q,E)$ for $q \in Q_{u}$ and $u \in \mor Y$, when we have a square in $X_{u}$:\\

\[\begin{tikzcd}
	x &&& x' \\
	\\
	\\
	{e} &&& {e'}
	\arrow["\tau_{x}", from=1-1, to=1-4]
	\arrow["m"', from=1-1, to=4-1]
	\arrow["{m'}", from=1-4, to=4-4]
	\arrow["{\tau_{e}}"', from=4-1, to=4-4]\\
\end{tikzcd}\]

where $f(\tau_{x}) = f(\tau_{e}) = u$, $f(m) = \id_{\dom(u)}$, and $f(m') = \id_{\cod(u)}$, then for $p_{\tau_{x}}, p_{\tau_{e}}, p_{m}, p_{m'} \in P_{q}$ such that $i'(p_{\tau_{x}}) = \tau_{x}$, $i'(p_{\tau_{e}}) = \tau_{e}$, $i'(m) = m$, and $i'(m') = m'$ we have $$\gamma(E(p_{m'}) \circ E(p_{\tau_{x}})) \sim \gamma(E(p_{\tau_{e}}) \circ E(p_{m})) $$ Each $E(p)$ is functorial or natural depending on whether $i'(p)$ is an object or morphism. When $w = w_{1} \circ w_{0}$, both $w_{1}$ and $w_{0}$ are natural.

\end{defn}

Now we construct $\sim$ as a partial equivalence relation generated by the following conditions:

\begin{itemize}
\item For $w,w',w'' \in W_{\overline{f}',t'_{1},t'_{2},z}$, when we have $\gamma(w) \sim \gamma(w')$ and $\gamma(w') \sim \gamma(w'')$, we have $\gamma(w) \sim \gamma(w'')$.\\

\item Given $q,q' \in Q_{u}$, $E \in \Pi_{p \in P_{q}} W_{\overline{f}',t'_{1},t'_{2},t'_{1}(p)}$, and $E' \in \Pi_{p \in P_{q'}} W_{\overline{f}',t'_{1},t'_{2},t'_{1}(p)}$ such that $w = \alpha(q,E)$ and $w' = \alpha(q',E')$ if for all $p \in P_{q}$ and $p' \in P_{q'}$ such that $i'(p) = i'(p')$, we have $\gamma(E(p)) \sim \gamma(E'(p'))$ and both $w$ and $w'$ are either both functorial or both natural, then $\gamma(w) \sim \gamma(w')$.\\

\item For $(u,q) \in Q$ and $E \in \Pi_{p \in P_{q}} W_{\overline{f}',t'_{1},t'_{2},t'_{1}(p)}$, such that $\alpha(q,E) \sim \alpha(q,E)$ holds by the second condition, if there exists $x \in \overline{X}_{u}$ and $r \in R$ such that $k(r) = x$ and $f(x) = u$ and $p \in P_{q}$ such that $i'(p) = x$, then $\gamma(\alpha(q,E)) \sim \gamma(E(p))$. \\

\item For $z \in \ob Z$ and $w, w' \in W_{\overline{f}',t'_{1},t'_{2},z}$, if $\gamma(w) \sim \gamma(w')$, we have $ \gamma(\id_{w})\sim \gamma(\id_{w'})$.\\

\item For $m \in \mor Z$ and $w, w' \in W_{\overline{f}',t'_{1},t'_{2},m}$,  if $\gamma(w) \sim \gamma(w')$, we have $ \gamma(\cod(w)) \sim \gamma(\cod(w'))$, $\gamma(\dom(w))\sim \gamma(\dom(w'))$ and $ \gamma(w^{-1}) \sim \gamma(w'^{-1})$ .\\

\item For $(m,m') \in \tilde{Z}$, $w_{m}, w'_{m} \in W_{\overline{f}',t'_{1},t'_{2},m}$ and $w_{m'}, w'_{m'} \in W_{\overline{f}',t'_{1},t'_{2},m'}$ if $\gamma(w_{m}) \sim \gamma(w'_{m})$, $\gamma(w_{m'}) \sim \gamma(w'_{m'})$ and $\gamma(\cod(w_{m})) \sim \gamma(\dom(w_{m'}))$  then we have $\gamma(w_{m'} \circ w_{m}) \sim \gamma(w'_{m'} \circ w'_{m})$.\\

\item For $(m,m'), (m',m'') \in \tilde{Z}$, $w_{m} \in W_{\overline{f}',t'_{1},t'_{2},m}$, $w_{m'} \in W_{\overline{f}',t'_{1},t'_{2},m'}$,
and $w_{m''} \in W_{\overline{f}',t'_{1},t'_{2},m''}$ if we have $\gamma(w_{m}) \sim \gamma(w_{m})$, $\gamma(w_{m'}) \sim \gamma(w_{m'})$, $\gamma(w_{m''}) \sim \gamma(w_{m''})$, $\gamma(\cod(w_{m})) \sim \gamma(\dom(w_{m'}))$ and $\gamma(\cod(w_{m'})) \sim \gamma(\dom(w_{m''}))$ then we have $\gamma(w_{m''} \circ (w_{m'} \circ w_{m}))\sim \gamma((w_{m''} \circ w_{m'})\circ w_{m})$ .\\

\item Given composable $u,u' \in \mor Y$, $q \in Q_{u}$, $q' \in Q_{u'}$, and $q'' \in Q_{u' \circ u}$, $E \in \Pi_{p \in P_{q}} W_{\overline{f}',t'_{1},t'_{2},t'_{1}(p)}$, $E' \in \Pi_{p \in P_{q'}} W_{\overline{f}',t'_{1},t'_{2},t'_{1}(p)}$, and $E'' \in \Pi_{p \in P_{q''}} W_{\overline{f}',t'_{1},t'_{2},t'_{1}(p)}$ such that $w = \alpha(q,E)$, $w' = \alpha(q',E')$, and $w'' = \alpha(q'',E'')$, if we have $\gamma(w) \sim \gamma(w)$, $\gamma(w') \sim \gamma(w')$, $\gamma(w'') \sim \gamma(w'')$, $\gamma(\dom(w)) \sim \gamma(\dom(w''))$, $\gamma(\cod(w')) \sim \gamma(\cod(w''))$, $\gamma(\cod(w)) \sim \gamma(\dom(w'))$, and for all $p'' \in P_{q''}$ such that $i'(p'') \in L_{u' \circ u}$, there exists $p \in P_{q}$ and $p' \in P_{q'}$ such that $i'(p') \circ i'(p) = i'(p'')$ and $\gamma(E''(p'')) \sim \gamma(E'(p') \circ E(p))$ and for $p^{*} \in P_{q'}$ and  $p^{-} \in P_{q}$ such that $i'(p') \circ i'(p) = i'(p^{*}) \circ i'(p^{-})$, $\gamma(E'(p') \circ E(p)) \sim \gamma(E'(p^{*}) \circ E(p^{-}))$, then $\gamma(w'') \sim \gamma(w' \circ w)$.\\

\item For $w \in W_{\overline{f}',t'_{1},t'_{2},m}$ and $m \in \mor Z$ such that $\gamma(w) \sim \gamma(w)$, we have $\gamma(w \circ \id_{\dom(w)}) \sim \gamma(w)$ and $\gamma(w) \sim \gamma(\id_{\cod(w)} \circ w)$.\\

\item For $w \in W_{\overline{f}',t'_{1},t'_{2},m}$ and $m \in \mor Z$ such that $\gamma(w) \sim \gamma(w)$, we have $\gamma(w^{-1} \circ w) \sim  \gamma(\id_{\dom(w)})$ and $\gamma(w \circ w^{-1}) \sim \gamma(\id_{\cod(w)})$.\\

\end{itemize}

\begin{defn}

We call $w$ {\bf well-defined} if $\gamma(w) \sim \gamma(w)$.

\end{defn}

\begin{rmk}
\label{wfnrmk}

The $\dom$, $\cod$, $\id$, and $(-)^{-1}$ operations preserve well defined elements.\\

For a hereditary graph morphism $\alpha(q,E)$, and for $p, p' \in P_{q}$ such that $i'(p) = i'(p') = x$, if $x$ is an object then there exists $p^{*} \in P_{q}$ such that $i'(p^{*}) = \id_{x}$ and so $\gamma(E(p)) \sim \gamma(\dom(E(p^{*}))) \sim \gamma(E(p'))$. If $x$ is a morphism, then there exists $p^{*} \in P_{q}$ such that $i'(p^{*}) = x^{-1}$, so $\gamma(E(p)) \sim \gamma(E(p^{*})^{-1}) \sim \gamma(E(p'))$. Thus $\gamma(E(p)) \sim \gamma(E(p'))$ for all $p, p' \in P_{q}$ such that $i'(p) = i'(p')$.\\

For a functorial or natural $\alpha(q,E)$ if $q \in Q_{u}$, then $i'_{(u,q)} = i': P_{q} \to \overline{X}_{u}$ is a cover so when we have $p, p' \in P_{q}$ such that $i'(p)$ and $i'(p')$ are morphisms and $\dom(i'(p')) = \cod(i'(p))$, there will exist some $p^{*},p^{-} \in P_{q}$ such that $i'(p^{*}) = \dom(i'(p')) = \cod(i'(p))$ and $i'(p^{-}) = i'(p') \circ i'(p)$, and so $\gamma(\dom(E(p'))) \sim \gamma(E(p^{*})) \sim \gamma(\cod(E(p)))$ and $\gamma(E(p^{-})) \sim \gamma(E(p') \circ E(p))$.\\

As a corollary of the above, for a functorial $w = \alpha(q,E)$ for $q \in Q_{y}$ and $y \in \ob Y$, $\id_{w}$ is natural since any square in $X_{\id_{y}}$

\[\begin{tikzcd}
	x &&& x' \\
	\\
	\\
	{e} &&& {e'}
	\arrow["\tau_{x}", from=1-1, to=1-4]
	\arrow["m"', from=1-1, to=4-1]
	\arrow["{m'}", from=1-4, to=4-4]
	\arrow["{\tau_{e}}"', from=4-1, to=4-4]\\
\end{tikzcd}\]

where $f(\tau_{x}) = f(\tau_{e}) = f(m) = f(m') = \id_{y}$ is a commutative square in $X_{y}$, so by functoriality of $w$,for any $p_{\tau_{x}}, p_{\tau_{e}}, p_{m}, p_{m'} \in P_{q}$ such that $i'(p_{\tau_{x}}) = \tau_{x}$, $i'(p_{\tau_{e}}) = \tau_{e}$, $i'(m) = m$: $$\gamma(\id_{E}(p_{m'}) \circ \id_{E}(p_{\tau_{x}})) = \gamma(E(p_{m'}) \circ E(p_{\tau_{x}})) \sim  \gamma(E(p^{-})) \sim \gamma(E(p_{\tau_{e}}) \circ E(p_{m})) = \gamma(\id_{E}(p_{\tau_{e}}) \circ \id_{E}(p_{m}))$$ where $i'(p^{-}) = m' \circ \tau_{x} = \tau_{e} \circ m$.\\

For a natural $w$, we will characterize $w$ as a morphism $w: \alpha(q,E) \to \alpha(q',E')$ where $\dom(w) = \alpha(q,E) $ and $\cod(w) = \alpha(q',E')$. We can rewrite the condition for naturality of $w$ as: when $w = \alpha(q'',\phi)$ for  $q'' \in Q_{u}$ and $u \in \mor Y$, when there exists $p,p' \in P_{q''}$ such that $i'(p) = \tau_{x}$ and $i'(p') = \tau_{e}$ and $p_{0} \in P_{q}$ and $p_{1} \in P_{q'}$ such that $i'(p_{0}) = m$ and $i'(p_{1}) = m'$ such that $$\gamma(E'(p_{1}) \circ \phi(p)) \sim \gamma(\phi(p') \circ E(p_{0}) )$$ 


If we have $p^{-},p^{*} \in P_{q''}$, such that $i'(p^{-}) = \tau_{x}^{-1}$ and $i'(p^{*}) = \tau_{e}^{-1}$ then:

\begin{align*}
\gamma(\phi(p^{*})  \circ E'(p_{1}) \circ \phi(p) \circ \phi(p^{-})) \sim \gamma(\phi(p^{*})  \circ  \phi(p') \circ E(p_{0}) \circ \phi(p^{-}))\\
\gamma(\phi(p^{*})  \circ E'(p_{1}) \circ \phi(p) \circ \phi(p)^{-1}) \sim \gamma(\phi(p')^{-1} \circ  \phi(p') \circ E(p_{0}) \circ \phi(p^{-}))\\
\gamma(\phi(p^{*})  \circ E'(p_{1}) \circ \id_{\cod(\phi(p))}) \sim \gamma(\id_{\dom(\phi(p'))}\circ E(p_{0}) \circ \phi(p^{-}))\\
\gamma(\phi(p^{*})  \circ E'(p_{1}) ) \sim  \gamma(E(p_{0}) \circ \phi(p^{-}))\\
\end{align*}

Thus for a natural $w$, $w^{-1}$ is also natural.\\

\end{rmk}

We now prove the following lemma:

\begin{lem}

$\gamma(w) \sim \gamma(w')$ implies both $w$ and $w'$ are either both functorial or both natural. Furthermore $\gamma(w) \sim \gamma(w)$ if and only if $w$ is functorial or natural.

\end{lem}
\begin{proof}

We prove the first statement of this lemma by induction on $\sim$. Transitivity is straightforward. Suppose that $\gamma(w) \sim \gamma(w')$ by either the second of third conditions, then both $w$ and $w'$ are either functorial or natural since the second condition includes the condition of functoriality or naturality and of course by definition of functoriality and naturality.
The groupoid operations preserves functoriality and naturality by definition of naturality and by Remark \ref{wfnrmk}. This covers all but the last three cases for $\sim$. If we have $\gamma(w'') \sim \gamma(w' \circ w)$, then $\gamma(w'') \sim \gamma(w'')$, $\gamma(w') \sim \gamma(w')$, and $\gamma(w) \sim \gamma(w)$. Thus $w''$ and $w' \circ w$ are natural by induction. The final two cases follow from the previous cases. \\

The second statement of this lemma follows from the second paragraph of Remark \ref{wfnrmk} and the first statement of this lemma.\\

\end{proof}


Recall that the functor $\Gamma:\Asm(A) \to \Set$ which takes an assembly to its underlying set has a right adjoint $\nabla: \Set \to \Asm(A)$ which takes a set $X$ to the assembly $(X,\const_{A})$. Furthermore the unit of the adjunction $\eta_{X}:X \to \nabla(\Gamma(X))$ is the identity function on $\Gamma(X)$ making $\eta_{X}$ a monomorphism, but not necessarily an isomorphism. We also have that $\Gamma \circ \nabla \cong \Id_{\Set}$ witnessed by the identity natural isomorphism. \\

With this in mind taking $$\Gamma(\omega_{\overline{f}',t'_{1},t'_{2}}):\Gamma(W_{\overline{f}',t'_{1},t'_{2}} ) \to \Gamma(\ob Z + \mor Z)$$ and the fibred partial equivalence relation $\sim$, we define the set $W_{\Gamma} = \{\gamma(w) |  \gamma(w) \sim \gamma(w) \}$ and we have a restriction $\Gamma(\omega_{\overline{f}',t'_{1},t'_{2}}): W_{\Gamma} \to \Gamma(\ob Z + \mor Z)$. We have that for $w \in W_{\Gamma}$, $w \sim w$ and so we take $\sim' \ \subseteq W_{\Gamma} \times_{\Gamma(\ob Z + \mor Z)} W_{\Gamma}$ as the restriction of $\sim$ to $W_{\Gamma}$ making $\sim'$ a fibred equivalence relation. We now have a function $\omega':W_{\Gamma}/\sim'  \to  \Gamma(\ob Z + \mor Z)$ with the square

\[\begin{tikzcd}
	{W_{\Gamma} } &&&& {W_{\Gamma}/\sim' } \\
	\\
	\\
	\\
	{\Gamma(\ob Z + \mor Z)} &&&& {\Gamma(\ob Z + \mor Z)}
	\arrow["{[-]}", from=1-1, to=1-5]
	\arrow["{\Gamma(\omega_{\overline{f}',t'_{1},t'_{2}})}"', from=1-1, to=5-1]
	\arrow["{\omega'}", from=1-5, to=5-5]
	\arrow["{\id_{\Gamma(\ob Z + \mor Z)}}"', from=5-1, to=5-5]
\end{tikzcd}\]

Now we have a injection $W_{\Gamma} \mono \Gamma(W_{\overline{f}',t'_{1},t'_{2}} )$ which gives us a monomorphism $\nabla(W_{\Gamma}) \mono \nabla(\Gamma(W_{\overline{f}',t'_{1},t'_{2}} ))$ since right adjoints preserve monomorphisms and a pullback:

\[\begin{tikzcd}
	{W_{\sim}} &&&& {\nabla(W_{\Gamma})} \\
	\\
	\\
	\\
	{W_{\overline{f}',t'_{1},t'_{2}} } &&&& {\nabla(\Gamma(W_{\overline{f}',t'_{1},t'_{2}} ))}
	\arrow["{\eta_{\sim}}", from=1-1, to=1-5]
	\arrow[from=1-1, to=5-1]
	\arrow["\lrcorner"{anchor=center, pos=0.125}, draw=none, from=1-1, to=5-5]
	\arrow[from=1-5, to=5-5]
	\arrow["{{\eta_{W_{\overline{f}',t'_{1},t'_{2}} }}}"', from=5-1, to=5-5]
\end{tikzcd}\]

where $W_{\sim} = \{w \in  W_{\overline{f}',t'_{1},t'_{2}} |  \gamma(w) \sim' \gamma(w)\}$ and $\eta_{\sim}$ is a monomorphism. Thus we have a restriction $\omega_{\overline{f}',t'_{1},t'_{2}}:W_{\sim} \to (\ob Z + \mor Z)$ and a diagram:

\[\begin{tikzcd}
	{W_{\sim}} &&& {\nabla(W_{\Gamma}) } &&&& {\nabla(W_{\Gamma}/\sim')} \\
	\\
	\\
	\\
	{\ob Z + \mor Z} &&& {\nabla(\Gamma(\ob Z + \mor Z))} &&&& {\nabla(\Gamma(\ob Z + \mor Z))}
	\arrow["{\eta_{\sim}}", from=1-1, to=1-4]
	\arrow["{\omega_{\overline{f}',t'_{1},t'_{2}}}"', from=1-1, to=5-1]
	\arrow["{{\nabla([-])}}", from=1-4, to=1-8]
	\arrow["{{\nabla(\Gamma(\omega_{\overline{f}',t'_{1},t'_{2}}))}}"{description}, from=1-4, to=5-4]
	\arrow["{{\nabla(\omega')}}", from=1-8, to=5-8]
	\arrow["{\eta_{\ob Z + \mor Z}}"', from=5-1, to=5-4]
	\arrow["{{\id_{\nabla(\Gamma(\ob Z + \mor Z))}}}"', from=5-4, to=5-8]
\end{tikzcd}\]

We apply image factorization to the morphism $\nabla([-]) \circ \eta_{\sim}: W_{\sim} \to \nabla(W_{\Gamma}/\sim')$ to get the following:

\[\begin{tikzcd}
	{W_{\sim}} &&&& {\nabla(W_{\Gamma}/\sim')} \\
	\\
	&& {U_{\overline{f}',t'_{1},t'_{2}} }
	\arrow["{\nabla([-]) \circ \eta_{\sim}}", from=1-1, to=1-5]
	\arrow["{[-]_{r}}"', from=1-1, to=3-3]
	\arrow["{i_{r}}"', from=3-3, to=1-5]
\end{tikzcd}\]

where $[-]_{r}$ is a regular epimorphism and $i_{r}$ is a monomorphism. Recall that $[-]_{r}$ is a coequalizer of the following diagram:

\[\begin{tikzcd}
	{W_{\sim} \times_{\nabla(W_{\Gamma}/\sim')}W_{\sim}} &&& {W_{\sim}} &&& {U_{\overline{f}',t'_{1},t'_{2}} }
	\arrow["{\pr{1}}"', shift right=3, from=1-1, to=1-4]
	\arrow["{\pr{0}}", shift left=3, from=1-1, to=1-4]
	\arrow["{[-]_{r}}", from=1-4, to=1-7]
\end{tikzcd}\]

which is to say that for $w, w' \in W_{\sim}$, $[w]_{r} = [w']_{r}$ if and only if $\nabla([-]) \circ \eta_{\sim}(w) = \nabla([-]) \circ \eta_{\sim}(w')$. However since we have $\id_{-}: \Gamma \circ \nabla \cong \Id_{\Set}$, the underlying function of $\eta_{W_{\overline{f}',t'_{1},t'_{2}} }$ is $\id_{\Gamma(W_{\overline{f}',t'_{1},t'_{2}} )}$ and the construction of $W_{\Gamma}$, once can see that elements $\eta_{\sim}(w) , \eta_{\sim}(w')  \in \nabla(W_{\Gamma})$ correspond to elements $\gamma(w), \gamma(w' ) \in W_{\Gamma}$ and $\nabla([-]) \circ \eta_{\sim}(w) = \nabla([-]) \circ \eta_{\sim}(w')$ if and only if $[\gamma(w)] = [\gamma(w')]$. Since $\Set$ is exact, $[\gamma(w)] = [\gamma(w')]$ if and only if $\gamma(w) \sim' \gamma(w')$. Hence $[w]_{r} = [w']_{r}$ if and only if $\gamma(w) \sim' \gamma(w')$.\\

When $[w]_{r} = [w']_{r}$, we have  $\nabla([-]) \circ \eta_{\sim}(w) = \nabla([-]) \circ \eta_{\sim}(w')$ and so $$\eta_{\ob Z + \mor Z}(\omega_{\overline{f}',t'_{1},t'_{2}}(w)) =  \eta_{\ob Z + \mor Z}(\omega_{\overline{f}',t'_{1},t'_{2}}(w')) $$ But this means that $\omega_{\overline{f}',t'_{1},t'_{2}}(w) = \omega_{\overline{f}',t'_{1},t'_{2}}(w')$ since $\eta_{\ob Z + \mor Z}$ is a monomorphism. So we have a morphism $\omega^{*}:U_{\overline{f}',t'_{1},t'_{2}} \to \ob Z + \mor Z$ with $\omega^{*} \circ [-]_{r} = \omega_{\overline{f}',t'_{1},t'_{2}}$. With everything we have demonstrated, we can now prove the following:

\begin{lem}

$\omega^{*}: U_{\overline{f}',t'_{1},t'_{2}} \to \ob Z + \mor Z$ induces a functor $\omega:U_{\overline{f}',t'_{1},t'_{2}} \to Z$.

\end{lem}
\begin{proof}

The objects of $U_{\overline{f}',t'_{1},t'_{2}}$ are the equivalence classes of functorial elements $[\alpha(q_{o},E)]_{r}$ and the morphisms are equivalence classes of natural elements, $[\alpha(q_{m},\phi)]_{r}$. The groupoid operations on $U_{\overline{f}',t'_{1},t'_{2}}$ are derived from the operations $\dom$, $\cod$, $(-)^{-1}$, $\id_{-}$, and $- \circ -$ restricted to  $W_{\sim}$ and are well defined by definition of $\sim$ and Remark \ref{wfnrmk}. Furthermore we also have associativity of composition and  identities and inverses functioning as intended by definition of $\sim$. The fact that $\omega$ is a functor can be inferred from Remark \ref{commuteZ}.

\end{proof}

\begin{lem}

Given a $(z,y,T:X_{y} \to U_{\overline{f}',t'_{1},t'_{2}} \times_{Z} X) \in \ob \poly_{f,g,h}(\omega)$, there exists $q \in Q_{y}$ and $E \in \Pi_{p \in P_{q}} W_{\sim,t'_{1}(p)} $ such that $\alpha(q,E)$ is functorial and $T(i'(p)) = [E(p)]_{r}$. Given a $(c:z \to z',v:y \to y',  \phi: T \gennattrans T') \in \mor \poly_{f,g,h}(\omega)$, there exists $q \in Q_{v}$ and $E \in \Pi_{p \in P_{q}} W'_{\overline{f}',t'_{1},t'_{2},t'_{1}(p)} $ such that $\alpha(q,E)$ is natural and when $P_{q} = P_{o_{0}} + P_{o_{1}} + C'_{c_{v}} + C'_{c^{i}_{v}}$ we have the following cases:

\begin{itemize}

\item For $p \in P_{o_{0}} $, $T(i'(p)) = [E(p)]_{r}$
\item For $p \in P_{o_{1}} $, $T'(i'(p)) = [E(p)]_{r}$
\item For $p \in C'_{c_{v}} $, $\phi_{i'(p)} = [E(p)]_{r}$
\item For $p \in C'_{c^{i}_{v}} $, $(\phi_{i'(p)^{-1}})^{-1} = [E(p)]_{r}$
\end{itemize}

Furthermore, for any two $(q,E)$ and $(q',E')$ satisfying the above conditions, $\gamma(\alpha(q,E)) \sim' \gamma(\alpha(q',E'))$.

\end{lem}
\begin{proof}
Given a $(z,y,T:X_{y} \to U_{\overline{f}',t'_{1},t'_{2}} \times_{Z} X) \in \ob \poly_{f,g,h}(\omega)$, we have a morphism $\overline{T}: \overline{X}_{y} \to U_{\overline{f}',t'_{1},t'_{2}}$ where $\overline{T}(x) = T(x)$. This induces a surjection over $ \overline{X}_{y}$ via the following pullback:

\[\begin{tikzcd}
	\Lambda &&& {W_{\sim}} \\
	\\
	\\
	{\overline{X}_{y}} &&& {U_{\overline{f}',t'_{1},t'_{2}}}
	\arrow[ from=1-1, to=1-4]
	\arrow["\epsilon"',from=1-1, to=4-1]
	\arrow["\lrcorner"{anchor=center, pos=0.125}, draw=none, from=1-1, to=4-4]
	\arrow["{[-]_{r}}", from=1-4, to=4-4]
	\arrow["{\overline{T}}"', from=4-1, to=4-4]
\end{tikzcd}\]

This gives us a $q \in Q_{y}$, and $e: P_{q}  \to \Lambda$ such that $\epsilon \circ e = i'_{(y,q)}$. Thus in turn gives us a $e':P_{q} \to W_{\sim}$ such that for $p \in P_{q}$, $T(i'(p)) = \overline{T}(i'(p)) = [e'(p)]_{r}$. $\alpha(q,e')$ being functorial follows from the fact that the groupoid operations defined on $W_{\sim}$ are well-defined for $U_{\overline{f}',t'_{1},t'_{2}}$ (i.e. $\dom([e'(p)]_{r}) = [\dom(e'(p))]_{r}$, $[e'(p)]_{r} \circ [e'(p')]_{r} = [e'(p) \circ e'(p')]_{r}$, etc.). Therefore, $\gamma(\alpha(q,e')) \sim \gamma(\alpha(q,e'))$. Suppose we have another such $(q',E)$ such that $T(i'(p)) = [E(p)]_{r}$ for $p \in P_{q'}$ and $\alpha(q',E)$ is functorial. Then for $p \in P_{q}$ and $p' \in P_{q'}$ such that $i'(p) = i'(p')$, then $[e'(p)]_{r} = T(i'(p)) = [E(p')]_{r}$. Thus $\gamma(e'(p)) \sim' \gamma(E(p'))$. Then by definition of $\sim'$, $\gamma(\alpha(q,e')) \sim' \gamma(\alpha(q',E))$.\\

Given a $(c:z \to z',v:y \to y',  \phi: T \gennattrans T') \in \mor \poly_{f,g,h}(\omega)$, we have a morphism $\overline{\phi}: \overline{X}_{v} \to U_{\overline{f}',t'_{1},t'_{2}}$ constructed as follows:

\begin{itemize}

\item $\overline{\phi}(x) = T(x)$ for $x \in \overline{X}_{y}$
\item $\overline{\phi}(x') = T'(x')$ for $x'\in \overline{X}_{y'}$
\item $\overline{\phi}(m) = \phi_{m}$ for $m \in L_{v}$
\item $\overline{\phi}(m) = (\phi_{m^{-1}})^{-1}$ for $m \in L_{v^{-1}}$

\end{itemize}

As with objects, this will give us a $q = (o_{0},o_{1},c_{v},c^{i}_{v}) \in Q_{v}$ and $E: P_{q} \to W_{\sim}$ such that 
\begin{itemize}

\item $[E(p)]_{r} = \overline{\phi}(i'(p)) = T(i'(p))$ for $p \in P_{o_{0}}$
\item $[E(p)]_{r} = \overline{\phi}(i'(p)) = T'(i'(p))$ for $p \in P_{o_{1}}$
\item $[E(p)]_{r} = \overline{\phi}(i'(p)) = \phi_{i'(p)}$ for $p \in C'_{c_{v}}$
\item $[E(p)]_{r} = \overline{\phi}(i'(p)) = (\phi_{i'(p)^{-1}})^{-1}$ for $p \in C'_{c^{i}_{v}}$

\end{itemize}

Similar to how we know for objects, we have that $\alpha(q,E)$ is a hereditary pre-functor. If for $p \in P_{o_{0}}$, $p' \in P_{o_{1}}$ and $c'_{0}, c'_{1} \in C'_{c_{v}}$ such that we have the following commutative square over $v$:

\[\begin{tikzcd}
	x &&& {e} \\
	\\
	\\
	{x'} &&& {e'}
	\arrow["i'(c'_{0})", from=1-1, to=1-4]
	\arrow["i'(p)"',from=1-1, to=4-1]
	\arrow["{i'(p')}", from=1-4, to=4-4]
	\arrow["{i'(c'_{1})}"', from=4-1, to=4-4]
\end{tikzcd}\]

We have $$[E(p') \circ E(c'_{0})]_{r} = [E(p')]_{r} \circ [E(c'_{0})]_{r} = T'(i'(p')) \circ \phi_{i'(c'_{0})} = \phi_{i'(c'_{1})} \circ T(i'(p)) = [E(c'_{1})]_{r} \circ [E(p)]_{r} = [E(c'_{1}) \circ E(p)]_{r}$$
Thus $\gamma(E(p') \circ E(c'_{0})) \sim' \gamma(E(c'_{1}) \circ E(p))$ making $\alpha(q,E)$ natural. If we have a similar $(q',E')$, then just like with objects, we have that $\gamma(\alpha(q,E)) \sim' \gamma(\alpha(q',E'))$.

\end{proof}

\begin{lem}

$\omega:U_{\overline{f}',t'_{1},t'_{2}} \to Z$ is a $\poly_{f,g,h,k}$-algebra.

\end{lem}
\begin{proof}

First we construct a functor $\alpha_{f,g,h}:\poly_{f,g,h}(\omega) \to \omega$. Given $(z,y,T:X_{y} \to U_{\overline{f}',t'_{1},t'_{2}} \times_{Z} X) \in \ob \poly_{f,g,h}(\omega)$ and $(c:z \to z',v:y \to y',  \phi: T \gennattrans T') \in \mor \poly_{f,g,h}(\omega)$, by by previous lemma, we have $(q \in Q_{y},E)$ and $(q' \in Q_{v},\psi)$ such that $[E(p)] = T(i'(p))$ for $p \in P_{q}$ and $[\psi(p)] = \phi(i'(p))$ for $p \in P_{q'}$. We also have that $\alpha(q,E)$ and $\alpha(q',\psi)$ are in the fibers over $z$ and $c:z \to z'$ respectively. So we set $\alpha_{f,g,h}(z,y,T) = [\alpha(q,E)]_{r}$ and $\alpha_{f,g,h}(c,v,\phi) = [\alpha(q',\psi)]_{r}$. Similarly to how we constructed the algebra structure for $W$-types with reductions in $\Asm(A)$ in section \ref{wtwr}, we can find $E$ and $\psi$ so that the underlying function of $\alpha_{f,g,h}$ is realized. \\

For preservation of domains and codomains, Suppose we have $\alpha_{f,g,h}(z,y,T) = [\alpha(q,E)]_{r}$, $\alpha_{f,g,h}(z',y',T') = [\alpha(q,E)]_{r}$ and $\alpha_{f,g,h}(c,v,\phi) = [\alpha(q',\psi)]_{r}$ where $q' = (o_{0},o_{1},c_{v},c^{i}_{v})$ then for $p \in P_{q}$ and $p' \in P_{o_{0}}$ such that $i'(p) = i'(p')$, then $[\psi(p')]_{r} = T(i'(p)) = [E(p)]_{r}$ and so $\gamma(\psi'(p')) \sim' \gamma(E(p))$. So $\dom([\alpha(q',\psi)]_{r}) = [\dom(\alpha(q',\psi))]_{r} = [\alpha(q,E)]_{r}$. We can show something similar for codomains.\\

For identities, suppose we have $\alpha_{f,g,h}(z,y,T) = [\alpha(q,E)]_{r}$ and $\alpha_{f,g,h}(\id_{z},\id_{y},\id_{T}) = [\alpha(q',\psi)]_{r}$ where $q' = (o_{0},o_{1},c_{v},c^{i}_{v})$ then for $p \in P_{\id_{q}}$ and $p' \in P_{q'}$ such that $i'(p) = i'(p')$ then 

\begin{itemize}

\item if $i'(p)$ is an object then $[\id_{E}(p)]_{r} = [E(p)]_{r} = T(i'(p)) = [\psi(p')]_{r}$
\item if $i'(p)$ is a morphism then $[\id_{E}(p)]_{r} = [E(p)]_{r} = T(i'(p)) = \id_{T,i'(p')} = [\psi(p')]_{r}$

\end{itemize}

thus $\gamma(E(p)) \sim' \gamma(\psi(p'))$, therefore $\gamma(\id_{\alpha(q,E)}) \sim' \gamma(\alpha(q',\psi))$.\\

For inverses, suppose we have $\alpha_{f,g,h}(c,v,\phi) = [\alpha(q,\psi)]_{r}$ and $\alpha_{f,g,h}(c^{-1},v^{-1},\phi^{-1}) = [\alpha(q',\psi')]_{r}$ then if $q = (o_{0},o_{1},c_{v},c^{i}_{v})$ and $q' = (o'_{0},o'_{1},c_{v^{-1}}, c^{i}_{v^{-1}})$ we examine the following cases for $\alpha(q^{-1},\psi)$ and $\alpha(q',\psi')$:

\begin{itemize}
\item For $p \in P_{o_{1}}$ and $p' \in P_{o'_{0}}$ such that $i'(p) = i'(p')$, $[\psi(p)]_{r} = T'(i'(p)) = [\psi'(p')]_{r}$
\item For $p \in P_{o_{0}}$ and $p' \in P_{o'_{1}}$ such that $i'(p) = i'(p')$, $[\psi(p)]_{r} = T(i'(p)) = [\psi'(p')]_{r}$
\item For $p \in C'_{c^{i}_{v}}$ and $p' \in C'_{c_{v^{-1}}}$ such that $i'(p) = i'(p')$, $[\psi(p)]_{r} = \phi(i'(p)^{-1})^{-1} = \phi^{-1}(i'(p')) = [\psi'(p')]_{r}$
\item For $p \in C'_{c_{v}}$ and $p' \in C'_{c^{i}_{v^{-1}}}$ such that $i'(p) = i'(p')$, $[\psi(p)]_{r} = \phi(i'(p)) = \phi^{-1}(i'(p')^{-1})^{-1} = [\psi'(p')]_{r}$

\end{itemize}
Thus $\gamma(\psi(p)) \sim' \gamma(\psi'(p'))$ and so $\gamma(\alpha(q,\psi)) \sim' \gamma(\alpha(q',\psi'))$.\\

Finally for composition, suppose we have $\alpha_{f,g,h}(c,v,\phi) = [\alpha(q,\psi)]_{r}$, $\alpha_{f,g,h}(c',v',\phi') = [\alpha(q',\psi')]_{r}$, and $\alpha_{f,g,h}(c' \circ c,v' \circ v,\phi' \circ \phi) = [\alpha(q'',\psi'')]_{r}$ then for $p'' \in P_{q''}$ such that $i'(p'') \in L_{v' \circ v}$, setting $m = \lift(\htpy{0},f)(\dom(i'(p'')), v)$, we have $m \in L_{v}$ and $i'(p'') \circ m^{-1} \in L_{v'}$. So we can find $p \in P_{q}$ and $p' \in P_{q'}$ such that $i'(p) = m$ and $i'(p') = i'(p'') \circ m^{-1}$ and $$[\psi'(p') \circ \psi(p)]_{r} = [\psi'(p')]_{r} \circ [\psi(p)]_{r} = \phi'(i'(p')) \circ \phi(i'(p)) = \phi'\circ \phi(i'(p'')) = [\psi''(p'')]_{r}$$
so $\gamma(\psi'(p') \circ \psi(p)) \sim' \gamma(\psi''(p''))$. For $p^{-} \in P_{q}$ and $p^{*} \in P_{q'}$ such that $i'(p^{*}) \circ i'(p^{-}) = i'(p') \circ i'(p)$ then in fact we have $\phi'(i'(p')) \circ \phi( i'(p)) = \phi'(i'(p^{*})) \circ \phi(i'(p^{-}))$ and so $$[\psi'(p') \circ \psi(p)]_{r} = [\psi'(p')]_{r} \circ [\psi(p)]_{r} = [\psi'(p^{*})]_{r} \circ [\psi(p^{-})]_{r} = [\psi'(p^{*}) \circ \psi(p^{-})]_{r}$$ meaning that $\psi'(p') \circ \psi(p) \sim \psi'(p^{*}) \circ \psi(p^{-})$. So by construction of $\sim'$, $\gamma(\alpha(q'',\psi'')) \sim' \gamma(\alpha(q',\psi')\circ \alpha(q,\psi))$.\\

Therefore, we have a functor $\alpha_{f,g,h}:\poly_{f,g,h}(\omega) \to \omega$.\\

For an object $r \in \ob R$ and $B \in \Pi_{x \in X_{f(k(r))}} U_{\overline{f}',t'_{1},t'_{2},g(x)}$, let $\alpha_{f,g,h}(g(k(r)), f(k(r)), B) = [\alpha(q,E)]_{r}$, then we have a $p \in P_{q}$ such that $i'(p) = k(r)$. Furthermore we have that for $p^{*}, p' \in P_{q}$, such that $i'(p^{*}) = i'(p')$, $\gamma(E(p^{*})) \sim' \gamma(E(p'))$. Thus $\gamma(\alpha(q,E)) \sim' \gamma(\alpha(q,E))$ by the second condition of $\sim'$. Thus $\gamma(E(p)) \sim' \gamma(\alpha(q,E))$. So $B(k(r)) = [E(p)]_{r} = [\alpha(q,E)]_{r} = \alpha_{f,g,h}(g(k(r)), f(k(r)), B)$.\\

For morphisms $s \in \mor R$ and $\phi: B \gennattrans B'$, let $\alpha_{f,g,h}(g(k(s)), f(k(s)), \phi) = [\alpha(q,\psi)]_{r}$, then we have a $p \in P_{q}$ such that $i'(p) = k(s)$ and for $p^{*}, p' \in P_{q}$ such that $i'(p^{*}) = i'(p')$ we have $\gamma(E(p^{*})) \sim' \gamma(E(p'))$. Thus we can show $\gamma(\psi(p)) \sim' \gamma(\alpha(q,\psi))$ as with objects and subsequently  $\phi_{k(r)} = [\psi(p)]_{r} = [\alpha(q,\psi)]_{r} = \alpha_{f,g,h}(g(k(s)), f(k(s)), \phi)$.\\

We can extend this to a functor $\poly^{*}_{r}(\omega) \to \omega$ by applying functoriality and naturality of elements $\alpha(q,E)$ and $\alpha(q',\phi)$. Therefore, $\omega$ is a $\poly_{f,g,h,k}$-algebra.

\end{proof}


\begin{lem}
\label{wtyper}

$\omega$ is the initial $\poly_{f,g,h,k}$-algebra.

\end{lem}
\begin{proof}

Given a $\poly_{f,g,h,k}$-algebra $d:D \to Z$ equipped with the $\poly_{f,g,h}$-algebra structure $\beta:\poly_{f,g,h}(d) \to d$, treating $D$ as the assembly $\ob D + \mor D $ and $d$ over $\overline{Z} = \ob Z + \mor Z$, we define a subobject  $J \subseteq D \times _{\overline{Z}} W_{\sim}$  in $\Asm(A)$ as follows:

\begin{itemize}

\item For $y \in \ob Y$, $q \in Q_{y}$, $\alpha(q,B) \in W_{\sim}$ and $(z,y,B') \in \ob \poly_{f,g,h}(d)$, and for all $p \in P_{q}$, $(B(p) ,B'(i'(p))) \in J$, then $(\alpha(q,B),\beta(z,y,B')) \in J$.\\

\item For $m \in \mor Y$,$q = (o_{0},o_{1}, c_{m},c^{i}_{m}) \in Q_{m}$ ,$\alpha(q,\phi) \in W_{\sim}$, $(n,m,\phi':B' \to L') \in \mor \poly_{f,g,h}(d)$, $(\dom(\alpha(q,\phi)), \alpha(\dom(n), \dom(m), B')) \in J$, $(\cod(\alpha(q,\phi)), \alpha(\cod(n), \cod(m), L')) \in J$, and for all $p \in C'_{c_{m}}$ and $p' \in C'_{c^{i}_{m}}$, $(\phi(p) ,\phi'_{i'(p)}), (\phi(p') ,(\phi'_{i'(p')^{-1}})^{-1}) \in J$ then $(\alpha(q,\phi),\beta(n,m,\phi')) \in J$.\\

\item For $(w,t) \in J$ and $(w',t') \in J$ such that $\dom(t') = \cod(t)$ and $\dom(w') \sim \cod(w)$ then we have $(w' \circ w, t' \circ t) \in J$.\\

\end{itemize}

Now we seek to show that for $w, w' \in W_{\sim}$, when $\gamma(w) \sim' \gamma(w')$ there exists a unique $t \in D$ such that $(w,t) \in J$ is inhabited and the same $t$ is unique for $w'$. We will refer to this condition as $C_{u}$. We need to prove this via induction on $\sim'$:\\

Suppose $\gamma(w) \sim \gamma(w')$ by transitivity, then there exists $w^{*} \in W_{\sim}$ such that $\gamma(w) \sim' \gamma(w^{*})$ and $\gamma(w^{*}) \sim \gamma(w')$. So by induction on both $\gamma(w) \sim' \gamma(w^{*})$ and $\gamma(w^{*}) \sim' \gamma(w')$, we have that $C_{u}$ holds for $\gamma(w) \sim' \gamma(w')$.\\

 Now we consider the $\id$ case. Suppose we have $\gamma(w) \sim' \gamma(w')$ for functorial elements satisfying $C_{u}$, then we have a unique object $t$ such that $J_{(w,t)}$. Then $w = \alpha(q,B)$ and $t = \beta(z,y,B')$ with $q \in Q_{y}$ and $(B(p) ,B'(i'(p))) \in J$ for $p \in P_{q}$. Since $w$ is functorial, then $\gamma(B(p)) \sim' \gamma(B(p'))$ for $p, p' \in P_{q}$ such that $i'(p) = i'(p')$. Suppose that all such instances satisfies $C_{u}$, note that $\id_{B}(p) = B(p)$ and $\id_{B',v} = B'(v)$. Since $(\id_{B',v^{-1}} )^{-1} = B'(v^{-1})$, it should be clear that $J_{(\id_{w},\id_{t})}$ is inhabited. Furthermore $\id_{t}$ is unique since each $\gamma(B(p)) \sim' \gamma(B(p'))$ satisfies $C_{u}$. $t$ is also unique for $w'$ and thus $\id_{t}$ is the unique element for $\id_{w'}$.\\

Fore $\dom$ and $\cod$ cases:\\

Suppose that we have natural elements $w \sim w'$ satisfying $C_{u}$, then we have a unique object $t$ such that $(w,t) \in J$. If $w = w_{1} \circ w_{0}$, then $t = t_{1} \circ t_{0}$ and  $(w_{1},t_{1}) \in J$ and $(w_{0},t_{0}) \in J$ are inhabited. Since $\dom(w) = \dom(w_{0})$ and $\cod(w) = \cod(w_{1})$, it suffices to examine the case where $w = \alpha(q,\phi)$.\\

If $w = \alpha(q,\phi)$ then $t = \beta(n,m,\phi':B' \to L')$ and we have $(\dom(w),\dom(t)) \in J$, $(\cod(w),\cod(t)) \in J$, $(\phi(p) ,\phi'_{i'(p)}) \in J $ for $p \in C'_{c_{m}}$ and $(\phi(p) ,(\phi'_{i'(p)^{-1}})^{-1}) \in J$ for $p \in C'_{c^{i}_{m}}$ where each $\gamma(\phi(p)) \sim' \gamma(\phi(p'))$ for $i'(p) = i'(p')$ satisfies $C_{u}$ by induction. Applying the definitions of $J$ to $(\dom(w),\dom(t)) \in J$ and $(\cod(w),\cod(t)) \in J$ and by induction one can see that $\dom(t)$ and $\cod(t)$ are the unique elements for $\dom(w)$ and $\cod(w)$. Since $w'$ satisfies $C_{u}$, $t$ is unique for $w'$ thus $\dom(t)$ and $\cod(t)$ are the unique elements for $\dom(w')$ and $\cod(w')$.\\

For the $- \circ -$ case:\\

Suppose we have $\gamma(w'_{m'}) \sim' \gamma(w_{m'})$ and $\gamma(w'_{m}) \sim' \gamma(w_{m})$ for natural elements and $\gamma(\cod(w_{m})) \sim' \gamma(\dom(w_{m'}))$ satisfying $C_{u}$. Then if $t'$ is the unique element for $w'_{m'}$ and $w_{m'}$ and $t$ is the unique element for $w'_{m}$ and $w_{m}$ then $(w'_{m'} \circ w'_{m},t'\circ t) \in J$ and $(w_{m'} \circ w_{m},t'\circ t) \in J$. Furthermore the composition $t' \circ t$ is well defined since $\dom(t')$ and $\cod(t)$ are the unique elements for $\gamma(\dom(w'_{m'})) \sim' \gamma(\dom(w_{m'}))$ and $\gamma(\cod(w'_{m})) \sim' \gamma(\cod(w_{m}))$ and so $\dom(t') = \cod(t)$. Suppose we have another $t''$ such that $(w_{m'} \circ w_{m},t'') \in J$. By construction of $J$, $t'' = t''_{1} \circ t''_{0}$ with $(w_{m'},t''_{1}) \in J$ and $(w_{m},t''_{0}) \in J$. Thus $t''_{1} = t'$ and $t''_{0} = t$, so $t'' = t' \circ t$.\\

For the $(-)^{-1}$ case:\\

Suppose we have natural elements $\gamma(w) \sim' \gamma(w')$ satisfying $C_{u}$, if $w = w_{1} \circ w_{0}$, then since $w^{-1} = w_{0}^{-1} \circ w_{1}^{-1}$ by the $- \circ -$ case if suffices to check $w = \alpha(q,\phi)$.\\

If $w = \alpha(q,\phi)$, then we have a unique $t = \beta(n,m,\phi')$ and $(\dom(w),\dom(t)) \in J$, $(\cod(w),\cod(t)) \in J$, $(\phi(p) ,\phi'_{i'(p)}) \in J $ for $p \in C'_{c_{m}}$ and $(\phi(p) ,(\phi'_{i'(p)^{-1}})^{-1}) \in J$ for $p \in C'_{c^{i}_{m}}$ where each $\gamma(\phi(p)) \sim' \gamma(\phi(p'))$ for $i'(p) = i'(p')$ satisfies $C_{u}$ by induction. This data also works for $(w^{-1} t^{-1}) \in J$ and $t^{-1}$ must be unique by induction. $t$ is also the unique element for $w'$ and thus $t^{-1}$ is the unique element for $w'$.\\

 For associativity:\\

Suppose we have that $(w_{0}, t_{0}) \in J$, $(w_{1},t_{1}) \in J$, and $(w_{2},t_{2}) \in J$ and $\gamma(w_{0}) \sim' \gamma(w_{0})$, $\gamma(w_{1}) \sim' \gamma(w_{1})$, and $\gamma(w_{2}) \sim' \gamma(w_{2})$ satisfies $C_{u}$, then by the $- \circ -$ case $t_{2} \circ t_{1} \circ t_{0}$ is the unique element for $(w_{2} \circ w_{1}) \circ w_{0}$ and $w_{2} \circ (w_{1} \circ w_{0})$.\\

Suppose we have $w = \alpha(q,E)$ and $w' = \alpha(q',E')$ such that $q, q' \in Q_{y}$ and for $p \in P_{q}$ and $p' \in P_{q'}$ such that $i'(p) = i'(p')$, $\gamma(E(p)) \sim' \gamma(E'(p'))$ satisfies $C_{u}$ when  $y$ is an object, We also assume the instances of $\sim'$ regarding functoriality of $w$ and $w'$ satisfies $C_{u}$. We define a functor $B: X_{y} \to D$, $B(x) = t_{x}$ where $t_{x}$ is the element of $D$ for $\gamma(E(p)) \sim' \gamma(E'(p'))$ such that $i'(p) = i'(p') = x$. $B$ is well defined since for any $p,p^{*} \in P_{q}$ and $p',p^{-} \in P_{q'}$ such that $i'(p) = i'(p^{*}) = i'(p^{-}) = i'(p') = x$, $\gamma(E(p)) \sim' \gamma(E(p^{*})) \sim' \gamma(E'(p^{-})) \sim' \gamma(E'(p'))$ by transitivity and so all such elements have the same $t_{x}$. Also both $w$ and $w'$ are functorial which by the previous cases makes $B$ a functor. Thus we have that $(w,\beta(z,y,B)) \in J$ and $(w',\beta(z,y,B)) \in J$ and $B$ is unique in such sense since each $\gamma(E(p)) \sim' \gamma(E'(p'))$ satisfies $C_{u}$. The morphism case is similar except we use naturality to construct a natural transformation.\\

Suppose we have $\gamma(w) \sim' \gamma(w')$ where $w = \alpha(q,E)$ for $q \in  Q_{y}$ and there exists a $r \in R$ and $p \in P_{q}$ such that $f(k(r)) = y$, $i'(p) = k(r)$, and $E(p) = w'$. We also assume that for $p,p' \in P_{q}$ such that $i'(p) = i'(p')$, $\gamma(E(p)) \sim' \gamma(E(p'))$ satisfies $C_{u}$ as well as instances of $\sim'$ related to either the naturality or functoriality of $w$. As with the previous case We can construct a unique $t$ such that $(w,t) \in J$ is inhabited. Furthermore when $t = \beta(z,y,B)$, $B(i'(p))$ is the unique element such that $(w',B(i'(p))) \in J$ is inhabited. Since $d$ is a $\poly_{f,g,h,k}$-algebra, $B(i'(p)) = t$, thus $w$ and $w'$ share the same $t$.\\

Suppose we have natural elements $\gamma(w'') \sim' \gamma(w' \circ w)$ where $w'' = \alpha(q'',\phi'')$, $w' = \alpha(q',\phi')$ and $w = \alpha(q,\phi)$. We assume that $\gamma(\dom(w'')) \sim' \gamma(\dom(w))$, $\gamma(\cod(w'')) \sim' \gamma(\cod(w'))$, $\gamma(\cod(w)) \sim' \gamma(\dom(w'))$, $\gamma(\phi''(p'')) \sim' \gamma(\phi'(p') \circ \phi(p))$ for $i'(p'') = i'(p') \circ i'(p)$ and every instance of $\sim'$ pertaining to the naturality of $w''$, $w'$, and $w'$ satisfies $C_{u}$. As above we can construct unique $t''$, $t'$, and $t$ such that $\dom(t'') = \dom(t)$, $\cod(t'') = \cod(t')$, and $\cod(t) = \dom(t')$ and $(w'',t'') \in J$, $(w',t') \in J$, and $(w,t) \in J$.  Furthermore for $p'' \in P_{q''}$, $p' \in P_{q'}$, and $p \in P_{q}$ such that $i'(p'') = i'(p') \circ i'(p)$, we have that $t''_{i'(p'')}$,  $t'_{i'(p')}$, and $t_{i'(p)}$ are unique for $\phi''(p'')$, $\phi'(p')$, and $\phi(p)$ and $t''_{i'(p'')} = t'_{i'(p')} \circ t_{i'(p)}$ since $\gamma(\phi''(p'')) \sim' \gamma(\phi'(p') \circ \phi(p))$ satisfies $C_{u}$. Then $t'' = t' \circ t$.\\

The final two cases follows from apply the $\dom$, $\cod$, $\id$,  $(-)^{-1}$ and $comp$ cases. \\

Thus for $w, w' \in W'_{\overline{f}',t'_{1},t'_{2}}$, when $\gamma(w) \sim' \gamma(w')$ there exists a unique $t$  such that $(w,t) \in J$ is inhabited and the same $t$ is unique for $w'$ as well. We can construct a functor $j :U_{\overline{f}',t'_{1},t'_{2}} \to D$ where $j([w]_{r})$ is the unique element for every $\gamma(w) \sim' \gamma(w')$. functoriality follows from the $\dom$, $\cod$, $\id$, $comp$, and the $lcomp$ and $rcomp$ cases. The fact that $\gamma$ extends to a functor $\gamma: \omega \to d$ is straightforward. By the construction of $J$ and the reduction cases, $\gamma$ is a  $\poly_{f,g,h,k}$-algebra morphism. We also know by the construction of $J$ and $C_{u}$ that $\gamma$ is the unique $\poly_{f,g,h,k}$-algebra morphism. Therefore $\omega$ is the initial $\poly_{f,g,h,k}$-algebra.

\end{proof}


\subsection{Uniform Normalized Small Object Argument}\label{soasec}
In this section, we use $W$-types with reductions to construct a version of the small object argument in $\GrpA{A}$. The best references for understanding the small object argument are \cite{Garner_2008} and \cite[chapter 12]{Riehl_2014}. We will denote a normal isofibration using the form $\sum_{i \in I} U_{i} \to I$ where $U_{i}$ is the fiber over $i \in \ob I$. We will also refer to normal isofibrations as simply fibrations. Given a groupoid assembly $I$, we have that the full subcategory of fibrations over $I$, $\Fib(I)$, is cartesian closed. We will review said structure:\\

Given $F:\sum_{i \in I} U_{i} \to I$ and $G:\sum_{i \in I} V_{i} \to I$, we have $F \times G = \Sigma_{F}  \circ \Delta_{F}(G)$ where $\Delta_{F}$ is the pullback functor from $\Fib(I)$ to $\Fib(\sum_{i \in I} U_{i})$ and $\Sigma_{F}$ is $F \circ -$. $F \to G = \Pi_{F} \circ \Delta_{F}(G)$. Breaking this down, we take the pullback of $G$ along $F$:

\[\begin{tikzcd}
	\sum_{i \in I} U_{i} \times_{I} \sum_{i \in I} V_{i}  &&& {\sum_{i \in I} V_{i} } \\
	\\
	\\
	{\sum_{i \in I} U_{i}} &&& {I}
	\arrow[ "\Delta_{G}(F)" ,from=1-1, to=1-4]
	\arrow["\Delta_{F}(G)"',from=1-1, to=4-1]
	\arrow["\lrcorner"{anchor=center, pos=0.125}, draw=none, from=1-1, to=4-4]
	\arrow["{G}", from=1-4, to=4-4]
	\arrow["{F}"', from=4-1, to=4-4]
\end{tikzcd}\]

For $F \times G$:

\begin{itemize}

\item Objects over $i \in \ob I$ are pairs $(u_{i},v_{i}) \in \ob U_{i} \times V_{i}$.\\

\item Morphisms over $k: i \to i' \in \mor I$ are pairs $(p_{k},m_{k}): (u_{i},v_{i}) \to (u_{i'},v_{i'})$ where $p_{k}$ is over $k$ in $F$ and $m_{k}$ is over $k$ in G. Note that a morphism $(p_{\id_{i}},m_{\id_{i}})$ exists in $\mor U_{i} \times V_{i}$.\\

\end{itemize}

We will denote $F \times G$ as $\sum_{i \in I} U_{i} \times V_{i} \to J$. As for $F \to G$:\\

\begin{itemize}

\item Objects over $i$ are functors $s_{i}:U_{i} \to \sum_{i \in I} U_{i} \times V_{i} $ such that $\Delta_{F}(G) \circ s_{i}$ can be restricted to the identity on $U_{i}$. This means $s_{i}$ can be represented as a functor $s'_{i}: U_{i} \to V_{i}$ where $s_{i}(u) = (u,s'_{i}(u))$.\\

\item Morphisms over $k:i \to i'$ are generalized natural transformations $\sigma_{k}: s_{i} \rightsquigarrow s_{i'}$ such that for $p_{k}:u_{i} \to u_{i'}$, $\Delta_{F}(G)(\sigma_{k,p_{k}}) = p_{k}$. \\

\item $\sigma_{k}$ can represented as a morphism taking $p_{k}:u_{i} \to u_{i'}$ to a morphism $m_{k}:s'_{i}(u_{i}) \to s'_{i'}(u_{i'})$ such that when we have another $p^{*}_{k}:u^{*}_{i} \to u^{*}_{i'}$ and morphisms $u^{i}:u_{i} \to u^{*}_{i} \in U_{i}$ and $u^{i'}:u_{i'} \to u^{*}_{i'} \in U_{i'}$ giving the follow commutative diagram:

\[\begin{tikzcd}
	u_{i} &&& {u_{i'}} \\
	\\
	\\
	{u^{*}_{i}} &&& {u^{*}_{i'}}
	\arrow["p_{k}", from=1-1, to=1-4]
	\arrow["u^{i}"',from=1-1, to=4-1]
	\arrow["{u^{i'}}", from=1-4, to=4-4]
	\arrow["{p^{*}_{k}}"', from=4-1, to=4-4]
\end{tikzcd}\]

Then we have $s'_{i'}(u^{i'}) \circ m_{k} = m^{*}_{k} \circ s'_{i}(u^{i})$ where $m^{*}_{k} = \sigma_{k,p^{*}_{k}}$.\\

\item $\sigma_{\id_{i}}$ can equivalently be characterized as natural transformations $s'_{i} \To s''_{i}$ in which case the fiber of $F \to G$ over $i$ is precisely the groupoid assembly $U_{i} \to V_{i}$.\\

\end{itemize}

We will denote $F \to G$ as $\sum_{i \in I} (U_{i} \to V_{i}) \to I$. Given a groupoid assembly $Z$, we can define the constant fibration on $Z$, $\pr{I}: I \times Z \to I$ which we denote as $\Sigma_{i \in I} Z \to I$.\\

 Given an indexing groupoid assembly $I$, fibrations $U:\sum_{i \in I} U_{i} \to I$, $V:\sum_{i \in I} V_{i} \to I$, and $N:\sum_{i \in I} N_{i} \to I$ and a morphism $J: \sum_{i \in I} U_{i} \to \sum_{i \in I} V_{i} $ over $I$, for every groupoid assembly $Z$ first we set $\sum_{i \in I} [\mathrm{nf}]_{i} = \sum_{i \in I} U_{i} \times V_{i} \times Z \times N_{i}$
can construct the following $W$-type with reduction diagram which we refer to as $SOA_{Z}$:

\[\begin{tikzcd}[column sep=small]
	{\sum_{i \in I} U_{i} \times (V_{i} \to Z) + \sum_{i \in I} [\mathrm{nf}]_{i}} && {\sum_{I \in i} U_{i} \times V_{i} \times (V_{i} \to Z)} &&& [.3 cm]{\sum_{i \in I} V_{i} \times (V_{i} \to Z)} \\
	\\
	\\
	& Z &&&&& Z
	\arrow["J^{*} + \const", tail, from=1-1, to=1-3]
	\arrow["\pr{V \times (V \to Z)}", from=1-3, to=1-6]
	\arrow["\app{-}_{J}"', from=1-3, to=4-2]
	\arrow["\app{-}", from=1-6, to=4-7]
\end{tikzcd}\]

where we define the functors used above:

For $J^{*}:\sum_{i \in I} U_{i} \times (V_{i} \to Z) \to \sum_{i \in I} U_{i} \times V_{i} \times (V_{i} \to Z)$:
\begin{itemize}

\item $J^{*}(u_{i},f_{i}) = (u_{i}, J(u_{i}),f_{i})$
\item $J^{*}(p_{k},\alpha_{k}) = (p_{k},J(p_{k}),\alpha_{k})$\\

\end{itemize}

For $\const:\sum_{i \in I} [\mathrm{nf}]_{i} \to \sum_{i \in I} U_{i} \times V_{i} \times (V_{i} \to Z)$:
\begin{itemize}

\item $\const(u_{i},v_{i},z,n_{i}) = (u_{i}, v_{i},\const_{z})$. $\const_{z}$ takes objects to $z$ and morphisms to $\id_{z}$.
\item $\const(p_{k},m_{k},\zeta:z \to z',\nu_{k}) = (p_{k},m_{k},\const_{\zeta})$. $\const_{\zeta}$ takes every morphism to $\zeta$.\\

\end{itemize}

For $\app{-}_{J}:\sum_{I \in i} U_{i} \times V_{i} \times (V_{i} \to Z) \to Z$:
\begin{itemize}

\item $\app{u_{i},v_{i},f_{i}}_{J} = f_{i}(J(u_{i}))$
\item $\app{p_{k},m_{k},\alpha_{k}}_{J} = \alpha_{k,J(p_{k})}$

\end{itemize}

$\pr{V \times (V \to Z)}$ and $\app{-}$ are just a projection operator and function application respectively. $\pr{V \times (V \to Z)}$ is a fibration since every object $(u_{i},v_{i},f_{i})$ over a morphism $(m_{k}:v_{i} \to v_{i'},\alpha_{k}:f_{i} \to f_{i'})$ lifts to a morphism $(\lift(\htpy{0},U )(u_{i},k), m_{k},\alpha_{k})$ where if $m_{k}$ and $\alpha_{k}$ are identity morphisms then $k = \id_{i}$ and so $\lift(\htpy{0},U )(u_{i},k) = \id_{u_{i}}$. One can also verify that $\app{-}_{J} \circ (J^{*} + \const) = \app{-} \circ \pr{V_{i} \times (V_{i} \to Z)} \circ (J^{*} + \const)$. Given a morphism $g: H \to Z$ we modify $SOA_{Z}$ as follows: 

\[\begin{tikzcd}[column sep=small]
	{\sum_{i \in I} U_{i} \times (V_{i} \to Z) + \sum_{i \in I} [\mathrm{nf}]_{i}} & &  [-.1cm] {\sum_{I \in i} U_{i} \times V_{i} \times (V_{i} \to Z)} &&& [.3 cm]{\sum_{i \in I} V_{i} \times (V_{i} \to Z)  +  H} \\
	\\
	\\
	& Z &&&&& [-.8cm] Z
	\arrow["J^{*} + \const", tail, from=1-1, to=1-3]
	\arrow["\pr{V \times (V \to Z)} ", from=1-3, to=1-6]
	\arrow["\app{-}_{J}"', from=1-3, to=4-2]
	\arrow["\app{-} + g", from=1-6, to=4-7]
\end{tikzcd}\]

We refer to this diagram as $SOA_{Z} + g$. The other morphism alongside $\pr{V \times (V \to Z)}$ is the unique morphism $0 \to H$ which is also a fibration since $0$ has no elements. $SOA_{Z} + g$ gives us the following $W$-type with reduction $\omega_{g}: \tilde{W}_{g} \to Z$ with the canonical $W$-type algebra structure denoted as $\alpha$. Furthermore we have the morphism $\overline{g}: H \to  \tilde{W}_{g} $ where $\overline{g}(h) = \alpha(h, 0 \to \tilde{W}_{g})$ and $\overline{g}(\eta:h \to h') = \alpha(\eta,0 \to \mor \tilde{W}_{g})$ such that $g = \omega_{g} \circ \overline{g}$. For $i \in \ob I$, we denote the restriction of $J$ to $i$ as $J_{i}:U_{i} \to V_{i}$. Not only can we show that $\omega_{g}$ has the right lifting property against $J_{i}$ for all $i \in \ob I$, but we can show an even stronger lifting property. First we define the following pullback square over $I$:\\
\[\begin{tikzcd}
	{\sum_{i \in I} (V_{i} \to \tilde{W}_{g})} \\
	\\
	&& {\sum_{i \in I} \Sq(J_{i}, \omega_{g})} &&& {\sum_{i \in I} (U_{i} \to \tilde{W}_{g})} \\
	\\
	\\
	&& {\sum_{i \in I}(V_{i} \to Z)} &&& {\sum_{i \in I}(U_{i} \to Z)}
	\arrow["{J  \pupw  \omega_{g}}"{description}, dashed, from=1-1, to=3-3]
	\arrow["{- \circ J}"{description}, from=1-1, to=3-6]
	\arrow["{\omega_{g} \circ -}"{description},  from=1-1, to=6-3]
	\arrow[from=3-3, to=3-6]
	\arrow[from=3-3, to=6-3]
	\arrow["\lrcorner"{anchor=center, pos=0.125}, draw=none, from=3-3, to=6-6]
	\arrow["{\omega_{g} \circ -}", from=3-6, to=6-6]
	\arrow["{- \circ J}"', from=6-3, to=6-6]
\end{tikzcd}\]

where $- \circ J$ and $\omega_{g}  \circ -$ can be defined explicitly as follows:\\

For $- \circ J$:
\begin{itemize}

\item $f_{i}: V_{i} \to Z$ and $f'_{i}: V_{i} \to \tilde{W}_{g}$ gets sent to $f_{i} \circ J_{i}$ and $f'_{i} \circ J_{i}$.
\item $\alpha_{k}:f_{i} \to f_{i'}$ and $\alpha'_{k}:f'_{i} \to f'_{i'}$ gets sent to $\alpha_{k,J_{k}(-)}$ and $\alpha'_{k,J_{k}(-)}$ where $J_{k}$ takes a morphism $p_{k}:u_{i} \to u_{i'}$ to $J_{k}(p_{k}) = J(p_{k}): J_{i}(u_{i}) \to J_{i'}(u_{i'})$.
\end{itemize}

For $\omega_{g} \circ -$:
\begin{itemize}

\item $f_{i}: U_{i} \to \tilde{W}_{g}$ and $f'_{i}: V_{i} \to \tilde{W}_{g}$ gets sent to $\omega_{g} \circ f_{i}$ and $\omega_{g} \circ f'_{i}$.
\item $\alpha_{k}:f_{i} \to f_{i'}$ and $\alpha'_{k}:f'_{i} \to f'_{i'}$ gets sent to $\omega_{g} \cdot \alpha_{k}$ and $\omega_{g} \cdot \alpha'_{k}$.
\end{itemize}

The groupoid $\sum_{i \in I} \Sq(J_{i}, \omega_{g})$ has as its objects $(c_{i}: U_{i} \to \tilde{W}_{g}, s_{i}: V_{i} \to Z)$ such that $\omega_{g} \circ c_{i} = s_{i} \circ J_{i}$ and as its morphisms $(\gamma_{k}, \sigma_{k}): (c_{i},s_{i}) \to (c_{i'},s_{i'})$ such that $\omega_{g} \cdot \gamma_{k} = \sigma_{k,J_{k}(-)}$. We also have that $\sum_{i \in I} \Sq(J_{i}, \omega_{g}) \to I$ is a fibration however we can only show that after proving the following lemma:

\begin{lem}
\label{liftI}

The functor $J  \pupw  \omega_{g}: \sum_{i \in I} (V_{i} \to \tilde{W}_{g}) \to \sum_{i \in I} \Sq(J_{i}, \omega_{g})$ over $I$ has a section $\lambda(J,\omega_{g})$ over $I$ which gives liftings for squares in $\sum_{i \in I} \Sq(J_{i}, \omega_{g})$. Furthermore, the liftings must satisfy the following normalization property:\\

When $n_{i} \in \ob N_{i}$ and $(c_{i}, s_{i}) \in \ob \sum_{i \in I} \Sq(J_{i}, \omega_{g})$ with $s_{i} = \const_{z}$ for all $u_{i} \in \ob U_{i}$: $$\lambda(J,\omega_{g})(c_{i},s_{i}) = \const_{c_{i}(u_{i})}$$ making $c_{i}$ constant. When $\nu_{k}: n_{i} \to n_{i'}$ exists in $N$ over $k$ and we have $(\gamma_{k},\sigma_{k}): (c_{i}, s_{i}) \to (c_{i'}, s_{i'}) $ with $\sigma_{k} = \const_{\zeta}$ for $\zeta: z \to z'$ for all $p_{k}:u_{i} \to u_{i'}$ over $k:i \to i'$: $$\lambda(J,\omega_{g})(\gamma_{k},\sigma_{k}) = \const_{\gamma_{k,p_{k}}}$$ which makes $\gamma_{k}$ constant.

\end{lem}

\begin{proof}

Given $(c_{i},s_{i}) \in \sum_{i \in I} \Sq(J_{i}, \omega_{g})$, we seek to construct a functor $\lambda(J,\omega_{g})(c_{i},s_{i}): V_{i} \to \tilde{W}_{g}$. Note that the condition $s_{i} \circ J_{i} = \omega_{g} \circ c_{i}$ makes $c_{i}$ of the type $\Pi_{u_{i} \in U_{i}} \tilde{W}_{g,s_{i} \circ J_{i}(u_{i})}$, thus for $v_{i} \in \ob V_{i}$, $((v_{i},s_{i}),c_{i})$ is in the domain of $\alpha$. So we set $\lambda(J,\omega_{g})(c_{i},s_{i})(v_{i}) = \alpha((v_{i},s_{i}),c_{i})$ and $\lambda(J,\omega_{g})(c_{i},s_{i})(m_{\id_{i}})  =  \alpha((m_{\id_{i}},\id_{s_{i}}),\id_{c_{i}})$. functoriality of $\lambda(J,\omega_{g})(c_{i},s_{i})$ follows from functoriality of $\alpha$. Given $(\gamma_{k},\sigma_{k}): (c_{i},s_{i}) \to (c_{i'},s_{i'})$ we seek to construct a generalized natural transformation $\lambda(J,\omega_{g})(\gamma_{k},\sigma_{k}): \lambda(J,\omega_{g})(c_{i},s_{i}) \to \lambda(J,\omega_{g})(c_{i'},s_{i'})$. For a moprhism $m_{k}: v_{i} \to v_{i'}$ over $k: i \to i'$, we set $\lambda(J,\omega_{g})(\gamma_{k},\sigma_{k})_{m_{k}} = \alpha((m_{k},\sigma_{k}),\gamma_{k})$. One can verify using functoriality of $\alpha$ that $\lambda(J,\omega_{g})(\gamma_{k},\sigma_{k})$ is a generalized natural transformation from $\lambda(J,\omega_{g})(c_{i},s_{i})$ to $\lambda(J,\omega_{g})(c_{i'},s_{i'})$. We have $$\id_{\lambda(J,\omega_{g})(c_{i},s_{i}),m_{\id_{i}} }= \lambda(J,\omega_{g})(c_{i},s_{i})(m_{\id_{i}}) = \alpha((m_{\id_{i}},\id_{s_{i}}),\id_{c_{i}}) = \lambda(J,\omega_{g})(\id_{c_{i}},\id_{s_{i}})_{m_{\id_{i}}}$$ and for $(\gamma_{k},\sigma_{k}):(c_{i},s_{i}) \to (c_{i'},s_{i'})$ and $(\gamma_{k'},\sigma_{k'}): (c_{i'},s_{i}) \to (c_{i''},s_{i''})$ and for $m_{k' \circ k} = m_{k'} \circ m_{k}$ over $k' \circ k$ we have: 

\begin{align*}
\lambda(J,\omega_{g})(\gamma_{k'} \circ \gamma_{k},\sigma_{k'} \circ \sigma_{k})_{m_{k' \circ k}}  &= \alpha((m_{k' \circ k},\sigma_{k'} \circ \sigma_{k}), \gamma_{k'} \circ \gamma_{k})\\
&= \alpha((m_{k'},\sigma_{k'}),\gamma_{k'}) \circ \alpha((m_{k},\sigma_{k}),\gamma_{k}) \\
&= \lambda(J,\omega_{g})(\gamma_{k'},\sigma_{k'}))_{m_{k'}} \circ \lambda(J,\omega_{g})(\gamma_{k},\sigma_{k}))_{m_{k}} 
\end{align*}

Thus $\lambda(J,\omega_{g})(\gamma_{k'} \circ \gamma_{k},\sigma_{k'} \circ \sigma_{k}) = \lambda(J,\omega_{g})(\gamma_{k'},\sigma_{k'})) \circ \lambda(J,\omega_{g})(\gamma_{k},\sigma_{k}))$. So $\lambda(J,\omega_{g})$ is a functor. It show be obvious by construction that $\lambda(J,\omega_{g})$ is a functor over $I$. To show that it is a section of $J  \pupw  \omega_{g}$ first we are considering $(c_{i},s_{i}) \in \ob \sum_{i \in I} \Sq(J_{i}, \omega_{g}) $. $J  \pupw  \omega_{g} \circ \lambda(J,\omega_{g})(c_{i},s_{i}) = (\lambda(J,\omega_{g})(c_{i},s_{i}) \circ J_{i}, \omega_{g} \circ \lambda(J,\omega_{g})(c_{i},s_{i}))$. Breaking this down:\\

For $\lambda(J,\omega_{g})(c_{i},s_{i}) \circ J_{i}$:
\begin{itemize}
\item $u_{i} \in \ob U_{i} \mapsto \alpha((J_{i}(u_{i}), s_{i}),c_{i}) = c_{i}(u_{i})$
\item $p_{\id_{i}} \in \mor U_{i} \mapsto \alpha((J_{i}(p_{\id_{i}}), \id_{s_{i}}),\id_{c_{i}}) = \id_{c_{i},p_{\id_{i}}} = c_{i}(p_{\id_{i}})$
\end{itemize}
This follows from the reduction induced by the functor $J^{*}$. As for $\omega_{g} \circ \lambda(J,\omega_{g})(c_{i},s_{i})$:
\begin{itemize}
\item $v_{i} \mapsto \omega_{g}(\alpha((v_{i}, s_{i}),c_{i}) ) = s_{i}(v_{i})$
\item $m_{\id_{i}} \mapsto \omega_{g}(\alpha((m_{\id_{i}}, \id_{s_{i}}),\id_{c_{i}}) ) = \id_{s_{i}}(m_{\id_{i}}) = s_{i}(m_{\id_{i}})$

\end{itemize}
This follows from the dependency of $\tilde{W}_{g}$ induced by $\app{-}$. A similar argument can be made for morphisms. Thus $\lambda(J,\omega_{g})$ is a section of $J  \pupw  \omega_{g}$ over $I$. \\

Suppose we have $n_{i} \in \ob N_{i}$ and $(c_{i},s_{i}) \in \sum_{i \in I} \Sq(J_{i}, \omega_{g})$ such that $s_{i} = \const_{z}$ for some $ z \in \ob Z$. For every $u_{i} \in \ob U_{i}$ and $v_{i} \in \ob V_{i}$, $(u_{i},v_{i},z,n_{i}) \in  [\mathrm{nf}]_{i}$, so $\lambda(J,\omega_{g})(c_{i},s_{i})(v_{i}) = \alpha((v_{i},s_{i}),c_{i}) = c_{i}(u_{i})$ for every $u_{i} \in U_{i}$. We also have that $(\id_{u_{i}},m_{\id_{i}}, \id_{z},\id_{n_{i}}) \in [\mathrm{nf}]_{i}$ so  $\lambda(J,\omega_{g})(c_{i},s_{i})(m_{\id_{i}}) = \alpha((m_{\id_{i}},\id_{s_{i}}),\id_{c_{i}}) = \id_{c_{i}}(\id_{u_{i}}) = c_{i}(\id_{u_{i}}) = \id_{c_{i}(u_{i})}$ for every $u_{i} \in \ob U_{i}$. Thus $\lambda(J,\omega_{g})(c_{i},s_{i}) = \const_{c_{i}(u_{i})}$ for every $u_{i} \in U_{i}$.\\

Suppose we have $\nu_{k}: n_{i} \to n_{i'}$ over $k$ in $N$ and $(\gamma_{k},\sigma_{k}) \in \mor \sum_{i \in I} \Sq(J_{i}, \omega_{g})$ such that $\sigma_{k} = \const_{\zeta}$ for $\zeta:z \to z'$ in $Z$. For every $p_{k} \in \mor \sum_{i \in I} U_{i}$ and $m_{k} \in \mor \sum_{i \in I} V_{i}$ over $k$ we have $(p_{k},m_{k},\zeta,\nu_{k})$ over $k$ in $\sum_{i \in I} [\mathrm{nf}]_{i}$, so $\lambda(J,\omega_{g})(\gamma_{k},\sigma_{k})_{m_{k}} = \alpha((m_{k},\sigma_{k}),\gamma_{k}) = \gamma_{k,p_{k}}$ for every $p_{k}$ over $k$. Therefore $\lambda(J,\omega_{g})(\gamma_{k},\sigma_{k}) = \const_{\gamma_{k,p_{k}}}$ for every $p_{k}$ over $k$. This completes the proof.

\end{proof}

\begin{cor}

$\sum_{i \in I} \Sq(J_{i}, \omega_{g}) \to I$ is a fibration.

\end{cor}
\begin{proof}

Given a morphism $k:i \to i'$ in $I$ and $(c_{i}, s_{i}) \in \ob \sum_{i \in I} \Sq(J_{i}, \omega_{g}) $, $f_{i} = \lambda(J,\omega_{g})(c_{i},s_{i})$ is an object in $\sum_{i \in I} (V_{i} \to \tilde{W}_{g})$ over $i$. Since $\sum_{i \in I} (V_{i} \to \tilde{W}_{g}) \to I$ is a fibration we use the fibration structure to get a lift $\lift(\htpy{0},V \to \tilde{W}_{g})(f_{i},k)$ over $k$ with $\dom(\lift(\htpy{0},V \to \tilde{W}_{g})(f_{i},k)) = f_{i}$. Thus $\beta_{k} = J  \pupw  \omega_{g}(\lift(\htpy{0},V \to \tilde{W}_{g})(f_{i},k))$ is a morphism in $\sum_{i \in I} \Sq(J_{i}, \omega_{g})$ over $k$ with $\dom(\beta_{k}) = (c_{i}, s_{i})$. When $k = \id_{i}$, then $\lift(\htpy{0},V \to \tilde{W}_{g})(f_{i},k) = \id_{f_{i}}$, so $\beta_{\id_{i}} = J  \pupw  \omega_{g}(\id_{f_{i}}) = \id_{(c_{i},s_{i})}$. Thus $\lift(\htpy{0},\Sq(J,\omega_{g}))((c_{i},s_{i}),k) = J  \pupw  \omega_{g}(\lift(\htpy{0},V \to \tilde{W}_{g})( \lambda(J,\omega_{g})(c_{i},s_{i}),k))$ making $\sum_{i \in I} \Sq(J_{i}, \omega_{g}) \to I$ a fibration.
 
\end{proof}

\begin{lem}

The class of $SOA_{Z} + g$ diagrams induces a functorial factorization $[\mathrm{SOA}]:\GrpA{A}^{\arr} \to \GrpA{A}^{\twarr}$.

\end{lem}
\begin{proof}

Given a square:

\[\begin{tikzcd}
	{X_{0}} &&& {X_{1}} \\
	\\
	\\
	{Y_{0}} &&& {Y_{1}}
	\arrow["{t'_{0}}", from=1-1, to=1-4]
	\arrow["{g_{0}}"', from=1-1, to=4-1]
	\arrow["{g_{1}}", from=1-4, to=4-4]
	\arrow["{t_{0}}"', from=4-1, to=4-4]\\
\end{tikzcd}\]

we seek to construct a commutative diagram:

\[\begin{tikzcd}
	{X_{0}} &&& {X_{1}} \\
	\\
	{\tilde{W}_{g_{0}}} &&& {\tilde{W}_{g_{1}}} \\
	\\
	{Y_{0}} &&& {Y_{1}}
	\arrow["{t'_{0}}", from=1-1, to=1-4]
	\arrow["{\overline{g_{0}}}"', from=1-1, to=3-1]
	\arrow["{\overline{g_{1}}}", from=1-4, to=3-4]
	\arrow["{[\mathrm{SOA}](t'_{0},t_{0})}"{description}, from=3-1, to=3-4]
	\arrow["{\omega_{g_{0}}}"', from=3-1, to=5-1]
	\arrow["{\omega_{g_{1}}}", from=3-4, to=5-4]
	\arrow["{t_{0}}"', from=5-1, to=5-4]
\end{tikzcd}\]

We define $[\mathrm{SOA}](t'_{0},t_{0})$ inductively using the algebra structure on $\tilde{W}_{g_{0}}$. We set $[\mathrm{SOA}](t'_{0},t_{0})(\overline{g}_{0}(x)) = \overline{g}_{1} \circ t'_{0}(x)$ for object or morphism $x \in X_{0}$. Now given an object $\alpha_{g_{0}}((v_{i},s_{i}:V_{i} \to Y_{0}),c_{i}: U_{i} \to \tilde{W}_{g_{0}})$, $[\mathrm{SOA}](t'_{0},t_{0})(\alpha_{g_{0}}((v_{i},s_{i}),c_{i})) = \alpha_{g_{1}}((v_{i}, t_{0} \circ s_{i}),[\mathrm{SOA}](t'_{0},t_{0}) \circ c_{i})$ and given a morphism $\alpha_{g_{0}}((m_{k},\sigma_{k}:s_{i} \to s_{i'}),\gamma_{k}:c_{i} \to c_{i'})$, $[\mathrm{SOA}](t'_{0},t_{0})(\alpha_{g_{0}}((m_{k},\sigma_{k}),\gamma_{k})) = \alpha_{g_{1}}((m_{k},t_{0} \cdot \sigma_{k}),[\mathrm{SOA}](t'_{0},t_{0}) \cdot \gamma_{k}) $. This clearly functorial. Now we must show that the reductions are preserved. Suppose we have a object $\alpha_{g_{0}}((J_{i}(u_{i}),s_{i}),c_{i})$, then $ [\mathrm{SOA}](t'_{0},t_{0})(\alpha_{g_{0}}((J_{i}(u_{i}),s_{i}),c_{i})) =  \alpha_{g_{1}}((J_{i}(u_{i}), t_{0} \circ s_{i}),[\mathrm{SOA}](t'_{0},t_{0}) \circ c_{i}) = [\mathrm{SOA}](t'_{0},t_{0})(c_{i}(u_{i}))$. A similar argument applies to morphisms. Note that when $s_{i}$ and $\sigma_{k}$ are constant then so are $t_{0} \circ s_{i}$ and $t_{0} \cdot \sigma_{k}$ thus $[\mathrm{SOA}](t'_{0},t_{0})$ respects the normalization reduction. Hence $[\mathrm{SOA}](t'_{0},t_{0})$ is a well-defined functor. The top square commute by construction of $[\mathrm{SOA}](t'_{0},t_{0})$ on $X_{0}$. As for the bottom square, for $\overline{g}_{0}(x)$, $$\omega_{g_{1}}([\mathrm{SOA}](t'_{0},t_{0})(\overline{g}_{0}(x))) = \omega_{g_{1}}(\overline{g}_{1} \circ t'_{0}(x)) = g_{1}(t'_{0}(x)) = t_{0}(g_{0}(x)) = t_{0}(\omega_{g_{0}}(\overline{g}_{0}(x)))$$ For $\alpha_{g_{0}}((v_{i},s_{i}),c_{i})$, $$\omega_{g_{1}}([\mathrm{SOA}](t'_{0},t_{0})(\alpha_{g_{0}}((v_{i},s_{i}),c_{i}))) = \omega_{g_{1}}(\alpha_{g_{1}}((v_{i},t_{0} \circ s_{i}),[\mathrm{SOA}](t'_{0},t_{0}) \circ c_{i})) = t_{0}(s_{i}(v_{i})) = t_{0}(\omega_{g_{0}}(\alpha_{g_{0}}((v_{i},s_{i}),c_{i}))) $$ and the same holds for morphisms. Since $\omega_{g_{0}}$, $\omega_{g_{1}}$ and $[\mathrm{SOA}](t'_{0},t_{0})$ respect the reductions, the bottom square also commutes.\\

Now we seek to show that $[\mathrm{SOA}]$ preserves composition and identities. Identities are straightforward since $[\mathrm{SOA}](\id_{X_{0}},\id_{Y_{0}})(\overline{g}_{0}(x)) = \overline{g}_{0} \circ \id_{X_{0}}(x) = \overline{g}_{0}(x)$ and $$[\mathrm{SOA}](\id_{X_{0}},\id_{Y_{0}})(\alpha_{g_{0}}((v_{i},s_{i}),c_{i})) = \alpha_{g_{0}}((v_{i},s_{i}),[\mathrm{SOA}](\id_{X_{0}},\id_{Y_{0}}) \circ c_{i}) = \alpha_{g_{0}}((v_{i},s_{i}),c_{i})$$ where we assume $c_{i} = [\mathrm{SOA}](\id_{X_{0}},\id_{Y_{0}}) \circ c_{i}$ by induction. A similar argument holds for morphisms. Thus $[\mathrm{SOA}](\id_{X_{0}},\id_{Y_{0}}) = \id_{\tilde{W}_{g_{0}}}$. Given a commutative diagram:
\[\begin{tikzcd}
	{X_{0}} &&& {X_{1}} &&& {X_{2}} \\
	\\
	\\
	{Y_{0}} &&& {Y_{1}} &&& {Y_{2}}
	\arrow["{{t'_{0}}}", from=1-1, to=1-4]
	\arrow["{{g_{0}}}"', from=1-1, to=4-1]
	\arrow["{t'_{1}}", from=1-4, to=1-7]
	\arrow["{{g_{1}}}", from=1-4, to=4-4]
	\arrow["{g_{2}}", from=1-7, to=4-7]
	\arrow["{{t_{0}}}"', from=4-1, to=4-4]
	\arrow["{t_{1}}"', from=4-4, to=4-7]
\end{tikzcd}\]

we seek to show that $[\mathrm{SOA}](t'_{1},t_{1}) \circ [\mathrm{SOA}](t'_{0},t_{0}) = [\mathrm{SOA}](t'_{1} \circ t'_{0},t_{1} \circ t_{0})$. We have $$[\mathrm{SOA}](t'_{1},t_{1}) \circ [\mathrm{SOA}](t'_{0},t_{0})(\overline{g}_{0}(x)) = [\mathrm{SOA}](t'_{1},t_{1})(\overline{g}_{1}(t'_{0}(x))) = \overline{g}_{2}(t'_{1}(t'_{0}(x)))= \overline{g}_{2}(t'_{1} \circ t'_{0}(x)) = [\mathrm{SOA}](t'_{1} \circ t'_{0},t_{1} \circ t_{0})(\overline{g}_{0}(x))$$ Given $\alpha_{g_{0}}((v_{i},s_{i}),c_{i}) \in \tilde{W}_{g_{0}}$ and suppose  $[\mathrm{SOA}](t'_{1},t_{1}) \circ [\mathrm{SOA}](t'_{0},t_{0}) \circ c_{i} = [\mathrm{SOA}](t'_{1} \circ t'_{0},t_{1} \circ t_{0}) \circ c_{i}$ by induction, then 
\begin{align*}
[\mathrm{SOA}](t'_{1},t_{1}) \circ [\mathrm{SOA}](t'_{0},t_{0})(\alpha_{g_{0}}((v_{i},s_{i}),c_{i})) &= \alpha_{g_{2}}((v_{i}, t_{1} \circ t_{0} \circ s_{i}),[\mathrm{SOA}](t'_{1},t_{1}) \circ [\mathrm{SOA}](t'_{0},t_{0}) \circ c_{i})\\
&= \alpha_{g_{2}}((v_{i}, t_{1} \circ t_{0} \circ s_{i}),[\mathrm{SOA}](t'_{1} \circ t'_{0},t_{1} \circ t_{0}) \circ c_{i}) \\
&= [\mathrm{SOA}](t'_{1} \circ t'_{0},t_{1} \circ t_{0})(\alpha_{g_{0}}((v_{i},s_{i}),c_{i}))\\
\end{align*}

A similar argument holds for morphisms. Thus $[\mathrm{SOA}](t'_{1},t_{1}) \circ [\mathrm{SOA}](t'_{0},t_{0}) = [\mathrm{SOA}](t'_{1} \circ t'_{0},t_{1} \circ t_{0})$. Therefore $[\mathrm{SOA}]$ is a functor.

\end{proof}

\begin{lem}

A functor $g_{0}:X_{0} \to Y_{0}$ satisfies the lifting property of Lemma \ref{liftI} if and only if we have a retract diagram:

\[\begin{tikzcd}
	{X_{0}} &&& {\tilde{W}_{g_{0}}} &&& {X_{0}} \\
	\\
	\\
	{Y_{0}} &&& {Y_{0}} &&& {Y_{0}}
	\arrow["{\overline{g_{0}}}", from=1-1, to=1-4]
	\arrow["{g_{0}}"', from=1-1, to=4-1]
	\arrow["l", from=1-4, to=1-7]
	\arrow["{\omega_{g_{0}}}"{description}, from=1-4, to=4-4]
	\arrow["{g_{0}}", from=1-7, to=4-7]
	\arrow["{\id_{Y_{0}}}"', from=4-1, to=4-4]
	\arrow["{\id_{Y_{0}}}"', from=4-4, to=4-7]
\end{tikzcd}\]

\end{lem}
\begin{proof}
Given the above retract, we can define a lift $\lambda(J,g_{0})$ where:

\begin{itemize}
\item $\lambda(J,g_{0})(c_{i}:U_{i} \to X_{0}, s_{i}: V_{i} \to Y_{0}) = l \circ \lambda(J,\omega_{g_{0}})(\overline{g_{0}} \circ c_{i}, s_{i})$ 
\item $\lambda(J,g_{0})(\gamma_{k}, \sigma_{k}) = l \cdot \lambda(J,\omega_{g_{0}})(\overline{g_{0}} \cdot \gamma_{k}, \sigma_{k})$ 
\end{itemize}

It is straightforward to show that $\lambda(J,g_{0})$ is a lift that preserves the normalization properties by the fact that $\lambda(J,\omega_{g_{0}})$ is such a lift and by using the retract diagram.\\

Suppose the functor $J  \pupw  g_{0}: \sum_{i \in I} (V_{i} \to X_{0}) \to \sum_{i \in I} \Sq(J_{i}, g_{0})$ over $I$ has a section $\lambda(J,g_{0})$ over $I$, we define $l: \tilde{W}_{g_{0}} \to X_{0}$ as follows

\begin{itemize}
\item For $x \in X_{0}$, $l(\overline{g_{0}}(x)) = x$
\item $l(\alpha_{g_{0}}((v_{i},s_{i}),c_{i}))  =   \lambda(J,g_{0})(l \circ c_{i},s_{i})(v_{i})$
\item $l(\alpha_{g_{0}}((m_{k},\sigma_{k}),\gamma_{k}))  =   \lambda(J,g_{0})(l \cdot \gamma_{k},\sigma_{k})(m_{k})$
\end{itemize}
 $l$ as defined is functorial since $\lambda(J,g_{0})$ is functorial and $\lambda(J,g_{0})(l \cdot \gamma_{k},\sigma_{k})$ is a generalized natural transformation. Now we show that $l$ preserves the reductions. Suppose we have $\alpha_{g_{0}}((J_{i}(u_{i}),s_{i}),c_{i})$ then $$l(\alpha_{g_{0}}((J_{i}(u_{i}),s_{i}),c_{i})) =  \lambda(J,g_{0})(l \circ c_{i},s_{i})(J_{i}(u_{i})) = l \circ c_{i}(u_{i})$$ A similar argument holds for morphisms. Now suppose we have $\alpha_{g_{0}}((v_{i},s_{i}),c_{i}))$ such that $s_{i}$ is constant and $n_{i} \in N_{i}$, then $\lambda(J,g_{0})(l \circ c_{i},s_{i})$ is constant thus for all $u_{i}$ and $v_{i}$, $l(\alpha_{g_{0}}((v_{i},s_{i}),c_{i})) = l \circ c_{i}(u_{i})$. We have a similar thing for morphisms, thus $l$ is well defined. We have $l \circ \overline{g_{0}} = \id_{X_{0}}$ by construction of $l$ and given $\alpha_{g_{0}}((v_{i},s_{i}),c_{i})$, $$g_{0} \circ \lambda(J,g_{0})(l \circ c_{i},s_{i})(v_{i}) = s_{i}(v_{i}) = \omega_{g_{0}}(\alpha_{g_{0}}((v_{i},s_{i}),c_{i}))$$ Thus we have a retract diagram

\[\begin{tikzcd}
	{X_{0}} &&& {\tilde{W}_{g_{0}}} &&& {X_{0}} \\
	\\
	\\
	{Y_{0}} &&& {Y_{0}} &&& {Y_{0}}
	\arrow["{\overline{g_{0}}}", from=1-1, to=1-4]
	\arrow["{g_{0}}"', from=1-1, to=4-1]
	\arrow["l", from=1-4, to=1-7]
	\arrow["{\omega_{g_{0}}}"{description}, from=1-4, to=4-4]
	\arrow["{g_{0}}", from=1-7, to=4-7]
	\arrow["{\id_{Y_{0}}}"', from=4-1, to=4-4]
	\arrow["{\id_{Y_{0}}}"', from=4-4, to=4-7]
\end{tikzcd}\]

\end{proof}

\begin{lem}
\label{natlift}

The $\lambda(J,\omega_{g})$ and $J  \pupw  \omega_{g}$ functors from Lemma \ref{liftI} extend to natural transformations between the functors $\sum_{i \in I} (V_{i} \to \tilde{W}_{-}), \sum_{i \in I} \Sq(J_{i}, \omega_{-}): \GrpA{A}^{\arr} \to \Fib(I)$.

\end{lem}
\begin{proof}

It is straightforward to exhibit $\sum_{i \in I} (V_{i} \to \tilde{W}_{-})$ and $\sum_{i \in I} \Sq(J_{i}, \omega_{-})$ as functors. A functor $g_{0}:X_{0} \to Y_{0}$ gets taken to $\sum_{i \in I} (V_{i} \to \tilde{W}_{g_{0}})$ and $\sum_{i \in I} \Sq(J_{i}, \omega_{g_{0}})$ and a square $(t'_{0},t_{0}): g_{0} \to g_{1}$ exhibits the following behavior:\\

For $\sum_{i \in I} \Sq(J_{i}, \omega_{(t'_{0},t_{0})})$:
\begin{itemize}
\item $(c_{i}, s_{i})$ gets sent to $([\mathrm{SOA}](t'_{0},t_{0}) \circ c_{i}, t_{0} \circ s_{i})$
\item $(\gamma_{k}, \sigma_{k})$ gets sent to $([\mathrm{SOA}](t'_{0},t_{0}) \cdot \gamma_{k}, t_{0} \cdot \sigma_{k})$
\end{itemize}

Note that $t_{0} \circ s_{i} \circ J_{i} = t_{0} \circ \omega_{g_{0}} \circ c_{i} = \omega_{g_{1}} \circ [\mathrm{SOA}](t'_{0},t_{0}) \circ c_{i}$ thus $([\mathrm{SOA}](t'_{0},t_{0}) \circ c_{i}, t_{0} \circ s_{i})$ is a square in $\sum_{i \in I} \Sq(J_{i}, \omega_{g_{1}})$. Furthermore not only can one verify that $\sum_{i \in I} \Sq(J_{i}, \omega_{(t'_{0},t_{0})})$ is a functor but also that $\sum_{i \in I} \Sq(J_{i}, \omega_{-})$ is functorial. $\sum_{i \in I} (V_{i} \to \tilde{W}_{(t'_{0},t_{0})})$ is simply post composition with $[\mathrm{SOA}](t'_{0},t_{0}) $ which clearly makes both $\sum_{i \in I} (V_{i} \to \tilde{W}_{(t'_{0},t_{0})})$ and $\sum_{i \in I} (V_{i} \to \tilde{W}_{-})$ functors.\\

Naturality of $J  \pupw  \omega_{-}$ follows from the fact that $[\mathrm{SOA}]$ is a functorial factorization. Now we seek to verify the following commutative square over $I$:

\[\begin{tikzcd}
	{\sum_{i \in I} \Sq(J_{i}, \omega_{g_{0}})} &&& {\sum_{i \in I} \Sq(J_{i}, \omega_{g_{1}})} \\
	\\
	\\
	{\sum_{i \in I} (V_{i} \to \tilde{W}_{g_{0}})} &&& {\sum_{i \in I} (V_{i} \to \tilde{W}_{g_{1}})}
	\arrow["{\sum_{i \in I} \Sq(J_{i}, \omega_{(t'_{0},t_{0})})}", from=1-1, to=1-4]
	\arrow["{\lambda(J,\omega_{g_{0}})}"', from=1-1, to=4-1]
	\arrow["{\lambda(J,\omega_{g_{1}})}", from=1-4, to=4-4]
	\arrow["{\sum_{i \in I} (V_{i} \to \tilde{W}_{(t'_{0},t_{0})})}"', from=4-1, to=4-4]\\
\end{tikzcd}\]

which in turn means verifying the following equations: $$[\mathrm{SOA}](t'_{0},t_{0}) \circ \lambda(J,\omega_{g_{0}})(c_{i},s_{i}) = \lambda(J,\omega_{g_{1}})([\mathrm{SOA}](t'_{0},t_{0}) \circ c_{i},t_{0} \circ s_{i})$$ and $$[\mathrm{SOA}](t'_{0},t_{0}) \cdot \lambda(J,\omega_{g_{0}})(\gamma_{k},\sigma_{k}) = \lambda(J,\omega_{g_{1}})([\mathrm{SOA}](t'_{0},t_{0}) \cdot \gamma_{k},t_{0} \cdot \sigma_{k})$$

which both hold by the construction of $\lambda(J,\omega_{g_{0}})$ and $\lambda(J,\omega_{g_{1}})$ which uses the canonical algebra structure on $\omega_{g_{0}}$ and $\omega_{g_{1}}$. Therefore $ \lambda(J,\omega_{-})$ is a natural transformation and a section of $J  \pupw  \omega_{-}$.

\end{proof}

We define $R_{[\mathrm{SOA}]}: \GrpA{A}^{\arr} \to \GrpA{A}^{\arr}$ as the functor $\Bot \circ [\mathrm{SOA}]$ and $L_{[\mathrm{SOA}]}: \GrpA{A}^{\arr} \to \GrpA{A}^{\arr}$ as the functor $\Top \circ [\mathrm{SOA}]$. 

\begin{lem}

$R_{[\mathrm{SOA}]}$ is a monad.

\end{lem}
\begin{proof}

Given a functor $g_{0}: X_{0} \to Y_{0}$, the component of the unit on $g_{0}$, $\eta_{R_{[\mathrm{SOA}]},g_{0}}$ is given by the square:\\

\[\begin{tikzcd}
	{X_{0}} &&& {\tilde{W}_{g_{0}}} \\
	\\
	\\
	{Y_{0}} &&& {Y_{0}}
	\arrow["{\overline{g}_{0}}", from=1-1, to=1-4]
	\arrow["{g_{0}}"', from=1-1, to=4-1]
	\arrow["{\omega_{g_{0}}}", from=1-4, to=4-4]
	\arrow["{\id_{Y_{0}}}"', from=4-1, to=4-4]\\
\end{tikzcd}\]

The unit $\eta_{R_{[\mathrm{SOA}]}}$ is a natural transformation by $[\mathrm{SOA}]$ being a functorial factorization. We seek to construct the component of multiplication $\mu_{R_{[\mathrm{SOA}]},g_{0}}$:\\

\[\begin{tikzcd}
	{\tilde{W}_{\omega_{g_{0}}}} &&& {\tilde{W}_{g_{0}}} \\
	\\
	\\
	{Y_{0}} &&& {Y_{0}}
	\arrow["{{\mu_{g_{0}}}}", dashed, from=1-1, to=1-4]
	\arrow["{{\omega_{\omega_{g_{0}}}}}"', from=1-1, to=4-1]
	\arrow["{{\omega_{g_{0}}}}", from=1-4, to=4-4]
	\arrow["{{\id_{Y_{0}}}}"', from=4-1, to=4-4]
\end{tikzcd}\]

For $w \in \tilde{W}_{g_{0}}$, $\mu_{g_{0}}(\overline{\omega}_{g_{0}}(w)) = w$. For $\alpha_{\omega_{g_{0}}}((v_{i},s_{i}:V_{i} \to Y_{0}), c_{i}: U_{i} \to \tilde{W}_{\omega_{g_{0}}})$, $\mu_{g_{0}}(\alpha_{\omega_{g_{0}}}((v_{i},s_{i}),c_{i})) = \alpha_{g_{0}}((v_{i},s_{i}), \mu_{g_{0}} \circ c_{i})$. The construction of $\mu_{g_{0}}$ on morphisms are straightforward as well as showing that $\mu_{g_{0}}$ respects the reductions. To show that $\mu_{R_{[\mathrm{SOA}]}}$ is a natural transformation we need to show that the following diagram is commutative for a square $(t'_{0},t_{0}):g_{0} \to g_{1}$:

\[\begin{tikzcd}
	{\tilde{W}_{\omega_{g_{0}}}} &&& {\tilde{W}_{\omega_{g_{1}}}} \\
	\\
	\\
	{\tilde{W}_{g_{0}}} &&& {\tilde{W}_{g_{1}}}
	\arrow["{{[\mathrm{SOA}]([\mathrm{SOA}](t'_{0},t_{0}),t_{0})}}", from=1-1, to=1-4]
	\arrow["{{\mu_{g_{0}}}}"', from=1-1, to=4-1]
	\arrow["{{\mu_{g_{1}}}}", from=1-4, to=4-4]
	\arrow["{{[\mathrm{SOA}](t'_{0},t_{0})}}"', from=4-1, to=4-4]
\end{tikzcd}\]

For $w \in \tilde{W}_{g_{0}}$, $$\mu_{g_{1}}([\mathrm{SOA}]([\mathrm{SOA}](t'_{0},t_{0}),t_{0})(\overline{\omega}_{g_{0}}(w)) ) = \mu_{g_{1}}(\overline{\omega}_{g_{1}}([\mathrm{SOA}](t'_{0},t_{0})(w))) = [\mathrm{SOA}](t'_{0},t_{0})(w) = [\mathrm{SOA}](t'_{0},t_{0})(\mu_{g_{0}}(\overline{\omega}_{g_{0}}(w)))$$ Suppose we have $\alpha_{\omega_{g_{0}}}((v_{i},s_{i}: V_{i} \to Y_{0}),c_{i}: U_{i} \to \tilde{W}_{\omega_{g_{0}}})$ where $$[\mathrm{SOA}](t'_{0},t_{0}) \circ \mu_{g_{0}} \circ c_{i} = \mu_{g_{1}} \circ [\mathrm{SOA}]([\mathrm{SOA}](t'_{0},t_{0}),t_{0}) \circ c_{i} $$ by induction then 
\begin{align*}
\mu_{g_{1}} \circ [\mathrm{SOA}]([\mathrm{SOA}](t'_{0},t_{0}),t_{0})(\alpha_{\omega_{g_{0}}}((v_{i},s_{i}),c_{i})) &= \mu_{g_{1}}(\alpha_{\omega_{g_{1}}}((v_{i},t_{0} \circ s_{i}),[\mathrm{SOA}]([\mathrm{SOA}](t'_{0},t_{0}),t_{0}) \circ c_{i}))\\
&= \alpha_{g_{1}}((v_{i},t_{0} \circ s_{i}),\mu_{g_{1}} \circ [\mathrm{SOA}]([\mathrm{SOA}](t'_{0},t_{0}),t_{0}) \circ c_{i})\\
&= \alpha_{g_{1}}((v_{i},t_{0} \circ s_{i}),[\mathrm{SOA}](t'_{0},t_{0}) \circ \mu_{g_{0}} \circ c_{i})\\
&= [\mathrm{SOA}](t'_{0},t_{0}) (\alpha_{g_{0}}((v_{i}, s_{i}), \mu_{g_{0}} \circ c_{i}))\\
&=  [\mathrm{SOA}](t'_{0},t_{0}) \circ \mu_{g_{0}}(\alpha_{\omega_{g_{0}}}((v_{i}, s_{i}),  c_{i}))
\end{align*}
As usual morphisms work similarly. Therefore $\mu_{R_{[\mathrm{SOA}]}}$ is a natural transformation. Now we must show that $\mu_{R_{[\mathrm{SOA}]}}$ and $\eta_{R_{[\mathrm{SOA}]}}$ satisfy the following equations:

\begin{itemize}
\item $\mu_{R_{[\mathrm{SOA}]}} \circ (\eta_{R_{[\mathrm{SOA}]}} \cdot R_{[\mathrm{SOA}]}) =  \mu_{R_{[\mathrm{SOA}]}} \circ (R_{[\mathrm{SOA}]} \cdot \eta_{R_{[\mathrm{SOA}]}}) = \Id_{R_{[\mathrm{SOA}]}} $

\item $\mu_{R_{[\mathrm{SOA}]}} \circ (R_{[\mathrm{SOA}]} \cdot \mu_{R_{[\mathrm{SOA}]}}) = \mu_{R_{[\mathrm{SOA}]}} \circ ( \mu_{R_{[\mathrm{SOA}]}} \cdot R_{[\mathrm{SOA}]} )$\\
\end{itemize}
which is tantamount to proving the following for every $g_{0}$:
\begin{itemize}
\item $\mu_{g_{0}} \circ \overline{\omega}_{g_{0}} =  \mu_{g_{0}} \circ [\mathrm{SOA}](\overline{g}_{0},\id_{Y_{0}}) = \id_{\tilde{W}_{g_{0}}} $

\item $\mu_{g_{0}} \circ [\mathrm{SOA}](\mu_{g_{0}},\id_{Y_{0}}) = \mu_{g_{0}} \circ \mu_{\omega_{g_{0}}}$\\

\end{itemize}

We already have $\mu_{g_{0}} \circ \overline{\omega}_{g_{0}}  = \id_{\tilde{W}_{g_{0}}} $ by construction of $\mu_{g_{0}}$.
We also have $\id_{\tilde{W}_{g_{0}}} =  \mu_{g_{0}} \circ [\mathrm{SOA}](\overline{g}_{0},\id_{Y_{0}})$ since $\overline{\omega}_{g_{0}} \circ \overline{g}_{0} = [\mathrm{SOA}](\overline{g}_{0},\id_{Y_{0}}) \circ \overline{g}_{0}$ by the functorial factorization of $[\mathrm{SOA}]$ and thus $$\overline{g}_{0} = \mu_{g_{0}} \circ \overline{\omega}_{g_{0}} \circ \overline{g}_{0} = \mu_{g_{0}} \circ [\mathrm{SOA}](\overline{g}_{0},\id_{Y_{0}}) \circ \overline{g}_{0}$$ The rest follows from induction on $c_{i} = \mu_{g_{0}} \circ [\mathrm{SOA}](\overline{g}_{0},\id_{Y_{0}}) \circ c_{i}$.\\

As for the second equality, for $w \in \tilde{W}_{\omega_{g_{0}}}$ we have $$[\mathrm{SOA}](\mu_{g_{0}},\id_{Y_{0}}) \circ \overline{\omega}_{\omega_{g_{0}}}(w) = \overline{\omega}_{g_{0}} \circ \mu_{g_{0}}(w)$$ by functorial factorization of $[\mathrm{SOA}]$ so we get $$\mu_{g_{0}} \circ [\mathrm{SOA}](\mu_{g_{0}},\id_{Y_{0}}) \circ \overline{\omega}_{\omega_{g_{0}}}(w) = \mu_{g_{0}} \circ \overline{\omega}_{g_{0}} \circ \mu_{g_{0}}(w) =  \mu_{g_{0}}(w) = \mu_{g_{0}} \circ \mu_{\omega_{g_{0}}}\circ \overline{\omega}_{\omega_{g_{0}}}(w)$$ The rest follows from induction on $\mu_{g_{0}} \circ [\mathrm{SOA}](\mu_{g_{0}},\id_{Y_{0}}) \circ c_{i} = \mu_{g_{0}} \circ \mu_{\omega_{g_{0}}} \circ c_{i}$. Therefore, $R_{[\mathrm{SOA}]} $ is a monad.

\end{proof}

\begin{defn}

A functor $g_{0}: X_{0} \to Y_{0}$ is an $R_{[\mathrm{SOA}]}$ algebra when there exists a square $(s_{0},t):\omega_{g_{0}} \to g_{0}$ such that $(s_{0},t) \circ \eta_{R_{[\mathrm{SOA}]},g_{0}} = \id_{g_{0}}$ and $(s_{0},t) \circ \mu_{R_{[\mathrm{SOA}]},g_{0}} = (s_{0},t) \circ R_{[\mathrm{SOA}]}(s_{0},t)$. Since $\eta_{R_{[\mathrm{SOA}]},g_{0}} = (\overline{g}_{0},\id_{Y_{0}})$, $t = \id_{Y_{0}}$. Thus we will present an $R_{[\mathrm{SOA}]}$ algebra as $(g_{0},s_{0})$.

\end{defn}

\begin{defn}

A morphism between two $R_{[\mathrm{SOA}]}$ algebras $(g_{0}:X_{0} \to Y_{0}, s_{0})$ and $(g_{1}:X_{1} \to Y_{1}, s_{1})$ is a square $(t,t'):g_{0} \to g_{1}$ such that $s_{1} \circ [\mathrm{SOA}](t,t') = t \circ s_{0}$.

\end{defn}


Given two $R_{[\mathrm{SOA}]}$ algebras $(g_{0}:X_{0} \to Y_{0}, s_{0})$ and $(h_{0}:Y_{0} \to Z_{0}, s'_{0})$ we construct a new $R_{[\mathrm{SOA}]}$ algebra $(h_{0} \circ g_{0}, s'_{0} \cdot s_{0})$ as follows:\\

We seek to construct a functor $\trans(s_{0},s'_{0}): \tilde{W}_{h_{0} \circ g_{0}} \to \tilde{W}_{g_{0}}$ making the following 2 diagrams commute:

\[\begin{tikzcd}
	{X_{0}} &&& {\tilde{W}_{h_{0} \circ g_{0}}} &&& {\tilde{W}_{g_{0}}} &&& {X_{0}} \\
	\\
	\\
	{Z_{0}} &&& {Z_{0}} &&& {Z_{0}} &&& {Z_{0}}
	\arrow["{\overline{h_{0} \circ g_{0}}}", from=1-1, to=1-4]
	\arrow["{h_{0} \circ g_{0}}"', from=1-1, to=4-1]
	\arrow["{\trans(s_{0},s'_{0})}", dashed, from=1-4, to=1-7]
	\arrow["{{{\omega_{h_{0} \circ g_{0}}}}}", from=1-4, to=4-4]
	\arrow["{s_{0}}", from=1-7, to=1-10]
	\arrow["{h_{0} \circ \omega_{g_{0}}}", from=1-7, to=4-7]
	\arrow["{h_{0} \circ g_{0}}", from=1-10, to=4-10]
	\arrow["{{{\id_{Z_{0}}}}}"', from=4-1, to=4-4]
	\arrow["{{{\id_{Z_{0}}}}}"', from=4-4, to=4-7]
	\arrow["{{{\id_{Z_{0}}}}}"', from=4-7, to=4-10]
\end{tikzcd}\]

where the leftmost square is $\eta_{R_{[\mathrm{SOA}]},h_{0} \circ g_{0}}$ and:

\[\begin{tikzcd}
	{\tilde{W}_{h_{0} \circ g_{0}}} &&& {\tilde{W}_{g_{0}}} \\
	\\
	\\
	{\tilde{W}_{h_{0}}} &&& {Y_{0}}
	\arrow["{\trans(s_{0},s'_{0})}", from=1-1, to=1-4]
	\arrow["{{{[\mathrm{SOA}](g_{0},\id_{Z_{0}})}}}"', from=1-1, to=4-1]
	\arrow["{{{\omega_{g_{0}}}}}", from=1-4, to=4-4]
	\arrow["{{s'_{0}}}"', from=4-1, to=4-4]
\end{tikzcd}\]

where the $[\mathrm{SOA}](g_{0},\id_{Z_{0}})$ comes from  applying $R_{[\mathrm{SOA}]}$ to the square $(g_{0},\id_{Z_{0}}): h_{0} \circ g_{0} \to h_{0}$. 
As usual we construct $\trans(s_{0},s'_{0})$ by induction on $\tilde{W}_{h_{0} \circ g_{0}} $. For $x \in X_{0}$, $\trans(s_{0},s'_{0})(\overline{h_{0} \circ g_{0}}(x)) = \overline{g_{0}}(x)$. Furthermore we have $$s'_{0} \circ [\mathrm{SOA}](g_{0},\id_{Z_{0}})(\overline{h_{0} \circ g_{0}}(x)) = s'_{0} \circ \overline{h_{0}}\circ g_{0}(x) = g_{0}(x) = \omega_{g_{0}} \circ \overline{g_{0}}(x) = \omega_{g_{0}}(\trans(s_{0},s'_{0})(\overline{h_{0} \circ g_{0}}(x)))$$

Given some $\alpha_{h_{0} \circ g_{0}}((v_{i},s_{i}: V_{i} \to Z_{0}),c_{i}: U_{i} \to \tilde{W}_{h_{0} \circ g_{0}})$ we have the following diagram:

\[\begin{tikzcd}
	{U_{i}} &&& {\tilde{W}_{h_{0} \circ g_{0}}} &&& {\tilde{W}_{h_{0}}} \\
	\\
	\\
	{V_{i}} &&& {Z_{0}} &&& {Z_{0}}
	\arrow["{{c_{i}}}", from=1-1, to=1-4]
	\arrow["{{J_{i}}}"', from=1-1, to=4-1]
	\arrow["{{[\mathrm{SOA}](g_{0},\id_{Z_{0}})}}", from=1-4, to=1-7]
	\arrow["{{{{\omega_{h_{0} \circ g_{0}}}}}}", from=1-4, to=4-4]
	\arrow["{{\omega_{h_{0}}}}", from=1-7, to=4-7]
	\arrow["{{s_{i}}}"', from=4-1, to=4-4]
	\arrow["{{{{\id_{Z_{0}}}}}}"', from=4-4, to=4-7]
\end{tikzcd}\]

This gives us a lift of the diagram $\lambda(J, \omega_{h_{0}})([\mathrm{SOA}](g_{0},\id_{Z_{0}}) \circ c_{i},s_{i}):V_{i} \to \tilde{W}_{h_{0}}$. This give us a morphism $s'_{1} \circ \lambda(J, \omega_{h_{0}})([\mathrm{SOA}](g_{0},\id_{Z_{0}}) \circ c_{i},s_{i}):V_{i} \to Y_{0}$. Assume by induction that $$\omega_{g_{0}} \circ \trans(s'_{0},s_{0}) \circ c_{i} = s'_{0} \circ [\mathrm{SOA}](g_{0},\id_{Z_{0}}) \circ c_{i} = s'_{0} \circ \lambda(J, \omega_{h_{0}})([\mathrm{SOA}](g_{0},\id_{Z_{0}}) \circ c_{i},s_{i}) \circ J$$ so that $$\trans(s_{0},s'_{0})(\alpha_{h_{0} \circ g_{0}}((v_{i},s_{i}),c_{i})) = \alpha_{g_{0}}((v_{i}, s'_{0} \circ \lambda(J, \omega_{h_{0}})([\mathrm{SOA}](g_{0},\id_{Z_{0}}) \circ c_{i},s_{i})), \trans(s'_{0},s_{0}) \circ c_{i})$$

For a morphism $\alpha_{h_{0} \circ g_{0}}((m_{k},\sigma_{k}),\gamma_{k})$ we set:
$$\trans(s_{0},s'_{0})(\alpha_{h_{0} \circ g_{0}}((m_{k},\sigma_{k}),\gamma_{k})) = \alpha_{g_{0}}((m_{k}, s'_{0} \cdot \lambda(J, \omega_{h_{0}})([\mathrm{SOA}](g_{0},\id_{Z_{0}}) \cdot \gamma_{k},\sigma_{k})), \trans(s'_{0},s_{0}) \cdot \gamma_{k})$$ assuming by induction $$(\omega_{g_{0}} \circ \trans(s'_{0},s_{0})) \cdot \gamma_{k} = (s'_{0} \circ [\mathrm{SOA}](g_{0},\id_{Z_{0}})) \cdot \gamma_{k}  = s'_{0} \cdot \lambda(J, \omega_{h_{0}})([\mathrm{SOA}](g_{0},\id_{Z_{0}}) \cdot \gamma_{k},\sigma_{k}) \cdot J$$ This construction is functorial by functoriality of $ \lambda(J, \omega_{h_{0}})$. Now we must show that $\trans(s_{0},s'_{0})$ is a morphism over $Z_{0}$ as well as showing that $\trans(s_{0},s'_{0})$ respects the reduction equations and that the second commutative square commutes. \\

For $x \in X_{0}$, $$h_{0} \circ \omega_{g_{0}} \circ \trans(s_{0},s'_{0})(\overline{h_{0} \circ g_{0}}(x)) = h_{0} \circ \omega_{g_{0}} \circ \overline{g_{0}}(x) = h_{0} \circ g_{0}(x) = \omega_{h_{0} \circ g_{0}} \circ \overline{h_{0} \circ g_{0}}(x)$$ Given an object $\alpha_{h_{0} \circ g_{0}}((v_{i},s_{i}),c_{i}))$, we have 

\begin{align*}
h_{0} &\circ \omega_{g_{0}} \circ \trans(s_{0},s'_{0})(\alpha_{h_{0} \circ g_{0}}((v_{i},s_{i}),c_{i})) \\
&= h_{0} \circ \omega_{g_{0}}(\alpha_{g_{0}}((v_{i}, s'_{0} \circ \lambda(J, \omega_{h_{0}})([\mathrm{SOA}](g_{0},\id_{Z_{0}}) \circ c_{i},s_{i})), \trans(s'_{0},s_{0}) \circ c_{i}))\\
&= h_{0}(s'_{0} \circ \lambda(J, \omega_{h_{0}})([\mathrm{SOA}](g_{0},\id_{Z_{0}}) \circ c_{i},s_{i})(v_{i}))\\
&= h_{0}(s'_{0}(\alpha_{h_{0}}((v_{i},s_{i}),[\mathrm{SOA}](g_{0},\id_{Z_{0}}) \circ c_{i})))\\
&= \omega_{h_{0}}(\alpha_{h_{0}}((v_{i},s_{i}),[\mathrm{SOA}](g_{0},\id_{Z_{0}}) \circ c_{i}))\\
&= s_{i}(v_{i})\\
&= \omega_{h_{0} \circ g_{0}}(\alpha_{h_{0} \circ g_{0}}((v_{i},s_{i}),c_{i}))
\end{align*}
A similar argument holds for morphisms thus $\trans(s_{0},s'_{0})$ is a morphism over $Z_{0}$.\\

Now we show that $\trans(s_{0},s'_{0})$ respects the reductions. Suppose we have $\alpha_{h_{0} \circ g_{0}}((J_{i}(u_{i}),s_{i}),c_{i})$, then 

\begin{align*}
\trans(s_{0},s'_{0})(\alpha_{h_{0} \circ g_{0}}((J_{i}(u_{i}),s_{i}),c_{i})) &= \alpha_{g_{0}}((J_{i}(u_{i}), s'_{0} \circ \lambda(J, \omega_{h_{0}})([\mathrm{SOA}](g_{0},\id_{Z_{0}}) \circ c_{i},s_{i})), \trans(s'_{0},s_{0}) \circ c_{i}))\\
&= \trans(s'_{0},s_{0})(c_{i}(u_{i}))
\end{align*}

When $n_{i} \in N_{i}$ and $s_{i}$ is constant, then $\lambda(J, \omega_{h_{0}})([\mathrm{SOA}](g_{0},\id_{Z_{0}}) \circ c_{i},s_{i})$ is constant by Lemma \ref{liftI} and so is $s'_{0} \circ \lambda(J, \omega_{h_{0}})([\mathrm{SOA}](g_{0},\id_{Z_{0}}) \circ c_{i},s_{i})$. So $\trans(s_{0},s'_{0})$ will respect the reduction. Since the argument for morphisms is similar to that of objects, we have that $\trans(s_{0},s'_{0})$ respects the reductions.\\

Now we seek to show that the following square is commutative:

\[\begin{tikzcd}
	{\tilde{W}_{h_{0} \circ g_{0}}} &&& {\tilde{W}_{g_{0}}} \\
	\\
	\\
	{\tilde{W}_{h_{0}}} &&& {Y_{0}}
	\arrow["{\trans(s_{0},s'_{0})}", from=1-1, to=1-4]
	\arrow["{{{[\mathrm{SOA}](g_{0},\id_{Z_{0}})}}}"', from=1-1, to=4-1]
	\arrow["{{{\omega_{g_{0}}}}}", from=1-4, to=4-4]
	\arrow["{{s'_{0}}}"', from=4-1, to=4-4]
\end{tikzcd}\]

We have already demonstrated the cases for $x \in X_{0}$,it suffices to consider $\alpha_{h_{0} \circ g_{0}}((v_{i},s_{i}),c_{i})$.

We have 

\begin{align*}
\omega_{g_{0}} \circ \trans(s_{0},s'_{0})(\alpha_{h_{0} \circ g_{0}}((v_{i},s_{i}),c_{i})) &= \omega_{g_{0}}(\alpha_{g_{0}}((v_{i}, s'_{0} \circ \lambda(J, \omega_{h_{0}})([\mathrm{SOA}](g_{0},\id_{Z_{0}}) \circ c_{i},s_{i})), \trans(s'_{0},s_{0}) \circ c_{i}))\\
&= s'_{0} \circ \lambda(J, \omega_{h_{0}})([\mathrm{SOA}](g_{0},\id_{Z_{0}}) \circ c_{i},s_{i})(v_{i})\\
&= s'_{0}(\alpha_{h_{0}}((v_{i},s_{i}), [\mathrm{SOA}](g_{0},\id_{Z_{0}}) \circ c_{i}, ))\\
&= s'_{0} \circ [\mathrm{SOA}](g_{0},\id_{Z_{0}})(\alpha_{h_{0} \circ g_{0}}((v_{i},s_{i}),c_{i}))
\end{align*}
Therefore, the above square commutes. We construct our potential $R_{[\mathrm{SOA}]}$ algebra structure on $h_{0} \circ g_{0}$ as $(h_{0} \circ g_{0}, s_{0} \circ \trans(s_{0},s'_{0}))$. Now we prove that $(h_{0} \circ g_{0}, s_{0} \circ \trans(s_{0},s'_{0}))$ is an $R_{[\mathrm{SOA}]}$ algebra:

\begin{lem}

$(h_{0} \circ g_{0}, s_{0} \circ \trans(s_{0},s'_{0}))$ is an $R_{[\mathrm{SOA}]}$ algebra.

\end{lem}
\begin{proof}

We already have that $(s_{0} \circ \trans(s_{0},s'_{0}), \id_{Z_{0}}) \circ \eta_{R_{[\mathrm{SOA}]}, h_{0} \circ g_{0}} = \id_{h_{0} \circ g_{0}}$ by the following commutative diagram:
\[\begin{tikzcd}
	{X_{0}} &&& {\tilde{W}_{h_{0} \circ g_{0}}} &&& {\tilde{W}_{g_{0}}} &&& {X_{0}} \\
	\\
	\\
	{Z_{0}} &&& {Z_{0}} &&& {Z_{0}} &&& {Z_{0}}
	\arrow["{\overline{h_{0} \circ g_{0}}}", from=1-1, to=1-4]
	\arrow["{h_{0} \circ g_{0}}"', from=1-1, to=4-1]
	\arrow["{\trans(s_{0},s'_{0})}", from=1-4, to=1-7]
	\arrow["{{{\omega_{h_{0} \circ g_{0}}}}}", from=1-4, to=4-4]
	\arrow["{s_{0}}", from=1-7, to=1-10]
	\arrow["{h_{0} \circ \omega_{g_{0}}}", from=1-7, to=4-7]
	\arrow["{h_{0} \circ g_{0}}", from=1-10, to=4-10]
	\arrow["{{{\id_{Z_{0}}}}}"', from=4-1, to=4-4]
	\arrow["{{{\id_{Z_{0}}}}}"', from=4-4, to=4-7]
	\arrow["{{{\id_{Z_{0}}}}}"', from=4-7, to=4-10]
\end{tikzcd}\]

and by $$s_{0} \circ \trans(s_{0},s'_{0}) \circ \overline{h_{0} \circ g_{0}} = s_{0} \circ \overline{g_{0}} =  \id_{X_{0}}$$
now we need to prove $$(s_{0} \circ \trans(s_{0},s'_{0}), \id_{Z_{0}}) \circ \mu_{R_{[\mathrm{SOA}]}, h_{0} \circ g_{0}}  = (s_{0} \circ \trans(s_{0},s'_{0}), \id_{Z_{0}}) \circ R_{[\mathrm{SOA}]}(s_{0} \circ \trans(s_{0},s'_{0}), \id_{Z_{0}})  $$ which means proving: $$s_{0} \circ \trans(s_{0},s'_{0}) \circ \mu_{h_{0} \circ g_{0}} = s_{0} \circ \trans(s_{0},s'_{0}) \circ [\mathrm{SOA}](s_{0} \circ \trans(s_{0},s'_{0}), \id_{Z_{0}})$$ 

 We wish the construct a functor $\trans_{\mu}(s_{0},s'_{0}): \tilde{W}_{\omega_{h_{0} \circ g_{0}}} \to \tilde{W}_{\omega_{g_{0}}}$ so that $$\trans(s_{0},s'_{0}) \circ \mu_{h_{0} \circ g_{0}} = \mu_{g_{0}} \circ \trans_{\mu}(s_{0},s'_{0})$$ and $$[\mathrm{SOA}](s_{0},\id_{Y_{0}}) \circ \trans_{\mu}(s_{0},s'_{0}) = \trans(s_{0},s'_{0}) \circ [\mathrm{SOA}](s_{0} \circ \trans(s_{0},s'_{0}) , \id_{Z_{0}})$$

We also wish to have the following diagram commute so that $\trans_{\mu}(s_{0},s'_{0})$  is well defined:

\[\begin{tikzcd}
	{\tilde{W}_{\omega_{h_{0} \circ g_{0}}}} &&& {\tilde{W}_{\omega_{g_{0}}}} \\
	\\
	\\
	{\tilde{W}_{h_{0}}} &&& {Y_{0}}
	\arrow["{\trans_{\mu}(s_{0},s'_{0})}", from=1-1, to=1-4]
	\arrow["{{{[\mathrm{SOA}](s'_{0} \circ [\mathrm{SOA}](g_{0},\id_{Z_{0}}),\id_{Z_{0}})}}}"', from=1-1, to=4-1]
	\arrow["{{{\omega_{\omega_{g_{0}}}}}}", from=1-4, to=4-4]
	\arrow["{{s'_{0}}}"', from=4-1, to=4-4]
\end{tikzcd}\]

First for $w \in \tilde{W}_{h_{0} \circ g_{0}}$, $\trans_{\mu}(s_{0},s'_{0})(\overline{\omega_{h_{0} \circ g_{0}}}(w)) = \overline{\omega_{ g_{0}}}(\trans(s_{0},s'_{0})(w))$. We have

\begin{align*}
\omega_{\omega_{g_{0}}} \circ \trans_{\mu}(s_{0},s'_{0})(\overline{\omega_{h_{0} \circ g_{0}}}(w)) &= \omega_{\omega_{g_{0}}} \circ \overline{\omega_{ g_{0}}}(\trans(s_{0},s'_{0})(w))\\
&= \omega_{g_{0}}(\trans(s_{0},s'_{0})(w))\\
&= s'_{0} \circ [\mathrm{SOA}](g_{0},\id_{Z_{0}})(w)\\
&= s'_{0} \circ [\mathrm{SOA}](g_{0},\id_{Z_{0}})\circ \mu_{h_{0} \circ g_{0}} \circ \overline{\omega_{h_{0} \circ g_{0}}}(w)\\
&= s'_{0} \circ \mu_{h_{0}}  \circ [\mathrm{SOA}]([\mathrm{SOA}](g_{0},\id_{Z_{0}}),\id_{Z_{0}}) \circ \overline{\omega_{h_{0} \circ g_{0}}}(w)\\
&= s'_{0} \circ [\mathrm{SOA}](s'_{0},\id_{Z_{0}})  \circ [\mathrm{SOA}]([\mathrm{SOA}](g_{0},\id_{Z_{0}}),\id_{Z_{0}}) \circ \overline{\omega_{h_{0} \circ g_{0}}}(w)\\
&= s'_{0} \circ  [\mathrm{SOA}](s'_{0} \circ [\mathrm{SOA}](g_{0},\id_{Z_{0}}),\id_{Z_{0}}) \circ \overline{\omega_{h_{0} \circ g_{0}}}(w)\\
\end{align*}

Given $\alpha_{\omega_{h_{0} \circ g_{0}}}((v_{i},s_{i}:V_{i} \to Z_{0}),c_{i}: U_{i} \to \tilde{W}_{\omega_{h_{0} \circ g_{0}}})$ we take note of the following commutative diagram:

\[\begin{tikzcd}
	{U_{i}} &&& {\tilde{W}_{\omega_{h_{0} \circ g_{0}}}} &&&& {\tilde{W}_{\omega_{h_{0}}}} &&& {\tilde{W}_{h_{0}}} &&& {Y_{0}} \\
	\\
	\\
	{V_{i}} &&& {Z_{0}} &&&& {Z_{0}} &&& {Z_{0}} &&& {Z_{0}}
	\arrow["{c_{i}}", from=1-1, to=1-4]
	\arrow["{J_{i}}"', from=1-1, to=4-1]
	\arrow["{[\mathrm{SOA}]([\mathrm{SOA}](g_{0},\id_{Z_{0}}),\id_{Z_{0}})}", from=1-4, to=1-8]
	\arrow["{\omega_{\omega_{h_{0} \circ g_{0}}}}"{description}, from=1-4, to=4-4]
	\arrow["{[\mathrm{SOA}](s'_{0},\id_{Z_{0}})}", from=1-8, to=1-11]
	\arrow["{\omega_{\omega_{h_{0}}}}"{description}, from=1-8, to=4-8]
	\arrow["{s'_{0}}", from=1-11, to=1-14]
	\arrow["{\omega_{h_{0}}}"{description}, from=1-11, to=4-11]
	\arrow["{h_{0}}", from=1-14, to=4-14]
	\arrow["{s_{i}}"', from=4-1, to=4-4]
	\arrow["{\id_{Z_{0}}}"', from=4-4, to=4-8]
	\arrow["{\id_{Z_{0}}}"', from=4-8, to=4-11]
	\arrow["{\id_{Z_{0}}}"', from=4-11, to=4-14]
\end{tikzcd}\]

where $[\mathrm{SOA}]([\mathrm{SOA}](g_{0},\id_{Z_{0}}),\id_{Z_{0}}))$ is the result of applying $R_{[\mathrm{SOA}]}$ to the diagram:

\[\begin{tikzcd}
	{\tilde{W}_{h_{0} \circ g_{0}}} &&& {\tilde{W}_{h_{0}}} \\
	\\
	\\
	{Z_{0}} &&& {Z_{0}}
	\arrow["{{{[\mathrm{SOA}](g_{0},\id_{Z_{0}})}}}", from=1-1, to=1-4]
	\arrow["{{{{{\omega_{h_{0} \circ g_{0}}}}}}}"', from=1-1, to=4-1]
	\arrow["{{{\omega_{h_{0}}}}}", from=1-4, to=4-4]
	\arrow["{{{{{\id_{Z_{0}}}}}}}"', from=4-1, to=4-4]
\end{tikzcd}\]

We have a lift $$\lambda(J,\omega_{h_{0}})([\mathrm{SOA}](s'_{0},\id_{Z_{0}}) \circ [\mathrm{SOA}]( [\mathrm{SOA}](g_{0},\id_{Z_{0}}),\id_{Z_{0}}) \circ c_{i},s_{i}): V_{i} \to \tilde{W}_{h_{0}}$$

and thus a functor:
$$l_{c_{i},s_{i}} = s'_{0} \circ \lambda(J,\omega_{h_{0}})([\mathrm{SOA}](s'_{0},\id_{Z_{0}}) \circ [\mathrm{SOA}]( [\mathrm{SOA}](g_{0},\id_{Z_{0}}),\id_{Z_{0}}) \circ c_{i},s_{i}): V_{i} \to Y_{0}$$

Assuming by induction $$\omega_{\omega_{g_{0}}} \circ \trans_{\mu}(s_{0},s'_{0}) \circ c_{i} =  s'_{0} \circ [\mathrm{SOA}](s'_{0} \circ [\mathrm{SOA}](g_{0},\id_{Z_{0}}),\id_{Z_{0}}) \circ c_{i} = l_{c_{i}, s_{i}} \circ J$$ we define:

$$\trans_{\mu}(s_{0},s'_{0})(\alpha_{\omega_{h_{0} \circ g_{0}}}((v_{i},s_{i}),c_{i})) = \alpha_{\omega_{g_{0}}}((v_{i}, l_{c_{i},s_{i}}),\trans_{\mu}(s_{0},s'_{0})\circ c_{i} )$$

And for a morphism $\omega_{h_{0} \circ g_{0}}((m_{k},\sigma_{k}),\gamma_{k})$, we define 

$$\trans_{\mu}(s_{0},s'_{0})(\alpha_{\omega_{h_{0} \circ g_{0}}}((m_{k},\sigma_{k}),\gamma_{k})) = \alpha_{\omega_{g_{0}}}((m_{k}, l_{\gamma_{k},\sigma_{k}}),\trans_{\mu}(s_{0},s'_{0})\cdot \gamma_{k} )$$ where

$$l_{\gamma_{k},\sigma_{k}} = s'_{0} \cdot \lambda(J,\omega_{h_{0}})(([\mathrm{SOA}](s'_{0},\id_{Z_{0}}) \circ [\mathrm{SOA}]( [\mathrm{SOA}](g_{0},\id_{Z_{0}}),\id_{Z_{0}})) \cdot \gamma_{k},\sigma_{k})$$

Showing that $\trans_{\mu}(s_{0},s'_{0})$ is a well defined functor is similar to the proof that $\trans(s_{0},s'_{0})$ is well defined. First we show that $\trans(s_{0},s'_{0}) \circ \mu_{h_{0} \circ g_{0}} = \mu_{g_{0}} \circ \trans_{\mu}(s_{0},s'_{0})$. For $w \in \tilde{W}_{h_{0} \circ g_{0}}$ we have 

\begin{align*}
\mu_{g_{0}} \circ \trans_{\mu}(s_{0},s'_{0})(\overline{\omega_{h_{0} \circ g_{0}}}(w)) &= \mu_{g_{0}} \circ \overline{\omega_{ g_{0}}}(\trans(s_{0},s'_{0})(w))\\
&= \trans(s_{0},s'_{0})(w)\\
&= \trans(s_{0},s'_{0})\circ \mu_{h_{0} \circ g_{0}}(\overline{\omega_{h_{0} \circ g_{0}}}(w)) \\
\end{align*}

Given $\alpha_{\omega_{h_{0} \circ g_{0}}}((v_{i},s_{i}),c_{i})$, assuming $\trans(s_{0},s'_{0}) \circ \mu_{h_{0} \circ g_{0}} \circ c_{i} = \mu_{g_{0}} \circ \trans_{\mu}(s_{0},s'_{0}) \circ c_{i}$

\begin{align*}
\mu_{g_{0}} \circ \trans_{\mu}(s_{0},s'_{0})(\alpha_{\omega_{h_{0} \circ g_{0}}}((v_{i},s_{i}),c_{i})) &= \mu_{g_{0}}(\alpha_{\omega_{g_{0}}}((v_{i}, l_{c_{i},s_{i}}),\trans_{\mu}(s_{0},s'_{0})\circ c_{i} )) \\
&= \alpha_{\omega_{\omega_{g_{0}}}}((v_{i}, l_{c_{i},s_{i}}),\mu_{g_{0}} \circ \trans_{\mu}(s_{0},s'_{0})\circ c_{i} )\\
&= \alpha_{\omega_{\omega_{g_{0}}}}((v_{i}, l_{c_{i},s_{i}}),\trans(s_{0},s'_{0}) \circ \mu_{h_{0} \circ g_{0}} \circ c_{i})\\
\end{align*}

We can break down $l_{c_{i},s_{i}}$ using the fact that $\lambda(J,\omega_{-})$ is a natural transformation as follows:\\
\begin{align*}
l_{c_{i},s_{i}} &= s'_{0} \circ \lambda(J,\omega_{h_{0}})( [\mathrm{SOA}](s'_{0},\id_{Z_{0}}) \circ  [\mathrm{SOA}]( [\mathrm{SOA}](g_{0},\id_{Z_{0}}),\id_{Z_{0}}) \circ c_{i},s_{i})\\
&= s'_{0} \circ [\mathrm{SOA}](s'_{0},\id_{Z_{0}}) \circ  \lambda(J,\omega_{\omega_{h_{0}}})([\mathrm{SOA}]( [\mathrm{SOA}](g_{0},\id_{Z_{0}}),\id_{Z_{0}}) \circ c_{i},s_{i})\\
&= s'_{0} \circ \mu_{h_{0}} \circ  \lambda(J,\omega_{\omega_{h_{0}}})([\mathrm{SOA}]( [\mathrm{SOA}](g_{0},\id_{Z_{0}}),\id_{Z_{0}}) \circ c_{i},s_{i})\\
&= s'_{0} \circ \lambda(J,\omega_{h_{0}})(\mu_{h_{0}} \circ  [\mathrm{SOA}]( [\mathrm{SOA}](g_{0},\id_{Z_{0}}),\id_{Z_{0}}) \circ c_{i},s_{i})\\
&= s'_{0} \circ \lambda(J,\omega_{h_{0}})([\mathrm{SOA}](g_{0},\id_{Z_{0}}) \circ \mu_{h_{0} \circ g_{0}} \circ c_{i},s_{i})\\
\end{align*}

Thus we continue the chain of equalities:
\begin{align*}
\mu_{g_{0}} \circ \trans_{\mu}(s_{0},s'_{0})(\alpha_{\omega_{h_{0} \circ g_{0}}}((v_{i},s_{i}),c_{i})) &= \alpha_{\omega_{\omega_{g_{0}}}}((v_{i}, l_{c_{i},s_{i}}),\trans(s_{0},s'_{0}) \circ \mu_{h_{0} \circ g_{0}} \circ c_{i})\\
&= \trans(s_{0},s'_{0}) (\alpha_{\omega_{\omega_{h_{0} \circ g_{0}}}}((v_{i}, s_{i}), \mu_{h_{0} \circ g_{0}} \circ c_{i}))\\
&= \trans(s_{0},s'_{0}) \circ  \mu_{h_{0} \circ g_{0}}(\alpha_{\omega_{h_{0} \circ g_{0}}}((v_{i},s_{i}),c_{i}))
\end{align*}
Of course we can show a similar thing for morphisms. Thus $\trans(s_{0},s'_{0}) \circ \mu_{h_{0} \circ g_{0}} = \mu_{g_{0}} \circ \trans_{\mu}(s_{0},s'_{0})$. Now we show that $ [\mathrm{SOA}](s_{0},\id_{Y_{0}}) \circ \trans_{\mu}(s_{0},s'_{0}) =  \trans(s_{0},s'_{0}) \circ [\mathrm{SOA}](s_{0} \circ \trans(s_{0},s'_{0}) , \id_{Z_{0}})$. For $w \in \tilde{W}_{h_{0} \circ g_{0}}$ we have:

\begin{align*}
 \trans(s_{0},s'_{0}) \circ [\mathrm{SOA}](s_{0} \circ \trans(s_{0},s'_{0}) , \id_{Z_{0}})(\overline{\omega_{h_{0} \circ g_{0}}}(w)) &= \trans(s_{0},s'_{0}) \circ \overline{h_{0} \circ g_{0}} \circ s_{0} \circ \trans(s_{0},s'_{0})(w)\\
&= \overline{g_{0}} \circ s_{0} \circ \trans(s_{0},s'_{0})(w)\\
&=  [\mathrm{SOA}](s_{0},\id_{Y_{0}}) \circ \overline{\omega_{g_{0}}} \circ \trans(s_{0},s'_{0})(w)\\
&=  [\mathrm{SOA}](s_{0},\id_{Y_{0}}) \circ \trans_{\mu}(s_{0},s'_{0})(\overline{\omega_{h_{0} \circ g_{0}}}(w))\\
\end{align*}

Suppose we have $\alpha_{\omega_{h_{0} \circ g_{0}}}((v_{i},s_{i}),c_{i})$ and assume by induction $$[\mathrm{SOA}](s_{0},\id_{Y_{0}}) \circ \trans_{\mu}(s_{0},s'_{0}) \circ c_{i} =  \trans(s_{0},s'_{0}) \circ [\mathrm{SOA}](s_{0} \circ \trans(s_{0},s'_{0}) , \id_{Z_{0}}) \circ c_{i}$$ first we take note of the following chain of equalities:

\begin{align*}
g_{0} \circ s_{0} \circ \trans(s_{0},s'_{0}) &= \omega_{g_{0}} \circ \trans(s_{0},s'_{0}) \\
&= \omega_{\omega_{g_{0}}} \circ \overline{\omega_{g_{0}}} \circ \trans(s_{0},s'_{0}) \\
&= \omega_{\omega_{g_{0}}} \circ \trans_{\mu}(s_{0},s'_{0})\circ \overline{\omega_{h_{0} \circ g_{0}}}\\
&= s'_{0} \circ [\mathrm{SOA}](s'_{0} \circ [\mathrm{SOA}](g_{0},\id_{Z_{0}}),\id_{Z_{0}}) \circ \overline{\omega_{h_{0} \circ g_{0}}}\\
&= s'_{0} \circ \overline{h_{0}} \circ s'_{0} \circ [\mathrm{SOA}](g_{0},\id_{Z_{0}})\\
&= s'_{0} \circ [\mathrm{SOA}](g_{0},\id_{Z_{0}})\\
\end{align*}

Thus we have $$l_{c_{i},s_{i}} = s'_{0} \circ \lambda(J,\omega_{h_{0}})( [\mathrm{SOA}](g_{0},\id_{Z_{0}}) \circ  [\mathrm{SOA}]( s_{0} \circ \trans(s_{0},s'_{0}) ,\id_{Z_{0}}) \circ c_{i},s_{i})$$

Now we can evaluate the following chain of equalities:
\begin{align*}
[\mathrm{SOA}](s_{0},\id_{Y_{0}}) \circ \trans_{\mu}(s_{0},s'_{0})(\alpha_{\omega_{h_{0} \circ g_{0}}}(&(v_{i},s_{i}),c_{i}))  = \alpha_{ g_{0} }((v_{i},l_{c_{i}, s_{i}}),[\mathrm{SOA}](s_{0},\id_{Y_{0}}) \circ \trans_{\mu}(s_{0},s'_{0}) \circ c_{i})\\
&=  \alpha_{ g_{0} }((v_{i},l_{c_{i}, s_{i}}),\trans(s_{0},s'_{0}) \circ [\mathrm{SOA}](s_{0} \circ \trans(s_{0},s'_{0}) , \id_{Z_{0}}) \circ c_{i})\\
&= \trans(s_{0},s'_{0}) \circ [\mathrm{SOA}](s_{0} \circ \trans(s_{0},s'_{0}) , \id_{Z_{0}})(\alpha_{\omega_{h_{0} \circ g_{0}}}((v_{i},s_{i}),c_{i}))\\
\end{align*}
Of course the argument is similar for morphisms. Thus $$[\mathrm{SOA}](s_{0},\id_{Y_{0}}) \circ \trans_{\mu}(s_{0},s'_{0})  =  \trans(s_{0},s'_{0}) \circ [\mathrm{SOA}](s_{0} \circ \trans(s_{0},s'_{0}) , \id_{Z_{0}})$$

Finally we have:
\begin{align*}
s_{0} \circ \trans(s_{0},s'_{0}) \circ \mu_{h_{0} \circ g_{0}} &= s_{0} \circ \mu_{g_{0}} \circ \trans_{\mu}(s_{0},s'_{0})\\
&= s_{0} \circ [\mathrm{SOA}](s_{0},\id_{Y_{0}})  \circ \trans_{\mu}(s_{0},s'_{0})\\
&= s_{0} \circ \trans(s_{0},s'_{0}) \circ [\mathrm{SOA}](s_{0} \circ \trans(s_{0},s'_{0}) , \id_{Z_{0}})\\
\end{align*}

Therefore, $(h_{0} \circ g_{0}, s_{0} \circ \trans(s_{0},s'_{0}))$ is an $R_{[\mathrm{SOA}]}$ algebra.

\end{proof}

\begin{lem}

If we have morphisms of $R_{[\mathrm{SOA}]}$ algebras: $(t,t'):(g_{0}, s_{0}) \to (g_{1},s_{1})$ and $(t',t''):(h_{0}, s'_{0}) \to (h_{1},s'_{1})$, then $(t,t''):(h_{0} \circ g_{0}, s_{0} \circ \trans(s_{0},s'_{0}) ) \to (h_{1} \circ g_{1}, s_{1} \circ \trans(s_{1},s'_{1}) )$ is an $R_{[\mathrm{SOA}]}$ algebra morphism. This shows that $R_{[\mathrm{SOA}]}$ algebras have a composition structure.

\end{lem}
\begin{proof}

We need to prove that $s_{1} \circ \trans(s_{1},s'_{1}) \circ [\mathrm{SOA}](t,t'') = t \circ s_{0} \circ \trans(s_{0},s'_{0})$. Using the $R_{[\mathrm{SOA}]}$ algebra morphism structure we have $t \circ s_{0} \circ \trans(s_{0},s'_{0}) = s_{1} \circ [\mathrm{SOA}](t,t') \circ \trans(s_{0},s'_{0})$. So it suffices to show that $\trans(s_{1},s'_{1}) \circ [\mathrm{SOA}](t,t'') = [\mathrm{SOA}](t,t') \circ \trans(s_{0},s'_{0})$. For $x \in X_{0}$, we have:
\begin{align*}
\trans(s_{1},s'_{1}) \circ [\mathrm{SOA}](t,t'') \circ \overline{h_{0} \circ g_{0}}(x) &= \trans(s_{1},s'_{1}) \circ \overline{h_{1} \circ g_{1}} \circ t(x)\\
&= \overline{g_{1}} \circ t(x)\\
&= [\mathrm{SOA}](t,t') \circ \overline{g_{0}}(x)\\
&= [\mathrm{SOA}](t,t') \circ \trans(s_{0},s'_{0}) \circ \overline{h_{0} \circ g_{0}}(x) 
\end{align*}

It suffices to check $\alpha_{h_{0} \circ g_{0}}((v_{i},s_{i}: V_{i} \to Z_{0}),c_{i}: U_{i} \to \tilde{W}_{h_{0} \circ g_{0}})$ assuming $\trans(s_{1},s'_{1}) \circ [\mathrm{SOA}](t,t'') \circ c_{i}= [\mathrm{SOA}](t,t') \circ \trans(s_{0},s'_{0}) \circ c_{i}$ by induction since the morphism case is similar.
\begin{align*}
\trans&(s_{1},s'_{1}) \circ [\mathrm{SOA}](t,t'')(\alpha_{h_{0} \circ g_{0}}((v_{i},s_{i}),c_{i})) = \trans(s_{1},s'_{1})(\alpha_{h_{1} \circ g_{1}}((v_{i},t'' \circ s_{i}),[\mathrm{SOA}](t,t'') \circ c_{i}))\\
&= \alpha_{g_{1}}((v_{i}, s'_{1} \circ \lambda(J,\omega_{h_{1}})([\mathrm{SOA}](g_{1} ,\id_{Z_{1}}) \circ [\mathrm{SOA}]( t,t'') \circ c_{i},t'' \circ s_{i}) ), \trans(s_{1},s'_{1}) \circ [\mathrm{SOA}](t,t'') \circ c_{i})\\
&= \alpha_{g_{1}}((v_{i}, s'_{1} \circ \lambda(J,\omega_{h_{1}})([\mathrm{SOA}](g_{1} \circ t,t'') \circ c_{i},t'' \circ s_{i}) ), \trans(s_{1},s'_{1}) \circ [\mathrm{SOA}](t,t'') \circ c_{i})\\
&= \alpha_{g_{1}}((v_{i}, s'_{1} \circ \lambda(J,\omega_{h_{1}})([\mathrm{SOA}](g_{1} \circ t,t'') \circ c_{i},t'' \circ s_{i}) ), [\mathrm{SOA}](t,t') \circ \trans(s_{0},s'_{0})  \circ c_{i})\\
&= \alpha_{g_{1}}((v_{i}, s'_{1} \circ \lambda(J,\omega_{h_{1}})([\mathrm{SOA}](t' \circ g_{0} ,t'') \circ c_{i},t'' \circ s_{i}) ), [\mathrm{SOA}](t,t') \circ \trans(s_{0},s'_{0})  \circ c_{i})\\
&= \alpha_{g_{1}}((v_{i}, s'_{1} \circ \lambda(J,\omega_{h_{1}})([\mathrm{SOA}](t' ,t'') \circ [\mathrm{SOA}]( g_{0} ,\id_{Z_{0}}) \circ c_{i},t'' \circ s_{i}) ), [\mathrm{SOA}](t,t') \circ \trans(s_{0},s'_{0})  \circ c_{i})\\
&= \alpha_{g_{1}}((v_{i}, s'_{1} \circ [\mathrm{SOA}](t' ,t'') \circ \lambda(J,\omega_{h_{0}})([\mathrm{SOA}](g_{0} ,\id_{Z_{0}}) \circ c_{i},s_{i}) ), [\mathrm{SOA}](t,t') \circ \trans(s_{0},s'_{0})  \circ c_{i})\\
&= \alpha_{g_{1}}((v_{i}, t' \circ s'_{0} \circ \lambda(J,\omega_{h_{0}})([\mathrm{SOA}](g_{0} ,\id_{Z_{0}}) \circ c_{i},s_{i}) ), [\mathrm{SOA}](t,t') \circ \trans(s_{0},s'_{0})  \circ c_{i})\\
&= [\mathrm{SOA}](t,t')(\alpha_{g_{0}}((v_{i}, s'_{0} \circ \lambda(J,\omega_{h_{0}})([\mathrm{SOA}](g_{0} ,\id_{Z_{0}}) \circ c_{i},s_{i}) ), \trans(s_{0},s'_{0})  \circ c_{i}))\\
&= [\mathrm{SOA}](t,t') \circ \trans(s_{0},s'_{0})(\alpha_{h_{0} \circ g_{0}}((v_{i},s_{i}),c_{i}))
\end{align*}

Therefore $\trans(s_{1},s'_{1}) \circ [\mathrm{SOA}](t,t'') = [\mathrm{SOA}](t,t') \circ \trans(s_{0},s'_{0})$ making $(t,t''):(h_{0} \circ g_{0}, s_{0} \circ \trans(s_{0},s'_{0}) ) \to (h_{1} \circ g_{1}, s_{1} \circ \trans(s_{1},s'_{1}) )$ is an $R_{[\mathrm{SOA}]}$ algebra morphism.

\end{proof}

\begin{cor}
\label{soa}

$([\mathrm{SOA}], R_{[\mathrm{SOA}]}, L_{[\mathrm{SOA}]})$ is an algebraic weak factorization system whose $R_{[\mathrm{SOA}]}$-maps are precisely those morphisms that satisfies the lifting property of Lemma \ref{liftI}.

\end{cor}

\subsection{Model Structure on $\GrpA{A}$}

Now we will construct the cofibrant algebraic weak factorization system. We define a morphism $J : U \to V$ on fibrations $U$ and $V$ over the assembly $\{0,1,2\}$ as follows: 

\[\begin{tikzcd}
	{U_{0} = 0} &&&& {V_{0} = 1} \\
	{U_{1} = 1} &&&& {V_{1} = \I} \\
	{U_{2} = 1 + 1} &&&& {V_{2} = \I}
	\arrow["{J_{0}}", from=1-1, to=1-5]
	\arrow["{J_{1} = \htpy{0}}", from=2-1, to=2-5]
	\arrow["{J_{2} = \htpy{0} +\htpy{1}}", from=3-1, to=3-5]
\end{tikzcd}\]

Note that any assembly can be treated as the groupoid assembly with only the identity morphisms. Any groupoid over an assembly is a fibration since assemblies only have identity morphisms. We also define a fibration $N$ over $\{0,1,2\}$ where $N_{0} = N_{2} = 0$ and $N_{1} = 1$. This defines a weak factorization system $([\mathrm{Cof}], R_{[\mathrm{Cof}]}, L_{[\mathrm{Cof}]})$ where any $R_{[\mathrm{Cof}]}$ map $g:X \to Y$ is precisely the morphism with the following lift functors:

\[\begin{tikzcd}
	{\Sq(J_{0}, g)} &&&& {1 \to X} \\
	{\Sq(J_{1}, g)} &&&& {\I \to X} \\
	{\Sq(J_{2}, g)} &&&& {\I \to X}
	\arrow["{\lambda(J_{0},g})", from=1-1, to=1-5]
	\arrow["{\lambda(J_{1},g})", from=2-1, to=2-5]
	\arrow["{\lambda(J_{2},g})", from=3-1, to=3-5]
\end{tikzcd}\]

where we treat each $\Sq(J_{n}, g)$ as the groupoid assemblies of square from $J_{n}$ to $g$. This leads us to the following lemma

\begin{lem}

$g$ is an $R_{[\mathrm{Cof}]}$ map if and only if $g$ is a fibration and a groupoid equivalence.

\end{lem}
\begin{proof}

We start with $g$ as a $R_{[\mathrm{Cof}]}$ map. Note that the lift $\lambda(J_{1},g)$ exhibits $g$ as a normal isofibration since the the property that identity morphisms lift to identity morphism is induced by the fibration $N$. Now we need to show that $g$ is a groupoid equivalence. Unpacking the lift $\lambda(J_{0},g)$, we see that it induces a functor $\lambda(J_{0},g): Y \to X$ since $$\Sq(J_{0}, g) = (0 \to X) \times_{(0 \to Y)} (1 \to Y) \cong 1 \times_{1} Y  \cong Y$$ and $(1 \to X) \cong X$. Furthermore we have that $\id_{Y} = g \circ \lambda(J_{0},g)$ since the functor $J_{0} \pupw g$ is just post composition by $g$. Now we must exhibit $\lambda(J_{0},g) \circ g \cong \id_{X}$. Given $x \in \ob X$, we have $g(x) = g(\lambda(J_{0},g) \circ g(x))$, so using $\lambda(J_{2},g)$ we exhibit the following isomorphism in $X$:

\[\begin{tikzcd}
	x &&&&& {\lambda(J_{0},g) \circ g(x)}
	\arrow["{{\lambda(J_{2},g)(\{x,\lambda(J_{0},g) \circ g(x)\},\id_{g(x)})}}", from=1-1, to=1-6]
\end{tikzcd}\]

we will call the above morphism $\alpha_{x}$. Given $\xi:x \to x' \in \ob X$, we have that $g(\xi) = g(\lambda(J_{0},g) \circ g(\xi))$ so we exhibit the commutative square $\lambda(J_{2},g)(\{\xi,\lambda(J_{0},g) \circ g(\xi)\},(g(\xi), g(\xi)):\id_{g(x)} \to \id_{g(x')} )$ in $X$:

\[\begin{tikzcd}
	x &&& {x'} \\
	\\
	{\lambda(J_{0},g) \circ g(x)} &&& {\lambda(J_{0},g) \circ g(x')}
	\arrow["\xi", from=1-1, to=1-4]
	\arrow["{\alpha_{x}}"', from=1-1, to=3-1]
	\arrow["{\alpha_{x'}}", from=1-4, to=3-4]
	\arrow["{\lambda(J_{0},g) \circ g(\xi)}"', from=3-1, to=3-4]
\end{tikzcd}\]

This gives us a natural isomorphism $\alpha: \lambda(J_{0},g) \circ g \cong \id_{X}$. Therefore $g$ is a groupoid equivalence.\\

Suppose $g$ is a groupoid equivalence and a fibration. By Lemma \ref{deformr}, $g$ is a strong deformation retract. As a reminder of what that is it means that $g$ has as data $g^{*}:Y \to X$, $\beta: g \circ g^{*} \to \id_{Y}$, and $\alpha: g^{*} \circ g \to \id_{X}$ that exhibit $g$ as a groupoid equivalence in addition with the following facts:

\begin{itemize}

\item $\beta \cdot g  = g \cdot \alpha$

\item $g \circ g^{*} = \id_{Y}$

\item $\beta_{y} = \id_{y} $

\end{itemize}

We will use this data to construct $\lambda(J_{0},g)$ and $\lambda(J_{2},g)$. For $\lambda(J_{0},g)$, as shown in the proof of the converse, it suffices to exhibit a section of $g$ which is $g^{*}$. Given the following commutative square:

\[\begin{tikzcd}
	{1 + 1} &&& {X} \\
	\\
	\\
	{\I} &&& {Y}
	\arrow["{\{x, x' \}}", from=1-1, to=1-4]
	\arrow["{J_{2}}"', from=1-1, to=4-1]
	\arrow["{g}", from=1-4, to=4-4]
	\arrow["{\nm{\id_{p}}}"', from=4-1, to=4-4]
\end{tikzcd}\]

we construct a lift:

\[\begin{tikzcd}
	x &&&&& {x'}
	\arrow["{\alpha_{x'} \circ g^{*}(p) \circ \alpha^{-1}_{x}}", from=1-1, to=1-6]
\end{tikzcd}\]

Observe the following:

$$g(\alpha_{x'} \circ g^{*}(p) \circ \alpha^{-1}_{x}) = g(\alpha_{x'}) \circ g(g^{*}(p)) \circ g(\alpha^{-1}_{x})  = \beta_{g(x')} \circ p \circ \beta^{-1}_{g(x)} = p$$

Given a square in $Y$:

\[\begin{tikzcd}
	{g(x)} &&& {g(x')} \\
	\\
	\\
	{g(e)} &&& {g(e')}
	\arrow["{p}", from=1-1, to=1-4]
	\arrow["{g(m)}"', from=1-1, to=4-1]
	\arrow["{g(m')}", from=1-4, to=4-4]
	\arrow["{p'}"', from=4-1, to=4-4]\\
\end{tikzcd}\]

we have the following commutative diagram in $X$:

\[\begin{tikzcd}
	x &&& {g^{*} \circ g(x)} &&& {g^{*} \circ g(x')} &&& {x'} \\
	\\
	\\
	e &&& {g^{*} \circ g(e)} &&& {g^{*} \circ g(e')} &&& {e'}
	\arrow["{\alpha^{-1}_{x}}", from=1-1, to=1-4]
	\arrow["m"{description}, from=1-1, to=4-1]
	\arrow["{{g^{*}(p)}}", from=1-4, to=1-7]
	\arrow["{{g^{*} \circ g(m)}}"{description}, from=1-4, to=4-4]
	\arrow["{\alpha_{x}}", from=1-7, to=1-10]
	\arrow["{{g^{*} \circ g(m')}}"{description}, from=1-7, to=4-7]
	\arrow["{m'}"{description}, from=1-10, to=4-10]
	\arrow["{\alpha^{-1}_{e}}"', from=4-1, to=4-4]
	\arrow["{{g^{*}(p')}}"', from=4-4, to=4-7]
	\arrow["{\alpha_{e'}}"', from=4-7, to=4-10]\\
\end{tikzcd}\]

Thus our lift is functorial once you consider identities and composition. This data gives us the functor $\lambda(J_{2},g)$.\\

Now we construct $\lambda(J_{1},g)$. Since $g$ is a normal isofibration, we need only show that $\lift(\htpy{0},g)$ can be extended to a functor. Given a square in $Y$:

\[\begin{tikzcd}
	{g(x)} &&& {z} \\
	\\
	\\
	{g(e)} &&& {z'}
	\arrow["{p}", from=1-1, to=1-4]
	\arrow["{g(m)}"', from=1-1, to=4-1]
	\arrow["{q}", from=1-4, to=4-4]
	\arrow["{p'}"', from=4-1, to=4-4]\\
\end{tikzcd}\]

This lifts to a square in $X$ :

\[\begin{tikzcd}
	{x} &&& {\lift(p)(x)} \\
	\\
	\\
	{e} &&& {\lift(p')(e)}
	\arrow["{\lift(\htpy{0},g)(x,p)}", from=1-1, to=1-4]
	\arrow["{m}"', from=1-1, to=4-1]
	\arrow["{\lift(\htpy{0},g)(e,p')  \circ  m  \circ  \lift(\htpy{0},g)(x,p)^{-1}}", from=1-4, to=4-4]
	\arrow["{\lift(\htpy{0},g)(e,p')}"', from=4-1, to=4-4]\\
\end{tikzcd}\]

we have $$g(\lift(\htpy{0},g)(e,p')  \circ  m  \circ  \lift(\htpy{0},g)(x,p)^{-1}) = g(\lift(\htpy{0},g)(e,p'))  \circ  g(m)  \circ  g(\lift(\htpy{0},g)(x,p)^{-1}) = p' \circ g(m) \circ p^{-1} = q$$

Furthermore when $p$ and $p'$ are identity morphisms, then $$\lift(\htpy{0},g)(e,p')  \circ  m  \circ  \lift(\htpy{0},g)(x,p)^{-1} = m$$ 

Once again, this gives us the data for the functor $\lambda(J_{1},g)$. \\

Therefore the $R_{[\mathrm{Cof}]}$ map are precisely the functors that are groupoid equivalences and fibrations.

\end{proof}

We seek to characterize the $L_{[\mathrm{Cof}]}$ maps which will serve as the cofibrations of the model structure on $\GrpA{A}$. Fortunately, we know that the $L_{[\mathrm{Cof}]}$ maps are precisely those morphisms that have the left lifting property against the $R_{[\mathrm{Cof}]}$ maps. We eventually want to show that the cofibrations are precisely those morphism which can be expressed as a decidable monomorphism on objects followed by a strong deformation section. We already have that strong deformation sections are $L_{[\mathrm{Cof}]}$ maps since by Lemma \ref{modstr1}, the strong deformation sections have the left lifting property against the normal isofibrations thus they have the left lifting property against the $R_{[\mathrm{Cof}]}$ maps. So we need to show that decidable monomorphisms on objects are $L_{[\mathrm{Cof}]}$ maps:

\begin{lem}

Decidable monomorphisms on objects are $L_{[\mathrm{Cof}]}$ maps.

\end{lem}
\begin{proof}

Given a decidable monomorphism of objects $i: U \to V$, we seek to construct a retract diagram:

\[\begin{tikzcd}
	U &&& U &&& U \\
	\\
	\\
	V &&& {[\mathrm{Cof}](i)} &&& V
	\arrow["{\id_{U}}", from=1-1, to=1-4]
	\arrow["i"', from=1-1, to=4-1]
	\arrow["{\id_{U}}", from=1-4, to=1-7]
	\arrow["{L_{[\mathrm{Cof}]}}"{description}, from=1-4, to=4-4]
	\arrow["i", from=1-7, to=4-7]
	\arrow["l"', from=4-1, to=4-4]
	\arrow["{R_{[\mathrm{Cof}]}}"', from=4-4, to=4-7]
\end{tikzcd}\]

so we construct the functor $l: V \to [\mathrm{Cof}](i)$. For an object $v \in V$, if $i(u) = v$ then $l(v) = L_{[\mathrm{Cof}]}(u)$, else $l(v) = \lambda(J_{0},R_{[\mathrm{Cof}]}(i))(v)$. Before we continue with morphisms, note that for $u \in \ob U$, we have a morphism between $L_{[\mathrm{Cof}]}(u)$, which will denote as $u$, and $\lambda(J_{0},R_{[\mathrm{Cof}]}(i))(i(u))$, which we will denote as $i(u)$, in $[\mathrm{Cof}](i)$:

\[\begin{tikzcd}
	u &&&&& {i(u)}
	\arrow["{{\lambda(J_{2},R_{[\mathrm{Cof}]}(i))(\{u,i(u)\},\id_{i(u)})}}", from=1-1, to=1-6]
\end{tikzcd}\]

We denote this morphism as $t(u,i(u))$. $t(-,i(-))$ is a natural isomorphism from $L_{[\mathrm{Cof}]}$ to $\lambda(J_{0},R_{[\mathrm{Cof}]}(i)) \circ i$ since when we have a morphism $p: u \to u'$ in $U$, we have the following commutative square exhibited by $\lambda(J_{2},R_{[\mathrm{Cof}]}(i))(\{p,i(p)\},(i(p),i(p)):\id_{u} \to \id_{u'} )$ in $[\mathrm{Cof}](i)$:

\[\begin{tikzcd}
	u &&& {u'} \\
	\\
	{i(u)} &&& {i(u')}
	\arrow["p", from=1-1, to=1-4]
	\arrow["{t(u,i(u))}"', from=1-1, to=3-1]
	\arrow["{t(u',i(u'))}", from=1-4, to=3-4]
	\arrow["{i(p)}"', from=3-1, to=3-4]
\end{tikzcd}\]

Thus $t(u',i(u'))^{-1} \circ \lambda(J_{0},R_{[\mathrm{Cof}]}(i))(i(p)) \circ t(u,i(u)) = L_{[\mathrm{Cof}]}(p)$. We will use $t(-,i(-))$ in defining the morphism part of the functor $l$. For a morphism $m \in V$, we have four cases:
\begin{itemize}

\item When $m:i(u) \to i(u')$, $l(m) = t(u',i(u'))^{-1} \circ \lambda(J_{0},R_{[\mathrm{Cof}]}(i))(m) \circ t(u,i(u))$
\item When $m:v \to v'$, $l(m) = \lambda(J_{0},R_{[\mathrm{Cof}]}(i))(m) $
\item When $m:v \to i(u')$, $l(m) = t(u',i(u'))^{-1} \circ \lambda(J_{0},R_{[\mathrm{Cof}]}(i))(m)$
\item When $m:i(u) \to v'$, $l(m) =  \lambda(J_{0},R_{[\mathrm{Cof}]}(i))(m) \circ t(u,i(u))$

\end{itemize}

functoriality of $l$ then follows from functoriality of $\lambda(J_{0},R_{[\mathrm{Cof}]}(i))$. In addition, by the establish property of $t(-,i(-))$, the leftmost square of the retract diagram commutes. Now we must show that $R_{[\mathrm{Cof}]} \circ l = \id_{V}$ but this is straightforward as for $u \in \ob U$, $$R_{[\mathrm{Cof}]} \circ l(i(u)) = R_{[\mathrm{Cof}]}  \circ L_{[\mathrm{Cof}]}(u) = i(u)$$
For $v \in \ob V$ not equal to $i(u)$, $$R_{[\mathrm{Cof}]} \circ l(v) = R_{[\mathrm{Cof}]}  \circ \lambda(J_{0},R_{[\mathrm{Cof}]}(i))(v) = v$$ We have $R_{[\mathrm{Cof}]}(t(u,i(u))) = \id_{i(u)}$ by definition of $t(u,i(u))$. So $l \circ R_{[\mathrm{Cof}]}$ is the identity on morphisms since $\lambda(J_{0},R_{[\mathrm{Cof}]}(i))$ is a section of $R_{[\mathrm{Cof}]}$. Therefore, $R_{[\mathrm{Cof}]} \circ l = \id_{V}$ making $i$ an $L_{[\mathrm{Cof}]}$ map.

\end{proof}

\begin{cor}
\label{cofif}
 Functors $i = s \circ d$ where $s$ is a strong deformation section and $d$ is a decidable monomorphism on objects are $L_{[\mathrm{Cof}]}$ maps.

\end{cor}

To show the converse, we define a alternative factorization:\\

For a functor $F: X \to Y$, we define a groupoid $F|Y$ where:

\begin{itemize}
\item $\ob F|Y = \ob X + \ob Y$
\item $F|Y(x,x') = Y(F(x),F(x'))$
\item $F|Y(y,x') = Y(y,F(x'))$
\item $F|Y(x,y') = Y(F(x),y')$
\item $F|Y(y,y') = Y(y,y')$
\end{itemize}

We have a functor $F^{-}:X \to F|Y$ which is just the identity on objects and the functor $F$ on morphisms. We have another morphism $F^{+}: F|Y \to Y$ which is the identity on morphisms and $\ob Y$ and $F$ on $\ob X$. It is straightforward to show that $F^{+} \circ F^{-} = F$. What is also straightforward to show is the following:

\begin{lem}
 $F^{+}$ is a groupoid equivalence.

\end{lem}
\begin{proof}

$F^{+*}:Y \to F|Y$ is simply the identity on objects and morphisms. Thus we have $F^{+} \circ F^{+*} = \id_{Y}$. We also have $\alpha: F^{+*} \circ F^{+} \cong \id_{F|Y}$ where $\alpha_{x} = \id_{F(x)}:F(x) \to x$ and $\alpha_{y} = \id_{y}:y \to y$. It is straightforward to verify that $\alpha$ is a natural isomorphism.

\end{proof}

Recall that we have a algebraic weak factorization system $(\Fib, R_{\Fib},L_{\Fib})$ where $\Fib$ gives the comma category construction, the $R_{\Fib}$ maps are the normal isofibrations and the $L_{\Fib}$ maps are the strong deformation sections.

\begin{cor}

$R_{\Fib}(F^{+})$ is a groupoid equivalence.

\end{cor}
\begin{proof}

We have $F^{+} = R_{\Fib}(F^{+}) \circ L_{\Fib}(F^{+})$ where both $F^{+}$ and $L_{\Fib}(F^{+})$ are groupoid equivalences. Thus by the 2 out of 3 property, $R_{\Fib}(F^{+})$ is one as well.

\end{proof}

Thus $R_{\Fib}(F^{+})$ is a groupoid equivalence and a normal isofibration. At last we can show the following:

\begin{lem}

$i:U \to V$ has the left lifting property against $R_{[\mathrm{Cof}]}$ maps if and only if $i = s \circ d$ where $d$ is a decidable mono on objects and $s$ is a strong deformation section.

\end{lem}
\begin{proof}

Corollary \ref{cofif} already gives the if direction, now we must show the only if direction. Suppose $i$ has the left lifting property against $R_{[\mathrm{Cof}]}$ maps, then we have a lifting diagram:
\[\begin{tikzcd}
	U &&& {(i^{+})\downarrow V} \\
	\\
	\\
	V &&& V
	\arrow["{{L_{\Fib}(i^{+}) \circ i^{-}}}", from=1-1, to=1-4]
	\arrow["i"', from=1-1, to=4-1]
	\arrow["{{R_{\Fib}(i^{+})}}", from=1-4, to=4-4]
	\arrow["l"{description}, dashed, from=4-1, to=1-4]
	\arrow["{{\id_{V}}}"', from=4-1, to=4-4]
\end{tikzcd}\]

this gives us a retract diagram:

\[\begin{tikzcd}
	U &&& U &&& U \\
	\\
	\\
	V &&& {(i^{+})\downarrow V} &&& V
	\arrow["{\id_{U}}", from=1-1, to=1-4]
	\arrow["i"', from=1-1, to=4-1]
	\arrow["{\id_{U}}", from=1-4, to=1-7]
	\arrow["{L_{\Fib}(i^{+}) \circ i^{-}}"{description}, from=1-4, to=4-4]
	\arrow["i", from=1-7, to=4-7]
	\arrow["l"', from=4-1, to=4-4]
	\arrow["{R_{\Fib}(i^{+})}"', from=4-4, to=4-7]
\end{tikzcd}\]

First we observe that $i$ a monomorphism on objects since if we have $u,u' \in \ob U$ such that $i(u) = i(u')$, then $$L_{\Fib}(i^{+}) \circ i^{-}(u) = l(i(u)) = l(i(u')) = L_{\Fib}(i^{+}) \circ i^{-}(u')$$ then $(u,i(u), \id_{i(u)}) = (u',i(u'), \id_{i(u')})$ thus $u = u'$. The objects of $(i^{+})\downarrow V$ are either of the form $(u, v, f: i(u) \to v)$ or of the form $(v, v', f: v \to v')$, thus we can split $\ob (i^{+})\downarrow V$ into $V_{base} + U_{=}$ where $U_{=}$ are the elements that $l$ maps into elements $(u, v, f: i(u) \to v)$ and $V_{base}$ are the elements that get mapped into $(v, v', f: v \to v')$. Furthermore for $u \in \ob U$, $i(u) \in U_{=}$ since $L_{\Fib}(i^{+}) \circ i^{-}(u) = (u,i(u), \id_{i(u)}) = l(i(u))$. We define $i \oplus V_{base}$ as the full sub groupoid of $V$ on the objects $\{v \in U_{=}|\exists u \in \ob U( i(u) = v) \} + V_{base}$. This gives us a factorization of $i$:

\[\begin{tikzcd}
	U && {i \oplus V_{base}} && V
	\arrow["{i^{\oplus}}", from=1-1, to=1-3]
	\arrow["{i^{=}}", from=1-3, to=1-5]
\end{tikzcd}\]

where $i^{\oplus}$ is a decidable monomorphism on objects. Now we seek to construct the following retract diagram:

\[\begin{tikzcd}
	i \oplus V_{base} &&& i \oplus V_{base} &&& i \oplus V_{base} \\
	\\
	\\
	V &&& {(i^{=})\downarrow V} &&& V
	\arrow["{\id_{i \oplus V_{base}}}", from=1-1, to=1-4]
	\arrow["i^{=}"', from=1-1, to=4-1]
	\arrow["{\id_{i \oplus V_{base}}}", from=1-4, to=1-7]
	\arrow["{L_{\Fib}(i^{=})}"{description}, from=1-4, to=4-4]
	\arrow["i^{=}", from=1-7, to=4-7]
	\arrow["l'"', from=4-1, to=4-4]
	\arrow["{R_{\Fib}(i^{=})}"', from=4-4, to=4-7]
\end{tikzcd}\]

We define $l'$ as follows:

\begin{itemize}

\item For $v \in U_{=}$, $l'(v) = (\iota_{v} = i(u), v, f_{v}:i(u) \to v)$ where $l(v) = (u, v, f_{v}:i(u) \to v)$ 
\item For $v \in V_{base}$, $l'(v) = (\iota_{v} = v,v, f_{v} = \id_{v})$
\item For a morphism $\nu: v \to v'$ in $V$, $l'(\nu) = (f_{v'}^{-1} \circ \nu \circ f_{v}, \nu)$ giving us the following commutative square:

\[\begin{tikzcd}
	{\iota_{v}} &&& v \\
	\\
	\\
	{\iota_{v'}} &&& {v'}
	\arrow["{f_{v}}", from=1-1, to=1-4]
	\arrow["{f_{v'}^{-1} \circ \nu \circ f_{v}}"', from=1-1, to=4-1]
	\arrow["\nu", from=1-4, to=4-4]
	\arrow["{f_{v'}}"', from=4-1, to=4-4]
\end{tikzcd}\]

\end{itemize}

It should be clear that $l'$ is functorial. Furthermore, $R_{\Fib}(i^{=}) \circ l' = \id_{V}$ since $R_{\Fib}(i^{=})$ projects the second element of a pair of morphisms or a triple. Now we seek to show that the leftmost square commutes. For $i(u)\in i \oplus V_{base}$, $$l'(i^{=}(i(u))) = l'(i(u)) = (i(u), i(u), \id_{i(u)}) = L_{\Fib}(i^{=})(i(u))$$ since $l(i(u)) = (u,i(u),\id_{i(u)})$. Thus for any morphism $\nu:v \to v'$ in $i \oplus V_{base}$, $f_{v} = \id_{v}$ and $f_{v'} = \id_{v'}$, thus $l'(\nu) = (\nu,\nu) = L_{\Fib}(i^{=})(\nu)$. Thus $i^{=}$ is a $L_{\Fib}$ map making it a strong deformation section. Therefore $i = i^{=} \circ i^{\oplus}$ where $i^{\oplus}$ is a decidable monomorphism on objects and $i^{=}$ is a strong deformation section. \\

Thus the $L_{[\mathrm{Cof}]}$ maps are precisely the functors $i = s \circ d$ where $d$ is a decidable monomorphism on objects and $s$ is a strong deformation section.

\end{proof}

\begin{lem}

If $i$ is an $L_{[\mathrm{Cof}]}$ map and a groupoid equivalence, then $i$ is a strong deformation section.

\end{lem}
\begin{proof}

Since $i = s \circ d$ where $s$ is a strong deformation section, by the 2 out of 3 property $d$ is a groupoid equivalence. Since $d$ is a decidable monomorphism on objects, by Lemma \ref{deciequi}, $d$ is a strong deformation section. Since strong deformation sections are the left maps of a weak factorization system then they are closed under composition. Therefore, $i$ is a strong deformation section.

\end{proof}

\begin{cor}
\label{mdlstr}

$\GrpA{A}$ has a model structure where 
\begin{itemize}

\item The fibrations are the normal isofibrations.
\item The cofibrations are the decidable monomorphisms on objects followed by the strong deformation sections.
\item The weak equivalences are the groupoid equivalences.
\item $(\Fib, R_{\Fib},L_{\Fib})$ and $([\mathrm{Cof}], R_{[\mathrm{Cof}]}, L_{[\mathrm{Cof}]})$ give the fibrant and  cofibrant factorizations respectively.

\end{itemize}

\end{cor}

\subsection{Localization Modalities}
In this section, we seek to use $W$-types with reductions to construct localization modalities as found in \cite{Rijke_2020}.\\

 Given an indexing groupoid assembly $I$, fibrations $U$ and $V$ over $I$ and a morphism $J:U \to V$ we have the following definitions:

\begin{defn}
We call a groupoid assembly $X$ {\bf (I,J)-local} if we have the following morphism between fibrations:

\[\begin{tikzcd}
	{\sum_{i \in I} (V_{i} \to X)} &&&& {\sum_{i \in I} (U_{i} \to X)} \\
	\\
	&& I
	\arrow["{- \circ J}", from=1-1, to=1-5]
	\arrow["{V \to X}"', from=1-1, to=3-3]
	\arrow["{U \to X}", from=1-5, to=3-3]
\end{tikzcd}\]

is an isomorphism over $I$.
\end{defn}

For this section, we will restrict $I$ to only being an discrete groupoid assembly (i.e. all morphisms are identity morphisms). Recall that any morphism over an assembly is a fibration, thus we can define the fibration $\sum_{i \in I} V_{i} +_{U_{i}} V_{i} \to I$ as the pushout of the following diagram over $I$:

\[\begin{tikzcd}
	{\sum_{i \in I} U_{i}} &&&& {\sum_{i \in I} V_{i}} \\
	\\
	\\
	\\
	{\sum_{i \in I} V_{i}} &&&& {\sum_{i \in I} V_{i} +_{U_{i}} V_{i}}
	\arrow["J", from=1-1, to=1-5]
	\arrow["J"', from=1-1, to=5-1]
	\arrow["{\iota_{0}}", from=1-5, to=5-5]
	\arrow["{\iota_{1}}"', from=5-1, to=5-5]
	\arrow["\lrcorner"{anchor=center, pos=0.125, rotate=180}, draw=none, from=5-5, to=1-1]
\end{tikzcd}\]

We set the morphism $J^{*}: \sum_{i \in I} V_{i} +_{U_{i}} V_{i} \to \sum_{i \in I} V_{i}$ over $I$ as the result of the universal property of pushouts:

\[\begin{tikzcd}
	{\sum_{i \in I} U_{i}} &&&& {\sum_{i \in I} V_{i}} \\
	\\
	\\
	\\
	{\sum_{i \in I} V_{i}} &&&& {\sum_{i \in I} V_{i} +_{U_{i}} V_{i}} \\
	&&&&& {\sum_{i \in I} V_{i}}
	\arrow["J", from=1-1, to=1-5]
	\arrow["J"', from=1-1, to=5-1]
	\arrow["{{\iota_{0}}}", from=1-5, to=5-5]
	\arrow["{\id_{V}}", from=1-5, to=6-6]
	\arrow["{{\iota_{1}}}"', from=5-1, to=5-5]
	\arrow["{\id_{V}}"', from=5-1, to=6-6]
	\arrow["\lrcorner"{anchor=center, pos=0.125, rotate=180}, draw=none, from=5-5, to=1-1]
	\arrow["{J^{*}}", from=5-5, to=6-6]
\end{tikzcd}\]

Now we take the functor $$J + J^{*}:\sum_{i \in I} U_{i} + \sum_{i \in I} V_{i} +_{U_{i}} V_{i}  \to \sum_{i \in I} V_{i} + \sum_{i \in I} V_{i}$$ as a functor over $I + I$. Take the fibration $N$ as the zero fibration, we obtain an algebraic weak factorization system $([\Loc(I,J)], R_{[\Loc(I,J)]},L_{[\Loc(I,J)]})$ such that the $R_{[\Loc(I,J)]}$ maps $g:X \to Y$ are precisely those functors such that for $i \in I$, we have the lifts: 

\[\begin{tikzcd}
	{\Sq(J_{i}, g)} &&&& {V_{i} \to X} \\
	{\Sq(J^{*}_{i}, g)} &&&& {V_{i} \to X}
	\arrow["{{{\lambda(J_{i},g})}}", from=1-1, to=1-5]
	\arrow["{{{\lambda(J^{*}_{i},g})}}", from=2-1, to=2-5]
\end{tikzcd}\]

with no normalization property. We call a groupoid assembly $X$ an $R_{[\Loc(I,J)]}$ object when $!: X \to 1$ is an $R_{[\Loc(I,J)]}$ map. 

\begin{lem}

The $R_{[\Loc(I,J)]}$ objects are precisely the $(I,J)$-local objects.

\end{lem}
\begin{proof}

Since the functor $- \to X$ sends limits to colimits, we have:
$$V_{i} +_{U_{i}} V_{i} \to X  \cong (V_{i} \to X)\times_{U_{i} \to X}(V_{i} \to X)$$ Hence the universal property of pushouts extend to natural transformations as well.\\

Suppose $X$ is $(I,J)$-local. Since we are dealing with a morphism $g:X \to 1$, it suffices to find for every $i \in I$ functors:
\[\begin{tikzcd}
	{U_{i} \to X} &&&& {V_{i} \to X} \\
	{V_{i} +_{U_{i}} V_{i} \to X} &&&& {V_{i} \to X}
	\arrow["{{{\lambda(J_{i},g})}}", from=1-1, to=1-5]
	\arrow["{{{\lambda(J^{*}_{i},g})}}", from=2-1, to=2-5]\\
\end{tikzcd}\]

such that $(- \circ J_{i}) \circ \lambda(J_{i},g) = \id_{U_{i} \to X}$ and $(- \circ J^{*}_{i}) \circ \lambda(J^{*}_{i},g) = \id_{V_{i} +_{U_{i}} V_{i} \to X}$. $X$ being $(I,J)$-local means that each $- \circ J_{i}$ is an isomorphism so it suffices to find $\lambda(J^{*}_{i},g)$. Given a morphism $t:V_{i} +_{U_{i}} V_{i} \to X$ by the universal property of pushouts we have $t \circ \iota_{0,i}: V_{i} \to X$ and $t \circ \iota_{1,i}: V_{i} \to X$ such that $t \circ \iota_{0,i} \circ J_{i} = t \circ \iota_{1,i} \circ J_{i}$. Since $- \circ J_{i}$ is an isomorphism, $t \circ \iota_{0,i} = t \circ \iota_{1,i}$. Also by the universal property of pushouts $t = t \circ \iota_{0,i} \circ J^{*}_{i}$ since $J^{*}_{i} \circ \iota_{0,i} = J^{*}_{i} \circ \iota_{1,i} = \id_{V_{i}}$.  We have a similar argument for natural isomorphisms $\tau: t \to t'$. Thus $\lambda(J^{*}_{i},g) = (- \circ \iota_{1,i})$ making $X$ an $R_{[\Loc(I,J)]}$ object.\\

Suppose $X$ is an $R_{[\Loc(I,J)]}$ object, then we have lifts

\[\begin{tikzcd}
	{U_{i} \to X} &&&& {V_{i} \to X} \\
	{V_{i} +_{U_{i}} V_{i} \to X} &&&& {V_{i} \to X}
	\arrow["{{{\lambda(J_{i},g})}}", from=1-1, to=1-5]
	\arrow["{{{\lambda(J^{*}_{i},g})}}", from=2-1, to=2-5]\\
\end{tikzcd}\]
such that $(- \circ J_{i}) \circ \lambda(J_{i},g) = \id_{U_{i} \to X}$ and $(- \circ J^{*}_{i}) \circ \lambda(J^{*}_{i},g) = \id_{V_{i} +_{U_{i}} V_{i} \to X}$. It suffices to show that $\lambda(J_{i},g) \circ (- \circ J_{i}) = \id_{V_{i} \to X}$. Suppose we have $t:V_{i} \to X$, then we have $\lambda(J_{i},g)(t \circ J_{i}): V_{i} \to X$ with $\lambda(J_{i},g)(t \circ J_{i}) \circ J_{i} = t \circ J_{i}$. This gives us a unique $t_{*}:V_{i} \times_{U_{i}} V_{i} \to X$ such that $t^{*} \circ \iota_{0,i} = t$ and $t^{*} \circ \iota_{1,i} = \lambda(J_{i},g)(t \circ J_{i})$. Furthermore we have a lift $\lambda(J^{*}_{i},g)(t_{*}):V_{i} \to X$ such that $\lambda(J^{*}_{i},g)(t_{*}) \circ J^{*} = t_{*}$. We now have the following: $$\lambda(J^{*}_{i},g)(t_{*}) = \lambda(J^{*}_{i},g)(t_{*}) \circ J^{*} \circ \iota_{1,i} = t_{*} \circ \iota_{1,i} = \lambda(J_{i},g)(t \circ J_{i})$$ and 
$$\lambda(J^{*}_{i},g)(t_{*}) = \lambda(J^{*}_{i},g)(t_{*}) \circ J^{*} \circ \iota_{0,i} = t_{*} \circ \iota_{0,i} = t$$ Thus $t = \lambda(J_{i},g)(t \circ J_{i})$. Since the same argument applies to morphisms we have $\lambda(J_{i},g) \circ (- \circ J_{i}) = \id_{V_{i} \to X}$. Thus $(- \circ J_{i})$ is an isomorphism making $X$ $(I,J)$-local.\\

Therefore, the $R_{[\Loc(I,J)]}$ objects are precisely the $(I,J)$-local objects.

\end{proof}

\begin{lem}
\label{jstarlocal}

$(I,J)$-local objects are also $(I,J^{*})$-local.

\end{lem}
\begin{proof}

It suffices to show that any $(I,J)$-local object $X$ is an $R_{[\Loc(I,J^{*})]}$ object. However, it is straightforward to show that we have the following diagram:

\[\begin{tikzcd}
	{\sum_{i \in I} V_{i} +_{U_{i}} V_{i} } &&&& {\sum_{i \in I} V_{i}} \\
	\\
	\\
	\\
	{\sum_{i \in I} V_{i}} &&&& {\sum_{i \in I} V_{i} } \\
	&&&&& {\sum_{i \in I} V_{i}}
	\arrow["J^{*}", from=1-1, to=1-5]
	\arrow["J^{*}"', from=1-1, to=5-1]
	\arrow["{{\id_{V}}}", from=1-5, to=5-5]
	\arrow["{\id_{V}}", from=1-5, to=6-6]
	\arrow["{{\id_{V}}}"', from=5-1, to=5-5]
	\arrow["{\id_{V}}"', from=5-1, to=6-6]
	\arrow["\lrcorner"{anchor=center, pos=0.125, rotate=180}, draw=none, from=5-5, to=1-1]
	\arrow["{\id_{V}}", from=5-5, to=6-6]
\end{tikzcd}\]

Hence $J^{**} = \id_{V}$, therefore $X$ is also an $R_{[\Loc(I,J^{*})]}$ object and thus is an $(I,J^{*})$-local object.

\end{proof}

We set $\Loc_{I}(J) = R_{[\Loc(I,J)]}|_{\GrpA{A}/1}$. Thus $\Loc_{I}(J): \GrpA{A} \to \GrpA{A}$ is still a monad with its unit $\eta$ and multiplication $\mu$ being the restrictions of $\eta_{R_{[\Loc(I,J)]}}$ and $\mu_{R_{[\Loc(I,J)]}}$.

\begin{lem}

The $\Loc_{I}(J)$ unit, $\eta_{X}$, is an isomorphism when $X$ is $(I,J)$-local.

\end{lem}
\begin{proof}

Since $X$ is $(I,J)$-local, $\eta_{X}: X \to \Loc_{I}(J)(X)$ has a section $h$. Recall that $\Loc_{I}(J)(X)$ is a $W$-type with reduction. We denote its algebra structure as $\alpha_{X}$. We seek to show $\eta_{X} \circ h = \id_{\Loc_{I}(J)(X)}$.
Obviously this holds for elements $\eta_{X}(x)$. Suppose we have an object $\alpha_{X}((v_{i},!:V_{i} \to 1), c_{i}:U_{i} \to \Loc_{I}(J)(X))$ such that we have $c_{i} = \eta_{X} \circ h \circ c_{i}$ by induction. We have a unique $k_{i}:V_{i} \to \Loc_{I}(J)(X)$ such that $k_{i} \circ J_{i} = c_{i}$ and $k_{i}(v_{i}) = \alpha_{X}((v_{i},!),c_{i})$. Using the fact that $X$ is $(I,J)$-local, we have a $q_{i}: V_{i} \to X$ such that $q_{i}  \circ J_{i} = h \circ c_{i}$. Thus we have $$\eta_{X} \circ q_{i}  \circ J_{i} = \eta_{X} \circ h \circ c_{i} = c_{i} = k_{i} \circ J_{i}$$ This gives us a unique $s: V_{i} +_{U_{i}} V_{i} \to \Loc_{I}(J)(X)$ such that $s \circ \iota_{0,i} = \eta_{X} \circ q_{i} $ and $s \circ \iota_{1,i} = k_{i}$. Thus we have a lift $l:V_{i} \to \Loc_{I}(J)(X)$ such that $s = l \circ J^{*}_{i}$. But this means that $l = k_{i} = \eta_{X} \circ q_{i}$. And so we have $$\eta_{X} \circ h \circ k_{i} = \eta_{X} \circ h \circ \eta_{X} \circ q_{i} = \eta_{X} \circ q_{i} = k_{i}$$ thus $$\eta_{X} \circ h(\alpha_{X}((v_{i},!),c_{i})) = \eta_{X} \circ h \circ k_{i}(v_{i}) = k_{i}(v_{i}) = \alpha_{X}((v_{i},!),c_{i})$$
We have a similar argument for morphisms. Preservation of reductions is obvious. Therefore $\eta_{X} \circ h = \id_{\Loc_{I}(J)(X)}$ making $\eta_{X}$ an isomorphism.

\end{proof}

\begin{cor}

$\mu_{X}$ is an isomorphism for every $X$.

\end{cor}
\begin{proof}

We have $\mu_{X} \circ \eta_{\Loc_{I}(J)(X)} = \id_{\Loc_{I}(J)(X)}$. Since $\Loc_{I}(J)(X)$ is $(I,J)$-local, $\eta_{\Loc_{I}(J)(X)}$ is an isomorphism, thus $\eta_{\Loc_{I}(J)(X)}^{-1} = \mu_{X}$.

\end{proof}

\begin{cor}

$\Loc_{I}(J)$ is an idempotent monad whose reflective subcategory is the full subcategory of $(I,J)$-local objects. 

\end{cor}

We will refer to the full subcategory of $(I,J)$-local objects as $\GrpA{A}_{(I,J)}$.\\

\begin{lem}

$\GrpA{A}_{(I,J)}$ is finitely complete and cocomplete. It is also an exponential ideal hence $\Loc_{I}(J)$ preserves finite products and $\GrpA{A}_{(I,J)}$ is cartesian closed.

\end{lem}
\begin{proof}

Since $\GrpA{A}_{(I,J)}$ is a reflective subcategory, $\Loc_{I}(J)$ is a left adjoint and thus preserves any colimits in $\GrpA{A}$ in particular the ones between $(I,J)$-local objects. Hence $\GrpA{A}_{(I,J)}$ has finite colimits. \\

Since we have an algebraic weak factorization system $([\Loc(I,J)], R_{[\Loc(I,J)]},L_{[\Loc(I,J)]})$, the $R_{[\Loc(I,J)]}$ algebras are closed under all limits that exists in $\GrpA{A}^{\arr}$. Recall that for a $(I,J)$-local object $X$, $\eta_{X}$ is an isomorphism. Hence any $(I,J)$-local object $X$ is an $R_{[\Loc(I,J)]}$ algebra and vice versa when considering $R_{[\Loc(I,J)]}$ algebras $X \to 1$. We concluded from this that $(I,J)$-local objects are closed under finite limits. \\

Given a groupoid assembly $Y$ and an $(I,J)$-local object $X$, we seek to show that for $i \in I$, $$(- \circ J_{i}): V_{i} \to (Y \to X)\to U_{i} \to (Y \to X)$$ is an isomorphism. However we can express $(- \circ J_{i})$ as the morphism: $$(- \circ J_{i})^{Y}:  Y \to (V_{i} \to X)\to Y \to (U_{i} \to X)$$ which is obviously an isomorphism. Thus $(- \circ J_{i})$ is an isomorphism making $(Y \to X)$ $(I,J)$-local. Therefore, $\GrpA{A}_{(I,J)}$ is an exponential ideal.

\end{proof}

Here we will talk about a few notable examples:\\

We have alluded to the notion of discrete groupoid assemblies before, but now we will formally define them:

\begin{defn}

We call $X$ a {\bf discrete groupoid assembly} if the only isomorphisms of $X$ are the identity morphisms.

\end{defn}

We can characterize the discrete groupoid assemblies $X$ as the groupoid assemblies with the morphism:

$$(- \circ  !_{\I}): X \to (\I \to X)$$ being an isomorphism since such a condition basically states that for an isomorphism $p$ in $X$ there exists a unique $x \in \ob X$ such that $p = \id_{x}$. So setting $\Loc_{1}(!_{\I})$ as the localization modality with respect to $!_{\I}:\I \to 1$, $\GrpA{A}_{(1,!_{\I})} \cong \Asm(A)$.\\

Recall that an assembly $(X,\phi)$ is modest when for $x, x' \in X$, if $x \neq x'$ then $\phi(x) \cap \phi(x') = \emptyset$.
Also recall that we have a functor $\nabla: \Set \to \Asm(A)$ which takes a set $X$ to $(X, \const_{A})$. Setting $2 = \{0,1\}$, we have the assembly $\nabla(2)$. We can now characterize the modest assemblies as follows:

\begin{lem}
$(X,\phi)$ is a modest assembly if and only if the morphism $$(- \circ !_{\nabla(2)}): (X,\phi) \to (\nabla(2) \to (X,\phi))$$ is an isomorphism.

\end{lem}
\begin{proof}

Suppose $(X,\phi)$ is a modest assembly, then if we have a morphism $f:\nabla(2) \to (X,\phi)$, $f$ is realized by $r_{f}$ in that for $i \in 2$ and $a \in A$, $r_{f} \cdot a \in \phi(f(i))$. This means that $\phi(f(0)) \cap \phi(f(1))$ is inhabited, so $f(0) = f(1)$. So $f = \nm{f(0)} \circ !_{\nabla(2)}$ where $\nm{f(0)}: 1 \to (X,\phi)$ picks out $f(0)$. So we have a morphism $s: (\nabla(2) \to (X,\phi)) \to (X,\phi)$ with $(- \circ !_{\nabla(2)}) \circ s = \id_{(\nabla(2) \to (X,\phi))}$ where $s(f) =  f(0)$ realized by $\lmbd{z}{z \cdot \iota}$ where $\iota \in A$ is the identity combinator. $s \circ (- \circ !_{\nabla(2)}) = \id_{(X,\phi)}$ is straightforward thus $(- \circ !_{\nabla(2)}) $ is an isomorphism.\\

Suppose $(- \circ !_{\nabla(2)})$ is an isomorphism, then take $x, x' \in X$ such that $x \neq x'$. If $a \in \phi(x) \cap \phi(x')$, Then we can define $f:\nabla(2) \to (X,\phi)$ with $f(0) = x$ and $f(1) = x'$ realized by $k \cdot a$. This is a contradiction since $f = \nm{x''} \circ !_{\nabla(2)}$ meaning $x' = x'' = x $. Thus  $\phi(x) \cap \phi(x')$ is empty making $(X,\phi)$ a modest assembly.\\

Therefore, $(X,\phi)$ is modest if and only if $(- \circ !_{\nabla(2)}): (X,\phi) \to (\nabla(2) \to (X,\phi))$ is an isomorphism.

\end{proof}

Returning to groupoid assemblies, we have $\nabla(2)$ as a discrete groupoid assembly and a functor $!_{\nabla(2)}:\nabla(2) \to 1$. This gives as a localization modality $\Loc_{1}(!_{\nabla(2)})$ where the $!_{\nabla(2)}$ local objects $X$ are the ones with the following isomorphisms:

$$(- \circ !_{\nabla(2)}):X \cong ({\nabla(2)} \to X)$$ $$(- \circ !_{\nabla(2)}):(\I \to X) \cong ({\nabla(2)} \to (\I \to X))$$

This means that both the object and morphism assemblies of $X$ are modest assemblies. This gives us the definition of the {\bf modest groupoid assemblies}. We can combine $!_{\nabla(2)}$ and $!_{\I}$ to obtain a localization modality $\Loc_{1 + 1}(!_{\nabla(2)}, !_{\I})$ whose local objects are precisely the modest assemblies.\\

Given a partial combinatory algebra $A$ and a partial function $f: A \rightharpoonup A$, we can construct a localization modality such that the local objects view $f$ as ``realizable". Setting $\dom(f)$ as the domain of $f$ we have the morphism $\id_{\dom(f)}: (\dom(f), [a \mapsto \{[a,f(a)]\}]) \to (\dom(f), [a \mapsto \{a\}])$ realized by $p_{0}$ which is the  first component projection combinator. This gives a localization modality $\Loc_{1}(f)$ whose local objects are the groupoid assemblies whose object and morphism assemblies are elements of $\Asm(A[f])$ where one can think of $A[f]$ as being the partial combinatory algebra that makes $f$ realizable with respect to $A$. The details of the construction of $A[f]$ are found in Section 1.7 of \cite{vanOosten2008}.\\

Finally we acknowledge the $0$-truncation and $-1$-truncation modalities, but we will cover those in a later section.


\section{$\GrpA{A}$ as a Model of Type Theory}
The section will be the final section in exhibiting $\GrpA{A}$ as a model of type theory. First we will show the existence of univalent universes of discrete groupoid assemblies. Building on that, we will then show the existence of a impredicative univalent universe of modest assemblies. Finally, we show that $\GrpA{A}$ has the structure of a $\pi$-tripos making it a model of type theory.\\


\subsection{Fibration Univalent and Impredicative Universes}
This section contains the construction of univalent universes of discrete assemblies and the impredicative univalent universe of modest assemblies.\\

\begin{defn}

We call a fibration $F:X \to Y$ a {\bf discrete fibration} when each $X_{y}$ is a discrete groupoid assembly.

\end{defn}

\begin{lem}

For a discrete fibration $F:X \to Y$, isomorphism $p:y \to y'$ in $Y$ and $x \in X_{y}$ and $x' \in X_{y'}$, there is at most one isomorphism $q:x \to x'$ over $p$. 

\end{lem}
\begin{proof}

Suppose we have two isomorphisms $q,q':x \to x'$ over $p$, then $q'^{-1} \circ q:x \to x$ exists over $\id_{y}$ and $q^{-1} \circ q':x' \to x'$ exists over $\id_{y'}$. Then $q'^{-1} \circ q  \in X_{y}$ and $q^{-1} \circ q' \in X_{y'}$ thus $q'^{-1} \circ q  = \id_{x}$ and $q^{-1} \circ q'  = \id_{x'}$. Therefore $q' = q$.

\end{proof}

\begin{lem}
\label{discfiberpath}
 When $F:X \to Y$ is a discrete fibration, for an isomorphism $p:y \to y'$ in $Y$, the induced equivalence $\lift(p):X_{y} \to X_{y'}$ is an isomorphism with inverse $\lift(p^{-1})$. Furthermore when we have an isomorphism $p' \circ p$ in $Y$, then $\lift(p' \circ p) = \lift(p') \circ \lift(p)$.

\end{lem}
\begin{proof}

By Lemma \ref{fiberpath}, we have a natural isomorphism $\lift(j,j'): \lift(p') \circ \lift(p) \cong \lift(p' \circ p)$. Since each fiber of $F$ is a discrete groupoid assembly, then each $\lift(j,j')(x)$ is the identity thus we have $\lift(p' \circ p) = \lift(p') \circ \lift(p)$. For $p:y \to y'$, we also have $\id_{X_{y}} = \lift(\id_{y}) = \lift(p^{-1}) \circ \lift(p)$ and $\id_{X_{y'}} = \lift(\id_{y'}) = \lift(p) \circ \lift(p^{-1})$. So $\lift(p)$ is an isomorphism with inverse $\lift(p^{-1})$.

\end{proof}

Given an object classifier $U$ in $\Asm(A)$, we extend $U$ into a groupoid $U^{grpd}$ where for $\ob U^{grpd} = U$ and for $X,Y \in U$, we have the assembly $U^{grpd}(X,Y) := \{f \in \Hom(X,Y)| f \text{ is an isomorphism}\}$. We define the groupoid $U^{grpd}_{*}$ where $\ob U^{grpd}_{*} = \Sigma_{X \in U} X$ and for $(X,x), (Y,y) \in \ob U^{grpd}_{*} $, $U^{grpd}_{*}((X,x), (Y,y)) := \{f \in U^{grpd}(X,Y)| f(x) = y\}$. We have an obvious projection functor $\pr{U}:U^{grpd}_{*} \to U^{grpd}$. We call a discrete fibration $F$ {\bf U-valued} when the fibers of $F$ are contained in $U$.

\begin{lem}
\label{disuni}

$\pr{U}$ classifies the $U$-valued discrete fibrations up to isomorphism.

\end{lem}
\begin{proof}

First we observe that $\pr{U}$ is a $U$-valued discrete fibration since the fibration lifting is given by the applying the isomorphism to an element in its domain. Hence any pullback of $\pr{U}$ will be a $U$-valued discrete fibration. Suppose we have a $U$-valued discrete fibration $F:X \to Y$, then we define a functor $\nm{F}: Y \to U^{grpd}$ where $\nm{F}(y) = X_{y}$ and for $p:y \to y'$, $\nm{F}(p) = \lift(p)$. $\nm{F}$ is a well defined functor since $F$ is $U$-valued and by Lemma \ref{discfiberpath}. This gives us a pullback:

\[\begin{tikzcd}
	{[F]} &&& {U^{grpd}_{*}} \\
	\\
	\\
	Y &&& {U^{grpd}}
	\arrow[from=1-1, to=1-4]
	\arrow["{F^{*}}"', from=1-1, to=4-1]
	\arrow["\lrcorner"{anchor=center, pos=0.125}, draw=none, from=1-1, to=4-4]
	\arrow["{\pr{U}}", from=1-4, to=4-4]
	\arrow["{\nm{F}}"', from=4-1, to=4-4]
\end{tikzcd}\]

We seek to show exhibit an isomorphism $X \cong [F]$ over $Y$. First we construct $\alpha: X \to [F]$. For $x \in \ob X$, $\alpha(x) = (F(x),(X_{F(x)},x))$ and for $q: x \to x'$, $\alpha(q) = (F(q),\lift(F(q)) )$. To see that $\alpha(q): \alpha(x) \to \alpha(x')$ works first observe that we have a isomorphism $\lift(\htpy{0},F)(x,F(q)):x \to \lift(F(q))(x)$ over $F(q)$ where $x' , \lift(F(q))(x) \in X_{F(x')}$. We then have an isomorphism $q': x' \to \lift(F(q))(x) $ in $X_{F(x')}$, but that makes $q'$ an identity morphism so $x' = \lift(F(q))(x)$. It straightforward to show that $\alpha$ is functorial. Now for $\beta: [F] \to X$. $\beta(y,(X_{y},x)) = x$ and for $(p,\lift(p)): (y,(X_{y},x)) \to (y',(X_{y'},x'))$, $\beta(p,\lift(p)) = \lift(\htpy{0},F)(x,p)$. $\beta$ is functorial since $F$ is normal isofibration so identities lift to identities and we have $\lift(\htpy{0},F)(x,p' \circ p) = \lift(\htpy{0},F)(\lift(p)(x),p') \circ \lift(\htpy{0},F)(x,p)$ since there exist at most one isomorphism between $x$ and $\lift(p' \circ p)(x) = \lift(p') \circ \lift(p)(x)$ over $p' \circ p$. \\

First we show $\alpha \circ \beta = \id_{[F]}$. For $(y,(X_{y},x)) \in [F]$, since $F(x) = y$, we have $$\alpha \circ \beta(y,(X_{y},x)) = \alpha(x) = (y,(X_{y},x))$$ and for a morphism $(p,\lift(p)) \in [F]$, $$\alpha \circ \beta(p,\lift(p)) = \alpha(\lift(\htpy{0},F)(x,p)) = (F(\lift(\htpy{0},F)(x,p)), \lift(F(\lift(\htpy{0},F)(x,p)))) = (p,\lift(p))$$

For showing $\beta \circ \alpha = \id_{X}$, for $x \in X$, $$\beta \circ \alpha(x) = \beta(F(x),(X_{F(x)},x)) = x$$ for a morphism $q:x \to x'$ since $q$ is the unique morphism from $x$ to $x'$ over $p$, then $q = \lift(\htpy{0},F)(x,F(q))$ so we have $$ \beta \circ \alpha(q) = \beta(F(q),\lift(F(q))) =  \lift(\htpy{0},F)(x,F(q)) = q$$ So $\alpha$ and $\beta$ are isomorphism pairs and it is straightforward to show that they are functors over $Y$, hence $F$ is the pullback of $\pr{U}$ over $\nm{F}$.\\

Given the following pullback square 

\[\begin{tikzcd}
	{[f]} &&& {U^{grpd}_{*}} \\
	\\
	\\
	Y &&& {U^{grpd}}
	\arrow[from=1-1, to=1-4]
	\arrow["{F}"', from=1-1, to=4-1]
	\arrow["\lrcorner"{anchor=center, pos=0.125}, draw=none, from=1-1, to=4-4]
	\arrow["{\pr{U}}", from=1-4, to=4-4]
	\arrow["{f}"', from=4-1, to=4-4]
\end{tikzcd}\]

with $\alpha:[F] \cong [f]$ over $Y$, we seek to show that $\nm{F} \cong f$. For $y \in \ob Y$, we can restrict $\alpha$ to an isomorphism $\alpha'_{y}:[F]_{y} \to [f]_{y}$. For $(y,(X_{y},x))$, $\alpha'_{y}(y,(X_{y},x)) = (y,(f(y),\alpha_{y}(x)))$ and for $(\id_{y},\id_{X_{y}})$, $\alpha'_{y}(\id_{y},\id_{X_{y}})= (\id_{y},\id_{f(y)})$. We can induce from this an isomorphism $\alpha_{y}: X_{y} \cong f(y)$. Now for $p:y \to y'$, we want to exhibit the following commutative square:

\[\begin{tikzcd}
	{X_{y}} &&& {f(y)} \\
	\\
	\\
	{X_{y'}} &&& {f(y')}
	\arrow["{\alpha_{y}}", from=1-1, to=1-4]
	\arrow["{\nm{F}(p)}"', from=1-1, to=4-1]
	\arrow["{f(p)}", from=1-4, to=4-4]
	\arrow["{{\alpha_{y'}}}"', from=4-1, to=4-4]
\end{tikzcd}\]

Observe that for $x \in X_{y}$, we have an isomorphism $$\nm{F}(p): (y,(X_{y},x)) \to (y', (X_{y'}, \nm{F}(p)(x)) )$$ in $[F]$. Applying $\alpha$ gives us an isomorphism $$\alpha(\nm{F}(p)): (y,(f(y),\alpha_{y}(x))) \to (y', (f(y'), \alpha_{y'}(\nm{F}(p)(x))) )$$ where $\alpha(\nm{F}(p)) = f(p)$ since $\alpha$ is a morphism over $Y$. So $f(p) \circ \alpha_{y}(x) = \alpha_{y'} \circ \nm{F}(p)(x)$ for all $x \in X_{y}$ thus $f(p) \circ \alpha_{y} = \alpha_{y'} \circ \nm{F}(p)$. Therefore $\nm{F} \cong f$ which completes the proof.

\end{proof}

Since $\pr{U}$ is a fibration, so are the functors $\pr{U} \times U^{grpd}$ and $U^{grpd} \times \pr{U}$. We now have a new fibration $\pr{U} \times U^{grpd} \to U^{grpd} \times \pr{U} $ which we can explicitly characterize using our previous characterization of exponent fibrations in Section \ref{soasec}:

\begin{itemize}

\item The objects over a pair $(X,Y)$ are functors $s_{X,Y}: U^{grpd}_{*,X}\times \{Y\} \to \{X\} \times U^{grpd}_{*,Y}$. Since $U^{grpd}_{*,X} = X$ by definition for every $X \in \ob U^{grpd}$, this means we have a functor $s_{X,Y}:X \to Y$.\\

\item Morphisms over $(i,i'):(X,Y) \cong (X',Y')$ are generalized natural transformations $\sigma_{(i,i')}: s_{X,Y} \rightsquigarrow s_{X',Y'}$. \\

\item This means that for a morphism $p_{i,i'} = (i:(X,x) \to (X',x'),i')$ over $(i,i')$, we have a morphism $\sigma_{(i,i'),p_{i,i'}}:(i,i':(Y,s_{X,Y}(x)) \to (Y',s_{X',Y'}(y)))$. This suffices to characterize $\sigma_{(i,i')}$ $i$ and $i'$ are isomorphisms between discrete groupoid assemblies.\\

\item In addition when $\sigma_{(i,i')}:s_{X,Y} \to s_{X',Y'}$ exists, it is unique since there is at most one morphism between objects in $U^{grpd}_{*}$.\\

\item Hence the unique morphism over $(i,i')$ between $s_{X,Y}$ and $s_{X',Y'}$ exists when $s_{X,Y} \circ i'  = i \circ s_{X',Y'}$ and is simply just the pair $(i,i')$.\\

\end{itemize}

We define $\Equ(U^{grpd}) \subseteq (\pr{U} \times U^{grpd} \to U^{grpd} \times \pr{U} )$ as the full subcategory of groupoid equivalences. This makes the description of $\Equ(U^{grpd})$ analogous to the description of the groupoid $\I \to U^{grpd}$ where instead of equivalences of discrete groupoids, we have the isomorphisms of $U^{grpd}$ which are isomorphisms of discrete groupoids. In fact, we can construct a canonical functor $univ: (\I \to U^{grpd}) \to \Equ(U^{grpd})$ which is the inclusion. With that in mind, we have the following definition:

\begin{defn}

We say $U^{grpd}$ is {\bf univalent} when $univ: (\I \to U^{grpd}) \to  \Equ(U^{grpd})$ is an equivalence over $U^{grpd} \times U^{grpd}$.
\end{defn}

\begin{lem}
\label{disuniv}

The groupoid equivalences between discrete groupoid assemblies are precisely the isomorphisms, hence $U^{grpd}$ is univalent.

\end{lem}
\begin{proof}

When we have a groupoid equivalence $F: X \to Y$ with the data $\alpha: F^{i} \circ F \cong \id_{X}$ and $\beta: F \circ F^{i} \cong \id_{Y}$, when $X$ and $Y$ are discrete then $\alpha$ and $\beta$ are fibred identity morphisms hence $F^{i} \circ F = \id_{X}$ and $F \circ F^{i} = \id_{Y}$ making $F$ an isomorphism.\\

Univalence of $U^{grpd}$ follows from the fact that the descriptions of $\Equ(U^{grpd})$ and $\I \to U^{grpd}$ are exactly the same.

\end{proof}

Now we wish to consider the univalent universe of modest assemblies. Recall from Lemma \ref{modu} that any modest assembly $(X,\phi)$ is isomorphic to an assembly of the form $(\{\phi(x)| x \in X\}, \id)$. Hence the underlying sets of the modest assemblies are bound by the cardinality of $\pow{\pow{A}}$. Given a Grothendieck universe $V_{\pow{\pow{A}}}$ which contains $\pow{\pow{A}}$, we have the object classifier of $V_{\pow{\pow{A}}}$-assemblies $V'_{\pow{\pow{A}}}$, which in turn gives us the object classifier $V'^{grpd}_{\pow{\pow{A}}}$. Every discrete fibration with modest assembly fibers, which we will refer to as a {\bf modest fibration}, is classified by $V'^{grpd}_{\pow{\pow{A}}}$. \\

Recall that the modest assemblies are the local objects of the localization modality $\Loc_{1 + 1}(!_{\nabla(2)}, !_{\I})$. Here we will characterize modalities of a similar form to $\Loc_{1 + 1}(!_{\nabla(2)}, !_{\I})$:\\

\begin{defn}

We call a localization modality $\Loc_{I}(J)$ a {\bf nullification modality} if for all $i \in I$, we have $J_{i}:U_{i} \to 1$.

\end{defn}

\begin{lem}

Suppose $\Loc_{I}(J)$ is a nullification modality and each $U_{i}$ is inhabited. If $X$ is $(I,J)$-local and $i:T \mono X$ is a monomorphism, then $T$ is $(I,J)$-local.

\end{lem}
\begin{proof}

The monomorphisms of groupoid assemblies are precisely those functors that are monomorphisms on objects and morphisms. Take note of the following commutative square:

\[\begin{tikzcd}
	{T} &&& {U_{i} \to T} \\
	\\
	\\
	{X} &&& {U_{i} \to X}
	\arrow["{- \circ J_{i}}", from=1-1, to=1-4]
	\arrow["{i}"', from=1-1, to=4-1]
	\arrow["{i \circ -}", from=1-4, to=4-4]
	\arrow["{- \circ J_{i}}"', from=4-1, to=4-4]
\end{tikzcd}\]

we will show that this square is a pullback. Suppose we have the commutative diagram

\[\begin{tikzcd}
	Z \\
	& T &&& {U_{i} \to T} \\
	\\
	\\
	& X &&& {U_{i} \to X}
	\arrow["{d_{0}}", from=1-1, to=2-5]
	\arrow["{d_{1}}"', from=1-1, to=5-2]
	\arrow["{{- \circ J_{i}}}", from=2-2, to=2-5]
	\arrow["{{i}}"', from=2-2, to=5-2]
	\arrow["{{i \circ -}}", from=2-5, to=5-5]
	\arrow["{{- \circ J_{i}}}"', from=5-2, to=5-5]
\end{tikzcd}\]

we will define a functor $d:Z \to T$. For an object $z \in Z$, we have $d_{0}(z): U_{i} \to T$ such that $i \circ d_{0}(z) = \const_{d_{1}(z)}$. Since $i$ is a monomorphism, for any $u, u' \in \ob U_{i}$, $d_{0}(z)(u) = d_{0}(z)(u')$ so $d_{0}(z)$ is constant. Since there exists a $u_{*} \in \ob U_{i}$, we have a $d_{0}(z)(u_{*}) \in \ob T$ with $i(d_{0}(z)(u_{*})) = d_{1}(z)$. So we set $d(z) = d_{0}(z)(u_{*})$. Suppose we have a morphism $\zeta: z \to z'$, this induces a natural transformation $d_{0}(\zeta): d_{0}(z) \to d_{0}(z')$ so that for all $u \in \ob U_{i}$, $i(d_{0}(\zeta)_{u}) = d_{1}(\zeta)$. Since $i$ is mono, $d_{0}(\zeta)$ is constant. Since $U_{i}$ is inhabited, there exists a $d_{0}(\zeta)_{u'}$, so we set $d(\zeta) = d_{0}(\zeta)_{u'}$. $d$ is a well-defined functor since $d_{0}$ picks out constant functors and natural transformations and $d_{0}$ is a functor. In addition we have $d_{1} = i \circ d$ and $d_{0} = (- \circ J_{i}) \circ d$ and is unique because of $d_{0}$. Thus the above square is a pullback square so since the bottom morphism is an isomorphism so is the top morphism. Therefore, $T$ is $(I,J)$-local.

\end{proof}

We set $V^{mod}$ as the full subgroupoid of $V'^{grpd}_{\pow{\pow{A}}}$ of modest assemblies. We can also regard $V^{mod}$ as the reflective subuniverse of $V'^{grpd}_{\pow{\pow{A}}}$ induced by $\Loc_{1 + 1}(!_{\nabla(2)}, !_{\I})$. We derive $\pr{V}:V_{*}^{mod} \to V^{mod}$ from the following pullback:

\[\begin{tikzcd}
	{V^{mod}_{*}} &&& {V'^{grpd}_{\pow{\pow{A}}},*} \\
	\\
	\\
	{V^{mod}} &&& {V'^{grpd}_{\pow{\pow{A}}}}
	\arrow[from=1-1, to=1-4]
	\arrow["{\pr{V}}"', from=1-1, to=4-1]
	\arrow["\lrcorner"{anchor=center, pos=0.125}, draw=none, from=1-1, to=4-4]
	\arrow["{{\pr{V_{\pow{\pow{A}}}}}}", from=1-4, to=4-4]
	\arrow[tail, from=4-1, to=4-4]
\end{tikzcd}\]

\begin{lem}

$\pr{V}$ classifies the modest fibrations and $V^{mod}$ is univalent.

\end{lem}
\begin{proof}

It has already been observed that all modest fibrations are classified by $\pr{V_{\pow{\pow{A}}}}$ since the cardinality of the fibers of a modest fibration are bound by $\pow{\pow{A}}$. $\pr{V}$ is simply the restriction of $\pr{V_{\pow{\pow{A}}}}$ to its modest fibers hence $\pr{V}$ classifies the modest fibrations. Since $V^{mod}$ is a full subcategory, $V^{mod}$ has as its isomorphisms all of the isomorphisms between modest assemblies. In addition, it is straightforward to show that $\Equ(V^{mod})$ is just the full subcategory of $\Equ(V'^{grpd}_{\pow{\pow{A}}})$ of isomorphisms between modest assemblies. Hence $\Equ(V^{mod})$ is still identical to $\I \to V^{mod}$. Therefore $V^{mod}$ is univalent.

\end{proof}

\begin{defn}

We say that $\pr{U}$ is {\bf impredicative} if for any fibration $F:X \to Y$, if $\pr{U}$ classifies $G:Z \to X$ then $\pr{U}$ classifies $\Pi_{F}(G)$. 

\end{defn}

We seek to show that $\pr{V}$ is impredicative. To do this we will use the fact that the modest assemblies are characterized by the modality $\Loc_{1 + 1}(!_{\nabla(2)}, !_{\I})$. Recall that a localization modality $\Loc_{I}(J)$ is the restriction of a functor $R_{[\Loc(I,J)]}$ which is the right functor of an algebraic weak factorization system: $([\Loc(I,J)], R_{[\Loc(I,J)]},L_{[\Loc(I,J)]})$.

\begin{defn}

We call a fibration $F:X \to Y$ an {\bf (I,J)-local fibration} if is also an $R_{[\Loc(I,J)]}$-map.

\end{defn}

\begin{lem}
\label{ijlocalfib}

Suppose $\Loc_{I}(J)$ is a nullification modality and each $U_{i}$ is inhabited. Then the $R_{[\Loc(I,J)]}$-maps are precisely those functors with $(I,J)$-locals fibers.

\end{lem}
\begin{proof}
Since the $R$-maps of an algebraic weak factorization system are preserved by pullbacks, then $R_{[\Loc(I,J)]}$-maps have $(I,J)$ local fibers. So it suffices to show the opposite direction. Furthermore, since each $U_{i}$ is inhabited then $1 +_{U_{i}} 1 \cong 1$ so it suffices to show that the functors with $(I,J)$-local fibers have a functorial right lifting property against each $!_{U_{i}}$. Given a functor $F:X \to Y$ with $(I,J)$-local fibers observe that a square

\[\begin{tikzcd}
	{U_{i}} &&& {X} \\
	\\
	\\
	{1} &&& {Y}
	\arrow["{c}", from=1-1, to=1-4]
	\arrow["{J_{i}}"', from=1-1, to=4-1]
	\arrow["{F}", from=1-4, to=4-4]
	\arrow["{\nm{y}}"', from=4-1, to=4-4]
\end{tikzcd}\]

is precisely the square

\[\begin{tikzcd}
	{U_{i}} &&& {X_{y}} \\
	\\
	\\
	{1} &&& {1}
	\arrow["{c}", from=1-1, to=1-4]
	\arrow["{J_{i}}"', from=1-1, to=4-1]
	\arrow[from=1-4, to=4-4]
	\arrow[ from=4-1, to=4-4]
\end{tikzcd}\]

So we have a lift $\nm{\lift(J_{i})(c,y)}:1 \to X$ such that $c = \const_{\lift(J_{i})(c,y)}$. Suppose we have a morphism $p:y \to y'$ in $Y$, functors $c:U_{i} \to X_{y}$ and $c':U_{i} \to X_{y'}$ and a natural transformation into $X$, $\gamma: c \to c'$ where for each $u \in U_{i}$ $\gamma_{u}$ is over $p$, since there exists $u^{*} \in \ob U_{i}$ and $c = \const_{\lift(J_{i})(c,y)}$ and $c' = \const_{\lift(J_{i})(c',y')}$, then we have a constant natural transformation $\const_{\gamma_{u^{*}}}:c \to c'$ over $p$ which induces natural transformations $$\gamma \circ (\const_{\gamma_{u^{*}}})^{-1}: c' \to c'$$ and $$\const_{\gamma_{u^{*}}}^{-1} \circ \gamma :c \to c$$ over $\id_{y'}$ and $\id_{y}$ respectively. This means that the above natural transformations are in $X_{y'}$ and $X_{y}$ respectively and since both $X_{y'}$ and $X_{y}$ are $(I,J)$ local both $\gamma \circ (\const_{\gamma_{u^{*}}})^{-1}$ and $\const_{\gamma_{u^{*}}} \circ \gamma^{-1}$ are constant. We also know that $$(\gamma \circ (\const_{\gamma_{u^{*}}})^{-1})_{u^{*}} = \gamma_{u^{*}} \circ \gamma_{u^{*}}^{-1} = \id_{\lift(J_{i})(y')}$$ and $$(\const_{\gamma_{u^{*}}}^{-1} \circ \gamma)_{u^{*}} = \gamma_{u^{*}}^{-1} \circ \gamma_{u^{*}} = \id_{\lift(J_{i})(y)}$$
Hence $\const_{\gamma_{u^{*}}}^{-1} \circ \gamma = \id_{c}$ and $\gamma \circ (\const_{\gamma_{u^{*}}})^{-1} = \id_{c'}$ therefore $\gamma = \const_{\gamma_{u^{*}}}$. Hence we set $\lift(J_{i})(\gamma,p) = \gamma_{u^{*}}$. $\lift(J_{i})$ is functorial since each $c$ and $\gamma:c \to c'$ is constant. Hence $F$ has the functorial right lifting property against $J_{i}$. Therefore $F$ is a $R_{[\Loc(I,J)]}$-map.

\end{proof}


\begin{lem}
\label{nulldfib}

When $\Loc_{I}(J)$ is a nullification modality, $R_{[\Loc(I,J)]}$-maps are preserved by dependent products along fibrations.

\end{lem}
\begin{proof}

Suppose we have an $(I,J)$-local fibration $F:X \to Y$ and a fibration $G:Y \to Z$ and suppose we have a diagram 

\[\begin{tikzcd}
	{U^{*}_{i}} &&& {\Pi_{G}(F)} \\
	\\
	\\
	{1} &&& {Z}
	\arrow["{c}", from=1-1, to=1-4]
	\arrow["{!_{U^{*}_{i}}}"', from=1-1, to=4-1]
	\arrow[ from=1-4, to=4-4]
	\arrow["{\nm{z}}"', from=4-1, to=4-4]
\end{tikzcd}\]

where either $U^{*}_{i} = U_{i}$ or $U^{*}_{i} = 1 +_{U_{i}} 1$. We know that $c$ maps $U^{*}_{i}$ into the fiber over $z$. This means that for a morphism $p_{i}: u_{i} \to u'_{i} \in U_{i}$, $c(p_{i}): c(u_{i}) \to c(u'_{i})$ is a natural transformation. We seek to construct an element $\lift(!_{U^{*}_{i}})(c,z):Y_{z} \to X$. Given $y \in \ob Y_{z}$ we have a square:

\[\begin{tikzcd}
	{U^{*}_{i}} &&& {X} \\
	\\
	\\
	{1} &&& {Y}
	\arrow["{c_{y}}", from=1-1, to=1-4]
	\arrow["{!_{U^{*}_{i}}}"', from=1-1, to=4-1]
	\arrow["F", from=1-4, to=4-4]
	\arrow["{\nm{y}}"', from=4-1, to=4-4]
\end{tikzcd}\]

where $c_{y}(u_{i}) = c(u_{i})(y)$ for $u_{i} \in \ob U^{*}_{i}$ and $c_{y}(p_{i}) = c(p_{i})_{y}$ for $p_{i} \in \mor U^{*}_{i}$. Let $l$ be the lift functor for $F$. This gives us a lift $l(c_{y}):1 \to X$ such that $c_{y} = \const_{l(c_{y})}$. So we set $\lift(!_{U^{*}_{i}})(c,z)(y) = l(c_{y})$. Given a morphism $q:y \to y'$, we have a natural transformation $c_{q}: c_{y} \to c_{y'}$ where $c_{q,u_{i}} = c(u_{i})(q)$. This gives us a morphism $l(c_{q}):l(c_{y}) \to l(c_{y'})$ where $c_{q} = \const_{l(c_{q})}$ and so we set $\lift(!_{U^{*}_{i}})(c,z)(q) = l(c_{q})$. It is straightforward to show that $\lift(!_{U^{*}_{i}})(c,z)(q)$ is a functor. In addition, $c = \const_{\lift(!_{U^{*}_{i}})(c,z)(q)}$ by construction.\\

Suppose we have an isomorphism $\zeta: z \to z'$,  functors $c: U^{*}_{i} \to \Pi_{G}(F)_{z}$ and $c': U^{*}_{i} \to \Pi_{G}(F)_{z'}$ and a natural transformation $\gamma: c \to c'$, we seek to construct a generalized natural transformation $\lift(!_{U^{*}_{i}})(\gamma,\zeta):\lift(!_{U^{*}_{i}})(c,z) \to \lift(!_{U^{*}_{i}})(c',z')$. We can treat $\gamma$ as a generalized natural transformation $\gamma'$ where for $p_{i} \in \mor U^{*}_{i}$, $\gamma'_{p_{i}} = c'(p_{i}) \circ \gamma_{\dom(p_{i})}$. Given $\sigma:y \to y' \in (\mor Y)_{\zeta}$, we have a generalized natural transformation $\gamma'_{\sigma}: c_{y} \to c'_{y'}$ where for $p_{i} \in \mor U^{*}_{i}$, $\gamma'_{\sigma,p_{i}} = \gamma'_{p_{i},\sigma}: c(\dom(p_{i}))(y) \to c'(\cod(p_{i}))(y')$. This induces a morphism $\lift(!_{U^{*}_{i}})(\gamma,\zeta)(\sigma) = l(\gamma'_{\sigma})$ where $\gamma'_{\sigma} = \const_{\lift(!_{U^{*}_{i}})(\gamma,\zeta)(\sigma)}$ and $l(\gamma'_{\sigma}): l(c_{y}) \to l(c'_{y})$. It is straightforward to show that $\lift(!_{U^{*}_{i}})(\gamma,\zeta)$ is a generalized natural transformation using the fact that $\gamma'$ and each $\gamma'_{p_{i}}$ is a generalized natural transformation. Plus we have $\gamma' = \const_{\lift(!_{U^{*}_{i}})(\gamma,\zeta)}$ by construction. \\

Since $c: U^{*}_{i} \to \Pi_{G}(F)_{z} = \const_{l(c)}$ then $\id_{c} = \const_{\id_{l(c)}}$ so $\lift(!_{U^{*}_{i}})(\id_{c},\id_{z}) = \id_{l(c)}$. We can show a similar thing for composition. Thus $\lift(!_{U^{*}_{i}})$ is functorial, therefore $\Pi_{G}(F)$ is $R_{[\Loc(I,J)]}$-map.

\end{proof}

\begin{cor}

The $(1 + 1, \{!_{\nabla(2)}, !_{\I}\})$-local fibrations are precisely the modest fibrations and are preserved by dependent products along fibrations.

\end{cor}

\begin{cor}
\label{impreduniv}

$\pr{V}$ is impredicative, thus making $V^{mod}$ an impredicative univalent universe.

\end{cor}

\subsection{$\pi$-tribes}

At last we will fulfill the titular purpose of the paper which is to exhibit $\GrpA{A}$ as a model of type theory. There are several categorical good axiomatizations of a model of type theory such as in \cite{GAMBINO_2021}, we choose to use the notion of $\pi$-tribe in \cite{joyal2017notesclanstribes}. The reason for this is to establish certain properties of $\HGA{A}$ as will be used later in the paper.\\

First we define $\pi$-tribe:

\begin{defn}

We call a category $\cC$ a {\bf $\bf \pi$-tribe} if we have the following structure on $\cC$:

\begin{itemize}

\item There exists an class of fibrations $\mathcal{E} \subseteq \cC$ containing isomorphisms and maps $X \to 1$ for every object $X \in \cC$ as well as being closed under composition and pullbacks.

\item Dependent products of fibrations along fibrations exists and is still a fibration.

\item $\cC$ has a weak factorization system where the right maps are the fibrations and the pullback of a left map along a fibration is still a left map.

\item The dependent product functor along a fibration takes left maps between fibrations to left maps between fibrations.

\end{itemize}

\end{defn}

For $\GrpA{A}$, it already satisfies 3 out of 4 conditions for being a $\pi$-tribe, so it suffices to show the final condition:

\begin{lem}

For a normal isofibration $F: X \to Y$, the dependent product functor $\Pi_{F}: \Fib(Y) \to  \Fib(X)$ preserves strong deformation sections between normal isofibrations.

\end{lem}
\begin{proof}

Suppose we have a strong deformation section between fibrations over $Y$, $N: G \to H$, then that gives us a morphism $M: H \to G$ with natural transformations $\alpha: M \circ N \to \id_{G}$ and $\beta: N \circ M \to \id_{H}$ satisfying the following properties:

\begin{itemize}

\item $\beta \cdot N = N \cdot \alpha$

\item $\alpha = \id_{\id_{G}}$, thus $M \circ N = \id_{G}$

\end{itemize}

Recall that the functor $\Pi_{F}(N): \Pi_{F}(G) \to \Pi_{F}(H)$ sends an object $(m,T)$ to $(m, N \circ T)$ and a morphism $(f,\eta)$ to $(f,[w \mapsto N(\eta(w))])$. Recall that we can regard $\alpha$ and $\beta$ as generalized natural transformations, so we have obvious natural transformations $\alpha \cdot (-): \Pi_{F}(N) \circ \Pi_{F}(M) \to \id_{\Pi_{F}(G)}$ and $\beta \cdot (-): \Pi_{F}(M) \circ \Pi_{F}(N) \to \id_{\Pi_{F}(H)}$ satisfying the above conditions where $\alpha \cdot (m,T) = (\id_{m}, \alpha \cdot T):(m, N \circ M \circ T) \to (m,T)$ and $\beta \cdot (m,T) = (\id_{m}, \beta \cdot T):(m, N \circ M \circ T) \to (m,T)$. Therefore, $\Pi_{F}(N)$ is a strong deformation section.

\end{proof}

\begin{cor}
\label{pitribe}

$\GrpA{A}$ equipped with the class of normal isofibrations and the algebraic weak factorization system $(\Fib, R_{\Fib},L_{\Fib})$ is a $\pi$-tribe hence it is a model of Martin-L\"{o}f type theory with $\Sigma$, $\Pi$, and $\Id$ types.

\end{cor}

We also want to show that reflective subcategories of nullification modalities give a model of type theory. Given a pair $(I,J)$ where $I$ is a discrete groupoid assembly and $J: U \to V$ is a morphism between fibrations over $I$, we can construct an algebraic weak factorization system $([\Fib(I,J)], R_{[\Fib(I,J)]},L_{[\Fib(I,J)]})$ where the right maps are precisely the morphisms that are normal isofibrations and right maps of $([\Loc(I,J)], R_{[\Loc(I,J)]},L_{[\Loc(I,J)]})$ (i.e. the (I,J)-local fibrations). Hence the fibrant objects are the $(I,J)$-local objects.

\begin{lem}
When $M$ is an $(I,J)$-local object, then $(\I \to M) \to M\times M$ is an $(I,J)$-local fibration.

\end{lem}
\begin{proof}

Recall that by Lemma \ref{jstarlocal}, $M$ is also $(I,J^{*})$ local where $J^{*}: \sum_{i \in I} V_{i} +_{U_{i}} V_{i} \to \sum_{i \in I} V_{i}$ over $I$. For $i \in I$, let us have $K_{i}: X_{i} \to V_{i}$ such that either $K_{i} = J_{i}$ or $K_{i} = J^{*}_{i}$. A square of the form:

\[\begin{tikzcd}
	{X_{i}} &&& {\I \to M} \\
	\\
	\\
	{V_{i}} &&& {M \times M}
	\arrow["{\alpha}", from=1-1, to=1-4]
	\arrow["{K_{i}}"', from=1-1, to=4-1]
	\arrow["{- \circ [\htpy{0},\htpy{1}]}",from=1-4, to=4-4]
	\arrow["{c}"', from=4-1, to=4-4]
\end{tikzcd}\]

corresponds to having functors $c_{0}, c_{1}:V_{i} \to M$ and a natural isomorphism $\alpha: a_{0} \to a_{1}$ such that $c_{0} \circ K_{i} = a_{0}$ and $c_{1} \circ K_{i} = a_{1}$. Since $M$ is $(I,J)$-local, $\alpha$ corresponds uniquely to a natural transformation $\gamma:c_{0} \to c_{1}$ such that $\gamma \cdot K_{i} = \alpha$. Given natural transformations $\gamma_{0}: c_{0} \to c'_{0}$ and $\gamma_{1}: c_{1} \to c'_{1}$ and a commutative square

\[\begin{tikzcd}
	{a_{0}} &&& {a_{1}} \\
	\\
	\\
	{a'_{0}} &&& {a'_{1}}
	\arrow["\alpha", from=1-1, to=1-4]
	\arrow["{\alpha_{0}}"', from=1-1, to=4-1]
	\arrow["{\alpha_{1}}", from=1-4, to=4-4]
	\arrow["{\alpha'}"', from=4-1, to=4-4]
\end{tikzcd}\]

such that $\gamma_{0} \cdot K_{i} = \alpha_{0}$ and $\gamma_{1} \cdot K_{i} = \alpha_{1}$. Since $M$ is $(I,J)$ local this corresponds unique to a commutative square
\[\begin{tikzcd}
	{c_{0}} &&& {c_{1}} \\
	\\
	\\
	{c'_{0}} &&& {c'_{1}}
	\arrow["\gamma", from=1-1, to=1-4]
	\arrow["{\gamma_{0}}"', from=1-1, to=4-1]
	\arrow["{\gamma_{1}}", from=1-4, to=4-4]
	\arrow["{\gamma'}"', from=4-1, to=4-4]
\end{tikzcd}\]

where $\gamma' \cdot K_{i} = \alpha'$ and $\gamma\cdot K_{i} = \alpha$. Functoriality of this lift is induced by the isomorphism pair of $(- \circ K_{i}): (V_{i} \to M)  \to (X_{i} \to M)$. Thus, $(\I \to M) \to M\times M$ is an $R_{[\Loc(I,J)]}$ map. By Lemma \ref{pathspace}, $(\I \to M) \to M\times M$ is a normal isofibration. Therefore, $(\I \to M) \to M\times M$ is an $(I,J)$-local fibration.

\end{proof}

\begin{lem}

Given an $L_{[\Fib(I,J)]}$-map between $(I,J)$ local objects $i:N \to M$, $i$ is a strong deformation section.

\end{lem}
\begin{proof}

This follows from the following two diagrams:

\[\begin{tikzcd}
	N &&& N \\
	\\
	\\
	M
	\arrow["{\id_{N}}", from=1-1, to=1-4]
	\arrow["i"', from=1-1, to=4-1]
	\arrow["{i^{*}}"', dashed, from=4-1, to=1-4]
\end{tikzcd}
\quad\quad
\begin{tikzcd}
	N &&& {(\I \to M)} \\
	\\
	\\
	M &&& {M \times M}
	\arrow["{{(i \circ -) \circ \id^{N}_{-}}}", from=1-1, to=1-4]
	\arrow["i"', from=1-1, to=4-1]
	\arrow[from=1-4, to=4-4]
	\arrow["\beta"{description}, dashed, from=4-1, to=1-4]
	\arrow["{{[\id_{M}, i \circ i^{*}]}}"', from=4-1, to=4-4]
\end{tikzcd}\]

where $\id^{N}_{-}: N \to (\I \to N)$ takes $n$ to the identity morphism on $n$ and $(i \circ -) : (\I \to N) \to (\I \to M)$ is post composition. Thus we have $i^{*} \circ i = \id_{N}$, $\beta: i \circ i^{*} \cong \id_{M}$ and $ \beta \cdot i = i \cdot \id_{\id_{N}}$. Setting $\alpha = \id_{\id_{N}}: i^{*} \circ i  \cong \id_{N}$, This makes $i$ a strong deformation section.

\end{proof}

\begin{cor}

The $L_{[\Fib(I,J)]}$-map between $(I,J)$ local objects are precisely the strong deformation sections between $(I,J)$ local objects.

\end{cor}
\begin{proof}
This holds from the pervious lemma and the fact that any strong deformation section is an $L_{[\Fib(I,J)]}$-map since any $R_{[\Fib(I,J)]}$-map is a normal isofibration.

\end{proof}

\begin{lem}
\label{lpt2}

In the category $\GrpA{A}_{(I,J)}$, pullbacks of $L_{[\Fib(I,J)]}$ maps along $R_{[\Fib(I,J)]}$ maps are still $L_{[\Fib(I,J)]}$ maps.

\end{lem}
\begin{proof}

Any $L_{[\Fib(I,J)]}$ maps in $\GrpA{A}_{(I,J)}$ will be a strong deformation section and any $R_{[\Fib(I,J)]}$ map is a fibration, thus the pullback of an  $L_{[\Fib(I,J)]}$ map along an $R_{[\Fib(I,J)]}$ map in $\GrpA{A}_{(I,J)}$ will be a strong deformation section between $(I,J)$-local objects since $\GrpA{A}_{(I,J)}$ is closed under finite limits. Therefore, it will be a $L_{[\Fib(I,J)]}$ map.

\end{proof}

\begin{lem}
\label{lpt1}

When $\Loc_{I}(J)$ is a nullification modality, the dependent product of $(I,J)$-local fibrations along $(I,J)$-local fibrations are $(I,J)$-local fibrations. Hence dependent of $(I,J)$-local fibrations along $(I,J)$-local fibrations exists in $\GrpA{A}_{(I,J)}$.

\end{lem}
\begin{proof}

By Lemma \ref{nulldfib}, dependent products of $(I,J)$-local fibration along any fibration is a $(I,J)$-local fibration. Furthermore, when we have an $(I,J)$-local fibration $F: X \to Y$ with $Y$ local, then $X$ is local so the dependent product of $(I,J)$-local fibrations between $(I,J)$-local fibrations when all objects are $(I,J)$-local result in $(I,J)$-local fibrations between $(I,J)$-local objects. Hence dependent of $(I,J)$-local fibrations along $(I,J)$-local fibrations exists in $\GrpA{A}_{(I,J)}$.

\end{proof}

\begin{cor}

When $\Loc_{I}(J)$ is a nullification modality, $\GrpA{A}_{(I,J)}$ equipped with the class of $(I,J)$-local fibrations and the algebraic weak factorization system $([\Fib(I,J)], R_{[\Fib(I,J)]},L_{[\Fib(I,J)]})$ restricted to $\GrpA{A}_{(I,J)}$ is a $\pi$-tribe.

\end{cor}
\begin{proof}
 The first 3 conditions for $\pi$-tribes are proven by Lemmas \ref{lpt1} and \ref{lpt2}, and the existence of the algebraic weak factorization system for $(I,J)$-local fibrations. The final condition follows from the fact that dependent products are constructed in $\GrpA{A}_{(I,J)}$ the same way they are in $\GrpA{A}$ and the $L_{[\Fib(I,J)]}$ maps are precisely the strong deformation sections in $\GrpA{A}_{(I,J)}$.
\end{proof}

\section{Homotopy Category of Groupoid Assemblies}
In this section, we expound on the bicategorical properties of $\GrpA{A}$: first we characterize the homotopy category of groupoid assemblies, then we construct a bicoreflective subcategory of $\GrpA{A}$, $\pGrA{A}$, via the bifunctor $\Part(-)$. We will show that both $\GrpA{A}$ and its subcategory has finite bilimits and bicolimits and are both cartesian closed up to homotopy. Finally we will induce a right biadjoint $\Clus(-)$ of $\Part(-)$ and prove that the homotopy category of the $0$-types of $\pGrA{A}$ is $\RT{A}$.\\

\subsection{$\HGA{A}$ Preliminaries}

In this section, we will construct the homotopy category of groupoid assemblies $\HGA{A}$ which we will show is the localization of $\GrpA{A}$ with respect to the groupoid equivalences. We will show that $\HGA{A}$ has finite bilimits and bicolimits, which are understood has homotopy limits and colimits,  and we will characterize the -1 and 0 types of $\HGA{A}$.\\


Here we will borrow definitions from \cite{joyal2017notesclanstribes}. 

\begin{defn}

Given a groupoid assembly $X$, we call a quadruple $(PX,\htpy{0}^{X}, \htpy{1}^{X}, \sigma_{X})$ a {\bf path object} when we have the following factorization:

\[\begin{tikzcd}
	X &&&& {X \times X} \\
	\\
	&& PX
	\arrow["{\Delta_{X}}", from=1-1, to=1-5]
	\arrow["{\sigma_{X}}"', from=1-1, to=3-3]
	\arrow["{(\htpy{0}^{X},\htpy{1}^{X})}"', from=3-3, to=1-5]
\end{tikzcd}\]

where $\sigma_{X}$ is a strong deformation section and $(\htpy{0}^{X},\htpy{1}^{X})$ is a normal isofibration.

\end{defn}

\begin{lem}

$(\I \to X)$ induces a path object on $X$.

\end{lem}
\begin{proof}

We set $\sigma_{X}:X \to (\I \to X)$ as taking an object $x \in \ob X$ to $\id_{x}:x \to x$ and a morphism in $X$, $p:x \to x'$, to the square:

\[\begin{tikzcd}
	{x} &&& {x'} \\
	\\
	\\
	{x} &&& {x'}
	\arrow["p", from=1-1, to=1-4]
	\arrow["{\id_{x}}"', from=1-1, to=4-1]
	\arrow["{\id_{x'}}", from=1-4, to=4-4]
	\arrow["{p}"', from=4-1, to=4-4]
\end{tikzcd}\]

We set $\htpy{0}^{X}, \htpy{1}^{X}: (\I \to X) \to X$ as the functors that project the domain and codomain of a path in $(\I \to X)$ so that $(\htpy{0}^{X},\htpy{1}^{X}): (\I \to X) \to X \times X$ is the normal isofibration in Lemma \ref{pathspace}. It should also be straightforward that $(\htpy{0}^{X},\htpy{1}^{X}) \circ \sigma_{X}$ is the diagonal on $X$. Now we must show that $\sigma_{X}$ is a strong deformation section, but this is realized by taking $\sigma_{X}^{i} = \htpy{0}^{X}$ where $\htpy{0}^{X}$ is a retract of $\sigma_{X}$ and $\beta: \sigma_{X} \circ \htpy{0}^{X} \cong \id_{(\I \to X)}$ is realized by assigning $\beta_{p}$ for $p:x \to x'$ the square 

\[\begin{tikzcd}
	{x} &&& {x} \\
	\\
	\\
	{x} &&& {x'}
	\arrow["\id_{x}", from=1-1, to=1-4]
	\arrow["{\id_{x}}"', from=1-1, to=4-1]
	\arrow["{p}", from=1-4, to=4-4]
	\arrow["{p}"', from=4-1, to=4-4]
\end{tikzcd}\]

It should also be straightforward that $\beta_{\sigma_{X}(x)} = \sigma_{X}(\id_{x})$. Hence $\sigma_{X}$ is a strong deformation section making $((\I \to X),\htpy{0}^{X}, \htpy{1}^{X}, \sigma_{X})$ a path object.

\end{proof}

\begin{defn}

For $f,g:X \to Y$, we say that $f$ and $g$ are {\bf homotopic}, denoted as $f \sim g$, if we have a functor $H:X \to (\I \to Y)$ making the following diagram commute:

\[\begin{tikzcd}
	&&& Y \\
	\\
	X &&& {\I \to Y} \\
	\\
	&&& Y
	\arrow["f", from=3-1, to=1-4]
	\arrow["H"{description}, from=3-1, to=3-4]
	\arrow["g"', from=3-1, to=5-4]
	\arrow["{\htpy{0}^{Y}}"', from=3-4, to=1-4]
	\arrow["{\htpy{1}^{Y}}", from=3-4, to=5-4]
\end{tikzcd}\]

\end{defn}

Lemma 3.3.2 of \cite{joyal2017notesclanstribes} implies that even if we defined the homotopy relation $\sim$ using a general path object, it makes no difference. Note that $f \sim g$ if and only if $f \cong g$. With that in mind, we have the following corollary:

\begin{cor}
 $\sim$ is an equivalence relation on $\GrpA{A}(X,Y)$. Furthermore, when $f \sim g$ then $u \circ f \sim u \circ g$ and $f \circ v \sim  g \circ v$ for composable $u$ and $v$. Hence when we have $u \sim u':X \to Y$ and $v \sim v':Y \to Z$, then $v \circ u \sim v' \circ u'$.

\end{cor} 

Now we define $\HGA{A}$, the {\bf homotopy category of groupoid assemblies}, to be the category whose objects are groupoid assemblies and for groupoid assemblies $X$ and $Y$, $$\HGA{A}(X,Y) = \GrpA{A}(X,Y)/\sim$$. For $[f]:X \to Y$ and $[g]:Y \to Z$, $[g] \circ [f] = [g \circ f]$ and the identity of $X$ in $\HGA{A}$ is $[\id_{X}]$. Now we have the following lemma:

\begin{lem}

$\HGA{A}$ is the localization category of $\GrpA{A}$ with respect to the groupoid equivalences.

\end{lem}
\begin{proof}

This is proposition 3.3.8 of \cite{joyal2017notesclanstribes}.

\end{proof}

Since $\GrpA{A}$ is a $\pi$-tribe, we have the following result regarding $\HGA{A}$:

\begin{lem}

$\HGA{A}$ is cartesian closed and the products and exponents are constructed in $\HGA{A}$ as in $\GrpA{A}$.

\end{lem}
\begin{proof}

This is proposition 3.8.5 of \cite{joyal2017notesclanstribes}.

\end{proof}

\begin{lem}

$\HGA{A}$ has coproducts.

\end{lem}
\begin{proof}
 Suppose we have groupoid assemblies $X$ and $Y$, the coproduct in $\HGA{A}$ is simply $X + Y$ equipped with the canonical inclusions $[\iota_{0}]:X \to X + Y $ and $[\iota_{1}]:Y \to X + Y$. To demonstrate this, suppose we have a groupoid assembly $Z$ and morphisms $[h_{0}]:X \to Z $ and $[h_{1}]:Y \to Z$, then we have morphisms $h_{0}:X \to Z$ and $h_{1}:Y \to Z$ in $\GrpA{A}$. Thus we have a unique morphism $h:X + Y \to Z$ such that $h_{0} =  h \circ \iota_{0}$ and $h_{1} =  h \circ \iota_{1}$. If we have $h': X + Y \to Z$ such that we have $\alpha_{0}:h_{0} \cong  h' \circ \iota_{0}$ and $\alpha_{1}:h_{1} \cong  h' \circ \iota_{1}$, these correspond to morphisms $\alpha_{0}:X \to (\I \to Z)$ and $\alpha_{1}:Y \to (\I \to Z)$. This gives us a new morphism $\alpha: X + Y \to (\I \to Z)$ with $\alpha_{0} = \alpha \circ \iota_{0}$ and $\alpha_{1} = \alpha \circ \iota_{1}$. Since we have $$\htpy{0}^{Z} \circ \alpha_{0} = \htpy{0}^{Z} \circ \alpha \circ \iota_{0} = h_{0}$$ and  $$\htpy{1}^{Z} \circ \alpha_{0} = \htpy{1}^{Z} \circ \alpha \circ \iota_{0} = h' \circ \iota_{1}$$ and similar thing for $\alpha_{1}$ by the universal property of coproducts, we have $\htpy{0}^{Z} \circ \alpha = h$ and $\htpy{1}^{Z} \circ \alpha = h'$. So we have $\alpha:h \cong h'$. Thus $[h]:X + Y \to Z$ is the unique morphism such that $[h_{0}] =  [h] \circ [\iota_{0}]$ and $[h_{1}] =  [h] \circ [\iota_{1}]$.

\end{proof}

So we have that $\HGA{A}$ is bicartesian closed. However, $\HGA{A}$ does not necessarily have limits and colimits. The proper notion of limit and colimit we seek is that of bilimit and bicolimit in $\GrpA{A}$. First we must first enrich $\GrpA{A}$ in $\Cat$. This is straightforward as we have a functor $\Gamma: \GrpA{A} \to \Grp$ which forgets the realizer structure of the groupoid assemblies. This functor is mentioned in Remark \ref{grpnabla}. This induces a functor \[\Hom_{\Cat}: (\GrpA{A})^{\op} \times\GrpA{A} \to \Cat\] where $\Hom_{\Cat}(X,Y) :=  \Gamma(X \to Y)$ and for $f:X \to X'$ and $g:Y' \to Y$, $\Hom_{\Cat}(f,g) = g \circ - \circ f: \Hom_{\Cat}(X,Y) \to \Hom_{\Cat}(X',Y')$. 
We give the definition of weighted bilimit and bicolimit from \cite{CTGDC_1980__21_2_111_0}. This definition makes use of the notion of bicategory, homomorphism of bicategories and transformations of homomorphisms which are define in \cite{leinster1998basicbicategories}.

\begin{lem}
\label{hexpon}

In particular, we have $\Hom_{\Cat}(X \times Y, Z) \cong \Hom_{\Cat}(X, Y \to Z)$.

\end{lem}
\begin{proof}

We have that $\GrpA{A}$ is cartesian closed so we have an isomorphism on the object level. Thus we must extend the isomorphism on the object level to an isomorphism on the morphism level. We denote the potential isomorphism as $\alpha$. Suppose we have $t:f \cong f':X \times Y \to Z$, we define $\alpha(t):\alpha(f) \cong \alpha(f'):X \to (Y \to Z)$ where $\alpha(t)_{x} = [y \in \ob Y \mapsto t_{(x,y)}]$ and for $t':g \cong g':X \to (Y \to Z)$ we define $\alpha^{-1}(t')$ where $\alpha^{-1}(t')_{(x,y)} = t'_{x,y}$. It should be clear from this that $\alpha$ and $\alpha^{-1}$ are functors and are isomorphism pairs. This completes the proof.

\end{proof}

\begin{defn}

We say that a $\Cat-enriched$ category $\cC$ has {\bf weighted bilimits} when for each bicategory $D$ and homomorphisms $W:D \to \Cat$ and $J:D \to \cC$, there exists an $(W \multimap J) \in \ob \cC$ such that we have an equivalence of categories $$\cC(X,(W \multimap J)) \simeq [D \to \Cat](W, \cC(X,J(-)))$$ for each $X \in \ob \cC$. We say that $\cC$ has {\bf weighted bicolimits}  when for each bicategory $D$ and homomorphisms $W:D^{\op} \to \Cat$ and $J:D \to \cC$, there exists an $W \otimes J \in \ob \cC$ such that we have an equivalence of categories $$\cC(W \otimes J,X) \simeq [D^{\op} \to \Cat](W, \cC(J(-), X))$$ for each $X \in \ob \cC$.

\end{defn}

\begin{defn}

Given a $\Cat-enriched$ category $\cC$, for $X \in \ob  \cC$ and $C \in  \ob \Cat$, if there exists an object $X \otimes C$ such that for all $Y \in \ob \cC$  $$(C \to \cC(X,Y))  \simeq \cC(X \otimes C, Y)$$ we call $X \otimes C$ the {\bf tensor product} of $X$ and $C$. When we have an object $C \multimap X$ such that such that for all $Y \in \ob \cC$,  $$(C \to \cC(Y,X))  \simeq \cC(Y, C \multimap X)$$ we call $C \multimap X $ the {\bf cotensor product} of $X$ with $C$.

\end{defn}

Now we give the following definitions and constructions:

The {\bf biproduct} and {\bf bicoproduct} is just the product and coproduct. Given a diagram 

\[\begin{tikzcd}
	X &&&& Y
	\arrow["{f}", shift left=3, from=1-1, to=1-5]
	\arrow["{g}"', shift right=3, from=1-1, to=1-5]
\end{tikzcd}\]

in $\GrpA{A}$, we construct the {\bf biequalizer} of the above diagram as resulting from the following pullback:

\[\begin{tikzcd}
	{E_{f,g}} &&& {\I \to Y} \\
	\\
	\\
	{X } &&& {Y \times Y}
	\arrow["\alpha", from=1-1, to=1-4]
	\arrow["{e_{f,g}}"', from=1-1, to=4-1]
	\arrow["\lrcorner"{anchor=center, pos=0.125}, draw=none, from=1-1, to=4-4]
	\arrow["{{(\htpy{0}^{Y},\htpy{1}^{Y})}}", from=1-4, to=4-4]
	\arrow["{{(f,g)}}"', from=4-1, to=4-4]
\end{tikzcd}\]

where $e_{f,g}:E_{f,g} \to X$ is the biequalizer morphism. The {\bf bicoequalizer} is constructed as follows:

\[\begin{tikzcd}
	{X + X} &&&& {X \times \I} \\
	\\
	\\
	{Y } &&&& {CE_{f,g}}
	\arrow["{[X \times \htpy{0}, X \times \htpy{1}]}", from=1-1, to=1-5]
	\arrow["{[f,g]}"', from=1-1, to=4-1]
	\arrow["\alpha", from=1-5, to=4-5]
	\arrow["{c_{f,g}}"', from=4-1, to=4-5]
	\arrow["\lrcorner"{anchor=center, pos=0.125, rotate=180}, draw=none, from=4-5, to=1-1]
\end{tikzcd}\]

where $c_{f,g}:Y \to CE_{f,g} $ is the bicoequalizer morphism. Now we prove the following:

\begin{lem}
\label{tensor}

$\GrpA{A}$ has tensors and cotensors for finite categories $C$.

\end{lem}
\begin{proof}

First we observe that since we have a partial combinatory algebra $A$ we can encode the natural numbers as combinators $\overline{n}$ where $\overline{0} = \iota$ the identity combinator and $\overline{n + 1} = [\overline{k}, \overline{n}]$. Recall that we have a combinator $p_{0} = \lmbd{z}{z \cdot k}$ which is the first projector of $[-,-]$. Hence we have $$p_{0} \cdot \overline{0} = \overline{0} \cdot k = \iota \cdot k = k$$ $$p_{0} \cdot \overline{n + 1} = p_{0} \cdot [\overline{k}, \overline{n}] = \overline{k}$$ Representing $k$ as $\text{``true"}$ and $\overline{k}$ as $\text{``false"}$, we can represent $Z = p_{0}$ as the combinator that distinguishes $0$ from non-zero numbers. In addition, we have a predecessor combinator: $$P = \lmbd{w}{\ifthen{Z \cdot w}{\iota}{p_{1} \cdot w}}$$

Denoting $(n) = \{i \in \nat|  i < n\}$, for any finite category $C$, we can construct a category object in $\Asm(A)$, $C_{A}$ as follows:

\begin{itemize}

\item Given an enumeration $\phi:\ob C \cong (n)$, we construct an assembly $\ob C_{A} = (\ob C, \overline{\phi})$ where $\overline{\phi}(c) = \{\overline{\phi(c)}\}$.

\item When for $c, c' \in \ob C$, we have an enumeration $\psi_{c,c'}:C(c,c') \cong (m)$ such that $\psi_{c,c}(\id_{c}) = 0$, we have $\mor C_{A} = (\mor C, \overline{\psi})$ where for $t \in C(c,c')$, we construct $\overline{\psi}(t)$ as follows:
\begin{itemize}
\item $[k,[[\overline{\phi(c)},\overline{\phi(c')}],[k, \overline{\psi_{c,c'}(t)}]]] \in \overline{\psi}(t)$

\item When we have $t_{0} \in C(c, c_{*})$ and $t_{1} \in C(c_{*}, c')$ such that $t_{1} \circ t_{0} = t$, for $r \in \overline{\psi}(t_{0})$ and $r' \in \overline{\psi}(t_{1})$, we have $[k,[[\overline{\phi(c)},\overline{\phi(c')}],[\overline{k},[r',r]]]] \in \overline{\psi}(t)$

\end{itemize}
\end{itemize}

So for every $w \in \overline{\psi}(t)$, $w = [k,[[\overline{\phi(c)},\overline{\phi(c')}],[t,p]]] $ where $t = k$ and $p = \overline{n}$ or $t = \overline{k}$ and $p = [r,r']$. It is straightforward to show that this gives a category object in $\Asm(A)$. From $C$, we can construct a groupoid $C_{g}$ as follows:

\begin{itemize}

\item We have an inclusion $i_{C}:C \mono C_{g}$ where the identities of $C_{g}$ are the identities of C.

\item For $c,c' \in \ob C$, $t \in C_{g}(c,c')$, we have $t^{-1} \in C_{g}(c',c)$ with $t^{-1}$ the inverse of $t$.

\item For $t \in C_{g}(c,c')$ and $t' \in C_{g}(c', c'')$, $t' \circ t \in C_{g}(c,c'')$ where the equalities for identities and inverses hold.

\end{itemize}

We can show that $C_{g}$ has the the following universal property:\\

For a functor $F:C \to G$ where $G$ is a groupoid, there exists a unique functor $\tilde{F}:C_{g} \to G$ allowing the following diagram to commute:

\[\begin{tikzcd}
	C &&& G \\
	\\
	\\
	{C_{g}}
	\arrow["F", from=1-1, to=1-4]
	\arrow["{i_{C}}"', from=1-1, to=4-1]
	\arrow["{\tilde{F}}"', from=4-1, to=1-4]
\end{tikzcd}\]

Furthermore, we can extend the realizer structure of $C_{A}$ to $C_{g}$ constructing a groupoid assembly $C_{g,A}$ as follows:

\begin{itemize}

\item $\ob C_{g,A} = \ob C_{A}$.

\item We have $\mor C_{g,A} = (\mor C_{g}, \overline{\psi}_{g})$ where

\begin{itemize}

\item For $t \in C_{A}(c,c')$, $ \overline{\psi}(t) \subseteq \overline{\psi}_{g}(i_{C}(t))$ so that we have a functor $i_{C}:C_{A} \to C_{g,A}$ realized by $[\iota,\iota]$.

\item For $[k,[[\overline{\phi(c)},\overline{\phi(c')}],[t,p]]] \in \overline{\psi}_{g}(t)$, we have $[\overline{k},[[\overline{\phi(c')},\overline{\phi(c)}],[t,p]]] \in \overline{\psi}_{g}(t^{-1})$.

\item For $t \in C_{g,A}(c,c')$ and $t' \in C_{g,A}(c',c'')$ with $r \in \overline{\psi}_{g}(t)$ and $r' \in \overline{\psi}_{g}(t')$, $[k,[[\overline{\phi(c)},\overline{\phi(c'')}],[\overline{k},[r',r]]]] \in \overline{\psi}_{g}(t' \circ t)$.

\end{itemize}

\end{itemize}

Given a groupoid assembly $X$, we define $X \otimes C := X \times C_{g,A}$ and $C \multimap X := C_{g,A} \to X$. Hence, if suffices to prove $$(C \to \Hom_{\Cat}(X,Y)) \cong \Hom_{\Cat}(C_{g,A}, X \to Y) $$ for groupoid assemblies $X$ and $Y$. We have a straightforward functor $$J: \Hom_{\Cat}(C_{g,A}, X \to Y) \to (C \to \Hom_{\Cat}(X,Y))$$ since a functor $f:C_{g,A} \to (X \to Y)$ induces a functor $\Gamma(f):C_{g} \to \Hom_{\Cat}(X,Y)$ and so we have a functor $J(f) = \Gamma(f) \circ i_{C}: C \to \Hom_{\Cat}(X,Y)$. A natural isomorphism $\alpha:f \cong f'$ in $C_{g,A} \to (X \to Y)$ induces a natural isomorphism $\alpha: \Gamma(f) \cong \Gamma(f')$ by forgetting the realizer structure, and hence we have $J(\alpha) = \alpha \cdot i_{C}$. \\

Now we see to construct a functor $$J_{*}: (C \to \Hom_{\Cat}(X,Y)) \to \Hom_{\Cat}(C_{g,A}, X \to Y) $$
Suppose we have a functor $g:C \to \Hom_{\Cat}(X,Y)$, since $\Hom_{\Cat}(X,Y)$ is a groupoid then there exists a unique $\tilde{g}:C_{g} \to \Hom_{\Cat}(X,Y)$ such that $g = \tilde{g} \circ i_{C}$. So we will extend this functor to a functor $\tilde{g}: C_{g,A} \to (X \to Y)$. First we show that the object part of $\tilde{g}$ is realized. For $c \in \ob C_{g,A}$, $\tilde{g}(c) = g(c)$ is realized. Given our enumeration $\phi: \ob C \cong (n)$, let $r_{i}$ be a realizer for $\tilde{g}(c)$ where $\phi(c) = i$. Recall that we have a realizer $\kappa = \lmbd{w,u,v}{u \cdot ((w \cdot w) \cdot u) \cdot v}$ and can define a realizer $\text{fixf} = \kappa \cdot \kappa$ so that $$\text{fixf} \cdot f = (\kappa \cdot \kappa) \cdot f = \lmbd{v}{f \cdot ((\kappa \cdot \kappa) \cdot f) \cdot v} = \lmbd{v}{f \cdot (\text{fixf} \cdot f) \cdot v}$$ We inductively define pairs $\pi_{i}$ for $i \in (n)$ where $\pi_{0} = [r_{n - 1},\iota]$ and $\pi_{j + 1} = [r_{n - j - 2},\pi_{j}]$ and define the realizers $$\text{ifthenN} = \lmbd{f,l,\overline{n'}}{\ifthen{Z \cdot \overline{n'}}{p_{0} \cdot l}{(f \cdot (p_{1} \cdot l)) \cdot (P \cdot \overline{n'})}}$$ and $$\text{base}_{\ob} = (\text{fixf} \cdot \text{ifthenN} ) \cdot \pi_{n - 1}$$ so that 

\begin{align*}
\text{base}_{\ob} &= (\text{fixf} \cdot \text{ifthenN} ) \cdot \pi_{n - 1}\\
&= (\lmbd{v}{ \text{ifthenN} \cdot (\text{fixf} \cdot  \text{ifthenN}) \cdot v}) \cdot \pi_{n - 1}\\
&= \text{ifthenN} \cdot (\text{fixf} \cdot  \text{ifthenN}) \cdot \pi_{n - 1}\\
&=  \lmbd{\overline{n'}}{\ifthen{Z \cdot \overline{n'}}{p_{0} \cdot \pi_{n - 1}}{((\text{fixf} \cdot  \text{ifthenN})  \cdot (p_{1} \cdot \pi_{n - 1})) \cdot (P \cdot \overline{n'})}}\\
&=  \lmbd{\overline{n'}}{\ifthen{Z \cdot \overline{n'}}{r_{0}}{((\text{fixf} \cdot  \text{ifthenN})  \cdot (\pi_{n - 2})) \cdot (P \cdot \overline{n'})}}
\end{align*}

$\text{base}_{\ob} $ is the realizer for the object part of $\tilde{g}$. Note that when $\ob C$ is empty, the object part is vacuously realized. Constructing the realizer for the morphism part of $\tilde{g}$ is a bit more complicated. Reusing the enumerations $\phi: \ob C \cong (n)$ and $\psi_{c,c'}: C(c,c') \cong (n_{c,c'})$ let $r_{(i,j,k)}$ be a realizer for $g(t)$ for $t:c \to c'$ so that $\phi(c) = i$, $\phi(c') = j$, and $\psi_{c,c'}(t) = k$. For $i,j \in (n)$, We define $\pi^{i,j}_{k}$ for $0 \leq k \leq n_{c,c'}$ inductively as follows: $\pi^{i,j}_{0} = \iota$ and $\pi^{i,j}_{m + 1} = [r_(i,j,n_{c,c'} - 1 - m),\pi^{i,j}_{t}]$. We define $\pi^{i,j} = \pi^{i,j}_{n_{c,c'}}$. Then we define $\pi^{i}$: for $j \in (n)$: $\pi^{i}_{0} = [\pi^{i,0},\iota]$ and $\pi^{i}_{m + 1} = [\pi^{i,m + 1}, \pi^{i}_{m}]$. So we have $\pi^{i} = \pi^{i}_{n - 1}$. Finally we define $\pi$: $\pi_{0} = [\pi^{0},\iota]$ and $\pi_{m + 1} = [\pi^{m + 1}, \pi_{m}]$ so that $\pi = \pi_{n - 1}$. \\

Similarly with $\text{base}_{\ob}$, we will define two combinators: $\text{base}_{1}$ and $\text{base}_{\mor}$. First we define a combinator $$\text{ifthen1} = \lmbd{f,l,p}{\ifthen{Z \cdot (p_{0} \cdot p)}{(\text{fixf} \cdot \text{ifthenN} ) \cdot (p_{0} \cdot l) \cdot (p_{1} \cdot p)}{f \cdot (p_{1} \cdot l) \cdot ([P \cdot (p_{0} \cdot p), p_{1} \cdot p])}}$$ so that we have $\text{base}_{1} = \text{fixf} \cdot \text{ifthen1} $. Now we will define a combinator: $$\text{ifthen}_{\mor} = \lmbd{f,l,p}{\ifthen{Z \cdot (p_{0} \cdot (p_{0} \cdot p))}{\text{base}_{1} \cdot (p_{0} \cdot l) \cdot [(p_{0} \cdot (p_{1} \cdot p)), (p_{1} \cdot p)]}{f \cdot (p_{1} \cdot l) \cdot ([[P \cdot (p_{0} \cdot (p_{0} \cdot p)), p_{0} \cdot (p_{1} \cdot p))], p_{1} \cdot p])}}$$ so that $\text{base}_{\mor} = \text{fixf} \cdot \text{ifthen}_{\mor} \cdot \pi$. When $\phi(c) = i$, $\phi(c') = j$, and $\phi_{c,c'}(t) = k$, $\text{base}_{\mor} \cdot [[\overline{i},\overline{j}],\overline{k}] = r_{(i,j,k)}$. \\

Recall that for a morphism $t \in \mor C_{g,A}$, a realizer $r \in \overline{\psi}_{g}(t)$ is of the form $[t_{0},[[\overline{\phi(c)},\overline{\phi(c')}],[t_{1},p]]]$ where both $t_{0}$ and $t_{1}$ are either $k$ or $\overline{k}$ and when $t_{1} = k$, $p = \overline{\phi_{c,c'}(t)}$ and when $t_{1} = \overline{k}$, $p = [r_{0}, r_{1}]$ where $r_{0}$ and $r_{1}$ realizes $s_{0}$ and $s_{1}$ and $s_{1} \circ s_{0} = t$. Furthermore when $t_{0} = k$, then $[\overline{k},[[\overline{\phi(c)},\overline{\phi(c')}],[t_{1},p]]] \in \overline{\psi}_{g}(t^{-1})$. With that in mind, we can begin to construct a realizer for the morphism part of $\tilde{g}$. First let $comp$ and $inv$ be the realizers for the composition and inverse morphism of $X \to Y$. This means that when $r$ realizes a morphism $t$, then $inv \cdot r$ realizes $t^{-1}$ and when $r$ and $r'$ realizes $t$ and $t'$, then $comp \cdot [r',r]$ realizes $t' \circ t$. First we construct a realizer, $\text{comppart}$:

$$\lmbd{f,w}{\ifthen{p_{0} \cdot (p_{1}  \cdot (p_{1} \cdot w))}{\text{base}_{\mor} \cdot [[p_{1}  \cdot (p_{1} \cdot w))],p_{1} \cdot (p_{1}  \cdot (p_{1} \cdot w))]}{comp \cdot [f \cdot (p_{1} \cdot (p_{1} \cdot (p_{1}  \cdot (p_{1} \cdot w)))) ,f \cdot (p_{0} \cdot (p_{1} \cdot (p_{1}  \cdot (p_{1} \cdot w))))] }}$$

Next we construct the following realizer:

$$\text{invpart} = \lmbd{f,w}{\ifthen{p_{0} \cdot w}{(\text{comppart} \cdot f) \cdot w}{inv \cdot (f \cdot [k, p_{1} \cdot w] ) }}$$

Finally, we construct the realizer for the morphism part of $\tilde{g}$, $\text{r}_{mor} = \text{fixf} \cdot \text{invpart}$. Thus $\tilde{g}: C_{g,A} \to (X \to Y)$ is realized by $[\text{base}_{\ob},\text{r}_{mor}]$ and $J_{*}(\tilde{g}) = \tilde{g}$. \\

The morphism part is straightforward, since $\ob C = \ob C_{g}$, any natural isomorphism $\alpha: f \cong f': C \to \Hom_{\Cat}(X,Y)$ becomes a natural isomorphism $\alpha:\tilde{f} \cong \tilde{f'}$. Realizing $\alpha$ can be done in a similar way to realizing the object part of $\tilde{g}$. Hence $J_{*}(\alpha) = \alpha$. It should be clear that $J$ and $J_{*}$ are isomorphisms, thus $$(C \to \Hom_{\Cat}(X,Y)) \cong \Hom_{\Cat}(C_{g,A}, X \to Y) $$ 

Therefore $\GrpA{A}$ has tensors and cotensors with finite categories. 
\end{proof}

Using the construction of bilimits in section 1.24 of \cite{CTGDC_1980__21_2_111_0}, suppose we have a bicategory $C$ such that $\ob C$ is finite and $C(x,y)$ is a finite category for $x, y \in \ob C$ with homomorphisms $W:C \to \Cat$ and $J:C \to \GrpA{A}$ with each $W(x)$ being a finite category, first observe that we have a functor $\phi_{x,y}:W(x) \times C(x,y) \to W(y)$ where $\phi_{x,y}(w_{x},f) = W(f)(w_{x})$ for objects and $\phi_{x,y}(m_{x},\alpha:f \to f') = W(\alpha)_{\cod(m_{x})} \circ W(f)(m_{x})$ for morphisms. We construct the weighted bilimit $W \multimap J$ as being the biequalizer of the following diagram:

\[\begin{tikzcd}
	\Pi_{x \in \ob C} (W(x) \multimap J(x)) &&&& \Pi_{f \in \mor C} ((W(\dom(f))\times C(\dom(f),\cod(f)) )\multimap J(\cod(f)))
	\arrow["{\pi_{0}}", shift left=3, from=1-1, to=1-5]
	\arrow["{\pi_{1}}"', shift right=3, from=1-1, to=1-5]
\end{tikzcd}\]

where $$\pi_{0}(p) = [f \in \mor C \mapsto  J(f) \circ p_{\dom(f)} \circ \phi_{\dom(f),\dom(f)}]$$ and $$\pi_{1}(p) = [f \in \mor C \mapsto p_{\cod(f)} \circ \phi_{\dom(f),\cod(f)}]$$ 

The weighted bicolimit $W \otimes J$ is the bicoequalizer of the following diagram:

\[\begin{tikzcd}
	\Sigma_{f \in \mor C} ((W(\dom(f))\times C(\dom(f),\cod(f)) )\otimes J(\dom(f))) &&&& \Sigma_{x \in \ob C} (W(x) \otimes J(x))
	\arrow["{\sigma_{0}}", shift left=3, from=1-1, to=1-5]
	\arrow["{\sigma_{1}}"', shift right=3, from=1-1, to=1-5]
\end{tikzcd}\]

where $$\sigma_{0}((w_{x},f),j_{x}) = (w_{x},j_{x})$$ and $$\sigma_{1}((w_{x},f),j_{x}) = (\phi_{\dom(f),\cod(f)}(w_{x},f), J(f)(j_{x}))$$ for objects and $$\sigma_{1}((m_{x},\alpha:f \to f'),k_{x}) = (\phi_{\dom(f),\cod(f)}(m_{x},\alpha), J(\alpha)_{\cod(k_{x})} \circ J(f)(k_{x}))$$ for morphisms. We will refer to these as {\bf finite bilimits and bicolimits}. This leads us to the following lemma:

\begin{lem}
\label{bilim}

$\GrpA{A}$ has finite bilimits and bicolimits.

\end{lem}

Homotopy limits and colimits can be understood weighted bilimits and bicolimits. There are two notable examples of finite bilimits and bicolimits being bipullbacks and the circle type:\\

Given a diagram in $\GrpA{A}$:

\[\begin{tikzcd}
	&&& X \\
	\\
	\\
	Z &&& Y
	\arrow["f", from=1-4, to=4-4]
	\arrow["g"', from=4-1, to=4-4]
\end{tikzcd}\]

and a weighted diagram in $\Cat$:

\[\begin{tikzcd}
	&&& 1 \\
	\\
	\\
	1 &&& \I
	\arrow["{\htpy{1}}", from=1-4, to=4-4]
	\arrow["{\htpy{0}}"', from=4-1, to=4-4]
\end{tikzcd}\]

We have the following diagram:\\
\[\begin{tikzcd}
	X \times (\I \to Y)\times Z &&&& Y \times Y
	\arrow["{\beta_{0} = [(x,p,z)   \mapsto (f(x),\dom(p))]}", shift left=3, from=1-1, to=1-5]
	\arrow["{\beta_{1} = [(x,p,z) \mapsto (g(z),\cod(p))]}"', shift right=3, from=1-1, to=1-5]
\end{tikzcd}\]

we construct the weighted bilimit as follows:

\[\begin{tikzcd}
	{P_{f,g}} &&& {\I \to (Y\times Y)} \\
	\\
	\\
	{X \times (\I \to Y)\times Z} &&& {(Y \times Y) \times (Y \times Y)}
	\arrow["\alpha", from=1-1, to=1-4]
	\arrow["{(p_{0},\alpha,p_{1})}"', from=1-1, to=4-1]
	\arrow["\lrcorner"{anchor=center, pos=0.125}, draw=none, from=1-1, to=4-4]
	\arrow["{{(\htpy{0}^{Y\times Y},\htpy{1}^{Y\times Y})}}", from=1-4, to=4-4]
	\arrow["{(\beta_{0},\beta_{1})}"', from=4-1, to=4-4]
\end{tikzcd}\]

The resultant bipullback is the morphism $p_{1}:P_{f,g} \to Z$. The {\bf circle type} $\Circ$ which is the result of the following bicoequalizer of the diagram:

\[\begin{tikzcd}
	1 &&&& 1
	\arrow[ shift left=3, from=1-1, to=1-5]
	\arrow[ shift right=3, from=1-1, to=1-5]
\end{tikzcd}\]

in $\GrpA{A}$ which is constructed as:

\[\begin{tikzcd}
	{1 + 1} &&&& { \I} \\
	\\
	\\
	{1} &&&& {\Circ}
	\arrow["{[ \htpy{0}, \htpy{1}]}", from=1-1, to=1-5]
	\arrow["{[!_{1},!_{1}]}"', from=1-1, to=4-1]
	\arrow["\text{loop}", from=1-5, to=4-5]
	\arrow["{\text{base}}"', from=4-1, to=4-5]
	\arrow["\lrcorner"{anchor=center, pos=0.125, rotate=180}, draw=none, from=4-5, to=1-1]
\end{tikzcd}\]

explicitly, $\Circ$ is generated by a single object $\text{base}$ and a non-trivial isomorphism $\text{loop}:\text{base} \to \text{base}$. Though we only discuss examples here, the general theory of bilimits and bicolimits is discussed in \cite{Lack_2009} .\\


Now we can talk about $0$ and $-1$ types, in particular, we can characterize $0$-truncation and $-1$ truncation.

\begin{defn}

We call a groupoid assembly $X$ a {\bf 0-type} if the morphism $- \circ !_{\Circ}:X \to (\Circ \to X)$ is a groupoid equivalence.  We call $X$ a {\bf -1-type} if the morphism $- \circ !_{1 + 1}:X \to (1 + 1 \to X)$ is a groupoid equivalence.

\end{defn}

The following lemmas allow us the characterize the $0$ and $-1$ types:

\begin{lem}
The following statements are equivalent:

\begin{enumerate}

\item $- \circ !_{\Circ}:X \to (\Circ \to X)$ is a groupoid equivalence. 
\item $- \circ !_{\Circ}:X \to (\Circ \to X)$ is an isomorphism.
\item There is at most one morphism between any two objects in $X$. 

\end{enumerate}

\end{lem}
\begin{proof}

We have that $(ii)$ implies $(i)$ and $(iii)$ implies $(ii)$ since $(iii)$ implies that the only morphism $p:x \to x$ is the identity morphism, thus we must show that $(i)$ implies $(iii)$. Suppose that we have $p,q:x \to x'$ in $X$, then we have $q^{-1} \circ p:x \to x$ and $p^{-1} \circ q:x' \to x'$. Since $- \circ !_{\Circ}:X \to (\Circ \to X)$ is a groupoid equivalence, there exists $x_{0},x'_{0} \in \ob X$ with $\alpha:x \to x_{0}$ and $\alpha':x' \to x'_{0}$ with $\alpha = \alpha \circ q^{-1} \circ p$ and $\alpha' = \alpha' \circ p^{-1} \circ q$. But then $p^{-1} \circ q = \id_{x'}$ and $q^{-1} \circ p = \id_{x}$. Therefore $p = q$. Thus $(i)$ implies $(iii)$ holds. This completes the proof.

\end{proof}

\begin{lem}
The following statements are equivalent:

\begin{enumerate}

\item $- \circ !_{1 + 1}:X \to (1 + 1 \to X)$ is a groupoid equivalence.
\item $- \circ [ \htpy{0}, \htpy{1}]:(\I \to X) \to (1 + 1 \to X)$ is an isomorphism.
\item For $x, x' \in \ob X$, there exists a unique morphism $p:x \to x'$.

\end{enumerate}

\end{lem}
\begin{proof}

It is straightforward to show that $(ii)$ and $(iii)$ are equivalent. So it suffices to show that $(i)$ and $(iii)$ are equivalent. Suppose $(i)$ holds and $p: (1 + 1 \to X) \to X$, $\alpha: (- \circ !_{1 + 1}) \circ p \cong \Id_{(1 + 1 \to X)}$, and $\beta: p \circ (- \circ !_{1 + 1})  \cong \id_{X}$ exhibit $- \circ !_{1 + 1}$ as a groupoid equivalence, we have $x, x' \in \ob X$, then we have an element $p(x,x') \in \ob X$ with isomorphisms $\alpha_{(x,x'),0}:p(x,x') \to x$ and $\alpha_{(x,x'),1}:p(x,x') \to x'$ giving us an isomorphism $\alpha_{(x,x'),1} \circ \alpha_{(x,x'),0}^{-1}:x \to x'$. If we have $q:x \to x$, then we can exhibit the following diagram: 

\[\begin{tikzcd}
	x &&& x \\
	\\
	\\
	{p(x,x)} &&& {p(x,x)} \\
	\\
	\\
	x &&& x
	\arrow["q", from=1-1, to=1-4]
	\arrow["{\alpha_{(x,x),0}}", from=4-1, to=1-1]
	\arrow["{p(\id_{x},q)}"{description}, from=4-1, to=4-4]
	\arrow["{\alpha_{(x,x),1}}"', from=4-1, to=7-1]
	\arrow["{\alpha_{(x,x),0}}"', from=4-4, to=1-4]
	\arrow["{\alpha_{(x,x),1}}", from=4-4, to=7-4]
	\arrow["{\id_{x}}"', from=7-1, to=7-4]
\end{tikzcd}\]

resulting in the following square:

\[\begin{tikzcd}
	{x} &&& {x} \\
	\\
	\\
	{x} &&& {x}
	\arrow["q", from=1-1, to=1-4]
	\arrow["{\alpha_{(x,x),1} \circ \alpha_{(x,x),0}^{-1}}"', from=1-1, to=4-1]
	\arrow["{\alpha_{(x,x),1} \circ \alpha_{(x,x),0}^{-1}}", from=1-4, to=4-4]
	\arrow["{\id_{x}}"', from=4-1, to=4-4]
\end{tikzcd}\]

Which implies that $q = \id_{x}$. This means that for $j, j':x \to x'$, $j \circ j'^{-1}$ and $j' \circ j^{-1}$ are identities thus $j = j'$. Therefore $(iii)$ holds.\\

Suppose $(iii)$ holds, then it suffices to construct a functor $p: (1 + 1 \to X) \to X$ since every two objects in $X$ have a unique isomorphism between them so verifying the natural isomorphisms is trivial. However, it is obvious that a functor $p: (1 + 1 \to X) \to X$ exists, Therefore, $(i)$ holds therefore, $(iii)$ and $(i)$ are equivalent. This completes the proof.
\end{proof}

So we can characterize $0$-types and $-1$-types as being the local types of $([\Loc(*,!_{\Circ})],R_{[\Loc(*,!_{\Circ})]},L_{[\Loc(*,!_{\Circ})]})$ and $([\Loc(*,[ \htpy{0}, \htpy{1}])],R_{[\Loc(*,[ \htpy{0}, \htpy{1}])]},L_{[\Loc(*,[ \htpy{0}, \htpy{1}])]})$ respectively. Hence $\Loc_{*}(!_{\Circ})$ and $\Loc_{*}([ \htpy{0}, \htpy{1}])$ gives us $0$-truncation and $-1$-truncation respectively.

\subsection{Partitioned Groupoid Assemblies and Regular Equivalences}
For this section and every other section after, we will use the fact that partitioned assemblies are projective which is also equivalent to Axiom of Choice in $\Set$. We introduce regular equivalences and partitioned groupoid assemblies and prove that regular equivalences over partitioned groupoid assemblies are groupoid equivalences.\\

\begin{defn}
\label{regeq}
Suppose we have a functor $F:X \to Y$ in $\GrpA{A}$ whose object part is $F_{\ob}: (X_{\ob},\phi_{\ob}) \to (Y_{\ob},\psi_{\ob})$ and whose morphism part is $F_{\mor}: (X_{\mor},\phi_{\mor}) \to (Y_{\mor},\psi_{\mor})$. We call $F$ a {\bf regular equivalence} iff it satisfies the following three conditions:

\begin{enumerate}

\item The proposition $\forall y \in Y_{\ob}(\psi_{\ob}(y) \imply \exists x \in X_{\ob}\exists f \in Y_{\mor}(F(x),y)(\phi_{\ob}(x)\  \land \ \psi_{\mor}(f)) )$ is realized. (F is essentially surjective)

\item The proposition $\forall x, x' \in X_{\ob}\forall f \in Y_{\mor}(F(x),F(x'))(\phi_{\ob}(x)\ \land \phi_{\ob}(x')\ \land \psi_{\mor}(f) \imply \exists g \in X_{\mor}(x,x')(\phi_{\mor}(g)\ \land \ (F(g) = f))) $ is realized. (F is full)

\item The proposition $\forall x,x' \in X_{\ob}\forall f,f' \in X_{\mor}(x,x')((F(f) = F(f'))\ \land \ \phi_{\mor}(f) \imply (f = f'))$ is realized. (F is faithful)

\end{enumerate}

We call $\RE$ the class of regular equivalences.

\end{defn}

As a reminder for a set $X$, the equality predicate for $X$, $( - = -)\in \pow{A}^{X \times X}$ is defined as $(x = x') = A$ when $x = x'$ and $(x = x') = \emptyset$ otherwise.\\

\begin{lem}

Groupoid equivalences are regular equivalences.

\end{lem}
\begin{proof}

Let $F:X \to Y$ be a groupoid equivalences. Using the notation in Definition \ref{regeq}, let $r_{F_{\ob}}$ and $r_{F_{\mor}}$ be the realizer for the object and morphism parts of $F$ respectively and $r_{F^{i}_{\ob}}$ and $r_{F^{i}_{\mor}}$ be the realizer for the object and morphism parts of $F^{i}:Y \to X$ respectively. Let $r_{\alpha}$ and $r_{\beta}$ realize $\alpha: F \circ F^{i} \cong \id_{Y}$ and $\beta: F^{i} \circ F \cong \id_{X}$ respectively.\\

First we show that $F$ is essentially surjective. Suppose $y \in \ob Y$ and $r_{y}$ is a realizer for $y$, then $r_{F^{i}_{\ob}} \cdot r_{y}$ is a realizer for $F^{i}(y)$ and $r_{\alpha} \cdot r_{y}$ realizes $\alpha_{y}:F(F^{i}(y)) \to y$. Therefore, $\lmbd{z}{[r_{F^{i}_{\ob}} \cdot z,  r_{\alpha} \cdot z]}$ realizes $F$ being essentially surjective.\\

Now we show fullness. Suppose we have $f:F(x) \to F(x')$ in $Y$ with $r_{x}$, $r_{x'}$, and $r_{f}$ being realizers for $x$, $x'$, and $f$ respectively. It can be shown using naturality of $\beta$ that $$F(\beta_{x'} \circ F^{i}(f) \circ \beta^{-1}_{x})  = f$$. Let $r_{comp,X}$ be the realizer for composition in $X$ and $r_{inv,X}$ be a realizer for inverse morphism in $X$, so that since $r_{\beta} \cdot r_{x'}$, $r_{\beta} \cdot r_{x}$, and $r_{F^{i}_{\mor}} \cdot r_{f}$ realizes $\beta_{x'}$, $\beta_{x}$, $F^{i}(f)$ respectively then $$r_{comp,X} \cdot [r_{comp,X} \cdot [r_{\beta} \cdot r_{x'},r_{F^{i}_{\mor}} \cdot r_{f} ], r_{inv,X} \cdot (r_{\beta} \cdot r_{x})]$$
 realizes $\beta_{x'} \circ F^{i}(f) \circ \beta^{-1}_{x}$. Hence fullness of $F$ is realized.\\

It suffices to prove faithfulness internally for $F$. Suppose we have $f, f':x \to x'$ in $X$ such that $F(f) = F(f')$. Then $f = f'$ follows from naturality of $\beta$. Therefore, $F$ is faithful.\\

Therefore, $F$ is a regular equivalence.\\

\end{proof}

\begin{rmk}

It suppose be clear from the previous lemma that the regular equivalences are the just the fully faithful and essentially surjective morphisms of $\GrpA{A}$. The name regular equivalences comes from the essentially surjective aspect since the surjections of a category are synonymous with the regular epimorphisms. Though in $\Grp$ under the assumption of Axiom of Choice in $\Set$, the fully faithful and essentially surjective functors are precisely the groupoid equivalences, this is not the case in $\GrpA{A}$ for a non-trivial pca $A$ as we will see later.

\end{rmk}




\begin{defn}

We call a groupoid assembly $G$ a {\bf partitioned groupoid assembly} when $\ob G$ is a partitioned assembly.
\end{defn}

\begin{lem} 

Suppose we have a regular equivalence $F:X \to Y$ in $\GrpA{A}$, for a partitioned groupoid assembly $Z$, and a functor $J: Z \to Y$, there exists a unique functor $G:Z \to X$ up to equivalence such that $F \circ G \cong J$.

\end{lem}
\begin{proof}
Given a functor $J:Z \to Y$, we need to construct a functor $G: Z \to X$. We have a surjection $F_{surj}: \sum_{x \in \ob X}\sum_{y \in \ob Y}Y(F(x),y) \to \ob Y$ such that $F_{surj}(x,y,f)  = y$. Since $\ob Z$ is a partitioned groupoid assembly, there exists a morphism $g_{0}: \ob Z \to \sum_{x \in \ob X}\sum_{y \in \ob Y}Y(F(x),y) $ such that $J_{\ob} = F_{surj}  \circ g_{0}$. \\

If $[r_{j,o},r_{j,m}]$ realizes $J$, then $u(r_{j,o}) = \lmbd{b}{[[p_{0} \cdot (ess \cdot (r_{j,o} \cdot b)), r_{j,o} \cdot b],p_{1} \cdot (ess \cdot (r_{j,o} \cdot b))]}$ realizes $g_{0}$ where $ess$ realizes $F$ being essentially surjective. So we set $G_{\ob}(z) = \pr{0}(g_{0}(z))$ which is realized by $\gamma_{0}(r_{j,o}) = \lmbd{b}{p_{0} \cdot (p_{0} \cdot(u(r_{j,o}) \cdot b))}$.\\

For $f \in Z(z,z')$, we define $G_{\mor}(f)$ as the unique morphism over the morphism $\pr{2}(g_{0}(z'))^{-1} \circ J(f) \circ \pr{2}(g_{0}(z)):F(G_{\ob}(z))\to F(G_{\ob}(z')) $. We know that this morphism exists since $F$ is full. If $ful$ is the realizer of $F$ being full, then $G_{\mor}: \mor Z\to \mor X$ is realized by $\gamma_{1}(r_{j,o},r_{j,m}) = \lmbd{b}{p_{0} \cdot (q \cdot [[\gamma_{0}(r_{j,o}) \cdot (d_{z} \cdot b),\gamma_{0}(r_{j,o}) \cdot (c_{z} \cdot b)],r_{j,m} \cdot b])}$ where $d_{z}$ is the realizer for  $\dom: \mor Z \to \ob Z$ and $c_{y}$ is the realizer for  $\cod: \mor Z \to \ob Z$.\\

$G$ preserves identity since for $z \in \ob Z$, $\id_{G(z)}$ exists over $\pr{2}(g_{0}(z))^{-1} \circ J(f) \circ \pr{2}(g_{0}(z))= \id_{F(G(z))}$, so $G(\id_{z}) = \id_{G(z)}$. $G$ preserves composition since for $i:z \to z'$ and $j:z' \to z''$, we have that $G(j) \circ G(i)$ exists over $\pr{2}(g_{0}(z''))^{-1} \circ J(j) \circ \pr{2}(g_{0}(z')) \circ \pr{2}(g_{0}(y'))^{-1} \circ J(i) \circ \pr{2}(g_{0}(z)) = \pr{2}(g_{0}(z''))^{-1} \circ J(j) \circ J(i) \circ \pr{2}(g_{0}(z))$, so $G(j) \circ G(i) = G(j \circ i)$. Thus $G$ is a functor.\\

Now we construct $\alpha : F \circ G \cong J$ where $\alpha(z) = \pr{2}(g_{0}(z))$. $\alpha$ is realized by $k(r_{j,o}) = \lmbd{b}{(p_{1} \cdot (u(r_{j,o})  \cdot b))}$. $\alpha$ is a natural isomorphism since for $f:z \to z'$, $G(f)$ is defined so that $F(G(f)) = \alpha(z')^{-1} \circ J(f) \circ \alpha(z)$.\\

Suppose we have another functor $G':Z \to X$ such that $F \circ G' \cong J$ exhibited by $\beta$, we seek to define $\delta \in \Pi_{z \in Z_{0}}X(G(z),G'(z))$ that will exhibit an equivalence $G \cong G'$. We set $\delta(z)$ as the unique morphism over $\beta(z)^{-1} \circ \alpha(z): F \circ G(z) \to F \circ G'(z)$. $\delta$ is realized by $\lmbd{b}{ful_{F} \cdot [r_{G,obj} \cdot b,r_{G',obj} \cdot b,com_{Y} \cdot [ inv_{Y} \cdot (r_{\beta} \cdot b) ,r_{\alpha} \cdot b]]}$. Suppose we have $f \in Z(z ,z')$, then 
\begin{align*}
F(\delta(z') \circ G(f)) &=  \beta(z')^{-1} \circ \alpha(z') \circ  \alpha(z')^{-1} \circ J(f) \circ \alpha(z)\\
& = \beta(z')^{-1} \circ J(f) \circ \alpha(z) \\
& =  \beta(z')^{-1} \circ J(f) \beta(z) \circ \beta(z)^{-1} \circ \alpha(z) \\
& = F(G'(f)) \circ F(\delta(z)) = F(G'(f) \circ \delta(z))\\
\end{align*}
$\delta(z') \circ G(f) = G'(f) \circ \delta(z)$ since $F$ is faithful. Therefore, $G \cong G'$.








\end{proof}

\begin{cor}
If $F: X \to Y$ is a regular equivalence with $Y$ a partitioned groupoid assembly, then $F$ is a groupoid equivalence.

\end{cor}
\begin{proof}
We know that there exists a $G:Y \to X$ such that $F \circ G \cong \id_{Y}$. Let $\alpha$ exhibit $F \circ G \cong \id_{Y}$. Now we need to construct $\beta: G \circ F \cong \id_{X}$. We construct $\beta \in \Pi_{x \in X_{0}} X(G(F(x)),x)$ where $\beta(x)$ is the unique morphism over $\alpha(F(x)): F(G(F(x))) \to F(x)$. $\beta$ is realized by $\lmbd{z}{ful_{F} \cdot ([r_{F \circ G,obbj} \cdot z,z,r_{\alpha} \cdot (r_{F,obj} \cdot z)])}$. $\beta$ is a natural isomorphism by virtue of the fact that $\alpha$ is as well as $F$ being fully faithful. Therefore $F$ is a groupoid equivalence.\\
\end{proof}

Recall that $\HGA{A}$ is the localization of $\GrpA{A}$ with respect to $\GE$. We have the canonical functor $Q:\HGA{A} \to \EGA{A}$ where $Q$ is the identity on objects which follows from the fact that $\GE \subseteq \RE$. \\

We call $\pGrA{A} \mono \GrpA{A}$ the full subcategory of partitioned groupoid assemblies with the full subcategory $\Ho(\pGrA{A}) \mono \HGA{A}$ of partitioned groupoid assemblies being the homotopy category of $\pGrA{A}$ as the result of inverting the groupoid equivalences.

\subsection{$\Ho(\pGrA{A})$ forms a Coreflective Subcategory of $\HGA{A}$}
In this section, we show that $\pGrA{A}$ is a bicoreflective subcategory of $\GrpA{A}$ via the bifunctor $\Part(-)$. Using this, we prove that $\pGrA{A}$ has bilimits and homotopy exponents and that $\Ho(\pGrA{A})$ is a coreflective subcategory of $\HGA{A}$.\\

\begin{lem}
\label{partition}

Given a functor of groupoid assemblies $F:G \to H$ and realizers for the object and morphism parts $r_{o}$ and $r_{m}$, we can construct a functor $\Part(F,r_{o}):\Part(G) \to \Part(H)$ such that the following holds: 

\begin{enumerate}
\item $\Part(G)$ and $\Part(H)$ are partitioned groupoid assemblies.
\item When $F$ is a regular equivalence, so is $\Part(F,r_{o})$. 
\item Given another choice of realizers for the object and morphism part of $F$, $s_{o}$ and $s_{m}$ respectively, we have that $\Part(F,r_{o}) \cong \Part(F,s_{o})$

\end{enumerate} 

\end{lem}
\begin{proof}

Given a groupoid assembly $G$, we will construct the partitioned groupoid assembly $\Part(G)$.  The object part of $\Part(G)$ is $\prt{G_{\ob}}{\phi_{\ob}}$ where $(G_{\ob},\phi_{\ob})$ is the assembly of objects of $G$. The set of morphisms is defined as $\Part(G)((x,a),(x',a')) = G(x,x')$ where for $f \in \Part(G)((x,a),(x',a'))$, the set of realizers of $f$ is defined as $\phi_{\mor}^{*}(f) = \{[[a,a'],t]| t \in \phi_{\mor}(f)\}$. The morphisms $\dom,\cod: \mor\Part(G) \to \ob\Part(G)$ are realized by $d_{G} = \lmbd{z}{(p_{0} \cdot (p_{0} \cdot z))}$ and $c_{G} = \lmbd{z}{(p_{1} \cdot (p_{0} \cdot z))}$ respectively. The morphism $\id : \ob\Part(G) \to \mor\Part(G)$ is realized by $i_{G} = \lmbd{z}{[[z,z],i \cdot z]}$ where $i$ realizes $\id : \ob G \to \mor G$. Composition in $\Part(G)$ is the same as composition in $G$ and is realized by $com_{G} = \lmbd{z}{[[d_{G} \cdot (p_{0} \cdot z),c_{G} \cdot (p_{1} \cdot z)],com \cdot [p_{1} \cdot (p_{0} \cdot z), p_{1} \cdot (p_{1} \cdot z)]]}$ where $com$ realizes composition in $G$. Thus, $\Part(G)$ is a partitioned groupoid assembly.\\

The construction of $\Part(F,r_{o})$ is straightforward. $\Part(F,r_{o})(x,a) = (F(x), r_{0} \cdot a)$ and $\Part(F,r_{o})(f) = F(f)$ for $f \in \Part(G)((x,a),(x',a'))$. The object part of $\Part(F,r_{o})$ is realized by $r_{o}$ and the morphism part is realized by $\lmbd{z}{[[r_{o} \cdot (d_{G} \cdot z),r_{o} \cdot (c_{G} \cdot z)],r_{m} \cdot (p_{1} \cdot z)]}$. functoriality of $\Part(F,r_{o})$ follows from functoriality of $F$.\\

Suppose $F$ is a regular equivalence. First suppose that the realizer $ess \in A $ realizes that $F$ is essentially surjective.  Suppose we have $(y,b) \in \ob \Part(H)$, then $ess \cdot b$ realizes $\exists x \in \ob G\exists f \in H(F(x),y)(\phi_{\ob}^{*}(x) \land \psi_{\mor}^{*}(f))$ where $(H_{\mor}, \psi_{\mor})$ is the assembly of morphisms of $H$. Then $p_{0} \cdot (ess \cdot b)$ realizes some $x \in G_{\ob}$ and $p_{1} \cdot (ess \cdot b)$ realizes some $f \in H(F(x),y)$. So $p_{0} \cdot (ess \cdot b)$ realizes $(x,p_{0} \cdot (ess \cdot b)) \in \ob \Part(G)$ and $f \in \Part(H)(\Part(F,r_{o})(x,p_{0} \cdot (ess \cdot b)) ,(y,b))$ and is realized by $[[r_{o} \cdot (p_{0} \cdot (ess \cdot b)),b],p_{1} \cdot (ess \cdot b)]$. So $\Part(F,r_{o})$ is essentially surjective and this fact is realized by $\lmbd{z}{[p_{0} \cdot (ess \cdot z),[[r_{o} \cdot (p_{0} \cdot (ess \cdot z)),z],p_{1} \cdot (ess \cdot z)]]}$.\\

Suppose that $ful \in A$ realizes $F$ being full. Given $x, x' \in G_{\ob}$, $f \in H(F(x),F(x'))$, $a \in \phi_{\ob}(x)$, $a' \in \phi_{0}(x')$, and $t \in \psi_{1}(f)$, then $p_{0} \cdot (ful \cdot [[a,a'],t])$ realizes some $g \in G(x,x')$ and $F(g) = f$. Then $[[a,a'],p_{0} \cdot (ful \cdot [[a,a'],t])]$ realizes $g \in \Part(G)((x,a),(x',a'))$ with $f \in \Part(H)((F(x), r_{o} \cdot a),(F(x'), r_{o} \cdot a'))$ realized by $[[r_{o} \cdot a, r_{o} \cdot a'], t]$. Then $\Part(F,r_{o})$ is full and that fact is realized by $\lmbd{z}{[[p_{0} \cdot z,p_{0} \cdot (ful \cdot [p_{0} \cdot z, p_{1} \cdot (p_{1} \cdot z)])],p_{1} \cdot (ful \cdot [p_{0} \cdot z, p_{1} \cdot (p_{1} \cdot z)])]}$.\\

Suppose $fai \in A$ realizes $F$ being faithful. Take $f, f' \in G((x,a),(x',a'))$, an realizer $b \in (\Part(F,r_{o})(f) = \Part(F,r_{o})(f'))$ and $[[a,a'],t] \in \phi_{1}^{*}(f)$. This means that $b \in (F(f) = F(f'))$ and so $fai \cdot [b,t] \in (f = f')$. Thus $\Part(F,r_{o})$ is faithful and that fact is realized by $\lmbd{z}{fai \cdot [p_{0} \cdot z, p_{1} \cdot (p_{1} \cdot z) ]}$.\\

Therefore, if $F$ is a regular equivalence, so is $\Part(F,r)$.\\

Suppose we have another choice of realizers for the object and morphism part of $F$, $s_{o}$ and $s_{m}$ respectively and we construct $\Part(F,s_{o}):\Part(G) \to \Part(H)$. We define $\alpha: \Pi_{(x,a) \in \Part(G)_{0}} \Part(H)((F(x),r_{o} \cdot a),(F(x),s_{o} \cdot a))$ where $\alpha(x,a) = \id_{F(x)}$. $\alpha$ is realized by $\lmbd{z}{[[r_{o} \cdot z, s_{o} \cdot z],s_{m} \cdot (p_{1} \cdot (i_{G} \cdot z))]}$. $\alpha$ clearly exhibits a natural isomorphism between $\Part(F,r_{o})$ and $\Part(F,s_{o})$ since composition in $\Part(H)$ is the same as in $H$. Therefore, $\Part(F,r_{o}) \cong \Part(F,s_{o})$.
 
\end{proof}

\begin{cor}

Given a functor of groupoid assemblies $F:G \to H$ and realizers for the object and morphism parts $r_{o}$ and $r_{m}$, if $F$ is a regular equivalence, then $\Part(F,r_{o})$ is a groupoid equivalence.

\end{cor}

\begin{cor}

For a groupoid assembly $G$, $\id_{\Part(G)} \cong \Part(\id_{G}, r_{o})$ where $(r_{o}, r_{m})$ realizes $\id_{G}$.

\end{cor}
\begin{proof}

This follows from that fact that $\id_{\Part(G)} = \Part(\id_{G}, \iota)$ where $\iota$ is the identity combinator which realizes both the object and morphism part of $\id_{G}$.

\end{proof}

\begin{cor}

Given a functor of groupoid assemblies $F:G \to H$ and realizers for the object and morphism parts of $F$, $r_{o}$ and $r_{m}$ and another functor of groupoid assemblies $J: H \to K$ and realizers for the object and morphism parts of $J$, $s_{o}$ and $s_{m}$. If $t_{o}$ and $t_{m}$ realize the object and morphism parts of $J \circ F$ respectively, then $\Part(J \circ F, t_{o}) \cong \Part(J,s_{o}) \circ \Part(F,r_{o})$.

\end{cor}
\begin{proof}

It is easy to see that $\Part(J,s_{o}) \circ \Part(F,r_{o}) = \Part(J \circ F, \lmbd{z}{s_{o} \cdot (r_{o} \cdot z)})$ where $\lmbd{z}{s_{o} \cdot (r_{o} \cdot z)}$ realizes the object part of $J \circ F$ and $\lmbd{z}{s_{m} \cdot (r_{m} \cdot z)}$ realizes the morphism part of $J \circ F$. Thus $\Part(J \circ F, t_{o})  \cong \Part(J,s_{o}) \circ \Part(F,r_{o})$.

\end{proof}

\begin{lem}
\label{natatom}

Given a groupoid assembly $G$, we have a regular equivalence $\atom_{G}: \Part(G) \to G$. Furthermore, if $F:G \to H$ is a functor of groupoid assemblies realized by $[r_{o},r_{m}]$ then $F \circ \atom_{G} \cong \atom_{H} \circ \Part(F,r_{o})$.

\end{lem}
\begin{proof}

$\atom_{G}$ takes $(x,a)$ to x and is the identity on morphisms. The fact that $\atom_{G}$ is a functor is straightforward. For showing that $\atom_{G}$ is essentially surjective, suppose we have $x \in \ob G$ and $a \in A$ a realizer of $x$, then $a$ is the unique realizer for $(x,a)$ and $i \cdot a$ realizes $\id_{x}$ where $i$ realizes $\id_{-}:\ob G \to \mor G$. Thus $\atom_{G}$ is essentially surjective and is realized by $\lmbd{z}{[z,i \cdot z]}$. Showing $\atom_{G}$ is straightforward since if we have objects $(x,a)$ and $(x',a')$ in $\Part(G)$ and a morphism $f:x \to x'$ realized by $t$, $f$ is a morphism from $(x,a)$ to $(x,a')$ and is realized by $[[a,a'],t]$. So $\atom_{G}$ being full is realized by the identity combinator. $\atom_{G}$ is plainly faithful since the morphisms $\Part(G)((x,a),(x',a'))$ are precisely the morphisms $G(x,x')$. Therefore, $\atom_{G}$ is a regular equivalence.\\

We define $\alpha: \Pi_{(x,a) \in \ob \Part(G)} H(F \circ \atom_{G}(x,a) , \atom_{H} \circ \Part(F,r_{o})(x,a))$ where $\alpha(x,a) = \id_{F(x)}$. We note that $F \circ \atom_{G}(x,a)  = \atom_{H} \circ \Part(F,r)(x,a) = F(x)$, furthermore $\alpha$ is realized by $\lmbd{z}{r_{m} \cdot (i \cdot z)}$. $\alpha$ is clearly a natural isomorphism between $F \circ \atom_{G}$ and $\atom_{H} \circ \Part(F,r_{o})$ since both functors are identities on morphisms. Therefore, $F \circ \atom_{G} \cong \atom_{H} \circ \Part(F,r_{o})$.

\end{proof}

\begin{rmk}
\label{noheexpl}

Given a non trivial partial combinatory algebra $A$, define a discrete groupoid assembly $A_{diff} = (A,[a \mapsto A/\{a\}])$ and take the regular equivalence $\atom_{A_{diff}}:\Part(A_{diff}) \to A_{diff}$ in $\GrpA{A}$. Suppose $\atom_{A_{diff}}$ is a groupoid equivalence and $f: (A,[a \mapsto A/\{a\}])  \to  \ob \Part(A_{diff})$ is the object part of the equivalence pair of $\atom_{A_{diff}}$. Then $\atom_{A_{diff},\ob} \circ f = \id_{(A,[a \mapsto A/\{a\}])}$ since $A_{diff}$ is discrete. Since $$ \ob \Part(A_{diff}) = (\{(a,a')| a \neq a'\}, [(a,a') \mapsto \{a'\}])$$ we have a realizer $r_{f}$ for $f$ such that for $b,b' \in A/\{a\}$, $(r_{f} \cdot b) \neq a$ and $r_{f} \cdot b = r_{f} \cdot b'$. Since $A$ is non trivial, $A$ has at least three distinct elements: $k$, $s$ and $\iota$. Then $r_{f} \cdot s = r_{f} \cdot \iota$ are distinct from $k$ but without loss of generality, if $r_{f} \cdot s \neq \iota$ then $r_{f} \cdot \iota \neq r_{f} \cdot s$ which is a contradiction. Thus $\atom_{A_{diff}}$ is not a groupoid equivalence. This also show why for non trivial $A$, 
$\HGA{A}$ and $\EGA{A}$ are not equivalent categories.

\end{rmk}

\begin{cor}
\label{corefl}

Given a groupoid assembly $G$, we have that $\atom_{\Part(G)}: \Part(\Part(G)) \to \Part(G)$ is a groupoid equivalence.

\end{cor}
\begin{proof}

This follows from the fact that $\atom_{G}$ is a regular equivalence, and thus for any pair of realizers $(r_{o},r_{m})$ for $\atom_{G}$, $\Part(\atom_{G},r_{o}):\Part(\Part(G)) \to \Part(G)$ is a groupoid equivalence.

\end{proof}

\begin{thm}
There is an equivalence of categories, $\Ho(\pGrA{A}) \cong \EGA{A}$.

\end{thm}
\begin{proof}

We have a functor $S: \Ho(\pGrA{A}) \to \EGA{A}$ which is the result of the composition of the inclusion $\Ho(\pGrA{A}) \mono \HGA{A}$ and the functor $Q:\HGA{A} \to \EGA{A}$. Explicitly, $S$ is the identity on objects and takes an equivalence class of morphisms, $[f]:G \to H$, in $\Ho(\pGrA{A})$ to its corresponding equivalence class in $\EGA{A}$. We define a functor $\Part(-):\GrpA{A} \to \Ho(\pGrA{A})$ where $\Part(G)$ is defined as in Lemma $\ref{partition}$ and $\Part(F)$ is the equivalence class of $\Part(F,r_{o})$ where $(r_{o},r_{m})$ realizes $F$. As shown in previous lemmas and corollaries, $\Part(-)$ is a well-defined functor. Furthermore, $\Part(-)$ takes regular equivalences in $\GrpA{A}$ to isomorphisms in $\Ho(\pGrA{A})$, thus this extends to a functor $\Part(-): \EGA{A} \to \Ho(\pGrA{A})$. \\

First we show that $\Part(S(-)) \cong \Id_{\Ho(\pGrA{A}) }$. We claim that $[\atom_{-}]: \Part(S(-)) \To \Id_{\Ho(\pGrA{A}) }$ is a natural isomorphism. If $G$ is a partitioned groupoid assembly, then $\atom_{G}$ is a groupoid equivalence since $\atom_{G}$ is always a regular equivalence. Naturality of $[\atom_{-}]$ follows from Lemma $\ref{natatom}$.\\

Now we will show that $S(\Part(-)) \cong \Id_{\EGA{A}}$ is exhibited by $[\atom_{-}]$ in $\EGA{A}$. Since $\atom_{G}$ is a regular equivalence for groupoid assembly $G$, then $[\atom_{G}]$ is an isomorphism in $\EGA{A}$. Once again naturality of $[\atom_{-}]$ follows from Lemma $\ref{natatom}$.\\

Therefore, $\Ho(\pGrA{A}) \cong \EGA{A}$.

\end{proof}

We also have a functor $\Part(-):\HGA{A} \to \HGA{A}$ and a naturally transformation $[\atom_{-}]: \Part(-) \To \Id_{\HGA{A}}$ withe following property:

\begin{lem}
$[\atom_{\Part(-)}] = \Part([\atom_{-}])$

\end{lem}
\begin{proof}
For a groupoid assembly $G$, $\atom_{\Part(G)}$ takes an object $((x,a),a) \in \ob \Part(\Part(G))$ to $(x,a) \in \ob \Part(G)$ and is the identity on morphisms. Given a pair of realizers $(r_{o},r_{m})$ for $\atom_{G}$, the functor $\Part(\atom_{G},r)$ takes an object $((x,a),a) \in \Part(\Part(G))$ to $(x, r_{o} \cdot a) \in \Part(G)$ and is the identity on morphisms. It is clear that $\atom_{\Part(G)} = \Part(\atom_{G},\iota)$ where $\iota$ is the identity combinator. Thus $\atom_{\Part(G)} \cong \Part(\atom_{G},r)$. Therefore, $[\atom_{\Part(-)}] = \Part([\atom_{-}])$.

\end{proof}

\begin{lem}
\label{adjin}
$\Ho(\pGrA{A})$ is the coreflective subcategory of $\HGA{A}$ induced by the idempotent comonad $\Part(-) :\HGA{A} \to \HGA{A}$.

\end{lem}
\begin{proof}

Firstly, for $\Part(-):\HGA{A} \to \HGA{A}$ we have the counit $\epsilon_{-} = [\atom_{-}]:\Part(-) \To \Id_{\HGA{A}}$ and comultiplication $\nu_{-} = \Part([\atom_{-}])^{-1}:\Part(-) \To \Part(\Part(-))$. $\Part([\atom_{G}])$ is an isomorphism,, since $\atom_{G}$ is always a regular equivalence. Furthermore since $\epsilon_{\Part(-)} = \Part(\epsilon_{-})$ and $\nu_{-} = \Part(\epsilon_{-})^{-1}$, we also have the counit condition, $\id_{\Part(-)} = \epsilon_{\Part(-)} \circ \nu_{-} = \Part(\epsilon_{-}) \circ \nu_{-} $. The condition for comultiplication, $\nu_{-} \circ \nu_{\Part(-)} = \nu_{-} \circ \Part(\nu_{-})$ follows once again from that the fact that $\epsilon_{\Part(-)} = \Part(\epsilon_{-})$ and $\nu_{-} = \Part(\epsilon_{-})^{-1}$. Thus $\Part(-)$ is an idempotent comonad.\\

Suppose we have a functor $[F]: G \to H$ where $G$ is a partitioned groupoid assembly, given a realizer for $F$, $[r_{o},r_{m}]$ we can construct a functor $F': G \to \Part(H)$. Without loss of generality we assume that for $\ob G = (G_{\ob},\phi)$, $\phi(g) = \{a_{g}\}$ is always a singleton, then for $g \in \ob G$, $F'(g) = (F(g), r_{o} \cdot a_{g})$ and $m \in G(g,g')$, $F'(m) = F(m)$. The object part of $F'$ is realized by $r_{o}$ and the morphism part by $$\lmbd{z}{[[r_{o} \cdot (r_{G,\dom} \cdot z),r_{o}\cdot( r_{G,\cod} \cdot z)],r_{m} \cdot z ]}$$. We clearly have $\atom_{H} \circ F' = F$. Suppose we have a $F^{*}: G \to \Part(H)$ realized by $(s_{o},s_{m})$ such that we have $\beta: \atom_{H} \circ F^{*} \cong F$, we define $\alpha : \Pi_{g \in \ob G} \Part(H)(F'(g),F^{*}(g))$ where $\alpha(g) = \beta(g)$. $\alpha$ is realized by $\lmbd{z}{[[r_{o} \cdot z, s_{o} \cdot z],b \cdot z]}$ where $b$ is a realizer for $\beta$. $\alpha$ exhibits a natural isomorphism since $\beta$ does as well and the morphism in $\Part(H)$ are the same as in $H$, thus $F' \cong F^{*}$.\\

What this means for $\HGA{A}$ is that for a functor $[F]: G \to H$ where $G$ is a partitioned groupoid assembly, we can construct a unique functor $[F']: G \to \Part(H)$ such that $[F] = [\atom_{H}] \circ [F']$. This means we have an adjunction $\incl \adj \Part(-): \Ho(\pGrA{A}) \to \HGA{A}$. Therefore, $\Ho(\pGrA{A})$ is the coreflective subcategory of $\HGA{A}$ induced by $\Part(-)$.

\end{proof}

\begin{lem}
\label{biadjin}

 Given a partitioned groupoid assembly $G$ and groupoid assembly $H$, we have an equivalence of categories $\atom_{H} \circ - : \Hom_{\Cat}(G,\Part(H)) \simeq  \Hom_{\Cat}(G,H) $.

\end{lem}
\begin{proof}
Given a realizer $[r_{f,o},r_{f,m}]$ for each $f:G \to H$, we construct the functor $L:\Hom_{\Cat}(G,H) \to \Hom_{\Cat}(G,\Part(H))$ in $\Cat$. Recall that for $f:G \to H$, we have a functor $L(f):G \to \Part(H)$ where $L(f)(g) = (f(g), r_{f,o} \cdot r_{g})$ where $g \in \ob G$ and $r_{g}$ is the unique realizer for $g$, $L(f)(m) = f(m)$ realized by $$\lmbd{z}{[[r_{f,o} \cdot (r_{G,\dom} \cdot z),r_{f,o}\cdot( r_{G,\cod} \cdot z)],r_{f,m} \cdot z ]}$$ and $\atom_{H} \circ L(f) = f$. If we have $t:f \cong f': G \to H$, then we have $t: L(f) \cong L(f')$ realized by $\lmbd{z}{[[r_{f,o} \cdot z, r_{f',o} \cdot z],r_{t} \cdot z]}$ where $r_{t}$ realizes $t$. This clearly induces a functor $L:\Hom_{\Cat}(G,H) \to \Hom_{\Cat}(G,\Part(H))$ where $\atom_{H} \circ L(-) = \id_{\Hom_{\Cat}(G,H)}$. So we must show that $L(\atom_{H} \circ -) \cong \id_{\Hom_{\Cat}(G,\Part(H))}$. \\

Take $f:G \to \Part(H)$ realized by $[s_{f,o},s_{f,m}]$ where $j = \atom_{H} \circ f:G \to H $ is realized by $[r_{j,o},r_{j,m}]$, then the functor $L(j):G \to \Part(H)$ can be explicitly constructed where $L(j)(g) = (\pr{0}(f(g)), r_{f,o} \cdot r_{g})$ and $L(j)(m) = f(m)$. So we have a natural isomorphism $\alpha_{f}: L(j) \cong f$ where $\alpha_{f,g} = \id_{\pr{0}(f(g))}$ realized by $$\lmbd{z}{[[r_{j,o} \cdot z, s_{f,o} \cdot z], r_{H, \id} \cdot (s_{f,o} \cdot z)  ]}$$

Suppose we have a natural isomorphism $t:f \cong f':G \to \Part(H)$, then $\atom_{H} \cdot t = t$, thus $L(\atom_{H} \cdot t) = t$ making naturality of $\alpha$ obvious since $\alpha$ is defined using the identity morphisms in $H$. Therefore, we have $\alpha: L(\atom_{H} \circ -) \cong \id_{\Hom_{\Cat}(G,\Part(H))}$ making $$\atom_{H} \circ - : \Hom_{\Cat}(G,\Part(H)) \cong  \Hom_{\Cat}(G,H) $$ an equivalence.

\end{proof}

We now have the following corollary:

\begin{cor}
$\pGrA{A}$ is a bicoreflective subcategory of $\GrpA{A}$.

\end{cor}
\begin{proof}

The previous lemma shows that $\Part(-)$ is a right biadjoint to the inclusion functor. Corollary \ref{corefl} states that $\atom_{-}: \Part(-) \to (-)$, which is the unit of the biadjoint, is a regular equivalence. So when $X$ is a partitioned groupoid assembly, $\atom_{X}$ is an equivalence. Hence $\atom_{-}$ is a fibred groupoid equivalence on partitioned groupoid equivalences which completes the proof.

\end{proof}

\begin{lem}
\label{pbilim}
$\pGrA{A}$ has finite bilimits.

\end{lem}
\begin{proof}

Suppose we have finite bilimit diagrams  $W:D \to \Cat$ and $J:D \to \pGrA{A}$, since $\pGrA{A}$ is a full subcategory of $\GrpA{A}$ and $\GrpA{A}$ has finite bilimits, there exists a groupoid assembly $W \multimap J \in \ob \GrpA{A}$ such that we have an equivalence of categories $$\Hom_{\Cat}(X,W \multimap J) \simeq [D \to \Cat](W, \Hom_{\Cat}(X, J(-)))$$ for each $X \in \ob \GrpA{A}$. We now show that $\Part(W \multimap J)$ is the bilimit of $W$ and $J$ in $\pGrA{A}$ by exhibiting the following chain of equivalences making use of the fact that $\pGrA{A}$ is a full subcategory of $\GrpA{A}$ and given a partitioned groupoid assembly X:

\begin{align*}
\Hom_{\Cat}(X,\Part(W \multimap J)) &\simeq \Hom_{\Cat}(X, W \multimap J)\\
&\simeq [D \to \Cat](W,  \Hom_{\Cat}(X,J(-)))\\
\end{align*}

\end{proof}

\begin{lem}

$\Ho(\pGrA{A})$ is cartesian closed.

\end{lem}
\begin{proof}
$\Ho(\pGrA{A})$ having products follows from the fact that products of partitioned assemblies are still partitioned assemblies. For partitioned groupoid assemblies $X$ and $Y$, we define the exponent to be $\Part(X \to Y)$. The following chain of equivalences verifies this fact:

\begin{align*}
\Ho(\pGrA{A})(Z, \Part(X \to Y)) &\cong \HGA{A}(Z, X \to Y)\\
&\cong  \HGA{A}(Z \times X,  Y)\\
&\cong \Ho(\pGrA{A})(Z \times X, Y)
\end{align*}

\end{proof}

\begin{lem}
\label{phexpon}

For partitioned groupoid assemblies $X$, $Y$, and $Z$, we have $$\Hom_{\Cat}(X \times Y, Z) \simeq \Hom_{\Cat}(X, \Part(Y \to Z))$$

\end{lem}
\begin{proof}
We have following chain of equivalences:
\begin{align*}
\Hom_{\Cat}(X \times Y, Z)  &\simeq \Hom_{\Cat}(X ,Y \to Z) \\
&\simeq  \Hom_{\Cat}(X ,\Part(Y \to Z))\\
\end{align*}

\end{proof}

\subsection{Right Adjoint to Partitioning}
In this section, we construct a right biadjoint $\Clus(-)$ for $\Part(-)$. As a consequence of this we prove that $\pGrA{A}$ has finite bicolimits and that $\Part(-) \adj \Clus(-)$ is an adjunction on $\HGA{A}$.\\

\begin{defn}

For a nonempty subset $U \subseteq A$, we define the groupoid assembly $\nabla(U)$ where the object assembly is $U_{*} = (U, [u \mapsto \{u\}])$ and the morphism assembly is $U_{*} \times U_{*}$. The domain and codomain functions are the first and second projections respectively, the identity morphism of $u$ is $(u,u)$ and $(v,w) \cdot (u,v) = (u,w) $.

\end{defn}

For groupoid assembly $G$, we can define $\Clus(G)$ as follows: 

\begin{enumerate}

\item the objects are functors $F:\nabla(U) \to G$ for non-empty $U \subseteq A$; we define the set of realizers for $F$ as $\phi_{*}(F) = \{[r,u]| \text{r realizes } F\text{ and }u \in U \}$

\item the morphisms $H:F \to F'$ are morphisms of assemblies $H: U_{*} \times U'_{*} \to \mor G$ such that $\dom(H(u,u')) = F(u)$, $\cod(H(u,u')) = F'(u')$, and for $u,v \in U$ and $u',v' \in U'$ we have the following commutative square in $G$:

\[\begin{tikzcd}
	{F(u)} &&& {F'(u')} \\
	\\
	{F(v)} &&& {F'(v')}
	\arrow["{H(u,u')}", from=1-1, to=1-4]
	\arrow["{F(u,v)}"', from=1-1, to=3-1]
	\arrow["{F'(u',v')}", from=1-4, to=3-4]
	\arrow["{H(v,v')}"', from=3-1, to=3-4]
\end{tikzcd}\]

The realizers of $H$ are the realizers of the morphisms $H: U_{*} \times U'_{*} \to \mor G$.

\end{enumerate}

The identity morphism $\id_{F}$ is the morphism part of $F$. Given morphisms $H:F \to F'$ and $J:F' \to F''$ and $u,v \in U$, $u',v' \in U'$, and $u'', v'' \in U''$ we have the following commutative square:
\[\begin{tikzcd}
	{F(u)} &&& {F'(u')} &&& {F''(u'')} \\
	\\
	{F(v)} &&& {F'(v')} &&& {F''(v'')}
	\arrow["{H(u,u')}", from=1-1, to=1-4]
	\arrow["{F(u,v)}"', from=1-1, to=3-1]
	\arrow["{J(u',u'')}", from=1-4, to=1-7]
	\arrow["{F'(u',v')}"', from=1-4, to=3-4]
	\arrow["{F''(u'',v'')}"', from=1-7, to=3-7]
	\arrow["{H(v,v')}"', from=3-1, to=3-4]
	\arrow["{J(v',v'')}"', from=3-4, to=3-7]
\end{tikzcd}\]

When $u = v$ and $u'' = v''$, we have $J(u',u'') \circ H(u,u') = J(v',v'') \circ H(v,v')$. Thus for $u', v' \in U'$, $J(u',u'') \circ H(u,u') = J(v',u'') \circ H(u,v')$. This gives a function $J \circ H: U_{*} \times U''_{*} \to \mor G$ where $J \circ H(u,u'') = J(v',u'') \circ H(u,v')$ for some $v' \in U$. $J \circ H$ is realized by $\lmbd{z}{c_{G} \cdot ([r_{J} \cdot [v', p_{1} \cdot z] , r_{H} \cdot [p_{0} \cdot z, v']])}$ where $c_{G}$ realizes composition of morphism in $G$, $r_{J}$ and $r_{H}$ realizes $J$ and $H$, and $v' \in U'$. Groupoid assemblies equivalent to $\Clus(G)$ are called {\bf clustered groupoid assemblies}.

\begin{lem}
\label{clusadj2}
For groupoid assemblies $E$ and $F$, $$\Hom_{\HGA{A}}(\Part(E),F) \cong \Hom_{\HGA{A}}(E,\Clus(F))$$.
 
\end{lem}
\begin{proof}
Suppose we have $[f]:\Part(E) \to F$, we define $\alpha(f):E \to \Clus(F)$ where for $ e \in \ob E$ and $\phi(e)$ the set of realizers for $e$, we have $\alpha(f)(e):\nabla(\phi(e)) \to F$ where $\alpha(f)(e)(r_{e}) = f(e,r_{e})$ and $\alpha(f)(e)(r_{e},r'_{e}) = f(\id_{e})$ where $\id_{e}: (e,r_{e}) \to (e,r'_{e})$. If $[r_{f,o},r_{f,m}]$ realizes $f$, then $r_{f,o}$ realizes the object part of $\alpha(f)(e)$ and for a realizer for $\id_{e}$, $r_{i,e}$, $\lmbd{z}{r_{f,m} \cdot [z,r_{i,e}]}$ realizes morphism part of $\alpha(f)(e)$. For a morphism $t:e \to e'$, we have $\alpha(f)(t):\phi(e) \times \phi(e') \to \mor F$, where $\alpha(f)(t)(r_{e},r_{e'}) = f(t:(e,r_{e}) \to (e',r_{e'}))$. $\alpha(f)(t)$ is realized by $\lmbd{z}{r_{f,m} \cdot [z,r_{t}]}$ where $r_{t}$ realizes $t$. $\alpha(f)(t): \alpha(f)(e) \to \alpha(f)(e')$ is a valid morphism by functoriality of $f$. It should be clear that $\alpha(f)$ is a functor. If we have $\beta:f \cong f'$, then $\alpha(\beta):\alpha(f) \cong \alpha(f')$ is exhibited by $\alpha(\beta)(e)(r_{e},r'_{e}) = f'(\id_{e}:(e,r_{e}) \to (e,r'_{e})) \circ \beta(e,r_{e}) = \beta(e,r'_{e}) \circ f(\id_{e}:(e,r_{e}) \to (e,r'_{e})) $. $\alpha(\beta)$ being a natural transformation follows from the fact that $\beta$ is one as well as the fact that $\alpha(\beta)$ uses $\id_{e}$ in its construction. This gives a function $\alpha:\Hom_{\HGA{A}}(\Part(E),F) \to \Hom_{\HGA{A}}(E,\Clus(F))$ where $\alpha([f]) = [\alpha(f)]$.\\

Given $[g]:E \to \Clus(F)$ and $[r_{f,o},r_{f,m}]$ being a realizer for $g$, We define $g':\Part(E) \to F$. $g'(e,r_{e}) = g(e)(p_{1} \cdot (r_{f,o} \cdot r_{e})))$ and for $t:(e,r_{e}) \to (e',r_{e'})$ we set $g'(t) = g(t)(p_{1} \cdot (r_{f,o} \cdot r_{e}),p_{1} \cdot (r_{f,o} \cdot r_{e'}))$. $g'$ is clearly a functor and is realized by $\lmbd{z}{(p_{0} \cdot (r_{f,o} \cdot z)) \cdot (p_{1} \cdot (r_{f,o} \cdot z))}$ for objects and $\lmbd{z}{(r_{f,m} \cdot (p_{1} \cdot z)) \cdot (p_{0} \cdot z)}$ for morphisms. We have  $\alpha(g')(e)(r_{e}) = g'(e,r_{e}) = g(e)(p_{1} \cdot (r_{f,o} \cdot r_{e}))$ and $\alpha(g')(e)(r_{e},r'_{e}) = g'(\id_{e}) = g(\id_{e})(p_{1} \cdot (r_{f,o} \cdot r_{e}),p_{1} \cdot (r_{f,o} \cdot r'_{e})) $ and for $t: e \to e'$ we have $\alpha(g')(t)(r_{e},r'_{e}) = g'(t:(e,r_{e}) \to (e',r'_{e})) = g(t)(p_{1} \cdot (r_{f,o} \cdot r_{e}),p_{1} \cdot (r_{f,o} \cdot r_{e'}))$. Now we show that $\gamma: \alpha(g') \cong g$ by exhibiting $\gamma(e)(r_{e},u_{e}) = g(\id_{e})(p_{1} \cdot (r_{f,o} \cdot r_{e'}),u_{e})$. Given $t:e \to e'$, we have $\alpha(g')(t)(r_{e},r_{e'}) = g(t)(p_{1} \cdot (r_{f,o} \cdot r_{e}),p_{1} \cdot (r_{f,o} \cdot r_{e'}))$ so the following square:

\[\begin{tikzcd}
	{\alpha(g')(e)(r_{e})} &&& {\alpha(g')(e')(r_{e'})} \\
	\\
	{g(e)(u_{e})} &&& {g(e')(u_{e'})}
	\arrow["{\alpha(g')(t)(r_{e},r_{e'}) }", from=1-1, to=1-4]
	\arrow["{\gamma(e)(r_{e},u_{e}) }"', from=1-1, to=3-1]
	\arrow["{\gamma(e)(r_{e'},u_{e'})} ", from=1-4, to=3-4]
	\arrow["{g(t)(u_{e},u_{e'}) }"', from=3-1, to=3-4]
\end{tikzcd}\]

is commutative by virtue how composition works in $\Clus(F)$. Thus $\alpha$ is surjective.\\

Suppose we have $f':\Part(E) \to F$ such that $\alpha(f') \cong g$, we need to show that $f' \cong g'$. Let $\beta_{0}:\alpha(f') \cong g$. We define  $\beta \in \Pi_{(e,r_{e}) \in \ob E} F(f'(e,r_{e}),g'(e,r_{e}))$ where $\beta(e,r_{e}) = \beta_{0}(e)(r_{e},p_{1} \cdot (r_{f,o} \cdot r_{e}))$. $\beta$ is realized by $\lmbd{z}{(r_{\beta} \cdot z) \cdot [z, p_{1} \cdot (r_{f,o} \cdot z)]}$ where $r_{\beta}$ is a realizer for $\beta$. Reminder that $\alpha(f')(e)(r_{e}) = f'(e,r_{e})$ and $g(e,p_{1} \cdot (r_{f,o} \cdot r_{e})) = g'(e,r_{e})$. Given $t:(e,r_{e}) \to (e',r_{e'})$, $f'(t) = \alpha(f')(t)(r_{e},r_{e'})$ and $g(t)(p_{1} \cdot (r_{f,o} \cdot r_{e}),p_{1} \cdot (r_{f,o} \cdot r_{e'})) = g'(t)$. Hence $\beta$ is a natural transformation by virtue of $\beta_{0}$ being one. Therefore $\alpha$ is injective.\\

Therefore, $\alpha$ exhibits $\Hom_{\HGA{A}}(\Part(E),F) \cong \Hom_{\HGA{A}}(E,\Clus(F))$.

\end{proof}

\begin{lem}
\label{clusadj1}
We have an equivalence of categories $e : \Hom_{\Cat}(\Part(G),H) \simeq  \Hom_{\Cat}(G,\Clus(H)) $.
\end{lem}
\begin{proof}
Given $f:\Part(G) \to H$, recall that we can construct a functor $e(f):G \to \Clus(H)$ where $e(f)(g):\nabla(\phi(g)) \to H$ where $e(f)(g)(r_{g}) = f(g,r_{g})$ and $e(f)(g)(r_{g},r'_{g}) = f(\id_{g}: (g,r_{g}) \to (g,r'_{g}))$ so that if $[r_{f,o},r_{f,m}]$ realizes $f$, then $r_{f,o}$ realizes the object part of $e(f)(g)$ and for a realizer for $\id_{g}$, $r_{i,g}$, $\lmbd{z}{r_{f,m} \cdot [z,r_{i,g}]}$ realizes morphism part of $e(f)(g)$. For a morphism $t:g \to g'$, we have $e(f)(t):\phi(g) \times \phi(g') \to \mor F$, where $e(f)(t)(r_{g},r_{g'}) = f(t:(g,r_{g}) \to (g',r_{g'}))$. $e(f)(t)$ is realized by $\lmbd{z}{r_{f,m} \cdot [z,r_{t}]}$ where $r_{t}$ realizes $t$. \\

If we have a natural isomorphism $j:f \cong f':\Part(G) \to H$, we can construct a natural isomorphism $e(j):e(f) \cong e(f')$ where $e(j)_{g}:\phi(g) \times \phi(g) \to H$ with $$e(j)_{g}(r_{g},r'_{g}) = f'(\id_{g}:(g,r_{g}) \to (g,r'_{g})) \circ j_{(g,r_{g})}$$ realized by $$\lmbd{z}{r_{comp,H} \cdot [r_{f',m} \cdot [z,r_{\id_{g}}], r_{j} \cdot (p_{0} \cdot z) ]}$$

Suppose we have a morphism $u:g \to g'$ in $G$, with $r_{g}, r'_{g} \in \phi(g)$ and $r_{g'},r'_{g'} \in \phi(g')$ then we have:

\begin{align*}
e(f')(t)(r'_{g},r_{g'}) \circ e(j)_{g}(r_{g},r'_{g}) &=   f'(t:(g,r'_{g}) \to (g',r_{g'})) \circ f'(\id_{g}:(g,r_{g}) \to (g,r'_{g})) \circ j_{(g,r_{g})}\\
&=  f'(t:(g,r_{g}) \to (g',r_{g'})) \circ j_{(g,r_{g})}\\
&= j_{(g',r_{g'})} \circ f(t:(g,r_{g}) \to (g',r_{g'}))\\
&=  j_{(g',r_{g'})} \circ f(\id_{g'}: (g',r'_{g'}) \to (g',r_{g'})) \circ f(t:(g,r_{g}) \to (g',r'_{g'}))\\
&= f'(\id_{g'}: (g',r'_{g'}) \to (g',r_{g'})) \circ j_{(g',r'_{g'})} \circ  f(t:(g,r_{g}) \to (g',r'_{g'}))\\
&= e(j)_{g'}(r'_{g'},r_{g'}) \circ e(f)(t)(r_{g},r'_{g'}) \\
\end{align*}

Hence $$e(f')(t) \circ e(j)_{g} = e(j)_{g'} \circ e(f)(t)$$ making $e(j)$ a natural isomorphism. If we have $\id_{f}:f \cong f:\Part(G) \to H$, then for $e(\id_{f})_{g}:\phi(g) \times \phi(g) \to H$ and $r_{g}, r'_{g} \in \phi(g)$ we have $$e(\id_{f})_{g}(r_{g}, r'_{g}) =  f(\id_{g}:(g,r_{g}) \to (g,r'_{g})) \circ (\id_{f})_{(g,r_{g})} = f(\id_{g}:(g,r_{g}) \to (g,r'_{g})) = e(f)(g)(r_{g},r'_{g})$$ Hence $e(\id_{f}) = \id_{e(f)}$. For $j:f \cong f'$, $j':f' \cong f''$ and $r_{g}, r'_{g}, r''_{g} \in \phi(g)$, we have
\begin{align*}
e(j')_{g}(r''_{g},r'_{g}) \circ e(j)_{g}(r'_{g},r_{g}) &= f''(\id_{g}:(g,r'_{g}) \to (g,r''_{g})) \circ j'_{(g,r'_{g})} \circ f'(\id_{g}:(g,r_{g}) \to (g,r'_{g})) \circ j_{(g,r_{g})}\\
&= f''(\id_{g}:(g,r'_{g}) \to (g,r''_{g}))  \circ f''(\id_{g}:(g,r_{g}) \to (g,r'_{g})) \circ j'_{(g,r_{g})} \circ j_{(g,r_{g})}\\
&= f''(\id_{g}:(g,r_{g}) \to (g,r''_{g})) \circ (j' \circ j)_{(g,r_{g})}\\
&= e(j' \circ j)(r''_{g}, r_{g})\\ 
\end{align*}
So $e(j') \circ e(j) = e(j' \circ j)$ making $e$ a functor.\\

For $f:G \to \Clus(H)$ we have a realizer $[r_{f,o},r_{f,m}]$. Using this, we construct a functor $e^{i}: \Hom_{\Cat}(G,\Clus(H)) \to  \Hom_{\Cat}(\Part(G),H)$. First we have $e^{i}(f):\Part(G) \to H$ where  $e^{i}(f)(g,r_{g}) = f(g)(p_{1} \cdot (r_{f,o} \cdot r_{g})))$ and for $t:(g,r_{g}) \to (g',r_{g'})$ we set $e^{i}(f)(t) = f(t)(p_{1} \cdot (r_{f,o} \cdot r_{g}),p_{1} \cdot (r_{f,o} \cdot r_{g'}))$. Furthermore $e^{i}(f)$ is realized by $\lmbd{z}{(p_{0} \cdot (r_{f,o} \cdot z)) \cdot (p_{1} \cdot (r_{f,o} \cdot z))}$ for objects and $\lmbd{z}{(r_{f,m} \cdot (p_{1} \cdot z)) \cdot (p_{0} \cdot z)}$ for morphisms. Given $k:f \cong f':G \to \Clus(H)$, we construct $e^{i}(k):e^{i}(f) \cong e^{i}(f')$ where $$e^{i}(k)_{(g,r_{g})} = k_{g}(p_{1} \cdot (r_{f,o} \cdot r_{g}), p_{1} \cdot (r_{f',o} \cdot r_{g}))$$ realized by $$\lmbd{z}{(r_{k} \cdot z) \cdot [p_{1} \cdot (r_{f,o} \cdot z), p_{1} \cdot (r_{f',o} \cdot z)] }$$ Suppose we have $t:(g,r_{g}) \to (g',r_{g'})$, then we have the following commutative square by the property of $k_{g}$ and $k_{g'}$:

\[\begin{tikzcd}
	{e^{i}(f)(g,r_{g})} &&& {e^{i}(f')(g',r_{g})} \\
	\\
	{e^{i}(f)(g,r_{g'})} &&& {e^{i}(f')(g',r_{g'})}
	\arrow["{e^{i}(k)_{(g,r_{g})} }", from=1-1, to=1-4]
	\arrow["{e^{i}(f)(t)}"', from=1-1, to=3-1]
	\arrow["{e^{i}(f')(t)} ", from=1-4, to=3-4]
	\arrow["{e^{i}(k)_{(g',r_{g'})}}"', from=3-1, to=3-4]
\end{tikzcd}\]

so $e^{i}(k)$ is natural. It is straightforward to show that $e^{i}$ is a functor.\\

So we need to exhibit $\alpha: e^{i} \circ e \cong \id_{\Hom_{\Cat}(\Part(G),H)}$ and $\beta: e \circ e^{i} \cong \id_{ \Hom_{\Cat}(G,\Clus(H))}$.\\

Suppose we have $f:\Part(G) \to H$, then we define $\alpha_{f}: e^{i} \circ e(f) \cong f$, where $$\alpha_{f,(g,r_{g})} = f(\id_{g}:(g,p_{1} \cdot (r_{e(f),o} \cdot r_{g}))  \to (g,r_{g}))$$ This works since $e^{i} \circ e(f)(g,r_{g}) = e(f)(g)(p_{1} \cdot (r_{e(f),o} \cdot r_{g})) = f(g,p_{1} \cdot (r_{e(f),o} \cdot r_{g}))$ and $\alpha_{f}$ is a natural isomorphism since it is defined using $f$ and $\id_{g}$. Furthermore $\alpha_{f}$ is realized by $$\lmbd{z}{s_{f,o} \cdot [[p_{1} \cdot (r_{e(f),o} \cdot z), z],r_{\id_{g}}]}$$ where $s_{f,o}$ is a realizer for the object part of $f$. Suppose we have $j:f \cong f':\Part(G) \to H$, then $$e^{i} \circ e(j)_{(g,r_{g})} = e(j)_{g}(p_{1} \cdot (r_{e(f),o} \cdot r_{g}), p_{1} \cdot (r_{e(f'),o} \cdot r_{g})) = f'(\id_{g}: (g,p_{1} \cdot (r_{e(f),o} \cdot r_{g}) ) \to (g, p_{1} \cdot (r_{e(f'),o} \cdot r_{g}))) \circ j_{(g,p_{1} \cdot (r_{e(f),o} \cdot r_{g}))}$$
So we have 

\begin{align*}
j_{(g,r_{g})} \circ &\alpha_{f,(g,r_{g})}  = j_{(g,r_{g})} \circ f(\id_{g}:(g,p_{1} \cdot (r_{e(f),o} \cdot r_{g}))  \to (g,r_{g}))\\
&= f'(\id_{g}:(g,p_{1} \cdot (r_{e(f),o} \cdot r_{g}))  \to (g,r_{g})) \circ j_{(g,p_{1} \cdot (r_{e(f),o} \cdot r_{g}))}\\
&= f'(\id_{g}: (g, p_{1} \cdot (r_{e(f'),o} \cdot r_{g})) \to (g,r_{g})) \circ f'(\id_{g}: (g,p_{1} \cdot (r_{e(f),o} \cdot r_{g}) ) \to (g, p_{1} \cdot (r_{e(f'),o} \cdot r_{g})) ) \circ j_{(g,p_{1} \cdot (r_{e(f),o} \cdot r_{g}))}\\
&= \alpha_{f',(g,r_{g})} \circ (e^{i} \circ e(j)_{(g,r_{g})})\\
\end{align*}
Hence $j \circ \alpha_{f} = \alpha_{f'} \circ (e^{i} \circ e(j))$ so we have $\alpha: e^{i} \circ e \cong \id_{\Hom_{\Cat}(\Part(G),H)}$.\\

Suppose we have $f: G \to \Clus(H)$, we define $\beta_{f}: e \circ e^{i}(f) \cong f$, so we have $\beta_{f,g}:\phi(g)\times U_{g} \to \mor H$ for $f(g):\Delta(U_{g}) \to H$ and for $u_{g} \in U_{g}$ and $r_{g} \in \phi(g)$ we have $\beta_{f,g}(r_{g},u_{g}) = f(\id_{g})(p_{1} \cdot (r_{f,o} \cdot r_{g}), u_{g})$. This works since $e \circ e^{i}(f)(g)_{\ob} = [r_{g} \in \phi(g) \mapsto  e^{i}(f)(g,r_{g})]$ and $ e^{i}(f)(g,r_{g}) = f(g)(p_{1} \cdot (r_{f,o} \cdot r_{g}))$. $\beta_{f,g}$ is realized by $$\lmbd{z}{(r_{f,\mor} \cdot r_{\id_{g}}) \cdot [p_{1} \cdot (r_{f,o} \cdot  (p_{0} \cdot z)),   p_{1} \cdot z]}$$ Suppose we have $t:f \cong f': G \to \Clus(H)$, we need to show that $t \circ \beta_{f} = \beta_{f'} \circ (e \circ e^{i}(t)) $. For $g \in \ob G$ with $f(g):\Delta(U_{g}) \to H$ and $f'(g):\Delta(U'_{g}) \to H$ and $u_{g} \in U_{g}$, $u'_{g} \in U'_{g}$, and $r_{g} \in \phi(g)$ we have a commutative diagram derived from $t$:

\[\begin{tikzcd}
	{f(g)(p_{1} \cdot (r_{f,o} \cdot r_{g}))} &&&& {f'(g)(p_{1} \cdot (r_{f',o} \cdot r_{g}))} \\
	\\
	{f(g)(u_{g})} &&&& {f'(g)(u'_{g})}
	\arrow["{{t_{g}(p_{1} \cdot (r_{f,o} \cdot r_{g}),p_{1} \cdot (r_{f',o} \cdot r_{g})) }}", from=1-1, to=1-5]
	\arrow["{{f(g)(p_{1} \cdot (r_{f,o} \cdot r_{g}), u'_{g}) }}"', from=1-1, to=3-1]
	\arrow["{f'(g)(p_{1} \cdot (r_{f',o} \cdot r_{g}), u'_{g})}", from=1-5, to=3-5]
	\arrow["{{t_{g}(u_{g},u'_{g})}}"', from=3-1, to=3-5]
\end{tikzcd}\]

giving us the following chain of equalities:

\begin{align*}
t_{g}(u_{g},u'_{g}) \circ \beta_{f,g}(r_{g},u_{g}) &= t_{g}(u_{g},u'_{g}) \circ f(\id_{g})(p_{1} \cdot (r_{f,o} \cdot r_{g}), u_{g})\\
&= t_{g}(u_{g},u'_{g}) \circ f(g)(p_{1} \cdot (r_{f,o} \cdot r_{g}), u_{g})\\
&= f'(g)(p_{1} \cdot (r_{f',o} \cdot r_{g}), u'_{g}) \circ t_{g}(p_{1} \cdot (r_{f,o} \cdot r_{g}), p_{1} \cdot (r_{f',o} \cdot r_{g}))\\
&= f'(\id_{g})(p_{1} \cdot (r_{f',o} \cdot r_{g}), u'_{g}) \circ t_{g}(p_{1} \cdot (r_{f,o} \cdot r_{g}), p_{1} \cdot (r_{f',o} \cdot r_{g}))\\
&= \beta_{f',g}(r_{g},u'_{g}) \circ (e \circ e^{i}(t))_{g}(r_{g},r_{g})\\
\end{align*}

This proves that $t \circ \beta_{f} = \beta_{f'} \circ (e \circ e^{i}(t)) $ exhibiting $\beta:e \circ e^{i} \cong \id_{ \Hom_{\Cat}(G,\Clus(H))}$. Therefore, $e$ is an equivalence.

\end{proof}

\begin{lem}
\label{pbicolim}
$\pGrA{A}$ has finite bicolimits.

\end{lem}
\begin{proof}

Suppose we have finite bilimit diagrams  $W:D^{\op} \to \Cat$ and $J:D \to \pGrA{A}$, since $\pGrA{A}$ is a full subcategory of $\GrpA{A}$ and $\GrpA{A}$ has finite bicolimits, there exists a groupoid assembly $W \otimes J \in \ob \GrpA{A}$ such that we have an equivalence of categories  $$\cC(W \otimes J,X) \simeq [D^{\op} \to \Cat](W, \cC(J(-), X))$$ for each $X \in \ob \GrpA{A}$. We now show that $\Part(W \otimes J)$ is the bicolimit of $W$ and $J$ in $\pGrA{A}$ by exhibiting the following chain of equivalences making use of the fact that $\pGrA{A}$ is a full subcategory of $\GrpA{A}$:

\begin{align*}
\Hom_{\Cat}(\Part(W \otimes J),X) &\simeq \Hom_{\Cat}(W \otimes J,\Clus(X))\\
&\simeq [D^{\op} \to \Cat](W, \Hom_{\Cat}(J(-), \Clus(X)))\\
&\simeq [D^{\op} \to \Cat](W, \Hom_{\Cat}(\Part(J(-)), X))\\
&\simeq [D^{\op} \to \Cat](W, \Hom_{\Cat}(J(-), X))\\
\end{align*}

The final line follows from the fact that each $J(d)$ is a partitioned groupoid assembly and we have a groupoid equivalence $\atom_{J(d)}:\Part(J(d)) \to J(d)$.

\end{proof}

\begin{rmk}

It is also possible to show that since $\pGrA{A}$ is a bicoreflective subcategory of $\GrpA{A}$, $\Part(-)$ creates bicolimits in $\pGrA{A}$ so $\Part(W \otimes J) \simeq (W \otimes J)$.

\end{rmk}

\begin{lem}

$\Clus(-)$ extends to an endofunctor on $\GrpA{A}$ where for $F \cong F':G \to H$, we have $\Clus(F) \cong \Clus(F')$.

\end{lem}
\begin{proof}

Suppose $F:G \to H$, We define $\Clus(F)$ to take functors $f: \nabla(U) \to G$ to $F \circ f: \nabla(U) \to H$ and takes morphisms $J: U_{*} \times U'_{*} \to \mor G$ to $F_{1} \circ J: U_{*} \times U'_{*} \to \mor H$ where $F_{1}$ is the morphism part of $F$. It it not only obvious that $\Clus(F)$ is realized but also that $\Clus(-)$ is a functor.\\

If we have $F \cong F':G \to H$ exhibited by $\alpha$, then we have $\alpha':\Clus(F) \cong \Clus(F')$ where for $b:\nabla(U) \to G$, we have $\alpha'(b):U \times U \to \mor H$ where $\alpha'(b)(u,u') = F'(b(u,u')) \circ \alpha(b(u))$. Realizer data is straightforward. Therefore, $\Clus(-):\GrpA{A} \to \GrpA{A}$ is an endofunctor preserving equivalent functors.

\end{proof}

\begin{cor}
$\Clus(-):\GrpA{A} \to \GrpA{A}$ extends to a functor $$\Clus(-):\HGA{A} \to \HGA{A}$$ where $\Clus([f]) = [\Clus(f)]$.

\end{cor}

\begin{cor}

$\Part(-):\HGA{A} \to \HGA{A}$ has a right adjoint $\Clus(-):\HGA{A} \to \HGA{A}$.

\end{cor}




\subsection{0-types of Partitioned Assemblies}
In this section, we prove that the homotopy category of $0$-types of $\pGrA{A}$ is $\RT{A}$, the realizability topos on $A$.\\


\begin{defn}
We define the category $\PEq{A}$ whose objects are pairs $(X, \sim_{X})$ where $\sim_{X}:X \times X \to \pow{A}$ is a partial equivalence relation and the morphisms $f: (X,\sim_{X}) \to (Y,\sim_{Y})$ are functions $f: X \to Y$ such that there is a realizer for $ \forall x,x' \in X (x \sim_{X} x' \leq f(x) \sim_{Y} f(x') )$.

\end{defn}

Similar to $\GrpA{A}$, $\PEq{A}$ also has a notion of regular equivalence:

\begin{defn}
\label{regeq}
Suppose we have a morphism $f: (X,\sim_{X}) \to (Y,\sim_{Y})$ in $\PEq{A}$, we call $f$ a {\bf regular equivalence} iff it satisfies the following conditions:

\begin{enumerate}

\item The proposition $\forall y \in Y(y \sim_{Y} y \imply \exists x \in X (x \sim_{X} x \land f(x) \sim_{Y} y) )$ is realized. (f is surjective)

\item The proposition $\forall x, x' \in X (x \sim_{X} x \land x' \sim_{X} x' \land f(x) \sim_{Y} f(x') \imply x \sim_{X} x' ) $ is realized. (f is injective )

\end{enumerate}

This is referred to as weak equivalence in \cite{frey2022categories}.

\end{defn}

\begin{defn}

We define the full subcategory $\EqRA{A} \mono \GrpA{A}$ whose objects are the 0-types which are precisely the groupoid assemblies for which there exists at most one morphism between objects. 

\end{defn}

\begin{lem}

$\EqRA{A}$ is equivalent to the full subcategory of $\PEq{A}$ whose objects are the pairs $(X,\sim_{X})$ such that for $x \in X$, $x \sim_{X} x$ is inhabited. We call such $(X,\sim_{X})$ an {\bf equivalence relation}. Furthermore the notions of regular equivalence coincides in both categories.

\end{lem}
\begin{proof}
Given a $0$-type $G$, we can construct the pair $(\ob G, \sim_{G})$ where $g \sim_{G} g' = \phi_{\mor G}(t)$ where $t:g \to g'$ in $G$. For a functor of $0$-types $F:G \to H$, $f:(\ob G, \sim_{G}) \to (\ob H, \sim_{H})$ is just the object part of $F$ which is realized since the morphism part of $F$ is realized. Let $| - |$ be the functor representing this construction\\

For an equivalence relation $(X,\sim_{X})$, we can construct a $0$-type $X^{*}$ where the object part is $(X,\phi_{X})$ where $\phi_{X}(x) = x \sim_{X} x$ and the morphism part is $(X', \phi_{X'})$ where $X' = \{(x,x')| x \sim_{X} x' \text{ is inhabited}\}$ and $\phi_{X}(x,x') = x \sim_{X} x'$. $X^{*}$ being a groupoid assembly is straightforward. Given a morphism $f:(X,\sim_{X}) \to (Y,\sim_{Y})$, we have $f^{*}:X^{*} \to Y^{*}$ where $f^{*}(x) = f(x)$ and $f^{*}(x,x') = (f(x),f(x'))$. Similarly, the fact that this is realized is straightforward. $(-)^{*}$ be the functor representing this construction\\

The functor $| (-)^{*} |$ is precisely the identity on equivalence relations and morphisms. The functor $(| - |)^{*}$ is plainly isomorphic to $\Id_{\EqRA{A}}$ since on the level of set the functors are isomorphic and for $0$-type $G$, we have $\forall x \in \ob G(\phi_{\ob G}(x) \bimp x \sim_{G} x)$ by virtue of the fact that $\dom$, $\cod$, and $\id_{-}$ are realized for $G$ and $\mor G$  and $\mor (| G |)^{*}$ are isomorphic. Hence $\EqRA{A}$ is equivalent to the full subcategory of $\PEq{A}$ of equivalence relations.\\

Since an functor between $0$-types, $F:G \to H$ is faithful we can simply the definition of regular equivalences between $0$-types by applying the definitions of $\sim_{G}$ and $\sim_{H}$ :

\begin{enumerate}

\item The proposition $\forall y \in \ob H (\phi_{\ob H}(y) \imply \exists x \in \ob G(\phi_{\ob G}(x) \land  F(x) \sim_{G} y  ) )$ is realized. 

\item The proposition $\forall x, x' \in \ob G (\phi_{\ob G}(x) \land \phi_{\ob G}(x') \land F(x) \sim_{H} F(x') \imply x \sim_{H} x' ) $ is realized.

\end{enumerate}

This definition is essentially the same as the one for $\PEq{A}$. Thus the definitions of regular equivalence in $\EqRA{A}$ and $\PEq{A}$ on equivalence relations coincide.

\end{proof}

\begin{lem}

We have an idempotent comonad $ER: \PEq{A} \to \PEq{A}$ such that the coreflective subcategory of $ER$ is the full subcategory of equivalence relations and the counit $\epsilon: ER \To \Id_{\PEq{A}}$ is a natural regular equivalence.

\end{lem}
\begin{proof}
Suppose we have a partial equivalence relation $(X, \sim_{X})$, We define $ER(X, \sim_{X}) = (X', \sim_{X'})$ where $X' = \{x \in X|  x \sim_{X} x\text{ is inhabited}\}$ and $x \sim_{X'} x' = x \sim_{X} x'$. Given a morphism $f:(X,\sim_{X}) \to (Y,\sim_{Y})$ if $ x' \sim_{X} x$ is inhabited then $f(x') \sim_{X} f(x)$ is as well, so we define $ER(f)(x) = f(x)$. $ER$ is clearly a functor. \\

We construct $\epsilon_{(X, \sim_{X})}:(X',\sim_{X'}) \to (X,\sim_{X})$ where $\epsilon_{(X, \sim_{X})}(x) = x$. $\epsilon$ is clearly a natural transformation. Given $x \in X$, if $x \sim_{Y} x$ is inhabited then $x \in X'$ and $x \sim_{X'} x$ is inhabited, thus surjectivity of $\epsilon_{(X, \sim_{X})}$ is realized by $\lmbd{z}{[z,z]}$. Injectivity is straightforward since for $x, x' \in X'$, $x \sim_{X'} x' = x \sim_{X} x'$. Thus $\epsilon$ is a natural regular equivalence. \\

Observe that $ER$ is the identity on equivalence relations so comultiplication is the identity $\Id_{ER}$. Thus $ER$ is an idempotent comonad. Furthermore, this also means that coreflective subcategory induced by $ER$ is the full subcategory of equivalence relations.

\end{proof}

Define $\Ho(\PEq{A})$ and $\Ho(\EqRA{A})$ to be the localization of $\PEq{A}$ and $\EqRA{A}$ on the regular equivalences. 

\begin{lem}

$\RT{A} \cong \Ho(\PEq{A}) \cong \Ho(\EqRA{A})$

\end{lem}
\begin{proof}
The second equivalence follows from the fact that $\epsilon:ER \To \Id_{\PEq{A}}$ is a natural regular equivalence and $\EqRA{A}$ is the coreflective subcategory of $\PEq{A}$ induced by $ER$. The first equivalence is a consequence of Theorem 5.5 of \cite{frey2022categories}.

\end{proof}

\begin{lem}

If $G$ is a $0$-type, then so is $\Part(G)$.

\end{lem}
\begin{proof}

The objects of $\Part(G)$ are pairs $(g,r_{g})$ where $r_{g}$ realizes $g \in \ob G$. We have that $$\Part(G)((g,r_{g}),(g',r_{g'})) = G(g,g')$$ Hence $\Part(G)$ is a $0$-type.

\end{proof}

\begin{cor}
\label{zerotype}

The homotopy category of $0$-types of $\pGrA{A}$ is the realizability topos on $A$, $\RT{A}$.

\end{cor}
\begin{proof}
This is a consequence of $\Ho(\pGrA{A}) \cong \EGA{A}$ and the fact that $\Part(-):\EGA{A} \to \Ho(\pGrA{A})$ preserves $0$-types.

\end{proof}

\begin{rmk}

Note that for nontrivial partial combinatory algebra $A$, the groupoid assembly $A_{diff}$ from Remark \ref{noheexpl} is a $0$-type; hence $\Part(A_{diff})$ is one as well and $\atom_{A_{diff}}:\Part(A_{diff}) \to A_{diff}$ is a regular equivalence between $0$-types. Since $\atom_{A_{diff}}$ is not a regular equivalence, it is not inverted in $\Ho(\GrpA{A})$ and so the homotopy category of $0$ types in $\Ho(\GrpA{A})$ is not equivalent to $\RT{A}$.

\end{rmk}


\section{Conclusion and Future Problems}

Here we express the key results of this thesis:

\begin{thm}

The following holds for $\GrpA{A}$:

\begin{itemize}

\item $\GrpA{A}$ equipped with the algebraic weak factorization system $(\Fib,R_{\Fib}, L_{\Fib})$ and the class of normal isofibrations is a $\pi$-tribe and thus a model of type theory. (Corollary \ref{pitribe})\\

\item $\GrpA{A}$ has a model structure where the weak equivalences are the groupoid equivalences, the fibrations are the normal isofibrations, and the cofibrations are the decidable monomorphisms on objects followed by the strong deformations sections. (Corollary \ref{mdlstr})\\

\item $\GrpA{A}$ has $W$-types with reductions and for a map of fibrations $J:U \to V$ over a groupoid assembly $I$, we can construct an algebraic weak factorization system $([\SOA], R_{\SOA},L_{\SOA})$ whose $R$-maps are precisely those maps that have a functorial right lifting property against $J$ as described in Lemma \ref{liftI}. (Lemma \ref{wtyper} and Corollary \ref{soa})\\

\item $\GrpA{A}$ has an impredicative univalent universe of modest assemblies and a univalent universe of assemblies. (Lemmas \ref{disuni} and \ref{disuniv} and Corollary \ref{impreduniv})\\

\end{itemize}

\end{thm}

\begin{thm}

The following holds for $\GrpA{A}$ and $\pGrA{A}$:

\begin{itemize}

\item $\GrpA{A}$ and $\pGrA{A}$ have finite bilimits and bicolimits as well as homotopy exponents. (Lemmas \ref{bilim} and \ref{hexpon} and Lemmas \ref{pbilim}, \ref{pbicolim} and \ref{phexpon})\\

\item There is a biadjunction between $\pGrA{A}$ and $\GrpA{A}$ which induces an adjunction $\incl \adj \Part(-)$ between $\Ho(\pGrA{A})$ and $\Ho(\GrpA{A})$ making $\Ho(\pGrA{A})$  a coreflective subcategory of $\Ho(\GrpA{A})$. (Lemmas \ref{biadjin} and  \ref{adjin}) \\

\item There is a biadjunction on $\GrpA{A}$ which induces a adjunction $\Part(-) \adj \Clus(-)$ on $\Ho(\GrpA{A})$. (Lemmas \ref{clusadj1} and \ref{clusadj2})\\

\item The homotopy category of the $0$-types of $\pGrA{A}$ is $\RT{A}$. (Corollary \ref{zerotype})\\

\end{itemize}

\end{thm}

The first theorem exhibits $\GrpA{A}$ as a model of type theory with $W$-types, localization modalities including truncation modalities, a univalent universe of assemblies and an impredicative univalent universe of modest assemblies. $\GrpA{A}$ also has a model structure whose fibrations are the normal isofibrations and weak equivalences are groupoid equivalences. There still is a question of whether every reflective subcategory induced by a localization modality is a model of type theory, however we will leave that as the following conjecture:

\begin{conj}

Every reflective subcategory induced by a localization modality on $\GrpA{A}$ is a model of type theory.

\end{conj}

The second serves as a prelude to uncovering the potential homotopical properties of $\pGrA{A}$. We believe that $\pGrA{A}$ serves as the bicategorical analogue of the elementary $(\infty,1)$-topos. First we seek to expand our definitions of $0$-types and $-1$-types to a proper notion of $n$-truncated maps in $\GrpA{A}$ so we express the following conjecture:

\begin{conj}
The following holds for $\pGrA{A}$:

\begin{itemize}

\item Bipullbacks along a morphism $f:X \to Y$ induces a bifunctor $f_{h}^{*}:\pGrA{A}/Y \to \pGrA{A}/X$ which has a right and left biadjoint. We refer to the right biadjoint as a {\bf bipushforward}.\\
\item There exists a morphism $\top:1 \to \Omega$ such that it induces a biequivalence between the hom-category $\Hom_{\Cat}(X,\Omega)$ and the category of $-1$-truncated morphism over $X$. We call this a {\bf sub-object biclassifier}.\\

\item For every Grothendieck universe $V$, there exists a morphism $\pr{V}:U^{*}[V] \to U[V]$ which induces a biequivalence between the hom-category $\Hom_{\Cat}(X,U[V])$ and the category of $V$-fibred $0$-truncated morphisms over $X$. We refer to this as a {\bf object biclassifier}.\\

\item Furthermore, we have a biclassifier for all modest $0$-truncated morphisms which class of morphisms is closed under bipushforward along any morphism.\\

\end{itemize}

\end{conj}

Adding additional context to the above conjecture, the notion of modest type we use are types $X$ such that the morphism $$- \circ !_{\nabla(2)}:X \to (\nabla(2) \to X)$$ is an equivalence. I am not quite sure what the proper notion of ``modest fibration" or ``modest morphism" is but I do know that it must coincide in the of a map into the terminal object with the above notion of modest type. In addition we require the two object biclassifiers mentioned in the conjecture to satisfy some notion of univalence which should state that the fibration of paths of the biclassifier $U$ are equivalent to the fibration of equivalences between objects in $U$. When all of these requirements are satisfied, I believe this should give a proper example of what should be called a {\bf realizability elementary bitopos} where an {\bf elementary bitopos} $\cC$ only needs to satisfy the following conditions:

\begin{itemize}

\item $\cC$ has finite bilimits and bicolimits.\\
\item $\cC$ has homotopy exponents.\\
\item The homotopy category of the $0$-types of $\cC$ is an elementary topos.\\
\item Bipullbacks of $\cC$ have a right and left biadjoint.\\
\item$\cC$ has a sub-object biclassifier.\\

\item $\cC$ has enough univalent object biclassifiers for $0$-truncated maps.\\

\end{itemize}

$\pGrA{A}$ already satisfies the first three conditions so we only need to check the rest of the conditions. This will require fleshing out the definitions of $n$-truncated morphism as well as that of ``modest morphisms". On a related note, there is also the question of whether $\pGrA{A}$ has the bicategorical analogue of $W$-types with reductions. This is important since this includes $W$-types as well as localization modalities and weak factorization systems  which might give a way of characterizing  $n$-truncated morphisms and  ``modest morphisms". We will refer to these as {\bf bicategorical W-types with reductions}. There is an argument to be made that the existence of bicategorical W-types with reductions should be added to the definition of elementary bitopos since such constructions include localization modalities and $W$-types like the natural numbers object. We leave that as a conjecture:

\begin{conj}
$\pGrA{A}$ has bicategorical $W$-types with reductions and is an elementary bitopos.

\end{conj}

We also seek to learn more about $\Clus(-)$ and $\Part(-)$. First we note that the construction of $\Clus(-)$ resembles branching monads like the powerset operator. In fact I believe that the best first step to understanding $\Clus(-)$ is to construct a branching monad $\Clus_{b}(-)$ on $\Asm(A)$ where for an assembly $(X,\phi)$, either $$\Clus_{b}(X,\phi) = (\{f:(U, [u \mapsto \{u\}])  \to (X,\phi)|  f\text{ is realized and }U \subseteq A\},[f \mapsto \{r_{f}| r_{f}\text{ realizes }f\} ])$$ or
\begin{align*}\Clus_{b}(X,\phi) = (\{f:(U, [u \mapsto \{u\}])  \to (X,\phi)|  f\text{ is realized and non-empty }U \subseteq A\},[f \mapsto \{[r_{f},u] | r_{f}\text{ realizes }f\text{ and }u \in U_{f}\} ])\end{align*}

This will give $\Clus(-)$ an analogue in $\Asm(A)$ as $\Part(-)$ has with the subcategories of partitioned assemblies. Of course a natural question is what is the relationship of the category of algebras of $\Clus_{b}$ to $\RT{A}$? With regard to $\Part(-)$ and $\pGrA{A}$, I believe that a good first step is considering the relationship between $\pAsm(A)$ and $\pGrA{A}$. First I believe it should be possible to show that $\GrpA{A}$ is the $2$-reg completion of $\Asm(A)$, but in light of the previous conjecture stating that $\pGrA{A}$ is an elementary bitopos, one might consider if $\pGrA{A}$ is some bicategorical or $2$-categorical ex/lex completion of $\pAsm(A)$. In fact the definition of partitioned groupoid assemblies seems to be similar to the explicit construction of the object of the ex/lex completion of $\pAsm(A)$. Finally, one ought to    consider how $\Part(-)$ is related to the reg/lex-completion as the unit of the biadjoint $\atom_{-}: \Part(-) \to \Id(-)$ makes use of the fact that every assembly is covered by a partitioned assembly.\\

Finally we discuss the potential insights that come from this thesis to the model of type theory on cubical (and possibly simplicial) assemblies and on the potential realizability $(\infty,1)$-topos. First we conjecture that the construction of $W$-types with reductions and the small object argument in $\GrpA{A}$ can be applied to cubical assemblies. From there we construct a fibrant algebraic weak factorization system of uniform normal fibrations in cubical assemblies. Uniform normal fibrations are discussed in \cite{awodey2016cubicalmodelhomotopytype}. This will be an important first step in constructing a model of type theory for cubical assemblies whether we use $\pi$-tribes or the model of type theory in \cite{GAMBINO_2021}. We should also be able construct models of type theory on reflexive subcategories induced by nullification modalities. In addition, once we begin to construct a univalent universe of modest cubical assemblies, using the idea in Lemma \ref{nulldfib}, we can show that said universe is impredicative.\\

Of course, we should also be able to construct a cofibrant algebraic weak factorization system and a potential model structure. \\

As for the potential realizability $(\infty,1)$-topos, first we ought to isolate the subcategory of fibrant cubical assemblies. The question is whether one can construct a analogous version of $\Part(-)$. Recall that for a groupoid assembly $X$, $\Part(X)$ makes the object part of $X$ into a partitioned assembly. For a cubical assembly $X$, there is an assembly $X_{n}$ for every $n$ comprising $X$. There is an interesting question about whether there is some $\Part_{n}(-)$ such that for a cubical assembly $X$, $\Part_{n}(X)$ makes the first $n + 1$ layers into partitioned assemblies. The presumed idea behind each $\Part_{n}$ is that the full subcategory of $n + 1$ types of the reflective subcategory of $\Part_{n}$ is the $n + 1$ topos analogue of the $(\infty,1)$-topos. If each $\Part_{n}$ exists, it is conjectured that the cubical assembly analogue of $\Part(-)$, $\Part_{\infty}(-)$ is the limit of the following chain of ``morphisms":

\[\begin{tikzcd}
	\dotsb && {\Part_{n}(-)} && \dotsb && {\Part_{1}(-)} && {\Part_{0}(-)} \\
	\\
	&&&& {\Part_{\infty}(-)}
	\arrow[from=1-1, to=1-3]
	\arrow[from=1-3, to=1-5]
	\arrow[from=1-5, to=1-7]
	\arrow[from=1-7, to=1-9]
	\arrow[from=3-5, to=1-1]
	\arrow[from=3-5, to=1-3]
	\arrow[from=3-5, to=1-5]
	\arrow[from=3-5, to=1-7]
	\arrow[from=3-5, to=1-9]
\end{tikzcd}\]

where $\Part_{\infty}(X)$ makes every $X_{n}$ a partitioned assembly. Finally, it is believed that the full subcategory of fibrant cubical assemblies $X$ whose $X_{n}$ is partitioned for every $n$ will result in the realizability $(\infty,1)$-topos. However, it is probably more likely that the subcategory corresponding to $\Part_{1}$ will give us the realizability $(\infty,1)$-topos since the homotopy category of its subcategory of $0$-types will likely be the realizability topos on $A$. This would then beg the question of whether $\Part_{n}$ even gives elementary $(\infty,1)$-toposes.

\printbibliography

\end{document}
